\documentclass[numbers=enddot,12pt,final,onecolumn,notitlepage]{scrartcl}%
\usepackage[headsepline,footsepline,manualmark]{scrlayer-scrpage}
\usepackage[all,cmtip]{xy}
\usepackage{amsfonts}
\usepackage{amssymb}
\usepackage{amsmath}
\usepackage{amsthm}
\usepackage{framed}
\usepackage{comment}
\usepackage{color}
\usepackage[breaklinks=true]{hyperref}
\usepackage[sc]{mathpazo}
\usepackage[T1]{fontenc}
\usepackage{needspace}
\usepackage{tabls}
\providecommand{\U}[1]{\protect\rule{.1in}{.1in}}
\newcounter{exer}

\theoremstyle{definition}
\newtheorem{theo}{Theorem}[section]
\newenvironment{theorem}[1][]
{\begin{theo}[#1]\begin{leftbar}}
{\end{leftbar}\end{theo}}
\newtheorem{lem}[theo]{Lemma}
\newenvironment{lemma}[1][]
{\begin{lem}[#1]\begin{leftbar}}
{\end{leftbar}\end{lem}}
\newtheorem{prop}[theo]{Proposition}
\newenvironment{proposition}[1][]
{\begin{prop}[#1]\begin{leftbar}}
{\end{leftbar}\end{prop}}
\newtheorem{defi}[theo]{Definition}
\newenvironment{definition}[1][]
{\begin{defi}[#1]\begin{leftbar}}
{\end{leftbar}\end{defi}}
\newtheorem{remk}[theo]{Remark}
\newenvironment{remark}[1][]
{\begin{remk}[#1]\begin{leftbar}}
{\end{leftbar}\end{remk}}
\newtheorem{coro}[theo]{Corollary}
\newenvironment{corollary}[1][]
{\begin{coro}[#1]\begin{leftbar}}
{\end{leftbar}\end{coro}}
\newtheorem{conv}[theo]{Convention}
\newenvironment{convention}[1][]
{\begin{conv}[#1]\begin{leftbar}}
{\end{leftbar}\end{conv}}
\newtheorem{quest}[theo]{Question}
\newenvironment{question}[1][]
{\begin{quest}[#1]\begin{leftbar}}
{\end{leftbar}\end{quest}}
\newtheorem{warn}[theo]{Warning}

\newtheorem{conj}[theo]{Conjecture}

\newtheorem{exam}[theo]{Example}
\newenvironment{example}[1][]
{\begin{exam}[#1]\begin{leftbar}}
{\end{leftbar}\end{exam}}
\newtheorem{exmp}[exer]{Exercise}

\newenvironment{statement}{\begin{quote}}{\end{quote}}
\newenvironment{fineprint}{\begin{small}}{\end{small}}
\newcommand{\arinj}{\ar@{_{(}->}}
\newcommand{\arinjrev}{\ar@{^{(}->}}
\newcommand{\arsurj}{\ar@{->>}}
\newcommand{\arelem}{\ar@{|->}}
\newcommand{\arback}{\ar@{<-}}

\newcommand{\kk}{\mathbf{k}}

\let\sumnonlimits\sum
\let\prodnonlimits\prod
\let\cupnonlimits\bigcup
\let\capnonlimits\bigcap
\let\oplusnonlimits\bigoplus
\let\otimesnonlimits\bigotimes
\let\sqcupnonlimits\bigsqcup
\renewcommand{\sum}{\sumnonlimits\limits}
\renewcommand{\prod}{\prodnonlimits\limits}
\renewcommand{\bigcup}{\cupnonlimits\limits}
\renewcommand{\bigcap}{\capnonlimits\limits}
\renewcommand{\bigoplus}{\oplusnonlimits\limits}
\renewcommand{\bigotimes}{\otimesnonlimits\limits}
\renewcommand{\bigsqcup}{\sqcupnonlimits\limits}
\newenvironment{noncompile}{}{}
\excludecomment{verlong}
\includecomment{vershort}
\excludecomment{noncompile}

\ihead{A Solomon Mackey formula for graded bialgebras, version 24 July 2026}
\ohead{page \thepage}
\begin{document}

\title{A Solomon Mackey formula for graded bialgebras}
\author{Darij Grinberg}
\date{corrected version,
24 July 2026}
\maketitle

\begin{abstract}
\textbf{Abstract.} Given a graded bialgebra $H$, we let $\Delta^{\left[
k\right]  }:H\rightarrow H^{\otimes k}$ and $m^{\left[  k\right]  }:H^{\otimes
k}\rightarrow H$ be its iterated (co)multiplications for all $k\in\mathbb{N}$.
For any $k$-tuple $\alpha=\left(  \alpha_{1},\alpha_{2},\ldots,\alpha
_{k}\right)  \in\mathbb{N}^{k}$ of nonnegative integers, and any permutation
$\sigma$ of $\left\{  1,2,\ldots,k\right\}  $, we consider the map
$p_{\alpha,\sigma}:=m^{\left[  k\right]  }\circ P_{\alpha}\circ\sigma
^{-1}\circ\Delta^{\left[  k\right]  }:H\rightarrow H$, where $P_{\alpha}$
denotes the projection of $H^{\otimes k}$ onto its multigraded component
$H_{\alpha_{1}}\otimes H_{\alpha_{2}}\otimes\cdots\otimes H_{\alpha_{k}}$, and
where $\sigma^{-1}:H^{\otimes k}\rightarrow H^{\otimes k}$ permutes the tensor factors.

We prove formulas for the composition $p_{\alpha,\sigma}\circ p_{\beta,\tau}$
and the convolution $p_{\alpha,\sigma}\star p_{\beta,\tau}$ of two such maps.
When $H$ is cocommutative, these generalize Patras's 1994 results (which, in
turn, generalize Solomon's Mackey formula).

We also construct a combinatorial Hopf algebra $\operatorname*{PNSym}$
(\textquotedblleft permuted noncommutative symmetric
functions\textquotedblright) that governs the maps $p_{\alpha,\sigma}$ for
arbitrary connected graded bialgebras $H$ in the same way as the well-known
$\operatorname*{NSym}$ governs them in the cocommutative case. We end by
outlining an application to checking identities for connected graded Hopf
algebras. \medskip

 \medskip


\end{abstract}

This is a preliminary report on a project. Its goal is to classify the
identities that hold between the natural ($\mathbf{k}$-linear) operations on
the category of graded $\mathbf{k}$-bialgebras. The following approach is
likely improvable (being entirely focused on unary operations $H\rightarrow
H$, despite the multi-input-multi-output operations $H^{\otimes k}\rightarrow
H^{\otimes\ell}$ being probably a more natural object of study) and rather
inchoate (the proofs lacking in elegance and readability), but the main result
(Theorem \ref{thm.sol-mac-1}) appears worthy of dissemination and -- to my
great surprise -- new. Even more surprisingly, a combinatorial Hopf algebra
(named $\operatorname*{PNSym}$ for \textquotedblleft permuted noncommutative
symmetric functions\textquotedblright) emerges from this study which, too,
seems to have hitherto escaped the eyes of the algebraic combinatorics
community. Thus I hope that this report will be of some use in the time before
the right proofs are found and written up (hopefully without requiring a
revision of the respective statements).

I would like to thank the organizers of the CATMI 2023 workshop in Bergen
(Gunnar Fl\o ystad, H\aa kon Gylterud, Hans Zanna Munthe-Kaas, Uwe Wolter) for
the occasion to present these ideas, and Marcelo Aguiar, Amy Pang and Sarah
Brauner for stimulating conversations. Two referees have furthermore given
useful and constructive feedback that has improved the paper. In July 2026,
GPT-5.6 Sol was used to correct some mistakes (probably too late for the
proceedings volume), in particular weakening Theorem \ref{thm.pas-lin-ind2}
and replacing a wrong Theorem 2.10 by what is now Remark \ref{rmk.splitting}.

\subsection{Introduction}

On any bialgebra $H$, we can define the \textit{Adams operations} (also known
as \textit{characteristic operations} or \textit{dilations})
$\operatorname*{id}\nolimits^{\star k}=m^{\left[  k\right]  }\circ
\Delta^{\left[  k\right]  }$ (where $\Delta^{\left[  k\right]  }:H\rightarrow
H^{\otimes k}$ is an iterated comultiplication, and $m^{\left[  k\right]
}:H^{\otimes k}\rightarrow H$ is an iterated multiplication\footnote{These are
commonly known as $\Delta^{\left(  k-1\right)  }$ and $m^{\left(  k-1\right)
}$, but we prefer the superscripts to match the number of tensorands.}). These
operations are natural in $H$ (that is, equivariant with respect to bialgebra
morphisms) and have been studied for years (see, e.g., \cite{Kashin19} and
\cite{AguLau14} for some recent work), as have been several similar operators
(e.g., \cite{Patras94}, \cite[\S 4.5]{Loday98}, \cite{PatReu98}) on Hopf
algebras and graded bialgebras. More generally, for any bialgebra $H$, we can
define the \textquotedblleft twisted Adams operations\textquotedblright%
\ $m^{\left[  k\right]  }\circ\sigma^{-1}\circ\Delta^{\left[  k\right]  }$ for
any permutation $\sigma\in\mathfrak{S}_{k}$ (where $\sigma^{-1}$ acts on
$H^{\otimes k}$ by permuting the tensorands)\footnote{The notation
$\mathfrak{S}_{k}$ means the $k$-th symmetric group. Our use of $\sigma^{-1}$
instead of $\sigma$ is meant to simplify the formula for the action.}. When
$H$ is a graded bialgebra, one can further refine these operations by
injecting a projection map $P_{\alpha}$ between the $m^{\left[  k\right]  }$
and the $\sigma$. To be specific: For any $k$-tuple $\alpha=\left(  \alpha
_{1},\alpha_{2},\ldots,\alpha_{k}\right)  \in\mathbb{N}^{k}$, we let
$P_{\alpha}:H^{\otimes k}\rightarrow H^{\otimes k}$ be the projection onto the
$\alpha$-th multigraded component (i.e., the tensor product of the projections
$H\rightarrow H_{\alpha_{1}}$, $H\rightarrow H_{\alpha_{2}}$, $\ldots$,
$H\rightarrow H_{\alpha_{k}}$).

The resulting \textquotedblleft twisted projecting Adams
operations\textquotedblright\ $p_{\alpha,\sigma}:=m^{\left[  k\right]  }\circ
P_{\alpha}\circ\sigma^{-1}\circ\Delta^{\left[  k\right]  }$ (for
$k\in\mathbb{N}$ and $\sigma\in\mathfrak{S}_{k}$ and $\alpha\in\mathbb{N}^{k}%
$) have an interesting history. When $H$ is cocommutative, they have been
studied under the name of \textquotedblleft op\'{e}rateurs de
descente\textquotedblright\ by Patras in \cite{Patras94} (see \cite[\S 5.1]%
{CarPat21} for a more recent exposition) and used by Reutenauer \cite[\S 9.1]%
{Reuten93} to prove Solomon's Mackey formula for the symmetric group. The dual
case -- when $H$ is commutative -- is essentially equivalent. In both of these
cases, the permutation $\sigma$ is immaterial, since it can be swallowed
either by the $\Delta^{\left[  k\right]  }$ (when $H$ is cocommutative) or by
the $m^{\left[  k\right]  }$ (when $H$ is commutative). Thus, in these cases,
the operators depend only on the $k$-tuple $\alpha$. Moreover, when the graded
bialgebra $H$ is connected, the $k$-tuple $\alpha$ can be compressed by
removing all $0$'s from it, thus becoming a composition (a tuple of positive
integers). One of Patras's key results (\cite[Th\'{e}or\`{e}me II,7]%
{Patras94}, \cite[Theorem 9.2]{Reuten93}, \cite[Theorem 5.1.1]{CarPat21}) is a
formula for the composition $p_{\alpha,\operatorname*{id}}\circ p_{\beta
,\operatorname*{id}}$ of two such operators: When $H$ is cocommutative and
$\alpha=\left(  \alpha_{1},\alpha_{2},\ldots,\alpha_{k}\right)  \in
\mathbb{N}^{k}$ and $\beta=\left(  \beta_{1},\beta_{2},\ldots,\beta_{\ell
}\right)  \in\mathbb{N}^{\ell}$, it claims that%
\begin{equation}
p_{\alpha,\operatorname*{id}}\circ p_{\beta,\operatorname*{id}}=\sum
_{\substack{\gamma_{i,j}\in\mathbb{N}\text{ for all }i\in\left[  k\right]
\text{ and }j\in\left[  \ell\right]  ;\\\gamma_{i,1}+\gamma_{i,2}%
+\cdots+\gamma_{i,\ell}=\alpha_{i}\text{ for all }i\in\left[  k\right]
;\\\gamma_{1,j}+\gamma_{2,j}+\cdots+\gamma_{k,j}=\beta_{j}\text{ for all }%
j\in\left[  \ell\right]  }}p_{\left(  \gamma_{1,1},\gamma_{1,2},\ldots
,\gamma_{k,\ell}\right)  ,\operatorname*{id}} \label{eq.intro.pp=sum.cocomm}%
\end{equation}
(where $\left[  n\right]  :=\left\{  1,2,\ldots,n\right\}  $ for each
$n\in\mathbb{N}$, and where $\left(  \gamma_{1,1},\gamma_{1,2},\ldots
,\gamma_{k,\ell}\right)  $ denotes the list of all $k\ell$ numbers
$\gamma_{i,j}$ listed in lexicographic order of their subscripts)\footnote{To
be precise, all the above-cited sources (\cite[Th\'{e}or\`{e}me II,7]%
{Patras94}, \cite[Theorem 9.2]{Reuten93}, \cite[Theorem 5.1.1]{CarPat21}) are
restricting themselves to the case when $H$ is connected.}. When $H$ is
furthermore connected (i.e., when $H_{0}\cong\mathbf{k}$), we can restrict
ourselves to compositions by removing all $0$'s from our tuples. In that
situation, the formula (\ref{eq.intro.pp=sum.cocomm}) is structurally
identical to the expansion of the internal product of two complete homogeneous
noncommutative symmetric functions in $\operatorname*{NSym}$ (see
\cite[Proposition 5.1]{ncsf1}), which reveals that the operators
$p_{\alpha,\operatorname*{id}}$ are images of the latter functions under an
algebra morphism from $\operatorname*{NSym}$ (under the internal product) to
$\operatorname*{End}H$. Thus, noncommutative symmetric functions govern many
natural operations for cocommutative bialgebras. A dual theory exists for
commutative $H$.

To my knowledge, no analogue of the formula (\ref{eq.intro.pp=sum.cocomm}) has
been proposed for general $H$. This general case is somewhat complicated by
the fact that the operators of the form $p_{\alpha,\operatorname*{id}}$ are no
longer closed under composition; but the ones of the more general form
$p_{\alpha,\sigma}$ still are. The main result of this paper (Theorem
\ref{thm.sol-mac-1}) is a formula that generalizes
(\ref{eq.intro.pp=sum.cocomm}) to this case. The rest of this paper is devoted
to studying the operators $p_{\alpha,\sigma}$ further, particularly under
connectedness assumptions, and to analyzing the formal structure behind their
composition and convolution rules. This structure is encoded in a new
combinatorial Hopf algebra I call $\operatorname*{PNSym}$ (\textquotedblleft
permuted noncommutative symmetric functions\textquotedblright), whose standard
basis is indexed by what I have christened \textquotedblleft
mopiscotions\textquotedblright\ (permuted compositions). Some properties of
this Hopf algebra and its internal multiplication are outlined, and several
questions are left open for further research. A connection to the work of
Aguiar and Mahajan -- in particular, their Janus monoid -- is discussed.
Finally, the results are applied to obtain an algorithm for verifying
functorial identities for connected graded bialgebras.

\subsection{Plan of this paper}

In Section \ref{sec.pas}, I introduce the operators $p_{\alpha,\sigma}$ for an
arbitrary graded $\mathbf{k}$-bialgebra $H$, and prove the main result
(Theorem \ref{thm.sol-mac-1}), which is a formula that expands a composition
$p_{\alpha,\sigma}\circ p_{\beta,\tau}$ of two such operators as a linear
combination of other such operators. This formula generalizes
(\ref{eq.intro.pp=sum.cocomm}). In this section, I also discuss how the
operators $p_{\alpha,\sigma}$ behave under convolution (Proposition
\ref{prop.sol-mac-0}), tensoring (Proposition \ref{prop.pas-on-tensors}) and
bialgebra duality (Proposition \ref{prop.pas-on-dual}), and show that they can
-- for an appropriate choice of $H$ -- be linearly independent (Theorems
\ref{thm.pas-lin-ind} and \ref{thm.pas-lin-ind2}). I suspect that the
operators $p_{\alpha,\sigma}$ and their infinite $\mathbf{k}$-linear
combinations span all possible natural $\mathbf{k}$-linear endomorphisms on
the category of connected graded bialgebras (i.e., all $\mathbf{k}$-linear
maps $H\rightarrow H$ that are defined for every connected graded bialgebra
$H$ and are functorial with respect to graded bialgebra morphisms); while I
have been unable to prove this, this conjecture is easily verified for all
known natural endomorphisms I have found in the literature.

\begin{noncompile}
and sketch how the classical Solomon's Mackey formula can be derived from ours
(or from (\ref{eq.intro.pp=sum.cocomm})).
\end{noncompile}

The remaining two sections are much terser and should be viewed as a
preliminary report on a field in flux. Only the slightest indications of
proofs are given, and several open questions are posed.

In Section \ref{sec.PNSym}, I define a combinatorial Hopf algebra
$\operatorname*{PNSym}$ (the \textquotedblleft permuted noncommutative
symmetric functions\textquotedblright) that governs the \textquotedblleft
twisted projecting descent operations\textquotedblright\ $p_{\alpha,\sigma}$
on an arbitrary connected graded bialgebra $H$ just like $\operatorname*{NSym}%
$ does for the cocommutative case. This Hopf algebra appears to be new, and I
attempt to uncover some of its structure. In particular, its internal
multiplication has much in common with that of $\operatorname*{NSym}$ (which
it lifts: there is a projection $\mathfrak{p}:\operatorname*{PNSym}%
\rightarrow\operatorname*{NSym}$ that respects all structures). Several
questions are left unanswered here.

In the final Section \ref{sec.idcheck}, I briefly discuss an application of
the above results: Namely, they can be used to mechanically verify any
identity between operators of the form $p_{\alpha,\sigma}$ (and their sums,
convolutions and compositions) that hold for arbitrary (connected) graded bialgebras.

\subsection{Notations}

Let $\mathbb{N}=\left\{  0,1,2,\ldots\right\}  $.

We fix a commutative ring $\mathbf{k}$, which shall serve as our base ring
throughout this paper. In particular, all modules, algebras, coalgebras,
bialgebras and Hopf algebras will be over $\mathbf{k}$. Tensor products and
hom spaces are defined over $\mathbf{k}$ as well. Algebras are understood to
be unital and associative unless said otherwise.

We will use standard concepts -- such as iterated (co)multiplications, tensor
products, (co)commutativity, etc. -- related to bialgebras (and occasionally
Hopf algebras). The reader can find them explained, e.g., in \cite[Chapter
1]{GriRei}.

If $H$ is any bialgebra, then the multiplication, the unit, the
comultiplication, and the counit of $H$ (regarded as linear maps) will be
denoted by
\begin{align*}
m_{H}  &  :H\otimes H\rightarrow H,\ \ \ \ \ \ \ \ \ \ u_{H}:\mathbf{k}%
\rightarrow H,\\
\Delta_{H}  &  :H\rightarrow H\otimes H,\ \ \ \ \ \ \ \ \ \ \epsilon
_{H}:H\rightarrow\mathbf{k},
\end{align*}
respectively. If no ambiguity is to be feared, then we will abbreviate them as
$m$, $u$, $\Delta$ and $\epsilon$. We furthermore denote the unity of a ring
$R$ (viewed as an element of $R$) by $1_{R}$.

Graded $\mathbf{k}$-modules are always understood to be $\mathbb{N}$-graded,
i.e., to have direct sum decompositions $V=\bigoplus_{n\in\mathbb{N}}V_{n}$.
The $n$-th graded component of a graded $\mathbf{k}$-module $V$ will be called
$V_{n}$. If $n<0$, then this is the zero submodule $0$. Tensor products of
graded $\mathbf{k}$-modules are equipped with the usual grading:
\[
\left(  A_{\left(  1\right)  }\otimes A_{\left(  2\right)  }\otimes
\cdots\otimes A_{\left(  k\right)  }\right)  _{n}=\sum_{\substack{\left(
i_{1},i_{2},\ldots,i_{k}\right)  \in\mathbb{N}^{k};\\i_{1}+i_{2}+\cdots
+i_{k}=n}}\left(  A_{\left(  1\right)  }\right)  _{i_{1}}\otimes\left(
A_{\left(  2\right)  }\right)  _{i_{2}}\otimes\cdots\otimes\left(  A_{\left(
k\right)  }\right)  _{i_{k}}.
\]

A $\mathbf{k}$-linear map $f:U\rightarrow V$ between two graded $\mathbf{k}%
$-modules $U$ and $V$ is said to be \emph{graded} if every $n\in\mathbb{N}$
satisfies $f\left(  U_{n}\right)  \subseteq V_{n}$.

Graded bialgebras are bialgebras whose operations ($m$, $u$, $\Delta$ and
$\epsilon$) are graded $\mathbf{k}$-linear maps. We do \textbf{not} twist our
tensor products by the grading (i.e., we do \textbf{not} follow the
topologists' sign conventions); in particular, the tensor product $A\otimes B$
of two algebras has its multiplication defined by $\left(  a\otimes b\right)
\left(  a^{\prime}\otimes b^{\prime}\right)  =aa^{\prime}\otimes bb^{\prime}$,
regardless of any possible gradings on $A$ and $B$.

We recall that a non-graded bialgebra can be viewed as a graded bialgebra
concentrated in degree $0$ (that is, a graded bialgebra $H$ with $H=H_{0}$).
Thus, all claims about graded bialgebras that we make below can be specialized
to non-graded bialgebras.

A graded $\mathbf{k}$-bialgebra $H$ is said to be \emph{connected} if
$H_{0}=\mathbf{k}\cdot1_{H}$. A known fact (e.g., \cite[Proposition
1.4.16]{GriRei}) says that any connected graded bialgebra is automatically a
Hopf algebra, i.e., has an antipode. Hopf algebras are not at the center of
our present work, but will occasionally appear in applications.

If $H$ is any $\mathbf{k}$-module, then $\operatorname*{End}H$ shall denote
the $\mathbf{k}$-module of $\mathbf{k}$-linear maps from $H$ to $H$. This
$\mathbf{k}$-module $\operatorname*{End}H$ becomes an algebra under
composition of maps. As usual, we denote this composition operation by $\circ$
(so that $f\circ g$ means the composition of two maps $f$ and $g$, sending
each $x\in H$ to $f\left(  g\left(  x\right)  \right)  $). The neutral element
of this composition is the identity map $\operatorname*{id}\nolimits_{H}%
\in\operatorname*{End}H$.

If $H$ is a bialgebra, then the $\mathbf{k}$-module $\operatorname*{End}H$ has
yet another canonical multiplication, known as \emph{convolution} and denoted
by $\star$; it is defined by\footnote{Actually, convolution is defined in the
same way for $\mathbf{k}$-linear maps from any given coalgebra to any given
algebra.}%
\[
f\star g:=m_{H}\circ\left(  f\otimes g\right)  \circ\Delta_{H}%
\ \ \ \ \ \ \ \ \ \ \text{for all }f,g\in\operatorname*{End}H.
\]
This operation $\star$ is $\mathbf{k}$-bilinear and associative and has the
neutral element $u_{H}\circ\epsilon_{H}$; thus, $\operatorname*{End}H$ becomes
an algebra under this operation. See \cite[Definition 1.4.1]{GriRei} for more
about convolution.

Thus, when $H$ is a bialgebra, $\operatorname*{End}H$ becomes a $\mathbf{k}%
$-algebra in two natural ways: once using the composition $\circ$, and once
using the convolution $\star$.

\section{\label{sec.pas}The maps $p_{\alpha,\sigma}$ for a graded bialgebra
$H$}

\subsection{Definitions}

Let $H$ be a graded $\mathbf{k}$-bialgebra. We fix it for the rest of Section
\ref{sec.pas}.

The set $\operatorname*{End}\nolimits_{\operatorname*{gr}}H$ of all graded
$\mathbf{k}$-module endomorphisms of $H$ is a $\mathbf{k}$-submodule of
$\operatorname*{End}H$, and is preserved under both composition and
convolution (and thus is a $\mathbf{k}$-subalgebra of both the composition
algebra $\operatorname*{End}H$ and the convolution algebra
$\operatorname*{End}H$). Furthermore, we can consider the $\mathbf{k}%
$-submodule $\mathbf{E}\left(  H\right)  $ of $\operatorname*{End}%
\nolimits_{\operatorname*{gr}}H$ that consists only of those $f\in
\operatorname*{End}\nolimits_{\operatorname*{gr}}H$ that annihilate all but
finitely many graded components of $H$ (that is, that satisfy $f\left(
H_{n}\right)  =0$ for all sufficiently high $n$). This submodule
$\mathbf{E}\left(  H\right)  $ is itself graded, with the $n$-th graded
component being canonically isomorphic to $\operatorname*{End}\left(
H_{n}\right)  $. This submodule $\mathbf{E}\left(  H\right)  $, too, is
preserved under both composition and convolution (but usually does not contain
$\operatorname*{id}\nolimits_{H}$, whence it is not a subalgebra of the
composition algebra $\operatorname*{End}H$).

This module $\mathbf{E}\left(  H\right)  $ has been studied, e.g., by
Hazewinkel \cite{Hazewi04}. Unlike him, we shall not consider $\mathbf{E}%
\left(  H\right)  $ for any specific $H$, but we shall instead focus on the
\textquotedblleft generic\textquotedblright\ $\mathbf{E}\left(  H\right)  $.
In other words, we will consider the functorial endomorphisms of $H$ that are
defined for all graded bialgebras $H$ and always belong to $\mathbf{E}\left(
H\right)  $. Such endomorphisms include

\begin{itemize}
\item the projections $p_{0},p_{1},p_{2},\ldots$ from $H$ onto the graded
components $H_{0},H_{1},H_{2},\ldots$ (regarded as endomorphisms of $H$);

\item their convolutions\footnote{As usual, the empty product $p_{i_{1}}\star
p_{i_{2}}\star\cdots\star p_{i_{0}}$ in $\mathbf{E}\left(  H\right)  $ is
understood to be the unity of $\mathbf{E}\left(  H\right)  $, that is, the map
$u\circ\epsilon:H\rightarrow H$.} $p_{\left(  i_{1},i_{2},\ldots,i_{k}\right)
}:=p_{i_{1}}\star p_{i_{2}}\star\cdots\star p_{i_{k}}$ with $i_{1}%
,i_{2},\ldots,i_{k}\in\mathbb{N}$;

\item the (nonempty) compositions $p_{\alpha}\circ p_{\beta}\circ\cdots\circ
p_{\kappa}$ of such convolutions.
\end{itemize}

However, we can define a broader class of such endomorphisms. To do so, we
need some notations:

\begin{definition}
For each $n\in\mathbb{N}$, we let $p_{n}:H\rightarrow H$ denote the canonical
projection of the graded module $H$ onto its $n$-th graded component $H_{n}$.
\end{definition}

\begin{definition}
For each $k\in\mathbb{N}$, we set $\left[  k\right]  :=\left\{  1,2,\ldots
,k\right\}  $.
\end{definition}

\begin{definition}
For each $k\in\mathbb{N}$, we let $\mathfrak{S}_{k}$ denote the $k$-th
symmetric group, i.e., the group of all permutations of $\left[  k\right]  $.
\end{definition}

\begin{definition}
\label{def.Sk-action}For each $k\in\mathbb{N}$, we let the group
$\mathfrak{S}_{k}$ act on $H^{\otimes k}$ from the left by permuting
tensorands, according to the rule%
\begin{align*}
\sigma\cdot\left(  h_{1}\otimes h_{2}\otimes\cdots\otimes h_{k}\right)   &
=h_{\sigma^{-1}\left(  1\right)  }\otimes h_{\sigma^{-1}\left(  2\right)
}\otimes\cdots\otimes h_{\sigma^{-1}\left(  k\right)  }\\
&  \ \ \ \ \ \ \ \ \ \ \text{for all }\sigma\in\mathfrak{S}_{k}\text{ and
}h_{1},h_{2},\ldots,h_{k}\in H.
\end{align*}
This is an action by bialgebra endomorphisms.
\end{definition}

\begin{definition}
\label{def.iterated}For any $k\in\mathbb{N}$, we let $m^{\left[  k\right]
}:H^{\otimes k}\rightarrow H$ and $\Delta^{\left[  k\right]  }:H\rightarrow
H^{\otimes k}$ be the iterated multiplication and iterated comultiplication
maps of the bialgebra $H$ (denoted $m^{\left(  k-1\right)  }$ and
$\Delta^{\left(  k-1\right)  }$ in \cite[Exercises 1.4.19 and 1.4.20]{GriRei},
and denoted $\Pi^{\left[  k\right]  }$ and $\Delta^{\left[  k\right]  }$ in
\cite{Patras94}). Note that $m^{\left[  k\right]  }$ sends each pure tensor
$a_{1}\otimes a_{2}\otimes\cdots\otimes a_{k}\in H^{\otimes k}$ to the product
$a_{1}a_{2}\cdots a_{k}$, whereas $\Delta^{\left[  k\right]  }$ sends each
element $x\in H$ to $\sum_{\left(  x\right)  }x_{\left(  1\right)  }\otimes
x_{\left(  2\right)  }\otimes\cdots\otimes x_{\left(  k\right)  }$ (using
Sweedler notation).
\end{definition}

We recall that these iterated multiplications and comultiplications satisfy
the rules\footnote{Note that we identify $H\otimes\mathbf{k}$ and
$\mathbf{k}\otimes H$ with $H$.}%
\begin{equation}
m^{\left[  u+v\right]  }=m\circ\left(  m^{\left[  u\right]  }\otimes
m^{\left[  v\right]  }\right)  \label{eq.mu+v}%
\end{equation}
and%
\begin{equation}
\Delta^{\left[  u+v\right]  }=\left(  \Delta^{\left[  u\right]  }\otimes
\Delta^{\left[  v\right]  }\right)  \circ\Delta\label{eq.Du+v}%
\end{equation}
for all $u,v\in\mathbb{N}$. (Indeed, these are the claims of \cite[Exercise
1.4.19(a)]{GriRei} and \cite[Exercise 1.4.20(a)]{GriRei}, respectively,
applied to $i=u-1$ and $k=u+v-1$, at least when $u,v\geq1$. The $u=0$ and
$v=0$ cases are easy.) Also, $m^{\left[  1\right]  }=\Delta^{\left[  1\right]
}=\operatorname*{id}\nolimits_{H}$ and $m^{\left[  0\right]  }=u:\mathbf{k}%
\rightarrow H$ and $\Delta^{\left[  0\right]  }=\epsilon:H\rightarrow
\mathbf{k}$.

\begin{definition}
\label{def.wc}\textbf{(a)} A \emph{weak composition} means a finite tuple of
nonnegative integers. \medskip

\textbf{(b)} If $\alpha=\left(  \alpha_{1},\alpha_{2},\ldots,\alpha
_{k}\right)  $ is a weak composition, then its \emph{size} $\left\vert
\alpha\right\vert $ is defined to be the number $\alpha_{1}+\alpha_{2}%
+\cdots+\alpha_{k}\in\mathbb{N}$. \medskip

\textbf{(c)} A \emph{composition} means a finite tuple of positive integers.
\end{definition}

\begin{definition}
\label{def.Palpha}For any weak composition $\alpha=\left(  \alpha_{1}%
,\alpha_{2},\ldots,\alpha_{k}\right)  $, we define the projection map
$P_{\alpha}:H^{\otimes k}\rightarrow H^{\otimes k}$ to be the tensor product
$p_{\alpha_{1}}\otimes p_{\alpha_{2}}\otimes\cdots\otimes p_{\alpha_{k}}$ of
the $\mathbf{k}$-linear maps $p_{\alpha_{1}},p_{\alpha_{2}},\ldots
,p_{\alpha_{k}}$. (Thus, if we regard $H^{\otimes k}$ as an $\mathbb{N}^{k}%
$-graded $\mathbf{k}$-module, then $P_{\alpha}$ is its projection onto its
degree-$\left(  \alpha_{1},\alpha_{2},\ldots,\alpha_{k}\right)  $ component.)
\end{definition}

\begin{definition}
\label{def.pas}For any weak composition $\alpha\in\mathbb{N}^{k}$ and any
permutation $\sigma\in\mathfrak{S}_{k}$, we define a map $p_{\alpha,\sigma
}:H\rightarrow H$ by the formula%
\begin{equation}
p_{\alpha,\sigma}:=m^{\left[  k\right]  }\circ P_{\alpha}\circ\sigma^{-1}%
\circ\Delta^{\left[  k\right]  }, \label{eq.pas.formal-def}%
\end{equation}
where the \textquotedblleft$\sigma^{-1}$\textquotedblright\ really means the
action of $\sigma^{-1}\in\mathfrak{S}_{k}$ on $H^{\otimes k}$ (as in
Definition \ref{def.Sk-action}).
\end{definition}

We can rewrite the definition of $p_{\alpha,\sigma}$ in more concrete terms
using the Sweedler notation:

\begin{remark}
\label{rmk.pas.informal-def}For any weak composition $\alpha=\left(
\alpha_{1},\alpha_{2},\ldots,\alpha_{k}\right)  \in\mathbb{N}^{k}$ and any
permutation $\sigma\in\mathfrak{S}_{k}$, the map $p_{\alpha,\sigma
}:H\rightarrow H$ is given by%
\[
p_{\alpha,\sigma}\left(  x\right)  =\sum_{\left(  x\right)  }p_{\alpha_{1}%
}\left(  x_{\left(  \sigma\left(  1\right)  \right)  }\right)  p_{\alpha_{2}%
}\left(  x_{\left(  \sigma\left(  2\right)  \right)  }\right)  \cdots
p_{\alpha_{k}}\left(  x_{\left(  \sigma\left(  k\right)  \right)  }\right)
\]
for every $x\in H$, where we are using the Sweedler notation $\sum_{\left(
x\right)  }x_{\left(  1\right)  }\otimes x_{\left(  2\right)  }\otimes
\cdots\otimes x_{\left(  k\right)  }$ for the iterated coproduct
$\Delta^{\left[  k\right]  }\left(  x\right)  \in H^{\otimes k}$.
\end{remark}

The following is near-obvious:

\begin{proposition}
\label{prop.pas.graded}Let $\alpha\in\mathbb{N}^{k}$ be a weak composition,
and $\sigma\in\mathfrak{S}_{k}$ a permutation. Then, the map $p_{\alpha
,\sigma}$ is a graded $\mathbf{k}$-module endomorphism of $H$ that sends
$H_{\left\vert \alpha\right\vert }$ to $H_{\left\vert \alpha\right\vert }$ and
sends all other graded components $H_{n}$ to $0$. Thus, $p_{\alpha,\sigma}$
lies in the $\left\vert \alpha\right\vert $-th graded component of
$\mathbf{E}\left(  H\right)  $.
\end{proposition}

\begin{proof}
The gradedness of $p_{\alpha,\sigma}$ follows from the fact that all four maps
$m^{\left[  k\right]  }$, $P_{\alpha}$, $\sigma^{-1}$ and $\Delta^{\left[
k\right]  }$ in (\ref{eq.pas.formal-def}) are graded. Thus, the map
$p_{\alpha,\sigma}$ sends $H_{\left\vert \alpha\right\vert }$ to
$H_{\left\vert \alpha\right\vert }$. To see that it sends all other $H_{n}$ to
$0$, we only need to observe that for every $n\neq\left\vert \alpha\right\vert
$, we have $\left(  \sigma^{-1}\circ\Delta^{\left[  k\right]  }\right)
\left(  H_{n}\right)  \subseteq\left(  H^{\otimes k}\right)  _{n}$ (since
$\sigma^{-1}$ and $\Delta^{\left[  k\right]  }$ are graded), and that
$P_{\alpha}$ annihilates $\left(  H^{\otimes k}\right)  _{n}$ (since
$n\neq\left\vert \alpha\right\vert $).
\end{proof}

\subsection{Particular cases}

Let us contrast our maps $p_{\alpha,\sigma}$ to Patras's descent operators
$p_{\alpha}$ (denoted $B_{\alpha}$ in \cite[\S II]{Patras94}). We recall how
the latter are defined (generalizing slightly from compositions to weak compositions):

\begin{definition}
\label{def.pa}For any weak composition $\alpha=\left(  \alpha_{1},\alpha
_{2},\ldots,\alpha_{k}\right)  \in\mathbb{N}^{k}$, we define a map $p_{\alpha
}:H\rightarrow H$ by the formula%
\[
p_{\alpha}:=p_{\alpha_{1}}\star p_{\alpha_{2}}\star\cdots\star p_{\alpha_{k}%
}.
\]
(Recall that $\star$ denotes convolution.)
\end{definition}

Our $p_{\alpha,\sigma}$ recover these for $\sigma=\operatorname*{id}$:

\begin{proposition}
\label{prop.paid}Let $\alpha\in\mathbb{N}^{k}$ be a weak composition. Then,
$p_{\alpha,\operatorname*{id}}=p_{\alpha}$ (where $\operatorname*{id}$ is the
identity permutation in $\mathfrak{S}_{k}$).
\end{proposition}

\begin{proof}
Write $\alpha$ as $\alpha=\left(  \alpha_{1},\alpha_{2},\ldots,\alpha
_{k}\right)  $. Then, (\ref{eq.pas.formal-def}) yields%
\begin{equation}
p_{\alpha,\operatorname*{id}}=m^{\left[  k\right]  }\circ P_{\alpha}%
\circ\underbrace{\operatorname*{id}\nolimits^{-1}}_{=\operatorname*{id}}%
\circ\,\Delta^{\left[  k\right]  }=m^{\left[  k\right]  }\circ P_{\alpha}%
\circ\Delta^{\left[  k\right]  }. \label{pf.prop.paid.1}%
\end{equation}

On the other hand, \cite[Exercise 1.4.23]{GriRei} yields\footnote{Recall that
the maps $m^{\left[  k\right]  }$ and $\Delta^{\left[  k\right]  }$ are called
$m^{\left(  k-1\right)  }$ and $\Delta^{\left(  k-1\right)  }$ in
\cite{GriRei}.}%
\[
p_{\alpha_{1}}\star p_{\alpha_{2}}\star\cdots\star p_{\alpha_{k}}=m^{\left[
k\right]  }\circ\underbrace{\left(  p_{\alpha_{1}}\otimes p_{\alpha_{2}%
}\otimes\cdots\otimes p_{\alpha_{k}}\right)  }_{\substack{=P_{\alpha
}\\\text{(by the definition of }P_{\alpha}\text{)}}}\circ\,\Delta^{\left[
k\right]  }=m^{\left[  k\right]  }\circ P_{\alpha}\circ\Delta^{\left[
k\right]  }.
\]
Now, the definition of $p_{\alpha}$ yields%
\begin{equation}
p_{\alpha}=p_{\alpha_{1}}\star p_{\alpha_{2}}\star\cdots\star p_{\alpha_{k}%
}=m^{\left[  k\right]  }\circ P_{\alpha}\circ\Delta^{\left[  k\right]  }.
\label{pf.prop.paid.2}%
\end{equation}
Comparing this with (\ref{pf.prop.paid.1}), we obtain $p_{\alpha
,\operatorname*{id}}=p_{\alpha}$. This proves Proposition \ref{prop.paid}.
\end{proof}

If the bialgebra $H$ is commutative or cocommutative, then we can bring all
our $p_{\alpha,\sigma}$ to the form $p_{\beta}$ for some weak composition
$\beta$:

\begin{proposition}
\label{prop.pas.comm}Let $\alpha\in\mathbb{N}^{k}$ be a weak composition, and
$\sigma\in\mathfrak{S}_{k}$ a permutation. \medskip

\textbf{(a)} If $H$ is commutative, then%
\[
p_{\alpha,\sigma}=p_{\sigma\cdot\alpha},
\]
where we are using the left action of $\mathfrak{S}_{k}$ on $\mathbb{N}^{k}$
defined by
\[
\sigma\cdot\left(  \alpha_{1},\alpha_{2},\ldots,\alpha_{k}\right)  :=\left(
\alpha_{\sigma^{-1}\left(  1\right)  },\alpha_{\sigma^{-1}\left(  2\right)
},\ldots,\alpha_{\sigma^{-1}\left(  k\right)  }\right)  .
\]

\textbf{(b)} If $H$ is cocommutative, then%
\[
p_{\alpha,\sigma}=p_{\alpha}.
\]

\end{proposition}

To prove this, we need a simple lemma:

\begin{lemma}
\label{lem.Pa-vers-perm}Let $f_{1},f_{2},\ldots,f_{k}$ be $k$ arbitrary
elements of $\operatorname*{End}H$. Let $\sigma\in\mathfrak{S}_{k}$. Then,%
\[
\sigma\circ\left(  f_{1}\otimes f_{2}\otimes\cdots\otimes f_{k}\right)
=\left(  f_{\sigma^{-1}\left(  1\right)  }\otimes f_{\sigma^{-1}\left(
2\right)  }\otimes\cdots\otimes f_{\sigma^{-1}\left(  k\right)  }\right)
\circ\sigma
\]
(where $\sigma$ and $f_{1}\otimes f_{2}\otimes\cdots\otimes f_{k}$ and
$f_{\sigma^{-1}\left(  1\right)  }\otimes f_{\sigma^{-1}\left(  2\right)
}\otimes\cdots\otimes f_{\sigma^{-1}\left(  k\right)  }$ are understood as
endomorphisms of $H^{\otimes k}$).
\end{lemma}

\begin{proof}
[Proof of Lemma \ref{lem.Pa-vers-perm}.]Both sides send any given pure tensor
$h_{1}\otimes h_{2}\otimes\cdots\otimes h_{k}\in H^{\otimes k}$ to
\[
f_{\sigma^{-1}\left(  1\right)  }\left(  h_{\sigma^{-1}\left(  1\right)
}\right)  \otimes f_{\sigma^{-1}\left(  2\right)  }\left(  h_{\sigma
^{-1}\left(  2\right)  }\right)  \otimes\cdots\otimes f_{\sigma^{-1}\left(
k\right)  }\left(  h_{\sigma^{-1}\left(  k\right)  }\right)  .
\]
Thus, they agree on all pure tensors, and hence are identical.
\end{proof}

\begin{proof}
[Proof of Proposition \ref{prop.pas.comm}.]Write $\alpha$ as $\alpha=\left(
\alpha_{1},\alpha_{2},\ldots,\alpha_{k}\right)  $. \medskip

\textbf{(a)} Assume that $H$ is commutative. Then, $m^{\left[  k\right]
}\circ\tau=m^{\left[  k\right]  }$ for any $\tau\in\mathfrak{S}_{k}$ (by
\cite[Exercise 1.5.10]{GriRei}). Thus, in particular, $m^{\left[  k\right]
}\circ\sigma^{-1}=m^{\left[  k\right]  }$. However, it is straightforward to
see (and actually true for any $\mathbf{k}$-module in place of $H$) that%
\[
\sigma\circ P_{\alpha}=P_{\sigma\cdot\alpha}\circ\sigma
\]
as maps from $H^{\otimes k}$ to $H^{\otimes k}$. (Indeed, this follows from
Lemma \ref{lem.Pa-vers-perm} (applied to $f_{i}=p_{\alpha_{i}}$), since
$P_{\alpha}=p_{\alpha_{1}}\otimes p_{\alpha_{2}}\otimes\cdots\otimes
p_{\alpha_{k}}$ and $P_{\sigma\cdot\alpha}=p_{\alpha_{\sigma^{-1}\left(
1\right)  }}\otimes p_{\alpha_{\sigma^{-1}\left(  2\right)  }}\otimes
\cdots\otimes p_{\alpha_{\sigma^{-1}\left(  k\right)  }}$.)

Now, the definition of $p_{\alpha,\sigma}$ yields%
\begin{align*}
p_{\alpha,\sigma}  &  =m^{\left[  k\right]  }\circ\underbrace{P_{\alpha}%
}_{\substack{=\sigma^{-1}\circ P_{\sigma\cdot\alpha}\circ\sigma\\\text{(since
}\sigma\circ P_{\alpha}=P_{\sigma\cdot\alpha}\circ\sigma\text{)}}%
}\circ\,\sigma^{-1}\circ\Delta^{\left[  k\right]  }\\
&  =\underbrace{m^{\left[  k\right]  }\circ\sigma^{-1}}_{=m^{\left[  k\right]
}}\circ\,P_{\sigma\cdot\alpha}\circ\underbrace{\sigma\circ\sigma^{-1}%
}_{=\operatorname*{id}}\circ\,\Delta^{\left[  k\right]  }=m^{\left[  k\right]
}\circ P_{\sigma\cdot\alpha}\circ\Delta^{\left[  k\right]  }.
\end{align*}
Comparing this with%
\[
p_{\sigma\cdot\alpha}=m^{\left[  k\right]  }\circ P_{\sigma\cdot\alpha}%
\circ\Delta^{\left[  k\right]  }\ \ \ \ \ \ \ \ \ \ \left(  \text{by
(\ref{pf.prop.paid.2}), applied to }\sigma\cdot\alpha\text{ instead of }%
\alpha\right)  ,
\]
we obtain $p_{\alpha,\sigma}=p_{\sigma\cdot\alpha}$. Proposition
\ref{prop.pas.comm} \textbf{(a)} is thus proved. \medskip

\textbf{(b)} Assume that $H$ is cocommutative. Then, $\tau\circ\Delta^{\left[
k\right]  }=\Delta^{\left[  k\right]  }$ for any $\tau\in\mathfrak{S}_{k}$ (by
\cite[Exercise 1.5.10]{GriRei}). Thus, in particular, $\sigma^{-1}\circ
\Delta^{\left[  k\right]  }=\Delta^{\left[  k\right]  }$. Now, the definition
of $p_{\alpha,\sigma}$ yields%
\[
p_{\alpha,\sigma}=m^{\left[  k\right]  }\circ P_{\alpha}\circ
\underbrace{\sigma^{-1}\circ\Delta^{\left[  k\right]  }}_{=\Delta^{\left[
k\right]  }}=m^{\left[  k\right]  }\circ P_{\alpha}\circ\Delta^{\left[
k\right]  }.
\]
Comparing this with (\ref{pf.prop.paid.2}), we find $p_{\alpha,\sigma
}=p_{\alpha}$. This proves Proposition \ref{prop.pas.comm} \textbf{(b)}.
\end{proof}

However, in general, when $H$ is neither commutative nor cocommutative, we
cannot \textquotedblleft simplify\textquotedblright\ $p_{\alpha,\sigma}$. (See
Theorem \ref{thm.pas-lin-ind} for a concretization of this claim.)

If the graded bialgebra $H$ is connected, then each $p_{\alpha,\sigma}$ for a
weak composition $\alpha$ and a permutation $\sigma$ can be rewritten in the
form $p_{\beta,\tau}$ for a composition (not just weak composition) $\beta$
and a permutation $\tau$. (Indeed, since $H$ is connected, we can remove all
$p_{0}\left(  x_{\left(  i\right)  }\right)  $ factors from the product in
Remark \ref{rmk.pas.informal-def}.) See Proposition \ref{prop.mopis.reduce}
below for an explicit statement of this claim.

\subsection{The convolution formula}

It is easy to see that any convolution of two maps of the form $p_{\alpha
,\sigma}$ is again a map of such form:

\begin{proposition}
\label{prop.sol-mac-0}Let $\alpha=\left(  \alpha_{1},\alpha_{2},\ldots
,\alpha_{k}\right)  $ be a weak composition, and let $\sigma\in\mathfrak{S}%
_{k}$ be a permutation.

Let $\beta=\left(  \beta_{1},\beta_{2},\ldots,\beta_{\ell}\right)  $ be a weak
composition, and let $\tau\in\mathfrak{S}_{\ell}$ be a permutation.

Let $\alpha\beta$ be the concatenation of $\alpha$ and $\beta$; this is the
weak composition $\left(  \alpha_{1},\alpha_{2},\ldots,\alpha_{k},\beta
_{1},\beta_{2},\ldots,\beta_{\ell}\right)  $.

Let $\sigma\oplus\tau$ be the image of $\left(  \sigma,\tau\right)  $ under
the obvious map $\mathfrak{S}_{k}\times\mathfrak{S}_{\ell}\rightarrow
\mathfrak{S}_{k+\ell}$. Explicitly, this is the permutation of $\left[
k+\ell\right]  $ that sends each $i\leq k$ to $\sigma\left(  i\right)  $ and
sends each $j>k$ to $k+\tau\left(  j-k\right)  $.

Then,
\begin{equation}
p_{\alpha,\sigma}\star p_{\beta,\tau}=p_{\alpha\beta,\sigma\oplus\tau}.
\label{eq.pas-conv}%
\end{equation}

\end{proposition}

\begin{proof}
Let us first observe that
\[
\underbrace{\sigma^{-1}}_{\substack{\text{this means}\\\text{the
action}\\\text{of }\sigma^{-1}\text{ on }H^{\otimes k}}}\otimes
\underbrace{\tau^{-1}}_{\substack{\text{this means}\\\text{the action}%
\\\text{of }\tau^{-1}\text{ on }H^{\otimes\ell}}}=\underbrace{\left(
\sigma\oplus\tau\right)  ^{-1}}_{\substack{\text{this means}\\\text{the
action}\\\text{of }\left(  \sigma\oplus\tau\right)  ^{-1}\text{ on }%
H^{\otimes\left(  k+\ell\right)  }}}.
\]
(This can be shown, e.g., by acting on a pure tensor: Any pure tensor
\[
h_{1}\otimes h_{2}\otimes\cdots\otimes h_{k}\otimes g_{1}\otimes g_{2}%
\otimes\cdots\otimes g_{\ell}\in H^{\otimes\left(  k+\ell\right)  }%
\]
is sent by both $\sigma^{-1}\otimes\tau^{-1}$ and $\left(  \sigma\oplus
\tau\right)  ^{-1}$ to the same image
\[
h_{\sigma\left(  1\right)  }\otimes h_{\sigma\left(  2\right)  }\otimes
\cdots\otimes h_{\sigma\left(  k\right)  }\otimes g_{\tau\left(  1\right)
}\otimes g_{\tau\left(  2\right)  }\otimes\cdots\otimes g_{\tau\left(
\ell\right)  }.
\]
Thus, the two $\mathbf{k}$-linear maps $\sigma^{-1}\otimes\tau^{-1}$ and
$\left(  \sigma\oplus\tau\right)  ^{-1}$ agree on each pure tensor, and thus
are identical.)

Furthermore, we have $P_{\alpha}\otimes P_{\beta}=P_{\alpha\beta}$ (as maps
from $H^{\otimes\left(  k+\ell\right)  }$ to $H^{\otimes\left(  k+\ell\right)
}$). (This follows easily from the definitions of $P_{\alpha}$, $P_{\beta}$
and $P_{\alpha\beta}$.)

The definition of convolution yields $p_{\alpha,\sigma}\star p_{\beta,\tau
}=m\circ\left(  p_{\alpha,\sigma}\otimes p_{\beta,\tau}\right)  \circ\Delta$.
In view of%
\begin{align*}
&  \underbrace{p_{\alpha,\sigma}}_{\substack{=m^{\left[  k\right]  }\circ
P_{\alpha}\circ\sigma^{-1}\circ\Delta^{\left[  k\right]  }\\\text{(by
(\ref{eq.pas.formal-def}))}}}\otimes\underbrace{p_{\beta,\tau}}%
_{\substack{=m^{\left[  \ell\right]  }\circ P_{\beta}\circ\tau^{-1}\circ
\Delta^{\left[  \ell\right]  }\\\text{(by (\ref{eq.pas.formal-def}))}}}\\
&  =\left(  m^{\left[  k\right]  }\circ P_{\alpha}\circ\sigma^{-1}\circ
\Delta^{\left[  k\right]  }\right)  \otimes\left(  m^{\left[  \ell\right]
}\circ P_{\beta}\circ\tau^{-1}\circ\Delta^{\left[  \ell\right]  }\right) \\
&  =\left(  m^{\left[  k\right]  }\otimes m^{\left[  \ell\right]  }\right)
\circ\underbrace{\left(  P_{\alpha}\otimes P_{\beta}\right)  }_{=P_{\alpha
\beta}}\circ\underbrace{\left(  \sigma^{-1}\otimes\tau^{-1}\right)
}_{=\left(  \sigma\oplus\tau\right)  ^{-1}}\circ\left(  \Delta^{\left[
k\right]  }\otimes\Delta^{\left[  \ell\right]  }\right) \\
&  =\left(  m^{\left[  k\right]  }\otimes m^{\left[  \ell\right]  }\right)
\circ P_{\alpha\beta}\circ\left(  \sigma\oplus\tau\right)  ^{-1}\circ\left(
\Delta^{\left[  k\right]  }\otimes\Delta^{\left[  \ell\right]  }\right)  ,
\end{align*}
we can rewrite this as
\begin{align*}
p_{\alpha,\sigma}\star p_{\beta,\tau}  &  =\underbrace{m\circ\left(
m^{\left[  k\right]  }\otimes m^{\left[  \ell\right]  }\right)  }%
_{\substack{=m^{\left[  k+\ell\right]  }\\\text{(by (\ref{eq.mu+v}))}}%
}\circ\,P_{\alpha\beta}\circ\left(  \sigma\oplus\tau\right)  ^{-1}%
\circ\underbrace{\left(  \Delta^{\left[  k\right]  }\otimes\Delta^{\left[
\ell\right]  }\right)  \circ\Delta}_{\substack{=\Delta^{\left[  k+\ell\right]
}\\\text{(by (\ref{eq.Du+v}))}}}\\
&  =m^{\left[  k+\ell\right]  }\circ P_{\alpha\beta}\circ\left(  \sigma
\oplus\tau\right)  ^{-1}\circ\Delta^{\left[  k+\ell\right]  }=p_{\alpha
\beta,\sigma\oplus\tau}%
\end{align*}
(since $p_{\alpha\beta,\sigma\oplus\tau}$ is defined as $m^{\left[
k+\ell\right]  }\circ P_{\alpha\beta}\circ\left(  \sigma\oplus\tau\right)
^{-1}\circ\Delta^{\left[  k+\ell\right]  }$). This proves Proposition
\ref{prop.sol-mac-0}.
\end{proof}

\subsection{The composition formula: statement}

It is more interesting to compute the composition of two maps of the form
$p_{\alpha,\sigma}$. It turns out that this is again a $\mathbf{k}$-linear
combination of maps of such form, and the explicit formula is similar to
Solomon's Mackey formula for the descent algebra (or, even more closely
related, \cite[Theorem 9.2 and \S 9.5.1]{Reuten93}). Before we state the
formula, let us introduce one more operation on permutations:

\begin{definition}
\label{def.tausigma}Let $\sigma\in\mathfrak{S}_{k}$ and $\tau\in
\mathfrak{S}_{\ell}$ be two permutations. Then, $\tau\left[  \sigma\right]
\in\mathfrak{S}_{k\ell}$ shall denote the permutation of $\left[
k\ell\right]  $ that sends each $\ell\left(  i-1\right)  +j$ (with
$i\in\left[  k\right]  $ and $j\in\left[  \ell\right]  $) to $k\left(
\tau\left(  j\right)  -1\right)  +\sigma\left(  i\right)  $. This permutation
is well-defined because of Remark \ref{rmk.tausigma.wd} below.
\end{definition}

\begin{remark}
\label{rmk.tausigma.wd}Let $\sigma\in\mathfrak{S}_{k}$ and $\tau
\in\mathfrak{S}_{\ell}$ be two permutations. Why is the permutation
$\tau\left[  \sigma\right]  \in\mathfrak{S}_{k\ell}$ in Definition
\ref{def.tausigma} well-defined? To see that it is well-defined as a map, it
suffices to observe that each element $p\in\left[  k\ell\right]  $ can be
written uniquely in the form $\ell\left(  i-1\right)  +j$ with $i\in\left[
k\right]  $ and $j\in\left[  \ell\right]  $. (This observation follows from
the fact that the $k+1$ numbers $0\ell,\ 1\ell,\ 2\ell,\ \ldots,\ k\ell$
subdivide the set $\left[  k\ell\right]  $ into $k$ intervals of length $\ell$ each.)

But it remains to show that the map $\tau\left[  \sigma\right]  $ is indeed a
permutation in $\mathfrak{S}_{k\ell}$. The best way to show this is by
factoring it as a composition of three bijections.

Namely, define two maps $\lambda$ and $\rho$ from $\left[  k\right]
\times\left[  \ell\right]  $ to $\left[  k\ell\right]  $ by setting%
\[
\lambda\left(  i,j\right)  :=k\left(  j-1\right)  +i\in\left[  k\ell\right]
\ \ \ \ \ \ \ \ \ \ \text{and}\ \ \ \ \ \ \ \ \ \ \rho\left(  i,j\right)
:=\ell\left(  i-1\right)  +j\in\left[  k\ell\right]
\]
for all $\left(  i,j\right)  \in\left[  k\right]  \times\left[  \ell\right]
$. It is easy to see that both of these maps $\lambda$ and $\rho$ are
well-defined. Moreover, these maps $\lambda$ and $\rho$ are bijections. In
fact, $\rho$ sends each pair $\left(  i,j\right)  \in\left[  k\right]
\times\left[  \ell\right]  $ to the position of $\left(  i,j\right)  $ in the
lexicographically ordered Cartesian product $\left[  k\right]  \times\left[
\ell\right]  $, whereas $\lambda$ sends each pair $\left(  i,j\right)
\in\left[  k\right]  \times\left[  \ell\right]  $ to the position of $\left(
j,i\right)  $ in the lexicographically ordered Cartesian product $\left[
\ell\right]  \times\left[  k\right]  $. Thus, $\rho$ is the order-isomorphism
between the lexicographically ordered Cartesian product $\left[  k\right]
\times\left[  \ell\right]  $ and the chain $\left[  k\ell\right]  $, whereas
$\lambda$ is the order-isomorphism between the lexicographically ordered
Cartesian product $\left[  \ell\right]  \times\left[  k\right]  $ and the
chain $\left[  k\ell\right]  $ composed with the standard swap map $\left[
k\right]  \times\left[  \ell\right]  \rightarrow\left[  \ell\right]
\times\left[  k\right]  ,\ \left(  i,j\right)  \mapsto\left(  j,i\right)  $.

Next, we consider the map%
\begin{align*}
\sigma\times\tau:\left[  k\right]  \times\left[  \ell\right]   &
\rightarrow\left[  k\right]  \times\left[  \ell\right]  ,\\
\left(  i,j\right)   &  \mapsto\left(  \sigma\left(  i\right)  ,\tau\left(
j\right)  \right)  .
\end{align*}
This is the Cartesian product of the two maps $\sigma$ and $\tau$, and thus is
a bijection (since $\sigma$ and $\tau$ are bijections). Clearly, the
composition $\lambda\circ\left(  \sigma\times\tau\right)  \circ\rho^{-1}$ is a
well-defined bijection from $\left[  k\ell\right]  $ to $\left[  k\ell\right]
$ (since it is a composition of bijections and their inverses), i.e., a
permutation in $\mathfrak{S}_{k\ell}$.

Now, we claim that the permutation $\tau\left[  \sigma\right]  $ in Definition
\ref{def.tausigma} is simply $\lambda\circ\left(  \sigma\times\tau\right)
\circ\rho^{-1}$. Indeed, for each $i\in\left[  k\right]  $ and $j\in\left[
\ell\right]  $, we have%
\begin{align*}
\left(  \lambda\circ\left(  \sigma\times\tau\right)  \circ\rho^{-1}\right)
\underbrace{\left(  \ell\left(  i-1\right)  +j\right)  }_{\substack{=\rho
\left(  i,j\right)  \\\text{(by the definition of }\rho\text{)}}}  &  =\left(
\lambda\circ\left(  \sigma\times\tau\right)  \circ\rho^{-1}\right)  \left(
\rho\left(  i,j\right)  \right) \\
&  =\lambda\left(  \underbrace{\left(  \sigma\times\tau\right)  \left(
i,j\right)  }_{=\left(  \sigma\left(  i\right)  ,\tau\left(  j\right)
\right)  }\right)  =\lambda\left(  \sigma\left(  i\right)  ,\tau\left(
j\right)  \right) \\
&  =k\left(  \tau\left(  j\right)  -1\right)  +\sigma\left(  i\right)
\end{align*}
(by the definition of $\lambda$). Thus, the permutation $\lambda\circ\left(
\sigma\times\tau\right)  \circ\rho^{-1}\in\mathfrak{S}_{k\ell}$ sends each
$\ell\left(  i-1\right)  +j$ (with $i\in\left[  k\right]  $ and $j\in\left[
\ell\right]  $) to $k\left(  \tau\left(  j\right)  -1\right)  +\sigma\left(
i\right)  $. In other words, $\lambda\circ\left(  \sigma\times\tau\right)
\circ\rho^{-1}$ satisfies the condition imposed on $\tau\left[  \sigma\right]
$ in Definition \ref{def.tausigma}. Thus, $\tau\left[  \sigma\right]
=\lambda\circ\left(  \sigma\times\tau\right)  \circ\rho^{-1}\in\mathfrak{S}%
_{k\ell}$. In other words, $\tau\left[  \sigma\right]  $ is really a
permutation in $\mathfrak{S}_{k\ell}$. This completes our proof.
\end{remark}

\begin{noncompile}
The most intuitive way to compute the permutation $\tau\left[  \sigma\right]
$ in Definition \ref{def.tausigma} is to arrange the $k\ell$ numbers
$1,2,\ldots,k\ell$ into a $k\times\ell$-matrix row-by-row, from left to right
in each row and from top to bottom:%
\[
\left(
\begin{array}
[c]{cccc}%
1 & 2 & \cdots & \ell\\
\ell+1 & \ell+2 & \cdots & 2\ell\\
\vdots & \vdots & \ddots & \vdots\\
\left(  k-1\right)  \ell+1 & \left(  k-1\right)  \ell+2 & \cdots & k\ell
\end{array}
\right)  .
\]
Then, $\tau\left[  \sigma\right]  $ is the permutation that sends each $i$. Aborted.
\end{noncompile}

\begin{example}
Let $s_{1}\in\mathfrak{S}_{2}$ be the transposition that swaps $1$ with $2$.
Let $\operatorname*{id}\nolimits_{3}\in\mathfrak{S}_{3}$ be the identity
permutation. Then, $s_{1}\left[  \operatorname*{id}\nolimits_{3}\right]
\in\mathfrak{S}_{6}$ is the permutation that sends $1,2,3,4,5,6$ to
$4,1,5,2,6,3$, respectively.
\end{example}

It can be shown that the operation that sends two permutations $\tau
\in\mathfrak{S}_{\ell}$ and $\sigma\in\mathfrak{S}_{k}$ to $\tau\left[
\sigma\right]  $ is associative (see Claim 1 in the Second proof idea for
Theorem \ref{thm.PNSym.alg} below for an outline of a proof), but we will not
use this fact in the present section. Furthermore, the identity permutation
$\operatorname*{id}\nolimits_{\left[  1\right]  }\in\mathfrak{S}_{1}$ is
neutral for this operation (i.e., satisfies $\operatorname*{id}%
\nolimits_{\left[  1\right]  }\left[  \sigma\right]  =\sigma\left[
\operatorname*{id}\nolimits_{\left[  1\right]  }\right]  =\sigma$ for all
$\sigma\in\mathfrak{S}_{k}$). However, if two positive integers $k$ and $\ell$
are both larger than $1$, then $\operatorname*{id}\nolimits_{\left[  k\right]
}\left[  \operatorname*{id}\nolimits_{\left[  \ell\right]  }\right]
\neq\operatorname*{id}\nolimits_{\left[  k\ell\right]  }$. Another description
of $\tau\left[  \sigma\right]  $ can be found in Lemma \ref{lem.sitaze} below.
\medskip

We can now state the explicit formula for composition of $p_{\alpha,\sigma}$'s:

\begin{theorem}
\label{thm.sol-mac-1}Let $\alpha=\left(  \alpha_{1},\alpha_{2},\ldots
,\alpha_{k}\right)  $ be a weak composition, and let $\sigma\in\mathfrak{S}%
_{k}$ be a permutation.

Let $\beta=\left(  \beta_{1},\beta_{2},\ldots,\beta_{\ell}\right)  $ be a weak
composition, and let $\tau\in\mathfrak{S}_{\ell}$ be a permutation.

Then,%
\[
p_{\alpha,\sigma}\circ p_{\beta,\tau}=\sum_{\substack{\gamma_{i,j}%
\in\mathbb{N}\text{ for all }i\in\left[  k\right]  \text{ and }j\in\left[
\ell\right]  ;\\\gamma_{i,1}+\gamma_{i,2}+\cdots+\gamma_{i,\ell}=\alpha
_{i}\text{ for all }i\in\left[  k\right]  ;\\\gamma_{1,j}+\gamma_{2,j}%
+\cdots+\gamma_{k,j}=\beta_{j}\text{ for all }j\in\left[  \ell\right]
}}p_{\left(  \gamma_{1,1},\gamma_{1,2},\ldots,\gamma_{k,\ell}\right)
,\tau\left[  \sigma\right]  }.
\]
Here, $\left(  \gamma_{1,1},\gamma_{1,2},\ldots,\gamma_{k,\ell}\right)  $
denotes the $k\ell$-tuple consisting of all $k\ell$ numbers $\gamma_{i,j}$
(for all $i\in\left[  k\right]  $ and $j\in\left[  \ell\right]  $) listed in
the order of lexicographically increasing pairs $\left(  i,j\right)  $ (i.e.,
the number $\gamma_{i,j}$ comes before $\gamma_{u,v}$ if and only if either
$i<u$ or $\left(  i=u\text{ and }j<v\right)  $).
\end{theorem}

We note that the sum in Theorem \ref{thm.sol-mac-1} can be viewed as a sum
over all $k\times\ell$-matrices $\left(  \gamma_{i,j}\right)  _{i\in\left[
k\right]  ,\ j\in\left[  \ell\right]  }\in\mathbb{N}^{k\times\ell}$ with row
sums $\alpha$ (that is, the sum of all entries in the $i$-th row of the matrix
equals $\alpha_{i}$) and column sums $\beta$. Such matrices are known as
\emph{contingency tables} with marginal distributions $\alpha$ and $\beta$
(see, e.g., \cite[\S 1]{LyuPak20}), and have found ample uses in combinatorics
(e.g., \cite[Corollary 1.3.11]{JamKer81} or \cite[Definition 4.3.4]{GriRei}).

\begin{example}
Let $s_{1}\in\mathfrak{S}_{2}$ be the transposition that swaps $1$ with $2$.
Then, the permutation $s_{1}\left[  s_{1}\right]  \in\mathfrak{S}_{4}$ sends
$1,2,3,4$ to $4,2,3,1$.

Let $\left(  a,b\right)  $ and $\left(  c,d\right)  $ be two weak compositions
in $\mathbb{N}^{2}$. Then, Theorem \ref{thm.sol-mac-1} says that
\[
p_{\left(  a,b\right)  ,s_{1}}\circ p_{\left(  c,d\right)  ,s_{1}}%
=\sum_{\substack{\gamma_{1,1},\gamma_{1,2},\gamma_{2,1},\gamma_{2,2}%
\in\mathbb{N};\\\gamma_{1,1}+\gamma_{1,2}=a;\ \gamma_{2,1}+\gamma
_{2,2}=b;\\\gamma_{1,1}+\gamma_{2,1}=c;\ \gamma_{1,2}+\gamma_{2,2}%
=d}}p_{\left(  \gamma_{1,1},\gamma_{1,2},\gamma_{2,1},\gamma_{2,2}\right)
,s_{1}\left[  s_{1}\right]  }.
\]
This is a sum over all $2\times2$-matrices $\left(  \gamma_{i,j}\right)
_{i\in\left[  2\right]  ,\ j\in\left[  2\right]  }\in\mathbb{N}^{2\times2}$
with row sums $\left(  a,b\right)  $ and column sums $\left(  c,d\right)  $.
How this sum looks like depends on whether $a+b$ equals $c+d$ or not:

\begin{itemize}
\item If $a+b\neq c+d$, then there are no such matrices, and therefore the sum
is empty. Thus,%
\[
p_{\left(  a,b\right)  ,s_{1}}\circ p_{\left(  c,d\right)  ,s_{1}%
}=0\ \ \ \ \ \ \ \ \ \ \text{in this case.}%
\]
This is not surprising, since the image of the map $p_{\left(  c,d\right)
,s_{1}}$ is contained in the graded component $H_{c+d}$ of $H$, whereas the
map $p_{\left(  a,b\right)  ,s_{1}}$ is $0$ on this component.

\item If $a+b=c+d$, then these matrices are precisely the matrices of the form
$\left(
\begin{array}
[c]{cc}%
i & a-i\\
c-i & i-g
\end{array}
\right)  $, where $g=c-b=a-d$ and where $i\in\mathbb{Z}$ satisfies
$\max\left\{  0,g\right\}  \leq i\leq\min\left\{  a,c\right\}  $. Thus, our
above formula becomes%
\[
p_{\left(  a,b\right)  ,s_{1}}\circ p_{\left(  c,d\right)  ,s_{1}}%
=\sum_{i=\max\left\{  0,g\right\}  }^{\min\left\{  a,c\right\}  }p_{\left(
i,\ a-i,\ c-i,\ i-g\right)  ,\ s_{1}\left[  s_{1}\right]  }%
\ \ \ \ \ \ \ \ \ \ \text{where }g=c-b=a-d.
\]
For example, for $a=b=c=d=1$, this simplifies to%
\[
p_{\left(  1,1\right)  ,s_{1}}\circ p_{\left(  1,1\right)  ,s_{1}}=p_{\left(
0,1,1,0\right)  ,s_{1}\left[  s_{1}\right]  }+p_{\left(  1,0,0,1\right)
,s_{1}\left[  s_{1}\right]  }.
\]

\end{itemize}

All this is not hard to check by hand.
\end{example}

\begin{remark}
A bialgebra without a grading can always be interpreted as a graded bialgebra
$H$ concentrated in degree $0$ (that is, satisfying $H_{0}=H$ and $H_{k}=0$
for all $k>0$). Thus, Theorem \ref{thm.sol-mac-1} can be applied to bialgebras
without a grading. If $H$ is such a bialgebra, then $p_{\alpha,\sigma}=0$
whenever the weak composition $\alpha$ has any nonzero entry. Thus, the claim
of Theorem \ref{thm.sol-mac-1} can be greatly simplified in this case. The
resulting claim is%
\[
p_{\mathbf{0},\sigma}\circ p_{\mathbf{0},\tau}=p_{\mathbf{0},\tau\left[
\sigma\right]  },
\]
where $\mathbf{0}$ means the tuple $\left(  0,0,\ldots,0\right)  $ of
appropriate length. This is precisely the equality $\Psi^{\left(
n,\sigma\right)  }\circ\Psi^{\left(  m,\tau\right)  }=\Psi^{\left(
nm,\Phi\left(  \sigma,\tau\right)  \right)  }$ in Pirashvili's \cite[\S 1]%
{Pirash02} and the Proposition in \cite[\S 1.3]{KaSoZh06} by Kashina,
Sommerh\"{a}user and Zhu. (Note that the $\Phi\left(  \sigma,\tau\right)  $ in
\cite[Proposition 5.3]{Pirash02} is exactly our $\sigma\left[  \tau\right]  $.
I do not know what the apparent reversal in the roles of $\sigma$ and $\tau$
is due to, but I have verified my version of the result.)
\end{remark}

\subsection{The composition formula: lemmas on tensors}

As preparation for the proof of Theorem \ref{thm.sol-mac-1}, we shall first
work out how the different kinds of operators composed in
(\ref{eq.pas.formal-def}) (iterated multiplications, comultiplications,
projections, permutations) can be \textquotedblleft commuted\textquotedblright%
\ past each other. We begin with the simplest operators: projections and
permutations. These commutation formulas have nothing to do with bialgebras,
and would equally make sense for any graded module instead of $H$.

Recall that $\mathbb{N}^{k}$ (for any given $k\in\mathbb{N}$) is the set of
all weak compositions of length $k$. The symmetric group $\mathfrak{S}_{k}$
acts from the right on this set $\mathbb{N}^{k}$ by permuting the entries: For
any $\left(  \gamma_{1},\gamma_{2},\ldots,\gamma_{k}\right)  \in\mathbb{N}%
^{k}$ and $\pi\in\mathfrak{S}_{k}$, we have%
\begin{equation}
\left(  \gamma_{1},\gamma_{2},\ldots,\gamma_{k}\right)  \cdot\pi=\left(
\gamma_{\pi\left(  1\right)  },\gamma_{\pi\left(  2\right)  },\ldots
,\gamma_{\pi\left(  k\right)  }\right)  . \label{eq.gammapi}%
\end{equation}
This action has the following property:

\begin{lemma}
\label{lem.Ppi}For any $\pi\in\mathfrak{S}_{k}$ and $\gamma\in\mathbb{N}^{k}$,
we have%
\begin{equation}
P_{\gamma}\circ\pi=\pi\circ P_{\gamma\cdot\pi} \label{pf.thm.sol-mac-1.Ppi}%
\end{equation}
and%
\begin{equation}
\pi\circ P_{\gamma}=P_{\gamma\cdot\pi^{-1}}\circ\pi.
\label{pf.thm.sol-mac-1.piP}%
\end{equation}

\end{lemma}

\begin{proof}
Let $\pi\in\mathfrak{S}_{k}$ and $\gamma\in\mathbb{N}^{k}$. Write $\gamma$ as
$\gamma=\left(  \gamma_{1},\gamma_{2},\ldots,\gamma_{k}\right)  $. Then, the
map $P_{\gamma}\circ\pi:H^{\otimes k}\rightarrow H^{\otimes k}$ sends each
pure tensor $h_{1}\otimes h_{2}\otimes\cdots\otimes h_{k}$ to%
\begin{align*}
&  P_{\gamma}\left(  \pi\left(  h_{1}\otimes h_{2}\otimes\cdots\otimes
h_{k}\right)  \right) \\
&  =P_{\gamma}\left(  h_{\pi^{-1}\left(  1\right)  }\otimes h_{\pi^{-1}\left(
2\right)  }\otimes\cdots\otimes h_{\pi^{-1}\left(  k\right)  }\right) \\
&  =p_{\gamma_{1}}\left(  h_{\pi^{-1}\left(  1\right)  }\right)  \otimes
p_{\gamma_{2}}\left(  h_{\pi^{-1}\left(  2\right)  }\right)  \otimes
\cdots\otimes p_{\gamma_{k}}\left(  h_{\pi^{-1}\left(  k\right)  }\right)  ,
\end{align*}
whereas the map $\pi\circ P_{\gamma\cdot\pi}:H^{\otimes k}\rightarrow
H^{\otimes k}$ sends it to%
\begin{align*}
&  \pi\left(  P_{\gamma\cdot\pi}\left(  h_{1}\otimes h_{2}\otimes\cdots\otimes
h_{k}\right)  \right) \\
&  =\pi\left(  p_{\gamma_{\pi\left(  1\right)  }}\left(  h_{1}\right)  \otimes
p_{\gamma_{\pi\left(  2\right)  }}\left(  h_{2}\right)  \otimes\cdots\otimes
p_{\gamma_{\pi\left(  k\right)  }}\left(  h_{k}\right)  \right) \\
&  \ \ \ \ \ \ \ \ \ \ \ \ \ \ \ \ \ \ \ \ \left(
\begin{array}
[c]{c}%
\text{since }\gamma=\left(  \gamma_{1},\gamma_{2},\ldots,\gamma_{k}\right) \\
\text{entails }\gamma\cdot\pi=\left(  \gamma_{\pi\left(  1\right)  }%
,\gamma_{\pi\left(  2\right)  },\ldots,\gamma_{\pi\left(  k\right)  }\right)
\end{array}
\right) \\
&  =p_{\gamma_{\pi\left(  \pi^{-1}\left(  1\right)  \right)  }}\left(
h_{\pi^{-1}\left(  1\right)  }\right)  \otimes p_{\gamma_{\pi\left(  \pi
^{-1}\left(  2\right)  \right)  }}\left(  h_{\pi^{-1}\left(  2\right)
}\right)  \otimes\cdots\otimes p_{\gamma_{\pi\left(  \pi^{-1}\left(  k\right)
\right)  }}\left(  h_{\pi^{-1}\left(  k\right)  }\right) \\
&  =p_{\gamma_{1}}\left(  h_{\pi^{-1}\left(  1\right)  }\right)  \otimes
p_{\gamma_{2}}\left(  h_{\pi^{-1}\left(  2\right)  }\right)  \otimes
\cdots\otimes p_{\gamma_{k}}\left(  h_{\pi^{-1}\left(  k\right)  }\right)  ,
\end{align*}
which is the same result. Thus, the linear maps $P_{\gamma}\circ\pi$ and
$\pi\circ P_{\gamma\cdot\pi}$ agree on each pure tensor, and therefore are
identical. This proves (\ref{pf.thm.sol-mac-1.Ppi}).

Applying (\ref{pf.thm.sol-mac-1.Ppi}) to $\gamma\cdot\pi^{-1}$ instead of
$\gamma$, we obtain%
\[
P_{\gamma\cdot\pi^{-1}}\circ\pi=\pi\circ\underbrace{P_{\gamma\cdot\pi
^{-1}\cdot\pi}}_{=P_{\gamma}}=\pi\circ P_{\gamma}.
\]
This proves (\ref{pf.thm.sol-mac-1.piP}).
\end{proof}

Moreover, the following holds:

\begin{lemma}
\label{lem.PP}Let $\alpha,\beta\in\mathbb{N}^{k}$ be two weak compositions.
Then,
\[
P_{\alpha}\circ P_{\beta}=%
\begin{cases}
P_{\alpha}, & \text{if }\alpha=\beta;\\
0, & \text{if }\alpha\neq\beta.
\end{cases}
\]

\end{lemma}

\begin{proof}
This follows easily from the fact that the projections $p_{k}:H\rightarrow H$
satisfy $p_{a}\circ p_{a}=p_{a}$ and $p_{a}\circ p_{b}=0$ for all $a\neq b$.
\end{proof}

The remainder of this section is devoted to lemmas whose purpose is to factor
the permutation $\tau\left[  \sigma\right]  $ (which appears in Theorem
\ref{thm.sol-mac-1}) into simpler permutations (Lemma \ref{lem.sitaze}) and to
establish how these simpler permutations interact with tensors and linear
maps. We begin by introducing one of these simpler permutations, which we call
the \emph{Zolotarev shuffle}:

\begin{lemma}
\label{lem.zol1}Let $k,\ell\in\mathbb{N}$. Let $\zeta\in\mathfrak{S}_{k\ell}$
be the permutation of $\left[  k\ell\right]  $ that sends each $k\left(
j-1\right)  +i$ (with $i\in\left[  k\right]  $ and $j\in\left[  \ell\right]
$) to $\ell\left(  i-1\right)  +j$. This is called the \emph{Zolotarev
shuffle} (and appears, e.g., as $\nu^{-1}\mu$ in \cite{Rousse94}).\footnotemark

Let $\left(  h_{i,j}\right)  _{i\in\left[  \ell\right]  ,\ j\in\left[
k\right]  }\in H^{\ell\times k}$ be any $\ell\times k$-matrix over $H$. Then,
\begin{align}
&  \zeta\left(  h_{1,1}\otimes h_{1,2}\otimes\cdots\otimes h_{1,k}\right.
\nonumber\\
&  \ \ \ \ \ \ \ \ \ \ \left.  \otimes\ h_{2,1}\otimes h_{2,2}\otimes
\cdots\otimes h_{2,k}\right. \nonumber\\
&  \ \ \ \ \ \ \ \ \ \ \left.  \otimes\ \cdots\right. \nonumber\\
&  \ \ \ \ \ \ \ \ \ \ \left.  \otimes\ h_{\ell,1}\otimes h_{\ell,2}%
\otimes\cdots\otimes h_{\ell,k}\right) \nonumber\\
&  =h_{1,1}\otimes h_{2,1}\otimes\cdots\otimes h_{\ell,1}\nonumber\\
&  \ \ \ \ \ \ \ \ \ \ \otimes h_{1,2}\otimes h_{2,2}\otimes\cdots\otimes
h_{\ell,2}\nonumber\\
&  \ \ \ \ \ \ \ \ \ \ \otimes\cdots\nonumber\\
&  \ \ \ \ \ \ \ \ \ \ \otimes h_{1,k}\otimes h_{2,k}\otimes\cdots\otimes
h_{\ell,k}. \label{pf.thm.sol-mac-1.z-prod}%
\end{align}
(Here, the $h_{i,j}$ on the left-hand side appear in the order of
lexicographically increasing pairs $\left(  i,j\right)  $, whereas the
$h_{i,j}$ on the right-hand side appear in the order of lexicographically
increasing pairs $\left(  j,i\right)  $.)
\end{lemma}

\footnotetext{To see that this $\zeta$ is well-defined, we can argue as
follows:
\par
\begin{itemize}
\item Each element of $\left[  k\ell\right]  $ can be uniquely written as
$k\left(  j-1\right)  +i$ with $i\in\left[  k\right]  $ and $j\in\left[
\ell\right]  $. Moreover, $\ell\left(  i-1\right)  +j\in\left[  k\ell\right]
$ for any such $i$ and $j$. Thus, $\zeta$ is a well-defined map from $\left[
k\ell\right]  $ to $\left[  k\ell\right]  $.
\par
\item Each element of $\left[  k\ell\right]  $ can be uniquely written as
$\ell\left(  i-1\right)  +j$ with $i\in\left[  k\right]  $ and $j\in\left[
\ell\right]  $. Thus, $\zeta$ is bijective, i.e., a permutation of $\left[
k\ell\right]  $.
\end{itemize}
}We note that the Zolotarev shuffle $\zeta$ in Lemma \ref{lem.zol1} can also
be described as $\left(  \operatorname*{id}\nolimits_{\left[  \ell\right]
}\left[  \operatorname*{id}\nolimits_{\left[  k\right]  }\right]  \right)
^{-1}$ using the notation $\tau\left[  \sigma\right]  $ from Definition
\ref{def.tausigma}. This follows easily from Lemma \ref{lem.sitaze} proved
further below.

\begin{proof}
[Proof of Lemma \ref{lem.zol1}.]Define two maps $\lambda$ and $\rho$ from
$\left[  k\right]  \times\left[  \ell\right]  $ to $\left[  k\ell\right]  $ by
setting%
\[
\lambda\left(  i,j\right)  :=k\left(  j-1\right)  +i\in\left[  k\ell\right]
\ \ \ \ \ \ \ \ \ \ \text{and}\ \ \ \ \ \ \ \ \ \ \rho\left(  i,j\right)
:=\ell\left(  i-1\right)  +j\in\left[  k\ell\right]
\]
for all $\left(  i,j\right)  \in\left[  k\right]  \times\left[  \ell\right]
$. These maps $\lambda$ and $\rho$ are bijections. In fact, $\rho$ sends each
pair $\left(  i,j\right)  $ to the position of $\left(  i,j\right)  $ in the
lexicographically ordered Cartesian product $\left[  k\right]  \times\left[
\ell\right]  $, whereas $\lambda$ sends each pair $\left(  i,j\right)  $ to
the position of $\left(  j,i\right)  $ in the lexicographically ordered
Cartesian product $\left[  \ell\right]  \times\left[  k\right]  $. (We already
observed all this in Remark \ref{rmk.tausigma.wd}.)

Recall that all $i\in\left[  k\right]  $ and $j\in\left[  \ell\right]  $
satisfy $\zeta\left(  k\left(  j-1\right)  +i\right)  =\ell\left(  i-1\right)
+j$ (by the definition of $\zeta$). In other words, all $i\in\left[  k\right]
$ and $j\in\left[  \ell\right]  $ satisfy $\zeta\left(  \lambda\left(
i,j\right)  \right)  =\rho\left(  i,j\right)  $ (since $\lambda\left(
i,j\right)  =k\left(  j-1\right)  +i$ and $\rho\left(  i,j\right)
=\ell\left(  i-1\right)  +j$). In other words,
\begin{equation}
\zeta\circ\lambda=\rho. \label{pf.lem.zol1.zl=r}%
\end{equation}

Set $g_{\left(  j,i\right)  }:=h_{i,j}$ for all $\left(  j,i\right)
\in\left[  k\right]  \times\left[  \ell\right]  $. Recall that the bijection
$\lambda$ sends each pair $\left(  i,j\right)  $ to the position of $\left(
j,i\right)  $ in the lexicographically ordered Cartesian product $\left[
\ell\right]  \times\left[  k\right]  $. Hence, the list $\left(  \lambda
^{-1}\left(  1\right)  ,\lambda^{-1}\left(  2\right)  ,\ldots,\lambda
^{-1}\left(  k\ell\right)  \right)  $ consists of all the $k\ell$ pairs
$\left(  i,j\right)  \in\left[  k\right]  \times\left[  \ell\right]  $ in the
order of (lexicographically) increasing pairs $\left(  j,i\right)  $. In other
words,%
\begin{align*}
\left(  \lambda^{-1}\left(  1\right)  ,\lambda^{-1}\left(  2\right)
,\ldots,\lambda^{-1}\left(  k\ell\right)  \right)   &  =\left(  \left(
1,1\right)  ,\ \left(  2,1\right)  ,\ \ldots,\ \left(  k,1\right)  ,\right. \\
&  \ \ \ \ \ \ \ \ \ \ \left.  \left(  1,2\right)  ,\ \left(  2,2\right)
,\ \ldots,\ \left(  k,2\right)  ,\right. \\
&  \ \ \ \ \ \ \ \ \ \ \left.  \ldots,\right. \\
&  \ \ \ \ \ \ \ \ \ \ \left.  \left(  1,\ell\right)  ,\ \left(
2,\ell\right)  ,\ \ldots,\ \left(  k,\ell\right)  \right)  .
\end{align*}
Thus,%
\begin{align*}
g_{\lambda^{-1}\left(  1\right)  }\otimes g_{\lambda^{-1}\left(  2\right)
}\otimes\cdots\otimes g_{\lambda^{-1}\left(  k\ell\right)  }  &  =g_{\left(
1,1\right)  }\otimes g_{\left(  2,1\right)  }\otimes\cdots\otimes g_{\left(
k,1\right)  }\\
&  \ \ \ \ \ \ \ \ \ \ \otimes g_{\left(  1,2\right)  }\otimes g_{\left(
2,2\right)  }\otimes\cdots\otimes g_{\left(  k,2\right)  }\\
&  \ \ \ \ \ \ \ \ \ \ \otimes\cdots\\
&  \ \ \ \ \ \ \ \ \ \ \otimes g_{\left(  1,\ell\right)  }\otimes g_{\left(
2,\ell\right)  }\otimes\cdots\otimes g_{\left(  k,\ell\right)  }\\
&  =h_{1,1}\otimes h_{1,2}\otimes\cdots\otimes h_{1,k}\\
&  \ \ \ \ \ \ \ \ \ \ \otimes h_{2,1}\otimes h_{2,2}\otimes\cdots\otimes
h_{2,k}\\
&  \ \ \ \ \ \ \ \ \ \ \otimes\cdots\\
&  \ \ \ \ \ \ \ \ \ \ \otimes h_{\ell,1}\otimes h_{\ell,2}\otimes
\cdots\otimes h_{\ell,k}%
\end{align*}
(since $g_{\left(  j,i\right)  }=h_{i,j}$ for all $j$ and $i$). By a similar
argument, we obtain%
\begin{align*}
g_{\rho^{-1}\left(  1\right)  }\otimes g_{\rho^{-1}\left(  2\right)  }%
\otimes\cdots\otimes g_{\rho^{-1}\left(  k\ell\right)  }  &  =h_{1,1}\otimes
h_{2,1}\otimes\cdots\otimes h_{\ell,1}\\
&  \ \ \ \ \ \ \ \ \ \ \otimes h_{1,2}\otimes h_{2,2}\otimes\cdots\otimes
h_{\ell,2}\\
&  \ \ \ \ \ \ \ \ \ \ \otimes\cdots\\
&  \ \ \ \ \ \ \ \ \ \ \otimes h_{1,k}\otimes h_{2,k}\otimes\cdots\otimes
h_{\ell,k}.
\end{align*}
In view of these two equalities, we must show that%
\[
\zeta\left(  g_{\lambda^{-1}\left(  1\right)  }\otimes g_{\lambda^{-1}\left(
2\right)  }\otimes\cdots\otimes g_{\lambda^{-1}\left(  k\ell\right)  }\right)
=g_{\rho^{-1}\left(  1\right)  }\otimes g_{\rho^{-1}\left(  2\right)  }%
\otimes\cdots\otimes g_{\rho^{-1}\left(  k\ell\right)  }.
\]
Since%
\[
\zeta\left(  g_{\lambda^{-1}\left(  1\right)  }\otimes g_{\lambda^{-1}\left(
2\right)  }\otimes\cdots\otimes g_{\lambda^{-1}\left(  k\ell\right)  }\right)
=g_{\lambda^{-1}\left(  \zeta^{-1}\left(  1\right)  \right)  }\otimes
g_{\lambda^{-1}\left(  \zeta^{-1}\left(  2\right)  \right)  }\otimes
\cdots\otimes g_{\lambda^{-1}\left(  \zeta^{-1}\left(  k\ell\right)  \right)
},
\]
this is equivalent to showing that
\[
g_{\lambda^{-1}\left(  \zeta^{-1}\left(  1\right)  \right)  }\otimes
g_{\lambda^{-1}\left(  \zeta^{-1}\left(  2\right)  \right)  }\otimes
\cdots\otimes g_{\lambda^{-1}\left(  \zeta^{-1}\left(  k\ell\right)  \right)
}=g_{\rho^{-1}\left(  1\right)  }\otimes g_{\rho^{-1}\left(  2\right)
}\otimes\cdots\otimes g_{\rho^{-1}\left(  k\ell\right)  }.
\]
Thus, we need to check that $\lambda^{-1}\left(  \zeta^{-1}\left(  q\right)
\right)  =\rho^{-1}\left(  q\right)  $ for each $q\in\left[  k\ell\right]  $.
In other words, we need to check that $\lambda^{-1}\circ\zeta^{-1}=\rho^{-1}$.
But this follows from (\ref{pf.lem.zol1.zl=r}), since $\lambda^{-1}\circ
\zeta^{-1}=\left(  \zeta\circ\lambda\right)  ^{-1}=\rho^{-1}$ (by
(\ref{pf.lem.zol1.zl=r})). Hence, Lemma \ref{lem.zol1} is proven.
\end{proof}

\begin{lemma}
\label{lem.pixl}Let $k,\ell\in\mathbb{N}$. Let $\sigma\in\mathfrak{S}_{k}$.
Let $\sigma^{\times\ell}$ denote the permutation in $\mathfrak{S}_{k\ell}$
that sends each $\ell\left(  i-1\right)  +j$ (with $i\in\left[  k\right]  $
and $j\in\left[  \ell\right]  $) to $\ell\left(  \sigma\left(  i\right)
-1\right)  +j$.

Let $f:H^{\otimes\ell}\rightarrow H$ be any $\mathbf{k}$-linear map. Then,%
\[
\sigma\circ f^{\otimes k}=f^{\otimes k}\circ\sigma^{\times\ell}%
\]
(as maps from $H^{\otimes k\ell}$ to $H^{\otimes k}$).
\end{lemma}

\begin{proof}
We begin with an even more basic claim:

\begin{statement}
\textit{Claim 1:} Let $a_{1},a_{2},\ldots,a_{k}\in H^{\otimes\ell}$. Then,%
\[
\sigma^{\times\ell}\left(  a_{1}\otimes a_{2}\otimes\cdots\otimes
a_{k}\right)  =a_{\sigma^{-1}\left(  1\right)  }\otimes a_{\sigma^{-1}\left(
2\right)  }\otimes\cdots\otimes a_{\sigma^{-1}\left(  k\right)  }%
\]
in $H^{\otimes k\ell}$. (Here, we are identifying $\left(  H^{\otimes\ell
}\right)  ^{\otimes k}$ with $H^{\otimes k\ell}$ in the obvious way.)
\end{statement}

\begin{proof}
[Proof of Claim 1.]Let $\omega=\sigma^{\times\ell}$.

The equality we need to prove depends linearly on each of $a_{1},a_{2}%
,\ldots,a_{k}$. Hence, we can WLOG assume that each of $a_{1},a_{2}%
,\ldots,a_{k}$ is a pure tensor. Assume this. Thus,%
\begin{align}
a_{1}  &  =g_{1}\otimes g_{2}\otimes\cdots\otimes g_{\ell}%
,\label{pf.lem.pixl.a1=}\\
a_{2}  &  =g_{\ell+1}\otimes g_{\ell+2}\otimes\cdots\otimes g_{2\ell
},\label{pf.lem.pixl.a2=}\\
&  \ldots,\nonumber\\
a_{k}  &  =g_{\left(  k-1\right)  \ell+1}\otimes g_{\left(  k-1\right)
\ell+2}\otimes\cdots\otimes g_{k\ell} \label{pf.lem.pixl.ak=}%
\end{align}
for some $g_{1},g_{2},\ldots,g_{k\ell}\in H$. Consider these $g_{1}%
,g_{2},\ldots,g_{k\ell}$. Thus,%
\[
a_{1}\otimes a_{2}\otimes\cdots\otimes a_{k}=g_{1}\otimes g_{2}\otimes
\cdots\otimes g_{k\ell}%
\]
(as we can see by tensoring together the equalities (\ref{pf.lem.pixl.a1=}),
(\ref{pf.lem.pixl.a2=}), $\ldots$, (\ref{pf.lem.pixl.ak=})). Applying the
permutation $\sigma^{\times\ell}\in\mathfrak{S}_{k\ell}$ to both sides of this
equality, we obtain%
\begin{align}
&  \sigma^{\times\ell}\left(  a_{1}\otimes a_{2}\otimes\cdots\otimes
a_{k}\right) \nonumber\\
&  =\underbrace{\sigma^{\times\ell}}_{=\omega}\left(  g_{1}\otimes
g_{2}\otimes\cdots\otimes g_{k\ell}\right) \nonumber\\
&  =\omega\left(  g_{1}\otimes g_{2}\otimes\cdots\otimes g_{k\ell}\right)
\nonumber\\
&  =g_{\omega^{-1}\left(  1\right)  }\otimes g_{\omega^{-1}\left(  2\right)
}\otimes\cdots\otimes g_{\omega^{-1}\left(  k\ell\right)  }.
\label{pf.lem.pixl.1}%
\end{align}

We shall now prove that
\begin{align}
&  a_{\sigma^{-1}\left(  1\right)  }\otimes a_{\sigma^{-1}\left(  2\right)
}\otimes\cdots\otimes a_{\sigma^{-1}\left(  k\right)  }\nonumber\\
&  =g_{\omega^{-1}\left(  1\right)  }\otimes g_{\omega^{-1}\left(  2\right)
}\otimes\cdots\otimes g_{\omega^{-1}\left(  k\ell\right)  }.
\label{pf.lem.pixl.goal2}%
\end{align}
Indeed, in order to prove this identity, it clearly suffices to show that%
\begin{equation}
a_{\sigma^{-1}\left(  i\right)  }=g_{\omega^{-1}\left(  \ell\left(
i-1\right)  +1\right)  }\otimes g_{\omega^{-1}\left(  \ell\left(  i-1\right)
+2\right)  }\otimes\cdots\otimes g_{\omega^{-1}\left(  \ell i\right)  }
\label{pf.lem.pixl.goal2.pf.1}%
\end{equation}
for each $i\in\left[  k\right]  $ (because then, tensoring the equalities
(\ref{pf.lem.pixl.goal2.pf.1}) together for all $i\in\left[  k\right]  $ will
yield (\ref{pf.lem.pixl.goal2})). But this is easy: Let $i\in\left[  k\right]
$. Then, the definition of $a_{\sigma^{-1}\left(  i\right)  }$ yields%
\begin{equation}
a_{\sigma^{-1}\left(  i\right)  }=g_{\ell\left(  \sigma^{-1}\left(  i\right)
-1\right)  +1}\otimes g_{\ell\left(  \sigma^{-1}\left(  i\right)  -1\right)
+2}\otimes\cdots\otimes g_{\ell\left(  \sigma^{-1}\left(  i\right)  \right)
}. \label{pf.lem.pixl.goal2.pf.2}%
\end{equation}
But each $j\in\left[  \ell\right]  $ satisfies%
\[
\ell\left(  \sigma^{-1}\left(  i\right)  -1\right)  +j=\omega^{-1}\left(
\ell\left(  i-1\right)  +j\right)
\]
(since $\omega=\sigma^{\times\ell}$ was defined to send $\ell\left(
\sigma^{-1}\left(  i\right)  -1\right)  +j$ to $\ell\left(  \underbrace{\sigma
\left(  \sigma^{-1}\left(  i\right)  \right)  }_{=i}-\,1\right)
+j=\ell\left(  i-1\right)  +j$) and thus $g_{\ell\left(  \sigma^{-1}\left(
i\right)  -1\right)  +j}=g_{\omega^{-1}\left(  \ell\left(  i-1\right)
+j\right)  }$. Hence, the right-hand side of (\ref{pf.lem.pixl.goal2.pf.2})
equals the right-hand side of (\ref{pf.lem.pixl.goal2.pf.1}). Therefore,
(\ref{pf.lem.pixl.goal2.pf.1}) follows from (\ref{pf.lem.pixl.goal2.pf.2})
(since these two equalities have the same left-hand side).

Forget that we fixed $i$. We thus have proved (\ref{pf.lem.pixl.goal2.pf.1})
for each $i\in\left[  k\right]  $. As explained, this proves
(\ref{pf.lem.pixl.goal2}).

Comparing (\ref{pf.lem.pixl.1}) with (\ref{pf.lem.pixl.goal2}), we find%
\[
\sigma^{\times\ell}\left(  a_{1}\otimes a_{2}\otimes\cdots\otimes
a_{k}\right)  =a_{\sigma^{-1}\left(  1\right)  }\otimes a_{\sigma^{-1}\left(
2\right)  }\otimes\cdots\otimes a_{\sigma^{-1}\left(  k\right)  }.
\]
This proves Claim 1.
\end{proof}

The rest is easy: Let $\mathbf{a}=a_{1}\otimes a_{2}\otimes\cdots\otimes
a_{k}$ be a pure tensor in $\left(  H^{\otimes\ell}\right)  ^{\otimes k}$
(with $a_{1},a_{2},\ldots,a_{k}\in H^{\otimes\ell}$). Then,%
\begin{align*}
\left(  \sigma\circ f^{\otimes k}\right)  \left(  \mathbf{a}\right)   &
=\left(  \sigma\circ f^{\otimes k}\right)  \left(  a_{1}\otimes a_{2}%
\otimes\cdots\otimes a_{k}\right) \\
&  =\sigma\left(  \underbrace{f^{\otimes k}\left(  a_{1}\otimes a_{2}%
\otimes\cdots\otimes a_{k}\right)  }_{=f\left(  a_{1}\right)  \otimes f\left(
a_{2}\right)  \otimes\cdots\otimes f\left(  a_{k}\right)  }\right) \\
&  =\sigma\left(  f\left(  a_{1}\right)  \otimes f\left(  a_{2}\right)
\otimes\cdots\otimes f\left(  a_{k}\right)  \right) \\
&  =f\left(  a_{\sigma^{-1}\left(  1\right)  }\right)  \otimes f\left(
a_{\sigma^{-1}\left(  2\right)  }\right)  \otimes\cdots\otimes f\left(
a_{\sigma^{-1}\left(  k\right)  }\right)
\end{align*}
and%
\begin{align*}
\left(  f^{\otimes k}\circ\sigma^{\times\ell}\right)  \left(  \mathbf{a}%
\right)   &  =\left(  f^{\otimes k}\circ\sigma^{\times\ell}\right)  \left(
a_{1}\otimes a_{2}\otimes\cdots\otimes a_{k}\right) \\
&  =f^{\otimes k}\left(  \underbrace{\sigma^{\times\ell}\left(  a_{1}\otimes
a_{2}\otimes\cdots\otimes a_{k}\right)  }_{\substack{=a_{\sigma^{-1}\left(
1\right)  }\otimes a_{\sigma^{-1}\left(  2\right)  }\otimes\cdots\otimes
a_{\sigma^{-1}\left(  k\right)  }\\\text{(by Claim 1)}}}\right) \\
&  =f^{\otimes k}\left(  a_{\sigma^{-1}\left(  1\right)  }\otimes
a_{\sigma^{-1}\left(  2\right)  }\otimes\cdots\otimes a_{\sigma^{-1}\left(
k\right)  }\right) \\
&  =f\left(  a_{\sigma^{-1}\left(  1\right)  }\right)  \otimes f\left(
a_{\sigma^{-1}\left(  2\right)  }\right)  \otimes\cdots\otimes f\left(
a_{\sigma^{-1}\left(  k\right)  }\right)  .
\end{align*}
Comparing these two equalities, we find $\left(  \sigma\circ f^{\otimes
k}\right)  \left(  \mathbf{a}\right)  =\left(  f^{\otimes k}\circ
\sigma^{\times\ell}\right)  \left(  \mathbf{a}\right)  $.

Forget that we fixed $\mathbf{a}$. We thus have proved that $\left(
\sigma\circ f^{\otimes k}\right)  \left(  \mathbf{a}\right)  =\left(
f^{\otimes k}\circ\sigma^{\times\ell}\right)  \left(  \mathbf{a}\right)  $ for
each pure tensor $\mathbf{a}$ in $\left(  H^{\otimes\ell}\right)  ^{\otimes
k}$. In other words, the two maps $\sigma\circ f^{\otimes k}$ and $f^{\otimes
k}\circ\sigma^{\times\ell}$ agree on all pure tensors in $\left(
H^{\otimes\ell}\right)  ^{\otimes k}$. Since these two maps are $\mathbf{k}%
$-linear (and since the pure tensors span $\left(  H^{\otimes\ell}\right)
^{\otimes k}$), this entails that they must be identical. In other words,
$\sigma\circ f^{\otimes k}=f^{\otimes k}\circ\sigma^{\times\ell}$. This proves
Lemma \ref{lem.pixl}.
\end{proof}

\begin{lemma}
\label{lem.pikx}Let $k,\ell\in\mathbb{N}$. Let $\tau\in\mathfrak{S}_{\ell}$.
Let $\tau^{k\times}$ denote the permutation in $\mathfrak{S}_{k\ell}$ that
sends each $k\left(  j-1\right)  +i$ (with $i\in\left[  k\right]  $ and
$j\in\left[  \ell\right]  $) to $k\left(  \tau\left(  j\right)  -1\right)  +i$.

Let $f:H\rightarrow H^{\otimes k}$ be any $\mathbf{k}$-linear map. Then,%
\[
f^{\otimes\ell}\circ\tau=\tau^{k\times}\circ f^{\otimes\ell}%
\]
(as maps from $H^{\otimes\ell}$ to $H^{\otimes k\ell}$).
\end{lemma}

\begin{proof}
We begin with an even more basic claim:

\begin{statement}
\textit{Claim 1:} Let $a_{1},a_{2},\ldots,a_{\ell}\in H^{\otimes k}$. Then,%
\[
\tau^{k\times}\left(  a_{1}\otimes a_{2}\otimes\cdots\otimes a_{\ell}\right)
=a_{\tau^{-1}\left(  1\right)  }\otimes a_{\tau^{-1}\left(  2\right)  }%
\otimes\cdots\otimes a_{\tau^{-1}\left(  \ell\right)  }%
\]
in $H^{\otimes k\ell}$. (Here, we are identifying $\left(  H^{\otimes
k}\right)  ^{\otimes\ell}$ with $H^{\otimes k\ell}$ in the obvious way.)
\end{statement}

\begin{proof}
[Proof of Claim 1.]Analogous to the Claim 1 in our above proof of Lemma
\ref{lem.pixl}.
\end{proof}

Now, let $\mathbf{a}=a_{1}\otimes a_{2}\otimes\cdots\otimes a_{\ell}$ be a
pure tensor in $H^{\otimes\ell}$ (with $a_{1},a_{2},\ldots,a_{\ell}\in H$).
Then,%
\begin{align*}
\left(  f^{\otimes\ell}\circ\tau\right)  \left(  \mathbf{a}\right)   &
=\left(  f^{\otimes\ell}\circ\tau\right)  \left(  a_{1}\otimes a_{2}%
\otimes\cdots\otimes a_{\ell}\right) \\
&  =f^{\otimes\ell}\left(  \underbrace{\tau\left(  a_{1}\otimes a_{2}%
\otimes\cdots\otimes a_{\ell}\right)  }_{=a_{\tau^{-1}\left(  1\right)
}\otimes a_{\tau^{-1}\left(  2\right)  }\otimes\cdots\otimes a_{\tau
^{-1}\left(  \ell\right)  }}\right) \\
&  =f^{\otimes\ell}\left(  a_{\tau^{-1}\left(  1\right)  }\otimes a_{\tau
^{-1}\left(  2\right)  }\otimes\cdots\otimes a_{\tau^{-1}\left(  \ell\right)
}\right) \\
&  =f\left(  a_{\tau^{-1}\left(  1\right)  }\right)  \otimes f\left(
a_{\tau^{-1}\left(  2\right)  }\right)  \otimes\cdots\otimes f\left(
a_{\tau^{-1}\left(  \ell\right)  }\right)
\end{align*}
and%
\begin{align*}
\left(  \tau^{k\times}\circ f^{\otimes\ell}\right)  \left(  \mathbf{a}\right)
&  =\left(  \tau^{k\times}\circ f^{\otimes\ell}\right)  \left(  a_{1}\otimes
a_{2}\otimes\cdots\otimes a_{\ell}\right) \\
&  =\tau^{k\times}\left(  \underbrace{f^{\otimes\ell}\left(  a_{1}\otimes
a_{2}\otimes\cdots\otimes a_{\ell}\right)  }_{=f\left(  a_{1}\right)  \otimes
f\left(  a_{2}\right)  \otimes\cdots\otimes f\left(  a_{\ell}\right)  }\right)
\\
&  =\tau^{k\times}\left(  f\left(  a_{1}\right)  \otimes f\left(
a_{2}\right)  \otimes\cdots\otimes f\left(  a_{\ell}\right)  \right) \\
&  =f\left(  a_{\tau^{-1}\left(  1\right)  }\right)  \otimes f\left(
a_{\tau^{-1}\left(  2\right)  }\right)  \otimes\cdots\otimes f\left(
a_{\tau^{-1}\left(  \ell\right)  }\right)
\end{align*}
(by Claim 1, applied to $f\left(  a_{i}\right)  $ instead of $a_{i}$).
Comparing these two equalities, we find $\left(  f^{\otimes\ell}\circ
\tau\right)  \left(  \mathbf{a}\right)  =\left(  \tau^{k\times}\circ
f^{\otimes\ell}\right)  \left(  \mathbf{a}\right)  $.

Forget that we fixed $\mathbf{a}$. We thus have proved that $\left(
f^{\otimes\ell}\circ\tau\right)  \left(  \mathbf{a}\right)  =\left(
\tau^{k\times}\circ f^{\otimes\ell}\right)  \left(  \mathbf{a}\right)  $ for
each pure tensor $\mathbf{a}$ in $H^{\otimes\ell}$. In other words, the two
maps $f^{\otimes\ell}\circ\tau$ and $\tau^{k\times}\circ f^{\otimes\ell}$
agree on all pure tensors in $H^{\otimes\ell}$. Since these two maps are
$\mathbf{k}$-linear (and since the pure tensors span $H^{\otimes\ell}$), this
entails that they must be identical. In other words, $f^{\otimes\ell}\circ
\tau=\tau^{k\times}\circ f^{\otimes\ell}$. This proves Lemma \ref{lem.pikx}.
\end{proof}

\begin{lemma}
\label{lem.sitaze}Let $k,\ell\in\mathbb{N}$. Let $\sigma\in\mathfrak{S}_{k}$
and $\tau\in\mathfrak{S}_{\ell}$. Define a permutation $\sigma^{\times\ell}%
\in\mathfrak{S}_{k\ell}$ as in Lemma \ref{lem.pixl}, and define a permutation
$\tau^{k\times}\in\mathfrak{S}_{k\ell}$ as in Lemma \ref{lem.pikx}. Define a
permutation $\zeta\in\mathfrak{S}_{k\ell}$ as in Lemma \ref{lem.zol1}. Recall
also the permutation $\tau\left[  \sigma\right]  $ defined in Definition
\ref{def.tausigma}. Then,%
\begin{equation}
\tau^{k\times}\circ\zeta^{-1}\circ\sigma^{\times\ell}=\tau\left[
\sigma\right]  \label{eq.lem.sitaze.eq.1}%
\end{equation}
and%
\begin{equation}
\left(  \sigma^{\times\ell}\right)  ^{-1}\circ\zeta\circ\left(  \tau^{k\times
}\right)  ^{-1}=\left(  \tau\left[  \sigma\right]  \right)  ^{-1}.
\label{eq.lem.sitaze.eq.2}%
\end{equation}

\end{lemma}

\begin{proof}
Let us first prove (\ref{eq.lem.sitaze.eq.1}). Let $n\in\left[  k\ell\right]
$. Write $n$ in the form $n=\ell\left(  i-1\right)  +j$ for some $i\in\left[
k\right]  $ and $j\in\left[  \ell\right]  $. (This is clearly possible.) Thus,%
\[
\sigma^{\times\ell}\left(  n\right)  =\sigma^{\times\ell}\left(  \ell\left(
i-1\right)  +j\right)  =\ell\left(  \sigma\left(  i\right)  -1\right)  +j
\]
(by the definition of $\sigma^{\times\ell}$). Hence,%
\[
\zeta^{-1}\left(  \sigma^{\times\ell}\left(  n\right)  \right)  =\zeta
^{-1}\left(  \ell\left(  \sigma\left(  i\right)  -1\right)  +j\right)
=k\left(  j-1\right)  +\sigma\left(  i\right)
\]
(since the definition of $\zeta$ yields $\zeta\left(  k\left(  j-1\right)
+\sigma\left(  i\right)  \right)  =\ell\left(  \sigma\left(  i\right)
-1\right)  +j$). Hence,%
\[
\tau^{k\times}\left(  \zeta^{-1}\left(  \sigma^{\times\ell}\left(  n\right)
\right)  \right)  =\tau^{k\times}\left(  k\left(  j-1\right)  +\sigma\left(
i\right)  \right)  =k\left(  \tau\left(  j\right)  -1\right)  +\sigma\left(
i\right)
\]
(by the definition of $\tau^{k\times}$). Comparing this with%
\begin{align*}
\left(  \tau\left[  \sigma\right]  \right)  \left(  n\right)   &  =\left(
\tau\left[  \sigma\right]  \right)  \left(  \ell\left(  i-1\right)  +j\right)
\ \ \ \ \ \ \ \ \ \ \left(  \text{since }n=\ell\left(  i-1\right)  +j\right)
\\
&  =k\left(  \tau\left(  j\right)  -1\right)  +\sigma\left(  i\right)
\ \ \ \ \ \ \ \ \ \ \left(  \text{by the definition of }\tau\left[
\sigma\right]  \right)  ,
\end{align*}
we obtain $\tau^{k\times}\left(  \zeta^{-1}\left(  \sigma^{\times\ell}\left(
n\right)  \right)  \right)  =\left(  \tau\left[  \sigma\right]  \right)
\left(  n\right)  $.

Forget that we fixed $n$. We thus have shown that $\tau^{k\times}\left(
\zeta^{-1}\left(  \sigma^{\times\ell}\left(  n\right)  \right)  \right)
=\left(  \tau\left[  \sigma\right]  \right)  \left(  n\right)  $ for each
$n\in\left[  k\ell\right]  $. In other words, $\tau^{k\times}\circ\zeta
^{-1}\circ\sigma^{\times\ell}=\tau\left[  \sigma\right]  $. This proves
(\ref{eq.lem.sitaze.eq.1}).

Now, taking inverses on both sides of (\ref{eq.lem.sitaze.eq.1}), we obtain
$\left(  \tau^{k\times}\circ\zeta^{-1}\circ\sigma^{\times\ell}\right)
^{-1}=\left(  \tau\left[  \sigma\right]  \right)  ^{-1}$. In other words,
$\left(  \sigma^{\times\ell}\right)  ^{-1}\circ\zeta\circ\left(  \tau
^{k\times}\right)  ^{-1}=\left(  \tau\left[  \sigma\right]  \right)  ^{-1}$
(since $\left(  \tau^{k\times}\circ\zeta^{-1}\circ\sigma^{\times\ell}\right)
^{-1}=\left(  \sigma^{\times\ell}\right)  ^{-1}\circ\zeta\circ\left(
\tau^{k\times}\right)  ^{-1}$). This proves (\ref{eq.lem.sitaze.eq.2}).
\end{proof}

Recall again that each symmetric group $\mathfrak{S}_{k}$ acts on the
corresponding set $\mathbb{N}^{k}$ of $k$-tuples from the right. For this
action, we have two further elementary combinatorial properties:

\begin{lemma}
\label{lem.tup-zeta}Let $k,\ell\in\mathbb{N}$. Define a permutation $\zeta
\in\mathfrak{S}_{k\ell}$ as in Lemma \ref{lem.zol1}.

Let $\theta_{i,j}\in\mathbb{N}$ be a nonnegative integer for each $i\in\left[
k\right]  $ and $j\in\left[  \ell\right]  $. Then,%
\[
\left(  \theta_{1,1},\theta_{1,2},\ldots,\theta_{k,\ell}\right)  \cdot
\zeta=\left(  \theta_{1,1},\theta_{2,1},\ldots,\theta_{k,\ell}\right)  .
\]
(Here, $\left(  \theta_{1,1},\theta_{1,2},\ldots,\theta_{k,\ell}\right)  $
denotes the list of all $k\ell$ numbers $\theta_{i,j}$ in the order of
lexicographically increasing pairs $\left(  i,j\right)  $, whereas $\left(
\theta_{1,1},\theta_{2,1},\ldots,\theta_{k,\ell}\right)  $ denotes the list of
all $k\ell$ numbers $\theta_{i,j}$ in the order of lexicographically
increasing pairs $\left(  j,i\right)  $.)
\end{lemma}

\begin{proof}
Let us write $\theta_{\left(  i,j\right)  }$ for $\theta_{i,j}$. Thus,
$\theta_{p}$ is defined for any pair $p\in\left[  k\right]  \times\left[
\ell\right]  $.

Define the two bijections $\lambda$ and $\rho$ as in the proof of Lemma
\ref{lem.zol1}. Then,%
\begin{align*}
\left(  \theta_{1,1},\theta_{1,2},\ldots,\theta_{k,\ell}\right)   &  =\left(
\theta_{\rho^{-1}\left(  1\right)  },\theta_{\rho^{-1}\left(  2\right)
},\ldots,\theta_{\rho^{-1}\left(  k\ell\right)  }\right)
\ \ \ \ \ \ \ \ \ \ \text{and}\\
\left(  \theta_{1,1},\theta_{2,1},\ldots,\theta_{k,\ell}\right)   &  =\left(
\theta_{\lambda^{-1}\left(  1\right)  },\theta_{\lambda^{-1}\left(  2\right)
},\ldots,\theta_{\lambda^{-1}\left(  k\ell\right)  }\right)  .
\end{align*}
Thus,%
\begin{align*}
\underbrace{\left(  \theta_{1,1},\theta_{1,2},\ldots,\theta_{k,\ell}\right)
}_{=\left(  \theta_{\rho^{-1}\left(  1\right)  },\theta_{\rho^{-1}\left(
2\right)  },\ldots,\theta_{\rho^{-1}\left(  k\ell\right)  }\right)  }%
\cdot\,\zeta &  =\left(  \theta_{\rho^{-1}\left(  1\right)  },\theta
_{\rho^{-1}\left(  2\right)  },\ldots,\theta_{\rho^{-1}\left(  k\ell\right)
}\right)  \cdot\zeta\\
&  =\left(  \theta_{\rho^{-1}\left(  \zeta\left(  1\right)  \right)  }%
,\theta_{\rho^{-1}\left(  \zeta\left(  2\right)  \right)  },\ldots
,\theta_{\rho^{-1}\left(  \zeta\left(  k\ell\right)  \right)  }\right) \\
&  =\left(  \theta_{\lambda^{-1}\left(  1\right)  },\theta_{\lambda
^{-1}\left(  2\right)  },\ldots,\theta_{\lambda^{-1}\left(  k\ell\right)
}\right) \\
&  \ \ \ \ \ \ \ \ \ \ \ \ \ \ \ \ \ \ \ \ \left(
\begin{array}
[c]{c}%
\text{since each }i\in\left[  k\ell\right] \\
\text{satisfies }\rho^{-1}\left(  \zeta\left(  i\right)  \right)
=\lambda^{-1}\left(  i\right) \\
\text{(since (\ref{pf.lem.zol1.zl=r}) yields }\rho^{-1}\circ\zeta=\lambda
^{-1}\text{)}\\
\text{and thus }\theta_{\rho^{-1}\left(  \zeta\left(  i\right)  \right)
}=\theta_{\lambda^{-1}\left(  i\right)  }%
\end{array}
\right) \\
&  =\left(  \theta_{1,1},\theta_{2,1},\ldots,\theta_{k,\ell}\right)  ,
\end{align*}
and this proves Lemma \ref{lem.tup-zeta}.
\end{proof}

\begin{lemma}
\label{lem.tup-sigma}Let $k,\ell\in\mathbb{N}$. Let $\sigma\in\mathfrak{S}%
_{k}$ be any permutation. Define a permutation $\sigma^{\times\ell}%
\in\mathfrak{S}_{k\ell}$ as in Lemma \ref{lem.pixl}.

Let $\theta_{i,j}\in\mathbb{N}$ be a nonnegative integer for each $i\in\left[
k\right]  $ and $j\in\left[  \ell\right]  $. Then,%
\[
\left(  \theta_{1,1},\theta_{1,2},\ldots,\theta_{k,\ell}\right)  \cdot
\sigma^{\times\ell}=\left(  \theta_{\sigma\left(  1\right)  ,1},\theta
_{\sigma\left(  1\right)  ,2},\ldots,\theta_{\sigma\left(  k\right)  ,\ell
}\right)  .
\]
(Here, $\left(  \theta_{1,1},\theta_{1,2},\ldots,\theta_{k,\ell}\right)  $
denotes the list of all $k\ell$ numbers $\theta_{i,j}$ in the order of
lexicographically increasing pairs $\left(  i,j\right)  $, whereas $\left(
\theta_{\sigma\left(  1\right)  ,1},\theta_{\sigma\left(  1\right)  ,2}%
,\ldots,\theta_{\sigma\left(  k\right)  ,\ell}\right)  $ denotes the list of
all $k\ell$ numbers $\theta_{\sigma\left(  i\right)  ,j}$ in the same order.)
\end{lemma}

\begin{proof}
Let us denote the $k\ell$-tuples $\left(  \theta_{1,1},\theta_{1,2}%
,\ldots,\theta_{k,\ell}\right)  $ and $\left(  \theta_{\sigma\left(  1\right)
,1},\theta_{\sigma\left(  1\right)  ,2},\ldots,\theta_{\sigma\left(  k\right)
,\ell}\right)  $ as $\left(  \alpha_{1},\alpha_{2},\ldots,\alpha_{k\ell
}\right)  $ and $\left(  \beta_{1},\beta_{2},\ldots,\beta_{k\ell}\right)  $,
respectively. Then, each $i\in\left[  k\right]  $ and $j\in\left[
\ell\right]  $ satisfy%
\begin{equation}
\alpha_{\ell\left(  i-1\right)  +j}=\theta_{i,j} \label{pf.lem.tup-sigma.=a}%
\end{equation}
(since $\left(  \alpha_{1},\alpha_{2},\ldots,\alpha_{k\ell}\right)  =\left(
\theta_{1,1},\theta_{1,2},\ldots,\theta_{k,\ell}\right)  $) and%
\begin{equation}
\beta_{\ell\left(  i-1\right)  +j}=\theta_{\sigma\left(  i\right)  ,j}
\label{pf.lem.tup-sigma.=b}%
\end{equation}
(since $\left(  \beta_{1},\beta_{2},\ldots,\beta_{k\ell}\right)  =\left(
\theta_{\sigma\left(  1\right)  ,1},\theta_{\sigma\left(  1\right)  ,2}%
,\ldots,\theta_{\sigma\left(  k\right)  ,\ell}\right)  $).

Let $\omega$ be the permutation $\sigma^{\times\ell}$. Then, for each
$i\in\left[  k\right]  $ and $j\in\left[  \ell\right]  $, we have%
\begin{align}
\omega\left(  \ell\left(  i-1\right)  +j\right)   &  =\sigma^{\times\ell
}\left(  \ell\left(  i-1\right)  +j\right) \nonumber\\
&  =\ell\left(  \sigma\left(  i\right)  -1\right)  +j
\label{pf.lem.tup-sigma.2}%
\end{align}
(by the definition of $\sigma^{\times\ell}$).

Now, let $n\in\left[  k\ell\right]  $. Write $n$ in the form $n=\ell\left(
i-1\right)  +j$ for some $i\in\left[  k\right]  $ and $j\in\left[
\ell\right]  $. (This is clearly possible.) Thus, $\omega\left(  n\right)
=\omega\left(  \ell\left(  i-1\right)  +j\right)  =\ell\left(  \sigma\left(
i\right)  -1\right)  +j$ (by (\ref{pf.lem.tup-sigma.2})). This yields%
\begin{align*}
\alpha_{\omega\left(  n\right)  }  &  =\alpha_{\ell\left(  \sigma\left(
i\right)  -1\right)  +j}\\
&  =\theta_{\sigma\left(  i\right)  ,j}\ \ \ \ \ \ \ \ \ \ \left(  \text{by
(\ref{pf.lem.tup-sigma.=a}), applied to }\sigma\left(  i\right)  \text{
instead of }i\right) \\
&  =\beta_{\ell\left(  i-1\right)  +j}\ \ \ \ \ \ \ \ \ \ \left(  \text{by
(\ref{pf.lem.tup-sigma.=b})}\right)  .\\
&  =\beta_{n}\ \ \ \ \ \ \ \ \ \ \left(  \text{since }\ell\left(  i-1\right)
+j=n\right)  .
\end{align*}

Forget that we fixed $n$. We thus have shown that $\alpha_{\omega\left(
n\right)  }=\beta_{n}$ for each $n\in\left[  k\ell\right]  $.

Now,
\begin{align*}
&  \underbrace{\left(  \theta_{1,1},\theta_{1,2},\ldots,\theta_{k,\ell
}\right)  }_{=\left(  \alpha_{1},\alpha_{2},\ldots,\alpha_{k\ell}\right)
}\cdot\underbrace{\sigma^{\times\ell}}_{=\omega}\\
&  =\left(  \alpha_{1},\alpha_{2},\ldots,\alpha_{k\ell}\right)  \cdot\omega\\
&  =\left(  \alpha_{\omega\left(  1\right)  },\alpha_{\omega\left(  2\right)
},\ldots,\alpha_{\omega\left(  k\ell\right)  }\right) \\
&  =\left(  \beta_{1},\beta_{2},\ldots,\beta_{k\ell}\right)
\ \ \ \ \ \ \ \ \ \ \left(  \text{since }\alpha_{\omega\left(  n\right)
}=\beta_{n}\text{ for each }n\in\left[  k\ell\right]  \right) \\
&  =\left(  \theta_{\sigma\left(  1\right)  ,1},\theta_{\sigma\left(
1\right)  ,2},\ldots,\theta_{\sigma\left(  k\right)  ,\ell}\right)  .
\end{align*}
This proves Lemma \ref{lem.tup-sigma}.
\end{proof}

\subsection{The composition formula: lemmas on bialgebras}

Now, we step to some lemmas that rely on the bialgebra structure on $H$.

\begin{lemma}
\label{lem.mmm}Let $k,\ell\in\mathbb{N}$. Then, $m^{\left[  k\right]  }%
\circ\left(  m^{\left[  \ell\right]  }\right)  ^{\otimes k}=m^{\left[
k\ell\right]  }$.
\end{lemma}

\begin{proof}
We induct on $k$. The \textit{base case} ($k=0$) is trivial (since $\left(
m^{\left[  \ell\right]  }\right)  ^{\otimes0}=\operatorname*{id}%
\nolimits_{\mathbf{k}}$ and $0=0\ell$). For the \textit{induction step}, we
fix a $k\in\mathbb{N}$, and we assume (as the induction hypothesis) that
$m^{\left[  k\right]  }\circ\left(  m^{\left[  \ell\right]  }\right)
^{\otimes k}=m^{\left[  k\ell\right]  }$. We must now prove that $m^{\left[
k+1\right]  }\circ\left(  m^{\left[  \ell\right]  }\right)  ^{\otimes\left(
k+1\right)  }=m^{\left[  \left(  k+1\right)  \ell\right]  }$. But $\left(
k+1\right)  \ell=k\ell+\ell$, and thus%
\begin{align*}
m^{\left[  \left(  k+1\right)  \ell\right]  }  &  =m^{\left[  k\ell
+\ell\right]  }=m\circ\left(  \underbrace{m^{\left[  k\ell\right]  }%
}_{\substack{=m^{\left[  k\right]  }\circ\left(  m^{\left[  \ell\right]
}\right)  ^{\otimes k}\\\text{(by the induction hypothesis)}}}\otimes
\underbrace{m^{\left[  \ell\right]  }}_{=\operatorname*{id}\circ m^{\left[
\ell\right]  }}\right)  \ \ \ \ \ \ \ \ \ \ \left(  \text{by (\ref{eq.mu+v}%
)}\right) \\
&  =m\circ\underbrace{\left(  \left(  m^{\left[  k\right]  }\circ\left(
m^{\left[  \ell\right]  }\right)  ^{\otimes k}\right)  \otimes\left(
\operatorname*{id}\circ m^{\left[  \ell\right]  }\right)  \right)  }_{=\left(
m^{\left[  k\right]  }\otimes\operatorname*{id}\right)  \circ\left(  \left(
m^{\left[  \ell\right]  }\right)  ^{\otimes k}\otimes m^{\left[  \ell\right]
}\right)  }\\
&  =m\circ\left(  m^{\left[  k\right]  }\otimes\underbrace{\operatorname*{id}%
}_{=m^{\left[  1\right]  }}\right)  \circ\underbrace{\left(  \left(
m^{\left[  \ell\right]  }\right)  ^{\otimes k}\otimes m^{\left[  \ell\right]
}\right)  }_{=\left(  m^{\left[  \ell\right]  }\right)  ^{\otimes\left(
k+1\right)  }}\\
&  =\underbrace{m\circ\left(  m^{\left[  k\right]  }\otimes m^{\left[
1\right]  }\right)  }_{\substack{=m^{\left[  k+1\right]  }\\\text{(by
(\ref{eq.mu+v}))}}}\circ\left(  m^{\left[  \ell\right]  }\right)
^{\otimes\left(  k+1\right)  }=m^{\left[  k+1\right]  }\circ\left(  m^{\left[
\ell\right]  }\right)  ^{\otimes\left(  k+1\right)  }.
\end{align*}
Thus, $m^{\left[  k+1\right]  }\circ\left(  m^{\left[  \ell\right]  }\right)
^{\otimes\left(  k+1\right)  }=m^{\left[  \left(  k+1\right)  \ell\right]  }$
is proved, so that the induction is complete. This proves Lemma \ref{lem.mmm}.
\end{proof}

\begin{lemma}
\label{lem.DDD}Let $k,\ell\in\mathbb{N}$. Then, $\left(  \Delta^{\left[
k\right]  }\right)  ^{\otimes\ell}\circ\Delta^{\left[  \ell\right]  }%
=\Delta^{\left[  k\ell\right]  }$.
\end{lemma}

\begin{proof}
Upon swapping $k$ and $\ell$, this becomes the dual claim to Lemma
\ref{lem.mmm}, so the same proof can be used (with all arrows reversed).
\end{proof}

\begin{lemma}
\label{lem.mDc}Let $k,\ell\in\mathbb{N}$. Define a permutation $\zeta
\in\mathfrak{S}_{k\ell}$ as in Lemma \ref{lem.zol1}. Then,
\[
\left(  m^{\left[  \ell\right]  }\right)  ^{\otimes k}\circ\zeta\circ\left(
\Delta^{\left[  k\right]  }\right)  ^{\otimes\ell}=\Delta^{\left[  k\right]
}\circ m^{\left[  \ell\right]  }.
\]

\end{lemma}

\begin{proof}
From \cite[Exercise 1.4.22(c)]{GriRei} (applied to $k-1$ and $\ell-1$ instead
of $k$ and $\ell$), we obtain\footnote{Strictly speaking, this use of
\cite[Exercise 1.4.22(c)]{GriRei} requires $k,\ell\geq1$. But the cases $k=0$
and $\ell=0$ are easily checked by hand (since in these cases, we have
$\zeta=\operatorname*{id}$ and $m^{\left[  0\right]  }=u$ and $\Delta^{\left[
0\right]  }=\epsilon$). To be specific: In the case $k=0$, the claim of Lemma
\ref{lem.mDc} boils down to $\epsilon^{\otimes\ell}=\epsilon\circ m^{\left[
\ell\right]  }$; whereas in the case $\ell=0$, it instead boils down to
$u^{\otimes k}=\Delta^{\left[  k\right]  }\circ u$. Both of these equalities
are easily proved (if needed, by induction on $\ell$ and $k$, respectively).}%
\begin{equation}
m_{H^{\otimes k}}^{\left[  \ell\right]  }\circ\left(  \Delta^{\left[
k\right]  }\right)  ^{\otimes\ell}=\Delta^{\left[  k\right]  }\circ m^{\left[
\ell\right]  }, \label{pf.thm.sol-mac-1.mD1}%
\end{equation}
where $m_{H^{\otimes k}}^{\left[  \ell\right]  }$ is the map defined just as
$m^{\left[  \ell\right]  }$ but for the algebra $H^{\otimes k}$ instead of
$H$. However, it is easy to see (and should be known) that
\begin{equation}
m_{H^{\otimes k}}^{\left[  \ell\right]  }=\left(  m^{\left[  \ell\right]
}\right)  ^{\otimes k}\circ\zeta. \label{pf.thm.sol-mac-1.mz}%
\end{equation}
Indeed, this can be checked on pure tensors, using Lemma \ref{lem.zol1}%
.\footnote{\textit{Proof.} If $\left(  h_{i,j}\right)  _{i\in\left[
\ell\right]  ,\ j\in\left[  k\right]  }\in H^{\ell\times k}$ is any
$\ell\times k$-matrix over $H$, then
\par%
\begin{align*}
&  \left(  \left(  m^{\left[  \ell\right]  }\right)  ^{\otimes k}\circ
\zeta\right)  \left(  h_{1,1}\otimes h_{1,2}\otimes\cdots\otimes
h_{1,k}\right. \\
&  \ \ \ \ \ \ \ \ \ \ \left.  \otimes\ h_{2,1}\otimes h_{2,2}\otimes
\cdots\otimes h_{2,k}\right. \\
&  \ \ \ \ \ \ \ \ \ \ \left.  \otimes\ \cdots\right. \\
&  \ \ \ \ \ \ \ \ \ \ \left.  \otimes\ h_{\ell,1}\otimes h_{\ell,2}%
\otimes\cdots\otimes h_{\ell,k}\right) \\
&  =\left(  m^{\left[  \ell\right]  }\right)  ^{\otimes k}\left(
h_{1,1}\otimes h_{2,1}\otimes\cdots\otimes h_{\ell,1}\right. \\
&  \ \ \ \ \ \ \ \ \ \ \left.  \otimes\ h_{1,2}\otimes h_{2,2}\otimes
\cdots\otimes h_{\ell,2}\right. \\
&  \ \ \ \ \ \ \ \ \ \ \left.  \otimes\ \cdots\right. \\
&  \ \ \ \ \ \ \ \ \ \ \left.  \otimes\ h_{1,k}\otimes h_{2,k}\otimes
\cdots\otimes h_{\ell,k}\right) \\
&  \ \ \ \ \ \ \ \ \ \ \ \ \ \ \ \ \ \ \ \ \left(
\begin{array}
[c]{c}%
\text{here, we have applied the map }\left(  m^{\left[  \ell\right]  }\right)
^{\otimes k}\\
\text{to both sides of the equality (\ref{pf.thm.sol-mac-1.z-prod})}%
\end{array}
\right) \\
&  =\underbrace{m^{\left[  \ell\right]  }\left(  h_{1,1}\otimes h_{2,1}%
\otimes\cdots\otimes h_{\ell,1}\right)  }_{=h_{1,1}h_{2,1}\cdots h_{\ell,1}}\\
&  \ \ \ \ \ \ \ \ \ \ \otimes\underbrace{m^{\left[  \ell\right]  }\left(
h_{1,2}\otimes h_{2,2}\otimes\cdots\otimes h_{\ell,2}\right)  }_{=h_{1,2}%
h_{2,2}\cdots h_{\ell,2}}\\
&  \ \ \ \ \ \ \ \ \ \ \otimes\cdots\\
&  \ \ \ \ \ \ \ \ \ \ \otimes\underbrace{m^{\left[  \ell\right]  }\left(
h_{1,k}\otimes h_{2,k}\otimes\cdots\otimes h_{\ell,k}\right)  }_{=h_{1,k}%
h_{2,k}\cdots h_{\ell,k}}\\
&  =h_{1,1}h_{2,1}\cdots h_{\ell,1}\otimes h_{1,2}h_{2,2}\cdots h_{\ell
,2}\otimes\cdots\otimes h_{1,k}h_{2,k}\cdots h_{\ell,k}\\
&  =\left(  h_{1,1}\otimes h_{1,2}\otimes\cdots\otimes h_{1,k}\right)
\cdot\left(  h_{2,1}\otimes h_{2,2}\otimes\cdots\otimes h_{2,k}\right)
\cdot\cdots\cdot\left(  h_{\ell,1}\otimes h_{\ell,2}\otimes\cdots\otimes
h_{\ell,k}\right) \\
&  =m_{H^{\otimes k}}^{\left[  \ell\right]  }\left(  h_{1,1}\otimes
h_{1,2}\otimes\cdots\otimes h_{1,k}\right. \\
&  \ \ \ \ \ \ \ \ \ \ \left.  \otimes\ h_{2,1}\otimes h_{2,2}\otimes
\cdots\otimes h_{2,k}\right. \\
&  \ \ \ \ \ \ \ \ \ \ \left.  \otimes\ \cdots\right. \\
&  \ \ \ \ \ \ \ \ \ \ \left.  \otimes\ h_{\ell,1}\otimes h_{\ell,2}%
\otimes\cdots\otimes h_{\ell,k}\right)  .
\end{align*}
In other words, the two maps $\left(  m^{\left[  \ell\right]  }\right)
^{\otimes k}\circ\zeta$ and $m_{H^{\otimes k}}^{\left[  \ell\right]  }$ agree
on each pure tensor. Since these two maps are $\mathbf{k}$-linear, they must
therefore be identical (since the pure tensors span $H^{\otimes k\ell}$). In
other words, $\left(  m^{\left[  \ell\right]  }\right)  ^{\otimes k}\circ
\zeta=m_{H^{\otimes k}}^{\left[  \ell\right]  }$. This proves
(\ref{pf.thm.sol-mac-1.mz}).}

Using (\ref{pf.thm.sol-mac-1.mz}), we can rewrite (\ref{pf.thm.sol-mac-1.mD1})
as%
\[
\left(  m^{\left[  \ell\right]  }\right)  ^{\otimes k}\circ\zeta\circ\left(
\Delta^{\left[  k\right]  }\right)  ^{\otimes\ell}=\Delta^{\left[  k\right]
}\circ m^{\left[  \ell\right]  }.
\]
This proves Lemma \ref{lem.mDc}.
\end{proof}

The next two lemmas combine the algebra and coalgebra structures with the grading:

\begin{lemma}
\label{lem.Pm}Let $k,\ell\in\mathbb{N}$. Let $\gamma=\left(  \gamma_{1}%
,\gamma_{2},\ldots,\gamma_{k}\right)  \in\mathbb{N}^{k}$. Then,%
\[
P_{\gamma}\circ\left(  m^{\left[  \ell\right]  }\right)  ^{\otimes k}%
=\sum_{\substack{\gamma_{i,j}\in\mathbb{N}\text{ for all }i\in\left[
k\right]  \text{ and }j\in\left[  \ell\right]  ;\\\gamma_{i,1}+\gamma
_{i,2}+\cdots+\gamma_{i,\ell}=\gamma_{i}\text{ for all }i\in\left[  k\right]
}}\left(  m^{\left[  \ell\right]  }\right)  ^{\otimes k}\circ P_{\left(
\gamma_{1,1},\gamma_{1,2},\ldots,\gamma_{k,\ell}\right)  }.
\]
Here, $\left(  \gamma_{1,1},\gamma_{1,2},\ldots,\gamma_{k,\ell}\right)  $
denotes the $k\ell$-tuple consisting of all $k\ell$ numbers $\gamma_{i,j}$
(for all $i\in\left[  k\right]  $ and $j\in\left[  \ell\right]  $) listed in
the order of lexicographically increasing pairs $\left(  i,j\right)  $.
\end{lemma}

\begin{proof}
Let $i\in\mathbb{N}$. Recall that $p_{i}:H\rightarrow H$ denotes the
projection of the graded $\mathbf{k}$-module $H$ onto its $i$-th graded
component. Likewise, let $p_{i}^{\prime}:H^{\otimes\ell}\rightarrow
H^{\otimes\ell}$ denote the projection of the graded $\mathbf{k}$-module
$H^{\otimes\ell}$ onto its $i$-th graded component. The map $m^{\left[
\ell\right]  }:H^{\otimes\ell}\rightarrow H$ is graded\footnote{This follows
by a straightforward induction on $\ell$ from the gradedness of the map
$m:H\otimes H\rightarrow H$.}. Thus, it commutes with the projection onto the
$i$-th graded component. In other words,%
\begin{equation}
p_{i}\circ m^{\left[  \ell\right]  }=m^{\left[  \ell\right]  }\circ
p_{i}^{\prime}. \label{pf.lem.Pm.1}%
\end{equation}

However, the definition of the grading on $H^{\otimes\ell}$ yields that the
$i$-th graded component of $H^{\otimes\ell}$ is $\bigoplus_{\substack{\left(
i_{1},i_{2},\ldots,i_{\ell}\right)  \in\mathbb{N}^{\ell};\\i_{1}+i_{2}%
+\cdots+i_{\ell}=i}}H_{i_{1}}\otimes H_{i_{2}}\otimes\cdots\otimes H_{i_{\ell
}}$. Hence, it is easy to see that the projection $p_{i}^{\prime}$ onto this
component is%
\[
\sum_{\substack{\left(  i_{1},i_{2},\ldots,i_{\ell}\right)  \in\mathbb{N}%
^{\ell};\\i_{1}+i_{2}+\cdots+i_{\ell}=i}}p_{i_{1}}\otimes p_{i_{2}}%
\otimes\cdots\otimes p_{i_{\ell}}.
\]
In view of this, we can rewrite (\ref{pf.lem.Pm.1}) as%
\begin{align}
p_{i}\circ m^{\left[  \ell\right]  }  &  =m^{\left[  \ell\right]  }%
\circ\left(  \sum_{\substack{\left(  i_{1},i_{2},\ldots,i_{\ell}\right)
\in\mathbb{N}^{\ell};\\i_{1}+i_{2}+\cdots+i_{\ell}=i}}p_{i_{1}}\otimes
p_{i_{2}}\otimes\cdots\otimes p_{i_{\ell}}\right) \nonumber\\
&  =\sum_{\substack{\left(  i_{1},i_{2},\ldots,i_{\ell}\right)  \in
\mathbb{N}^{\ell};\\i_{1}+i_{2}+\cdots+i_{\ell}=i}}m^{\left[  \ell\right]
}\circ\left(  p_{i_{1}}\otimes p_{i_{2}}\otimes\cdots\otimes p_{i_{\ell}%
}\right)  \label{pf.lem.Pm.3}%
\end{align}
(here, we have distributed the summation sign to the beginning of the product,
since all the maps involved are linear).

Forget that we fixed $i$. We have thus proved the equality (\ref{pf.lem.Pm.3})
for each $i\in\mathbb{N}$.

Now, from $\gamma=\left(  \gamma_{1},\gamma_{2},\ldots,\gamma_{k}\right)  $,
we obtain $P_{\gamma}=p_{\gamma_{1}}\otimes p_{\gamma_{2}}\otimes\cdots\otimes
p_{\gamma_{k}}$ (by the definition of $P_{\gamma}$). Hence,%
\begin{align}
&  \underbrace{P_{\gamma}}_{=p_{\gamma_{1}}\otimes p_{\gamma_{2}}\otimes
\cdots\otimes p_{\gamma_{k}}}\circ\underbrace{\left(  m^{\left[  \ell\right]
}\right)  ^{\otimes k}}_{=m^{\left[  \ell\right]  }\otimes m^{\left[
\ell\right]  }\otimes\cdots\otimes m^{\left[  \ell\right]  }}\nonumber\\
&  =\left(  p_{\gamma_{1}}\otimes p_{\gamma_{2}}\otimes\cdots\otimes
p_{\gamma_{k}}\right)  \circ\left(  m^{\left[  \ell\right]  }\otimes
m^{\left[  \ell\right]  }\otimes\cdots\otimes m^{\left[  \ell\right]  }\right)
\nonumber\\
&  =\left(  p_{\gamma_{1}}\circ m^{\left[  \ell\right]  }\right)
\otimes\left(  p_{\gamma_{2}}\circ m^{\left[  \ell\right]  }\right)
\otimes\cdots\otimes\left(  p_{\gamma_{k}}\circ m^{\left[  \ell\right]
}\right) \nonumber\\
&  =\bigotimes_{s=1}^{k}\left(  p_{\gamma_{s}}\circ m^{\left[  \ell\right]
}\right)  \ \ \ \ \ \ \ \ \ \ \left(  \text{here, we use the symbol
}\bigotimes_{s=1}^{k}f_{s}\text{ for }f_{1}\otimes f_{2}\otimes\cdots\otimes
f_{k}\right) \nonumber\\
&  =\bigotimes_{s=1}^{k}\left(  \sum_{\substack{\left(  i_{1},i_{2}%
,\ldots,i_{\ell}\right)  \in\mathbb{N}^{\ell};\\i_{1}+i_{2}+\cdots+i_{\ell
}=\gamma_{s}}}m^{\left[  \ell\right]  }\circ\left(  p_{i_{1}}\otimes p_{i_{2}%
}\otimes\cdots\otimes p_{i_{\ell}}\right)  \right) \nonumber\\
&  \ \ \ \ \ \ \ \ \ \ \ \ \ \ \ \ \ \ \ \ \left(  \text{here, we have applied
(\ref{pf.lem.Pm.3}) to }i=\gamma_{s}\text{ for each }s\in\left[  k\right]
\right) \nonumber\\
&  =\bigotimes_{s=1}^{k}\left(  \sum_{\substack{\left(  \gamma_{s,1}%
,\gamma_{s,2},\ldots,\gamma_{s,\ell}\right)  \in\mathbb{N}^{\ell}%
;\\\gamma_{s,1}+\gamma_{s,2}+\cdots+\gamma_{s,\ell}=\gamma_{s}}}m^{\left[
\ell\right]  }\circ\left(  p_{\gamma_{s,1}}\otimes p_{\gamma_{s,2}}%
\otimes\cdots\otimes p_{\gamma_{s,\ell}}\right)  \right) \nonumber\\
&  \ \ \ \ \ \ \ \ \ \ \ \ \ \ \ \ \ \ \ \ \left(
\begin{array}
[c]{c}%
\text{here, we have renamed the index }\left(  i_{1},i_{2},\ldots,i_{\ell
}\right) \\
\text{as }\left(  \gamma_{s,1},\gamma_{s,2},\ldots,\gamma_{s,\ell}\right)
\text{ in the sum}%
\end{array}
\right) \nonumber\\
&  =\sum_{\substack{\gamma_{s,j}\in\mathbb{N}\text{ for all }s\in\left[
k\right]  \text{ and }j\in\left[  \ell\right]  ;\\\gamma_{s,1}+\gamma
_{s,2}+\cdots+\gamma_{s,\ell}=\gamma_{s}\text{ for all }s\in\left[  k\right]
}}\ \ \bigotimes_{s=1}^{k}\left(  m^{\left[  \ell\right]  }\circ\left(
p_{\gamma_{s,1}}\otimes p_{\gamma_{s,2}}\otimes\cdots\otimes p_{\gamma
_{s,\ell}}\right)  \right)  \label{pf.lem.Pm.2}%
\end{align}
(by the product rule for tensor products, which says
\[
\bigotimes\limits_{s=1}^{k}\ \ \sum_{a\in A_{s}}f_{s,a}=\sum_{\left(
a_{1},a_{2},\ldots,a_{k}\right)  \in A_{1}\times A_{2}\times\cdots\times
A_{k}}\ \ \bigotimes\limits_{s=1}^{k}f_{s,a_{s}},
\]
and which we here have applied to \newline$A_{s}=\left\{  \left(  \gamma
_{s,1},\gamma_{s,2},\ldots,\gamma_{s,\ell}\right)  \in\mathbb{N}^{\ell}%
\ \mid\ \gamma_{s,1}+\gamma_{s,2}+\cdots+\gamma_{s,\ell}=\gamma_{s}\right\}  $
and $f_{s,\left(  \gamma_{s,1},\gamma_{s,2},\ldots,\gamma_{s,\ell}\right)
}=m^{\left[  \ell\right]  }\circ\left(  p_{\gamma_{s,1}}\otimes p_{\gamma
_{s,2}}\otimes\cdots\otimes p_{\gamma_{s,\ell}}\right)  $). However, if
$\gamma_{s,j}$ is a nonnegative integer for all $s\in\left[  k\right]  $ and
all $j\in\left[  \ell\right]  $, then%
\begin{align*}
&  \bigotimes_{s=1}^{k}\left(  m^{\left[  \ell\right]  }\circ\left(
p_{\gamma_{s,1}}\otimes p_{\gamma_{s,2}}\otimes\cdots\otimes p_{\gamma
_{s,\ell}}\right)  \right) \\
&  =\underbrace{\left(  \bigotimes_{s=1}^{k}m^{\left[  \ell\right]  }\right)
}_{=\left(  m^{\left[  \ell\right]  }\right)  ^{\otimes k}}\circ
\underbrace{\left(  \bigotimes_{s=1}^{k}\left(  p_{\gamma_{s,1}}\otimes
p_{\gamma_{s,2}}\otimes\cdots\otimes p_{\gamma_{s,\ell}}\right)  \right)
}_{\substack{=P_{\left(  \gamma_{1,1},\gamma_{1,2},\ldots,\gamma_{k,\ell
}\right)  }\\\text{(by the definition of }P_{\left(  \gamma_{1,1},\gamma
_{1,2},\ldots,\gamma_{k,\ell}\right)  }\text{)}}}\\
&  =\left(  m^{\left[  \ell\right]  }\right)  ^{\otimes k}\circ P_{\left(
\gamma_{1,1},\gamma_{1,2},\ldots,\gamma_{k,\ell}\right)  }.
\end{align*}
Thus, (\ref{pf.lem.Pm.2}) can be rewritten as%
\begin{align*}
P_{\gamma}\circ\left(  m^{\left[  \ell\right]  }\right)  ^{\otimes k}  &
=\sum_{\substack{\gamma_{s,j}\in\mathbb{N}\text{ for all }s\in\left[
k\right]  \text{ and }j\in\left[  \ell\right]  ;\\\gamma_{s,1}+\gamma
_{s,2}+\cdots+\gamma_{s,\ell}=\gamma_{s}\text{ for all }s\in\left[  k\right]
}}\left(  m^{\left[  \ell\right]  }\right)  ^{\otimes k}\circ P_{\left(
\gamma_{1,1},\gamma_{1,2},\ldots,\gamma_{k,\ell}\right)  }\\
&  =\sum_{\substack{\gamma_{i,j}\in\mathbb{N}\text{ for all }i\in\left[
k\right]  \text{ and }j\in\left[  \ell\right]  ;\\\gamma_{i,1}+\gamma
_{i,2}+\cdots+\gamma_{i,\ell}=\gamma_{i}\text{ for all }i\in\left[  k\right]
}}\left(  m^{\left[  \ell\right]  }\right)  ^{\otimes k}\circ P_{\left(
\gamma_{1,1},\gamma_{1,2},\ldots,\gamma_{k,\ell}\right)  }.
\end{align*}
This proves Lemma \ref{lem.Pm}.
\end{proof}

\begin{lemma}
\label{lem.DelP}Let $k,\ell\in\mathbb{N}$. Let $\gamma=\left(  \gamma
_{1},\gamma_{2},\ldots,\gamma_{k}\right)  \in\mathbb{N}^{k}$. Then,%
\begin{equation}
\left(  \Delta^{\left[  \ell\right]  }\right)  ^{\otimes k}\circ P_{\gamma
}=\sum_{\substack{\gamma_{i,j}\in\mathbb{N}\text{ for all }i\in\left[
k\right]  \text{ and }j\in\left[  \ell\right]  ;\\\gamma_{i,1}+\gamma
_{i,2}+\cdots+\gamma_{i,\ell}=\gamma_{i}\text{ for all }i\in\left[  k\right]
}}P_{\left(  \gamma_{1,1},\gamma_{1,2},\ldots,\gamma_{k,\ell}\right)  }%
\circ\left(  \Delta^{\left[  \ell\right]  }\right)  ^{\otimes k}.\nonumber
\end{equation}
Here, $\left(  \gamma_{1,1},\gamma_{1,2},\ldots,\gamma_{k,\ell}\right)  $
denotes the $k\ell$-tuple consisting of all $k\ell$ numbers $\gamma_{i,j}$
(for all $i\in\left[  k\right]  $ and $j\in\left[  \ell\right]  $) listed in
the order of lexicographically increasing pairs $\left(  i,j\right)  $.
\end{lemma}

\begin{proof}
This is the dual statement to Lemma \ref{lem.Pm}, and is proved by the same
argument \textquotedblleft with all arrows reversed\textquotedblright%
\ (meaning that any composition $f_{1}\circ f_{2}\circ\cdots\circ f_{\ell}$ of
several maps should be reversed to become $f_{\ell}\circ f_{\ell-1}\circ
\cdots\circ f_{1}$) and with all $m$'s replaced by $\Delta$'s.
\end{proof}

\subsection{The composition formula: proof}

We are finally able to prove Theorem \ref{thm.sol-mac-1}. The proof follows
the same paradigm as Patras's proof of the commutative and cocommutative
particular cases (\cite[preuve de Th\'{e}or\`{e}me II,7]{Patras94} and
\cite[proof of Theorem 5.11]{CarPat21}), but involves more complexity due to
the presence of permutations.

\begin{proof}
[Proof of Theorem \ref{thm.sol-mac-1}.]The formula (\ref{eq.pas.formal-def})
yields%
\begin{align}
p_{\alpha,\sigma}  &  =m^{\left[  k\right]  }\circ P_{\alpha}\circ\sigma
^{-1}\circ\Delta^{\left[  k\right]  }\ \ \ \ \ \ \ \ \ \ \text{and}%
\label{pf.thm.sol-mac-1.pas}\\
p_{\beta,\tau}  &  =m^{\left[  \ell\right]  }\circ P_{\beta}\circ\tau
^{-1}\circ\Delta^{\left[  \ell\right]  }. \label{pf.thm.sol-mac-1.pbt}%
\end{align}
Multiplying these two equalities (using the operation $\circ$), we obtain%
\[
p_{\alpha,\sigma}\circ p_{\beta,\tau}=m^{\left[  k\right]  }\circ P_{\alpha
}\circ\sigma^{-1}\circ\Delta^{\left[  k\right]  }\circ m^{\left[  \ell\right]
}\circ P_{\beta}\circ\tau^{-1}\circ\Delta^{\left[  \ell\right]  }.
\]
Our main task now is to \textquotedblleft commute\textquotedblright\ the
operators in this equality past each other, moving both $m$ factors to the
left, both $\Delta$ factors to the right, and ideally obtaining some
expression of the form $m^{\left[  k\ell\right]  }\circ P_{\gamma}\circ
\rho^{-1}\circ\Delta^{\left[  k\ell\right]  }$ (more precisely, it will be a
sum of such expressions). For this purpose, we shall derive some
\textquotedblleft commutation-like\textquotedblright\ formulas.

Lemma \ref{lem.Pm} (applied to $\alpha$ and $\alpha_{i}$ instead of $\gamma$
and $\gamma_{i}$) yields%
\begin{align}
&  P_{\alpha}\circ\left(  m^{\left[  \ell\right]  }\right)  ^{\otimes
k}\nonumber\\
&  =\sum_{\substack{\gamma_{i,j}\in\mathbb{N}\text{ for all }i\in\left[
k\right]  \text{ and }j\in\left[  \ell\right]  ;\\\gamma_{i,1}+\gamma
_{i,2}+\cdots+\gamma_{i,\ell}=\alpha_{i}\text{ for all }i\in\left[  k\right]
}}\left(  m^{\left[  \ell\right]  }\right)  ^{\otimes k}\circ P_{\left(
\gamma_{1,1},\gamma_{1,2},\ldots,\gamma_{k,\ell}\right)  }.
\label{pf.thm.sol-mac-1.Pm1}%
\end{align}

Lemma \ref{lem.DelP} (applied to $k$, $\ell$, $\beta$ and $\beta_{i}$ instead
of $\ell$, $k$, $\gamma$ and $\gamma_{i}$) yields%
\begin{align}
&  \left(  \Delta^{\left[  k\right]  }\right)  ^{\otimes\ell}\circ P_{\beta
}\nonumber\\
&  =\sum_{\substack{\gamma_{i,j}\in\mathbb{N}\text{ for all }i\in\left[
\ell\right]  \text{ and }j\in\left[  k\right]  ;\\\gamma_{i,1}+\gamma
_{i,2}+\cdots+\gamma_{i,k}=\beta_{i}\text{ for all }i\in\left[  \ell\right]
}}P_{\left(  \gamma_{1,1},\gamma_{1,2},\ldots,\gamma_{\ell,k}\right)  }%
\circ\left(  \Delta^{\left[  k\right]  }\right)  ^{\otimes\ell}\nonumber\\
&  =\sum_{\substack{\theta_{j,i}\in\mathbb{N}\text{ for all }i\in\left[
\ell\right]  \text{ and }j\in\left[  k\right]  ;\\\theta_{1,i}+\theta
_{2,i}+\cdots+\theta_{k,i}=\beta_{i}\text{ for all }i\in\left[  \ell\right]
}}P_{\left(  \theta_{1,1},\theta_{2,1},\ldots,\theta_{k,\ell}\right)  }%
\circ\left(  \Delta^{\left[  k\right]  }\right)  ^{\otimes\ell}\nonumber\\
&  \ \ \ \ \ \ \ \ \ \ \ \ \ \ \ \ \ \ \ \ \left(  \text{here, we have renamed
the }\gamma_{i,j}\text{ as }\theta_{j,i}\right) \nonumber\\
&  =\sum_{\substack{\theta_{i,j}\in\mathbb{N}\text{ for all }j\in\left[
\ell\right]  \text{ and }i\in\left[  k\right]  ;\\\theta_{1,j}+\theta
_{2,j}+\cdots+\theta_{k,j}=\beta_{j}\text{ for all }j\in\left[  \ell\right]
}}P_{\left(  \theta_{1,1},\theta_{2,1},\ldots,\theta_{k,\ell}\right)  }%
\circ\left(  \Delta^{\left[  k\right]  }\right)  ^{\otimes\ell}\nonumber\\
&  \ \ \ \ \ \ \ \ \ \ \ \ \ \ \ \ \ \ \ \ \left(  \text{here, we have renamed
the indices }i\text{ and }j\text{ as }j\text{ and }i\right) \nonumber\\
&  =\sum_{\substack{\theta_{i,j}\in\mathbb{N}\text{ for all }i\in\left[
k\right]  \text{ and }j\in\left[  \ell\right]  ;\\\theta_{1,j}+\theta
_{2,j}+\cdots+\theta_{k,j}=\beta_{j}\text{ for all }j\in\left[  \ell\right]
}}P_{\left(  \theta_{1,1},\theta_{2,1},\ldots,\theta_{k,\ell}\right)  }%
\circ\left(  \Delta^{\left[  k\right]  }\right)  ^{\otimes\ell}
\label{pf.thm.sol-mac-1.dP2}%
\end{align}
(here, we have just rewritten the \textquotedblleft$\theta_{i,j}\in\mathbb{N}$
for all $j\in\left[  \ell\right]  $ and $i\in\left[  k\right]  $%
\textquotedblright\ under the summation sign as \textquotedblleft$\theta
_{i,j}\in\mathbb{N}$ for all $i\in\left[  k\right]  $ and $j\in\left[
\ell\right]  $\textquotedblright).

Next, we shall use a nearly trivial fact: If $f,g,u,v$ are four maps such that
the compositions $f\circ u$ and $g\circ v$ are well-defined (i.e., the target
of $u$ is the domain of $f$, and the target of $v$ is the domain of $g$) and
satisfy $f\circ u=g\circ v$, and if the maps $f$ and $v$ are invertible, then%
\begin{equation}
u\circ v^{-1}=f^{-1}\circ g \label{pf.thm.sol-mac-1.fu}%
\end{equation}
(since $\underbrace{f\circ u}_{=g\circ v}\circ\,v^{-1}=g\circ
\underbrace{v\circ v^{-1}}_{=\operatorname*{id}}=g$ and thus $g=f\circ u\circ
v^{-1}$, so that $f^{-1}\circ g=\underbrace{f^{-1}\circ f}%
_{=\operatorname*{id}}\circ\,u\circ v^{-1}=u\circ v^{-1}$).

Define a permutation $\sigma^{\times\ell}\in\mathfrak{S}_{k\ell}$ as in Lemma
\ref{lem.pixl}. Define a permutation $\tau^{k\times}\in\mathfrak{S}_{k\ell}$
as in Lemma \ref{lem.pikx}. Define a permutation $\zeta\in\mathfrak{S}_{k\ell
}$ as in Lemma \ref{lem.zol1}.

Lemma \ref{lem.pixl} (applied to $f=m^{\left[  \ell\right]  }$) yields
$\sigma\circ\left(  m^{\left[  \ell\right]  }\right)  ^{\otimes k}=\left(
m^{\left[  \ell\right]  }\right)  ^{\otimes k}\circ\sigma^{\times\ell}$.
Hence, (\ref{pf.thm.sol-mac-1.fu}) (applied to $f=\sigma$ and $u=\left(
m^{\left[  \ell\right]  }\right)  ^{\otimes k}$ and $g=\left(  m^{\left[
\ell\right]  }\right)  ^{\otimes k}$ and $v=\sigma^{\times\ell}$) yields%
\begin{equation}
\left(  m^{\left[  \ell\right]  }\right)  ^{\otimes k}\circ\left(
\sigma^{\times\ell}\right)  ^{-1}=\sigma^{-1}\circ\left(  m^{\left[
\ell\right]  }\right)  ^{\otimes k}. \label{pf.thm.sol-mac-1.msig}%
\end{equation}

Lemma \ref{lem.pikx} (applied to $f=\Delta^{\left[  k\right]  }$) yields
$\left(  \Delta^{\left[  k\right]  }\right)  ^{\otimes\ell}\circ\tau
=\tau^{k\times}\circ\left(  \Delta^{\left[  k\right]  }\right)  ^{\otimes\ell
}$. In other words, $\tau^{k\times}\circ\left(  \Delta^{\left[  k\right]
}\right)  ^{\otimes\ell}=\left(  \Delta^{\left[  k\right]  }\right)
^{\otimes\ell}\circ\tau$. Hence, (\ref{pf.thm.sol-mac-1.fu}) (applied to
$f=\tau^{k\times}$ and $u=\left(  \Delta^{\left[  k\right]  }\right)
^{\otimes\ell}$ and $g=\left(  \Delta^{\left[  k\right]  }\right)
^{\otimes\ell}$ and $v=\tau$) yields%
\begin{equation}
\left(  \Delta^{\left[  k\right]  }\right)  ^{\otimes\ell}\circ\tau
^{-1}=\left(  \tau^{k\times}\right)  ^{-1}\circ\left(  \Delta^{\left[
k\right]  }\right)  ^{\otimes\ell}. \label{pf.thm.sol-mac-1.tauDel}%
\end{equation}

Furthermore, if $\theta_{i,j}\in\mathbb{N}$ is a nonnegative integer for each
$i\in\left[  k\right]  $ and $j\in\left[  \ell\right]  $, then we have%
\begin{align}
&  P_{\left(  \theta_{1,1},\theta_{1,2},\ldots,\theta_{k,\ell}\right)  }%
\circ\zeta\nonumber\\
&  =\zeta\circ P_{\left(  \theta_{1,1},\theta_{1,2},\ldots,\theta_{k,\ell
}\right)  \cdot\zeta}\nonumber\\
&  \ \ \ \ \ \ \ \ \ \ \ \ \ \ \ \ \ \ \ \ \left(  \text{by
(\ref{pf.thm.sol-mac-1.Ppi}), applied to }\pi=\zeta\text{ and }\gamma=\left(
\theta_{1,1},\theta_{1,2},\ldots,\theta_{k,\ell}\right)  \right) \nonumber\\
&  =\zeta\circ P_{\left(  \theta_{1,1},\theta_{2,1},\ldots,\theta_{k,\ell
}\right)  } \label{pf.thm.sol-mac-1.Pz}%
\end{align}
(since Lemma \ref{lem.tup-zeta} yields $\left(  \theta_{1,1},\theta
_{1,2},\ldots,\theta_{k,\ell}\right)  \cdot\zeta=\left(  \theta_{1,1}%
,\theta_{2,1},\ldots,\theta_{k,\ell}\right)  $) and%
\begin{align*}
&  P_{\left(  \theta_{1,1},\theta_{1,2},\ldots,\theta_{k,\ell}\right)  }%
\circ\sigma^{\times\ell}\\
&  =\sigma^{\times\ell}\circ P_{\left(  \theta_{1,1},\theta_{1,2}%
,\ldots,\theta_{k,\ell}\right)  \cdot\sigma^{\times\ell}}\\
&  \ \ \ \ \ \ \ \ \ \ \ \ \ \ \ \ \ \ \ \ \left(  \text{by
(\ref{pf.thm.sol-mac-1.Ppi}), applied to }\pi=\sigma^{\times\ell}\text{ and
}\gamma=\left(  \theta_{1,1},\theta_{1,2},\ldots,\theta_{k,\ell}\right)
\right) \\
&  =\sigma^{\times\ell}\circ P_{\left(  \theta_{\sigma\left(  1\right)
,1},\theta_{\sigma\left(  1\right)  ,2},\ldots,\theta_{\sigma\left(  k\right)
,\ell}\right)  }%
\end{align*}
(since Lemma \ref{lem.tup-sigma} yields $\left(  \theta_{1,1},\theta
_{1,2},\ldots,\theta_{k,\ell}\right)  \cdot\sigma^{\times\ell}=\left(
\theta_{\sigma\left(  1\right)  ,1},\theta_{\sigma\left(  1\right)  ,2}%
,\ldots,\theta_{\sigma\left(  k\right)  ,\ell}\right)  $) and therefore
$\sigma^{\times\ell}\circ P_{\left(  \theta_{\sigma\left(  1\right)
,1},\theta_{\sigma\left(  1\right)  ,2},\ldots,\theta_{\sigma\left(  k\right)
,\ell}\right)  }=P_{\left(  \theta_{1,1},\theta_{1,2},\ldots,\theta_{k,\ell
}\right)  }\circ\sigma^{\times\ell}$, so that%
\begin{equation}
P_{\left(  \theta_{\sigma\left(  1\right)  ,1},\theta_{\sigma\left(  1\right)
,2},\ldots,\theta_{\sigma\left(  k\right)  ,\ell}\right)  }\circ\left(
\sigma^{\times\ell}\right)  ^{-1}=\left(  \sigma^{\times\ell}\right)
^{-1}\circ P_{\left(  \theta_{1,1},\theta_{1,2},\ldots,\theta_{k,\ell}\right)
} \label{pf.thm.sol-mac-1.Ps}%
\end{equation}
(by (\ref{pf.thm.sol-mac-1.fu}), applied to $f=\sigma^{\times\ell}$ and
$u=P_{\left(  \theta_{\sigma\left(  1\right)  ,1},\theta_{\sigma\left(
1\right)  ,2},\ldots,\theta_{\sigma\left(  k\right)  ,\ell}\right)  }$ and
$g=P_{\left(  \theta_{1,1},\theta_{1,2},\ldots,\theta_{k,\ell}\right)  }$ and
$v=\sigma^{\times\ell}$).

Now, (\ref{pf.thm.sol-mac-1.pas}) and (\ref{pf.thm.sol-mac-1.pbt}) yield%
\begin{align}
&  p_{\alpha,\sigma}\circ p_{\beta,\tau}\nonumber\\
&  =m^{\left[  k\right]  }\circ P_{\alpha}\circ\sigma^{-1}\circ
\underbrace{\Delta^{\left[  k\right]  }\circ m^{\left[  \ell\right]  }%
}_{\substack{=\left(  m^{\left[  \ell\right]  }\right)  ^{\otimes k}\circ
\zeta\circ\left(  \Delta^{\left[  k\right]  }\right)  ^{\otimes\ell
}\\\text{(by Lemma \ref{lem.mDc})}}}\circ\,P_{\beta}\circ\tau^{-1}\circ
\Delta^{\left[  \ell\right]  }\nonumber\\
&  =m^{\left[  k\right]  }\circ P_{\alpha}\circ\underbrace{\sigma^{-1}%
\circ\left(  m^{\left[  \ell\right]  }\right)  ^{\otimes k}}%
_{\substack{=\left(  m^{\left[  \ell\right]  }\right)  ^{\otimes k}%
\circ\left(  \sigma^{\times\ell}\right)  ^{-1}\\\text{(by
(\ref{pf.thm.sol-mac-1.msig}))}}}\circ\,\zeta\nonumber\\
&  \ \ \ \ \ \ \ \ \ \ \circ\underbrace{\left(  \Delta^{\left[  k\right]
}\right)  ^{\otimes\ell}\circ P_{\beta}}_{\substack{=\sum_{\substack{\theta
_{i,j}\in\mathbb{N}\text{ for all }i\in\left[  k\right]  \text{ and }%
j\in\left[  \ell\right]  ;\\\theta_{1,j}+\theta_{2,j}+\cdots+\theta
_{k,j}=\beta_{j}\text{ for all }j\in\left[  \ell\right]  }}P_{\left(
\theta_{1,1},\theta_{2,1},\ldots,\theta_{k,\ell}\right)  }\circ\left(
\Delta^{\left[  k\right]  }\right)  ^{\otimes\ell}\\\text{(by
(\ref{pf.thm.sol-mac-1.dP2}))}}}\circ\,\tau^{-1}\circ\Delta^{\left[
\ell\right]  }\nonumber\\
&  =\sum_{\substack{\theta_{i,j}\in\mathbb{N}\text{ for all }i\in\left[
k\right]  \text{ and }j\in\left[  \ell\right]  ;\\\theta_{1,j}+\theta
_{2,j}+\cdots+\theta_{k,j}=\beta_{j}\text{ for all }j\in\left[  \ell\right]
}}m^{\left[  k\right]  }\circ\underbrace{P_{\alpha}\circ\left(  m^{\left[
\ell\right]  }\right)  ^{\otimes k}}_{\substack{=\sum_{\substack{\gamma
_{i,j}\in\mathbb{N}\text{ for all }i\in\left[  k\right]  \text{ and }%
j\in\left[  \ell\right]  ;\\\gamma_{i,1}+\gamma_{i,2}+\cdots+\gamma_{i,\ell
}=\alpha_{i}\text{ for all }i\in\left[  k\right]  }}\left(  m^{\left[
\ell\right]  }\right)  ^{\otimes k}\circ P_{\left(  \gamma_{1,1},\gamma
_{1,2},\ldots,\gamma_{k,\ell}\right)  }\\\text{(by (\ref{pf.thm.sol-mac-1.Pm1}%
))}}}\nonumber\\
&  \ \ \ \ \ \ \ \ \ \ \circ\left(  \sigma^{\times\ell}\right)  ^{-1}%
\circ\zeta\circ P_{\left(  \theta_{1,1},\theta_{2,1},\ldots,\theta_{k,\ell
}\right)  }\circ\underbrace{\left(  \Delta^{\left[  k\right]  }\right)
^{\otimes\ell}\circ\tau^{-1}}_{\substack{=\left(  \tau^{k\times}\right)
^{-1}\circ\left(  \Delta^{\left[  k\right]  }\right)  ^{\otimes\ell
}\\\text{(by (\ref{pf.thm.sol-mac-1.tauDel}))}}}\circ\,\Delta^{\left[
\ell\right]  }\nonumber\\
&  \ \ \ \ \ \ \ \ \ \ \ \ \ \ \ \ \ \ \ \ \left(
\begin{array}
[c]{c}%
\text{here, we have distributed the summation sign to the}\\
\text{beginning of the product, since all our maps are }\mathbf{k}%
\text{-linear}%
\end{array}
\right) \nonumber\\
&  =\sum_{\substack{\gamma_{i,j}\in\mathbb{N}\text{ for all }i\in\left[
k\right]  \text{ and }j\in\left[  \ell\right]  ;\\\gamma_{i,1}+\gamma
_{i,2}+\cdots+\gamma_{i,\ell}=\alpha_{i}\text{ for all }i\in\left[  k\right]
}}\ \ \sum_{\substack{\theta_{i,j}\in\mathbb{N}\text{ for all }i\in\left[
k\right]  \text{ and }j\in\left[  \ell\right]  ;\\\theta_{1,j}+\theta
_{2,j}+\cdots+\theta_{k,j}=\beta_{j}\text{ for all }j\in\left[  \ell\right]
}}\nonumber\\
&  \ \ \ \ \ \ \ \ \ \ m^{\left[  k\right]  }\circ\left(  m^{\left[
\ell\right]  }\right)  ^{\otimes k}\circ P_{\left(  \gamma_{1,1},\gamma
_{1,2},\ldots,\gamma_{k,\ell}\right)  }\circ\left(  \sigma^{\times\ell
}\right)  ^{-1}\nonumber\\
&  \ \ \ \ \ \ \ \ \ \ \circ\zeta\circ P_{\left(  \theta_{1,1},\theta
_{2,1},\ldots,\theta_{k,\ell}\right)  }\circ\left(  \tau^{k\times}\right)
^{-1}\circ\left(  \Delta^{\left[  k\right]  }\right)  ^{\otimes\ell}%
\circ\Delta^{\left[  \ell\right]  }\label{pf.thm.sol-mac-1.5}\\
&  \ \ \ \ \ \ \ \ \ \ \ \ \ \ \ \ \ \ \ \ \left(
\begin{array}
[c]{c}%
\text{here, again, we have distributed the summation sign to the}\\
\text{beginning of the product, since all our maps are }\mathbf{k}%
\text{-linear}%
\end{array}
\right)  .\nonumber
\end{align}

However, if $\gamma_{i,j}$ and $\theta_{i,j}$ are nonnegative integers for all
$i\in\left[  k\right]  $ and all $j\in\left[  \ell\right]  $, then%
\begin{align*}
&  \underbrace{m^{\left[  k\right]  }\circ\left(  m^{\left[  \ell\right]
}\right)  ^{\otimes k}}_{\substack{=m^{\left[  k\ell\right]  }\\\text{(by
Lemma \ref{lem.mmm})}}}\circ\,P_{\left(  \gamma_{1,1},\gamma_{1,2}%
,\ldots,\gamma_{k,\ell}\right)  }\circ\left(  \sigma^{\times\ell}\right)
^{-1}\\
&  \ \ \ \ \ \ \ \ \ \ \circ\underbrace{\zeta\circ P_{\left(  \theta
_{1,1},\theta_{2,1},\ldots,\theta_{k,\ell}\right)  }}_{\substack{=P_{\left(
\theta_{1,1},\theta_{1,2},\ldots,\theta_{k,\ell}\right)  }\circ\zeta
\\\text{(by (\ref{pf.thm.sol-mac-1.Pz}))}}}\circ\left(  \tau^{k\times}\right)
^{-1}\circ\underbrace{\left(  \Delta^{\left[  k\right]  }\right)
^{\otimes\ell}\circ\Delta^{\left[  \ell\right]  }}_{\substack{=\Delta^{\left[
k\ell\right]  }\\\text{(by Lemma \ref{lem.DDD})}}}\\
&  =m^{\left[  k\ell\right]  }\circ P_{\left(  \gamma_{1,1},\gamma
_{1,2},\ldots,\gamma_{k,\ell}\right)  }\circ\underbrace{\left(  \sigma
^{\times\ell}\right)  ^{-1}\circ P_{\left(  \theta_{1,1},\theta_{1,2}%
,\ldots,\theta_{k,\ell}\right)  }}_{\substack{=P_{\left(  \theta
_{\sigma\left(  1\right)  ,1},\theta_{\sigma\left(  1\right)  ,2}%
,\ldots,\theta_{\sigma\left(  k\right)  ,\ell}\right)  }\circ\left(
\sigma^{\times\ell}\right)  ^{-1}\\\text{(by (\ref{pf.thm.sol-mac-1.Ps}))}%
}}\circ\,\zeta\circ\left(  \tau^{k\times}\right)  ^{-1}\circ\Delta^{\left[
k\ell\right]  }\\
&  =m^{\left[  k\ell\right]  }\circ\underbrace{P_{\left(  \gamma_{1,1}%
,\gamma_{1,2},\ldots,\gamma_{k,\ell}\right)  }\circ P_{\left(  \theta
_{\sigma\left(  1\right)  ,1},\theta_{\sigma\left(  1\right)  ,2}%
,\ldots,\theta_{\sigma\left(  k\right)  ,\ell}\right)  }}_{\substack{=%
\begin{cases}
P_{\left(  \gamma_{1,1},\gamma_{1,2},\ldots,\gamma_{k,\ell}\right)  }, &
\text{if }\left(  \gamma_{1,1},\gamma_{1,2},\ldots,\gamma_{k,\ell}\right)
=\left(  \theta_{\sigma\left(  1\right)  ,1},\theta_{\sigma\left(  1\right)
,2},\ldots,\theta_{\sigma\left(  k\right)  ,\ell}\right)  ;\\
0, & \text{otherwise}%
\end{cases}
\\\text{(by Lemma \ref{lem.PP})}}}\\
&  \ \ \ \ \ \ \ \ \ \ \circ\underbrace{\left(  \sigma^{\times\ell}\right)
^{-1}\circ\zeta\circ\left(  \tau^{k\times}\right)  ^{-1}}_{\substack{=\left(
\tau\left[  \sigma\right]  \right)  ^{-1}\\\text{(by (\ref{eq.lem.sitaze.eq.2}%
))}}}\circ\,\Delta^{\left[  k\ell\right]  }\\
&  =m^{\left[  k\ell\right]  }\circ%
\begin{cases}
P_{\left(  \gamma_{1,1},\gamma_{1,2},\ldots,\gamma_{k,\ell}\right)  }, &
\text{if }\left(  \gamma_{1,1},\gamma_{1,2},\ldots,\gamma_{k,\ell}\right)
=\left(  \theta_{\sigma\left(  1\right)  ,1},\theta_{\sigma\left(  1\right)
,2},\ldots,\theta_{\sigma\left(  k\right)  ,\ell}\right)  ;\\
0, & \text{otherwise}%
\end{cases}
\\
&  \ \ \ \ \ \ \ \ \ \ \circ\left(  \tau\left[  \sigma\right]  \right)
^{-1}\circ\Delta^{\left[  k\ell\right]  }\\
&  =m^{\left[  k\ell\right]  }\circ%
\begin{cases}
P_{\left(  \gamma_{1,1},\gamma_{1,2},\ldots,\gamma_{k,\ell}\right)  }, &
\text{if }\gamma_{i,j}=\theta_{\sigma\left(  i\right)  ,j}\text{ for all
}i,j;\\
0, & \text{otherwise}%
\end{cases}
\\
&  \ \ \ \ \ \ \ \ \ \ \circ\left(  \tau\left[  \sigma\right]  \right)
^{-1}\circ\Delta^{\left[  k\ell\right]  }%
\end{align*}
(since the condition \textquotedblleft$\left(  \gamma_{1,1},\gamma
_{1,2},\ldots,\gamma_{k,\ell}\right)  =\left(  \theta_{\sigma\left(  1\right)
,1},\theta_{\sigma\left(  1\right)  ,2},\ldots,\theta_{\sigma\left(  k\right)
,\ell}\right)  $\textquotedblright\ is equivalent to \textquotedblleft%
$\gamma_{i,j}=\theta_{\sigma\left(  i\right)  ,j}$ for all $i,j$%
\textquotedblright). Thus, we can rewrite (\ref{pf.thm.sol-mac-1.5}) as%
\begin{align*}
p_{\alpha,\sigma}\circ p_{\beta,\tau}  &  =\sum_{\substack{\gamma_{i,j}%
\in\mathbb{N}\text{ for all }i\in\left[  k\right]  \text{ and }j\in\left[
\ell\right]  ;\\\gamma_{i,1}+\gamma_{i,2}+\cdots+\gamma_{i,\ell}=\alpha
_{i}\text{ for all }i\in\left[  k\right]  }}\ \ \sum_{\substack{\theta
_{i,j}\in\mathbb{N}\text{ for all }i\in\left[  k\right]  \text{ and }%
j\in\left[  \ell\right]  ;\\\theta_{1,j}+\theta_{2,j}+\cdots+\theta
_{k,j}=\beta_{j}\text{ for all }j\in\left[  \ell\right]  }}\\
&  \ \ \ \ \ \ \ \ \ \ m^{\left[  k\ell\right]  }\circ%
\begin{cases}
P_{\left(  \gamma_{1,1},\gamma_{1,2},\ldots,\gamma_{k,\ell}\right)  }, &
\text{if }\gamma_{i,j}=\theta_{\sigma\left(  i\right)  ,j}\text{ for all
}i,j;\\
0, & \text{otherwise}%
\end{cases}
\\
&  \ \ \ \ \ \ \ \ \ \ \circ\left(  \tau\left[  \sigma\right]  \right)
^{-1}\circ\Delta^{\left[  k\ell\right]  }\\
&  =\sum_{\substack{\gamma_{i,j}\in\mathbb{N}\text{ for all }i\in\left[
k\right]  \text{ and }j\in\left[  \ell\right]  ;\\\gamma_{i,1}+\gamma
_{i,2}+\cdots+\gamma_{i,\ell}=\alpha_{i}\text{ for all }i\in\left[  k\right]
}}\ \ \sum_{\substack{\theta_{i,j}\in\mathbb{N}\text{ for all }i\in\left[
k\right]  \text{ and }j\in\left[  \ell\right]  ;\\\theta_{1,j}+\theta
_{2,j}+\cdots+\theta_{k,j}=\beta_{j}\text{ for all }j\in\left[  \ell\right]
;\\\gamma_{i,j}=\theta_{\sigma\left(  i\right)  ,j}\text{ for all }i,j}}\\
&  \ \ \ \ \ \ \ \ \ \ m^{\left[  k\ell\right]  }\circ P_{\left(  \gamma
_{1,1},\gamma_{1,2},\ldots,\gamma_{k,\ell}\right)  }\circ\left(  \tau\left[
\sigma\right]  \right)  ^{-1}\circ\Delta^{\left[  k\ell\right]  }%
\end{align*}
(here, we have discarded all the addends that do \textbf{not} satisfy $\left(
\gamma_{i,j}=\theta_{\sigma\left(  i\right)  ,j}\text{ for all }i,j\right)  $,
because these addends are $0$). Thus,%
\begin{align}
p_{\alpha,\sigma}\circ p_{\beta,\tau}  &  =\sum_{\substack{\gamma_{i,j}%
\in\mathbb{N}\text{ for all }i\in\left[  k\right]  \text{ and }j\in\left[
\ell\right]  ;\\\gamma_{i,1}+\gamma_{i,2}+\cdots+\gamma_{i,\ell}=\alpha
_{i}\text{ for all }i\in\left[  k\right]  }}\ \ \sum_{\substack{\theta
_{i,j}\in\mathbb{N}\text{ for all }i\in\left[  k\right]  \text{ and }%
j\in\left[  \ell\right]  ;\\\theta_{1,j}+\theta_{2,j}+\cdots+\theta
_{k,j}=\beta_{j}\text{ for all }j\in\left[  \ell\right]  ;\\\gamma
_{i,j}=\theta_{\sigma\left(  i\right)  ,j}\text{ for all }i,j}}\nonumber\\
&  \ \ \ \ \ \ \ \ \ \ \underbrace{m^{\left[  k\ell\right]  }\circ P_{\left(
\gamma_{1,1},\gamma_{1,2},\ldots,\gamma_{k,\ell}\right)  }\circ\left(
\tau\left[  \sigma\right]  \right)  ^{-1}\circ\Delta^{\left[  k\ell\right]  }%
}_{\substack{=p_{\left(  \gamma_{1,1},\gamma_{1,2},\ldots,\gamma_{k,\ell
}\right)  ,\tau\left[  \sigma\right]  }\\\text{(by (\ref{eq.pas.formal-def}%
))}}}\nonumber\\
&  =\sum_{\substack{\gamma_{i,j}\in\mathbb{N}\text{ for all }i\in\left[
k\right]  \text{ and }j\in\left[  \ell\right]  ;\\\gamma_{i,1}+\gamma
_{i,2}+\cdots+\gamma_{i,\ell}=\alpha_{i}\text{ for all }i\in\left[  k\right]
}}\ \ \sum_{\substack{\theta_{i,j}\in\mathbb{N}\text{ for all }i\in\left[
k\right]  \text{ and }j\in\left[  \ell\right]  ;\\\theta_{1,j}+\theta
_{2,j}+\cdots+\theta_{k,j}=\beta_{j}\text{ for all }j\in\left[  \ell\right]
;\\\gamma_{i,j}=\theta_{\sigma\left(  i\right)  ,j}\text{ for all }%
i,j}}p_{\left(  \gamma_{1,1},\gamma_{1,2},\ldots,\gamma_{k,\ell}\right)
,\tau\left[  \sigma\right]  }. \label{pf.thm.sol-mac-1.7}%
\end{align}

Now, let us simplify the inner sum in (\ref{pf.thm.sol-mac-1.7}). If
$\gamma_{i,j}$ and $\theta_{i,j}$ are nonnegative integers for all
$i\in\left[  k\right]  $ and $j\in\left[  \ell\right]  $ that satisfy $\left(
\gamma_{i,j}=\theta_{\sigma\left(  i\right)  ,j}\text{ for all }i,j\right)  $,
then each $j\in\left[  \ell\right]  $ satisfies%
\begin{align*}
\theta_{1,j}+\theta_{2,j}+\cdots+\theta_{k,j}  &  =\theta_{\sigma\left(
1\right)  ,j}+\theta_{\sigma\left(  2\right)  ,j}+\cdots+\theta_{\sigma\left(
k\right)  ,j}\\
&  \ \ \ \ \ \ \ \ \ \ \ \ \ \ \ \ \ \ \ \ \left(  \text{since }\sigma
\in\mathfrak{S}_{k}\text{ is a bijection from }\left[  k\right]  \text{ to
}\left[  k\right]  \right) \\
&  =\gamma_{1,j}+\gamma_{2,j}+\cdots+\gamma_{k,j}%
\end{align*}
(since the condition $\left(  \gamma_{i,j}=\theta_{\sigma\left(  i\right)
,j}\text{ for all }i,j\right)  $ allows us to rewrite each $\theta
_{\sigma\left(  i\right)  ,j}$ as $\gamma_{i,j}$). Thus, the \textquotedblleft%
$\theta_{1,j}+\theta_{2,j}+\cdots+\theta_{k,j}$\textquotedblright\ under the
second summation sign on the right-hand side of (\ref{pf.thm.sol-mac-1.7}) can
be rewritten as \textquotedblleft$\gamma_{1,j}+\gamma_{2,j}+\cdots
+\gamma_{k,j}$\textquotedblright. As a consequence, we can rewrite
(\ref{pf.thm.sol-mac-1.7}) as%
\begin{align*}
&  p_{\alpha,\sigma}\circ p_{\beta,\tau}\\
&  =\sum_{\substack{\gamma_{i,j}\in\mathbb{N}\text{ for all }i\in\left[
k\right]  \text{ and }j\in\left[  \ell\right]  ;\\\gamma_{i,1}+\gamma
_{i,2}+\cdots+\gamma_{i,\ell}=\alpha_{i}\text{ for all }i\in\left[  k\right]
}}\ \ \sum_{\substack{\theta_{i,j}\in\mathbb{N}\text{ for all }i\in\left[
k\right]  \text{ and }j\in\left[  \ell\right]  ;\\\gamma_{1,j}+\gamma
_{2,j}+\cdots+\gamma_{k,j}=\beta_{j}\text{ for all }j\in\left[  \ell\right]
;\\\gamma_{i,j}=\theta_{\sigma\left(  i\right)  ,j}\text{ for all }%
i,j}}p_{\left(  \gamma_{1,1},\gamma_{1,2},\ldots,\gamma_{k,\ell}\right)
,\tau\left[  \sigma\right]  }\\
&  =\sum_{\substack{\gamma_{i,j}\in\mathbb{N}\text{ for all }i\in\left[
k\right]  \text{ and }j\in\left[  \ell\right]  ;\\\gamma_{i,1}+\gamma
_{i,2}+\cdots+\gamma_{i,\ell}=\alpha_{i}\text{ for all }i\in\left[  k\right]
;\\\gamma_{1,j}+\gamma_{2,j}+\cdots+\gamma_{k,j}=\beta_{j}\text{ for all }%
j\in\left[  \ell\right]  ;}}\ \ \underbrace{\sum_{\substack{\theta_{i,j}%
\in\mathbb{N}\text{ for all }i\in\left[  k\right]  \text{ and }j\in\left[
\ell\right]  ;\\\gamma_{i,j}=\theta_{\sigma\left(  i\right)  ,j}\text{ for all
}i,j}}p_{\left(  \gamma_{1,1},\gamma_{1,2},\ldots,\gamma_{k,\ell}\right)
,\tau\left[  \sigma\right]  }}_{\substack{=p_{\left(  \gamma_{1,1}%
,\gamma_{1,2},\ldots,\gamma_{k,\ell}\right)  ,\tau\left[  \sigma\right]
}\\\text{(since the condition \textquotedblleft}\gamma_{i,j}=\theta
_{\sigma\left(  i\right)  ,j}\text{ for all }i,j\text{\textquotedblright%
}\\\text{ensures that }\theta_{i,j}=\gamma_{\sigma^{-1}\left(  i\right)
,j}\text{ for all }i,j\text{,}\\\text{and thus uniquely determines the }%
\theta_{i,j}\text{;}\\\text{hence, the sum has exactly one addend)}}}\\
&  =\sum_{\substack{\gamma_{i,j}\in\mathbb{N}\text{ for all }i\in\left[
k\right]  \text{ and }j\in\left[  \ell\right]  ;\\\gamma_{i,1}+\gamma
_{i,2}+\cdots+\gamma_{i,\ell}=\alpha_{i}\text{ for all }i\in\left[  k\right]
;\\\gamma_{1,j}+\gamma_{2,j}+\cdots+\gamma_{k,j}=\beta_{j}\text{ for all }%
j\in\left[  \ell\right]  }}p_{\left(  \gamma_{1,1},\gamma_{1,2},\ldots
,\gamma_{k,\ell}\right)  ,\tau\left[  \sigma\right]  }.
\end{align*}
This proves Theorem \ref{thm.sol-mac-1}.
\end{proof}

As we mentioned above, (\ref{eq.intro.pp=sum.cocomm}) is a particular case of
Theorem \ref{thm.sol-mac-1} using Proposition \ref{prop.pas.comm}
\textbf{(b)}. Likewise, the analogue of (\ref{eq.intro.pp=sum.cocomm}) for
commutative bialgebras can be recovered from Theorem \ref{thm.sol-mac-1} using
Proposition \ref{prop.pas.comm} \textbf{(a)}.

\begin{question}
In \cite[Proposition 12]{Pang21}, Pang generalized
(\ref{eq.intro.pp=sum.cocomm}) to Hopf algebras with involutions (as long as
they are commutative or cocommutative). Is a similar generalization possible
for Theorem \ref{thm.sol-mac-1}?
\end{question}

\subsection{Linear independence of $p_{\alpha,\sigma}$'s}

Both Theorem \ref{thm.sol-mac-1} and Proposition \ref{prop.sol-mac-0} hold for
arbitrary graded (not necessarily connected) bialgebras. Let us now focus our
view on the connected graded bialgebras (which, as we recall, are
automatically Hopf algebras).

Using Theorem \ref{thm.sol-mac-1} and (\ref{eq.pas-conv}), we can expand any
nested convolution-and-composition of $p_{\alpha,\sigma}$'s as a $\mathbf{k}%
$-linear combination of single $p_{\alpha,\sigma}$'s. Using Proposition
\ref{prop.mopis.reduce} further below, we can furthermore transform the
$\alpha$'s in the resulting combination into actual compositions, not just
weak compositions. This allows us to mechanically prove equalities for
$p_{\alpha,\sigma}$'s that involve only convolution, composition and
$\mathbf{k}$-linear combination and that are supposed to be valid for any
connected graded Hopf algebra $H$. The reason why this works is the following
\textquotedblleft generic linear independence\textquotedblright\ theorem:

\begin{theorem}
\label{thm.pas-lin-ind}\textbf{(a)} There is a connected graded Hopf algebra
$H$ such that the family%
\[
\left(  p_{\alpha,\sigma}\right)  _{\substack{k\in\mathbb{N}\text{;}%
\\\alpha\text{ is a composition of length }k\text{;}\\\sigma\in\mathfrak{S}%
_{k}}}
\]
(of endomorphisms of $H$) is $\mathbf{k}$-linearly independent, and such that
each $H_{n}$ is a free $\mathbf{k}$-module. \medskip

\textbf{(b)} Let $n\in\mathbb{N}$. Then, there is a connected graded Hopf
algebra $H$ such that the family%
\[
\left(  p_{\alpha,\sigma}\right)  _{\substack{k\in\mathbb{N}\text{;}%
\\\alpha\text{ is a composition of length }k\text{ and size }<n\text{;}%
\\\sigma\in\mathfrak{S}_{k}}}
\]
(of endomorphisms of $H$) is $\mathbf{k}$-linearly independent, and such that
each $H_{m}$ is a free $\mathbf{k}$-module with a finite basis.
\end{theorem}

\begin{proof}
[Proof of Theorem \ref{thm.pas-lin-ind} (sketched).]\textbf{(b)} For any weak
composition $\alpha$, we denote by $\ell\left(  \alpha\right)  $ the length of
$\alpha$ (that is, the unique $k\in\mathbb{N}$ satisfying $\alpha\in
\mathbb{N}^{k}$).

Let $H$ be the free $\mathbf{k}$-algebra with generators
\[
x_{i,j}\ \ \ \ \ \ \ \ \ \ \text{with }i,j\in\mathbb{Z}\text{ satisfying
}1\leq i<j\leq n.
\]
We also set%
\[
x_{k,k}:=1_{H}\ \ \ \ \ \ \ \ \ \ \text{for each }k\in\left\{  1,2,\ldots
,n\right\}  .
\]
We make the $\mathbf{k}$-algebra $H$ graded by declaring each $x_{i,j}$ to be
homogeneous of degree $j-i$. We define a comultiplication $\Delta:H\rightarrow
H\otimes H$ on $H$ to be the $\mathbf{k}$-algebra homomorphism that satisfies%
\[
\Delta\left(  x_{i,j}\right)  =\sum_{k=i}^{j}x_{i,k}\otimes x_{k,j}%
\ \ \ \ \ \ \ \ \ \ \text{for each }i,j\in\mathbb{Z}\text{ satisfying }1\leq
i<j\leq n.
\]
We define a counit $\epsilon:H\rightarrow\mathbf{k}$ in the obvious way to
preserve the grading (so that $\epsilon\left(  x_{i,j}\right)  =0$ whenever
$i<j$). It is easy to see that $H$ is a connected graded $\mathbf{k}%
$-bialgebra, thus a connected graded Hopf algebra.\footnote{This $H$ resembles
certain bialgebras that appear in the literature. In particular, it can be
regarded either as a noncommutative version of a reduced incidence algebra of
the chain poset with $n$ elements, or as a noncommutative variant of the
subalgebra of the Schur algebra corresponding to the upper-triangular matrices
with equal numbers on the diagonal. But these connections are tangential for
our purpose.} Moreover, each $H_{m}$ is a free $\mathbf{k}$-module with a
finite basis (the basis consists of all noncommutative monomials of degree $m$).

It is easy to see that%
\[
\Delta^{\left[  k\right]  }\left(  x_{i,j}\right)  =\sum_{i=u_{0}\leq
u_{1}\leq\cdots\leq u_{k}=j}x_{u_{0},u_{1}}\otimes x_{u_{1},u_{2}}%
\otimes\cdots\otimes x_{u_{k-1},u_{k}}%
\]
for any $1\leq i\leq j\leq n$ and any $k\in\mathbb{N}$. Thus, for any $i\leq
j$ and any $k\in\mathbb{N}$ and any permutation $\sigma\in\mathfrak{S}_{k}$,
we have%
\begin{align*}
\left(  \sigma^{-1}\circ\Delta^{\left[  k\right]  }\right)  \left(
x_{i,j}\right)   &  =\sigma^{-1}\cdot\sum_{i=u_{0}\leq u_{1}\leq\cdots\leq
u_{k}=j}x_{u_{0},u_{1}}\otimes x_{u_{1},u_{2}}\otimes\cdots\otimes
x_{u_{k-1},u_{k}}\\
&  =\sum_{i=u_{0}\leq u_{1}\leq\cdots\leq u_{k}=j}x_{u_{\sigma\left(
1\right)  -1},u_{\sigma\left(  1\right)  }}\otimes x_{u_{\sigma\left(
2\right)  -1},u_{\sigma\left(  2\right)  }}\otimes\cdots\otimes x_{u_{\sigma
\left(  k\right)  -1},u_{\sigma\left(  k\right)  }}%
\end{align*}
and therefore (by applying the projection $P_{\alpha}$ to both sides)%
\begin{align*}
&  \left(  P_{\alpha}\circ\sigma^{-1}\circ\Delta^{\left[  k\right]  }\right)
\left(  x_{i,j}\right) \\
&  =\sum_{\substack{i=u_{0}\leq u_{1}\leq\cdots\leq u_{k}=j;\\u_{\sigma\left(
i\right)  }-u_{\sigma\left(  i\right)  -1}=\alpha_{i}\text{ for each }%
i\in\left[  k\right]  }}x_{u_{\sigma\left(  1\right)  -1},u_{\sigma\left(
1\right)  }}\otimes x_{u_{\sigma\left(  2\right)  -1},u_{\sigma\left(
2\right)  }}\otimes\cdots\otimes x_{u_{\sigma\left(  k\right)  -1}%
,u_{\sigma\left(  k\right)  }}.
\end{align*}
Hence, for any composition $\alpha=\left(  \alpha_{1},\alpha_{2},\ldots
,\alpha_{k}\right)  \in\mathbb{N}^{k}$ and any $\sigma\in\mathfrak{S}_{k}$ and
any $1\leq i\leq j\leq n$ satisfying $j-i=\left\vert \alpha\right\vert $, we
have
\[
p_{\alpha,\sigma}\left(  x_{i,j}\right)  =x_{u_{\sigma\left(  1\right)
-1},u_{\sigma\left(  1\right)  }}x_{u_{\sigma\left(  2\right)  -1}%
,u_{\sigma\left(  2\right)  }}\cdots x_{u_{\sigma\left(  k\right)
-1},u_{\sigma\left(  k\right)  }},
\]
where $\left(  u_{0}<u_{1}<\cdots<u_{k}\right)  $ is the unique strictly
increasing sequence of integers satisfying $u_{0}=i$ and $u_{k}=j$ and
$u_{\sigma\left(  m\right)  }-u_{\sigma\left(  m\right)  -1}=\alpha_{m}$ for
all $m\in\left[  k\right]  $\ \ \ \ \footnote{Why is there a unique sequence
with these properties? The easiest way to see this is as follows: Rewrite the
condition \textquotedblleft$u_{\sigma\left(  m\right)  }-u_{\sigma\left(
m\right)  -1}=\alpha_{m}$ for all $m\in\left[  k\right]  $\textquotedblright%
\ as \textquotedblleft$u_{j}-u_{j-1}=\alpha_{\sigma^{-1}\left(  j\right)  }$
for all $j\in\left[  k\right]  $\textquotedblright\ (here we just substituted
$\sigma^{-1}\left(  j\right)  $ for $m$). In the latter form, the condition
simply dictates the differences between consecutive entries of the sequence
$\left(  u_{0},u_{1},\ldots,u_{k}\right)  $. Thus, clearly, there exists a
unique sequence $\left(  u_{0},u_{1},\ldots,u_{k}\right)  $ of integers
satisfying this condition and having starting entry $u_{0}=i$. Moreover, this
sequence will be strictly increasing (since the differences $u_{j}%
-u_{j-1}=\alpha_{\sigma^{-1}\left(  j\right)  }$ are positive) and will
satisfy%
\begin{align*}
u_{k}-u_{0}  &  =\sum_{j=1}^{k}\underbrace{\left(  u_{j}-u_{j-1}\right)
}_{=\alpha_{\sigma^{-1}\left(  j\right)  }}\ \ \ \ \ \ \ \ \ \ \left(
\text{by the telescope principle}\right) \\
&  =\sum_{j=1}^{k}\alpha_{\sigma^{-1}\left(  j\right)  }=\alpha_{\sigma
^{-1}\left(  1\right)  }+\alpha_{\sigma^{-1}\left(  2\right)  }+\cdots
+\alpha_{\sigma^{-1}\left(  k\right)  }=\alpha_{1}+\alpha_{2}+\cdots
+\alpha_{k}=\left\vert \alpha\right\vert =j-i
\end{align*}
and therefore $u_{k}=j-i+\underbrace{u_{0}}_{=i}=j-i+i=j$. Hence, it will be a
strictly increasing sequence $\left(  u_{0}<u_{1}<\cdots<u_{k}\right)  $ of
integers satisfying $u_{0}=i$ and $u_{k}=j$ and $u_{\sigma\left(  m\right)
}-u_{\sigma\left(  m\right)  -1}=\alpha_{m}$ for all $m\in\left[  k\right]
$.}. Hence, for any choice of $1\leq i\leq j\leq n$, the images $p_{\alpha
,\sigma}\left(  x_{i,j}\right)  $ as $\alpha$ runs over all compositions of
$j-i$ and $\sigma$ runs over all permutations of $\left[  \ell\left(
\alpha\right)  \right]  $ are distinct monomials\footnote{Why are they
distinct? Because the noncommutative monomial%
\[
p_{\alpha,\sigma}\left(  x_{i,j}\right)  =x_{u_{\sigma\left(  1\right)
-1},u_{\sigma\left(  1\right)  }}x_{u_{\sigma\left(  2\right)  -1}%
,u_{\sigma\left(  2\right)  }}\cdots x_{u_{\sigma\left(  k\right)
-1},u_{\sigma\left(  k\right)  }}%
\]
allows us to uniquely reconstruct $\alpha$ and $\sigma$ as follows: Recall
that all the indeterminates $x_{p,q}$ with $p<q$ are distinct. The
noncommutative monomial $p_{\alpha,\sigma}\left(  x_{i,j}\right)
=x_{u_{\sigma\left(  1\right)  -1},u_{\sigma\left(  1\right)  }}%
x_{u_{\sigma\left(  2\right)  -1},u_{\sigma\left(  2\right)  }}\cdots
x_{u_{\sigma\left(  k\right)  -1},u_{\sigma\left(  k\right)  }}$ is a product
of $k$ such indeterminates (since $u_{0}<u_{1}<\cdots<u_{k}$, so that
$u_{\sigma\left(  m\right)  -1}<u_{\sigma\left(  m\right)  }$ for each $m$).
Hence, from this product, we can recover the pairs $\left(  u_{\sigma\left(
1\right)  -1},u_{\sigma\left(  1\right)  }\right)  ,\ \left(  u_{\sigma\left(
2\right)  -1},u_{\sigma\left(  2\right)  }\right)  ,\ \ldots,\ \left(
u_{\sigma\left(  k\right)  -1},u_{\sigma\left(  k\right)  }\right)  $, and
thus reconstruct the values $u_{\sigma\left(  1\right)  },u_{\sigma\left(
2\right)  },\ldots,u_{\sigma\left(  k\right)  }$. These, in turn, allow us to
reconstruct the permutation $\sigma$ (since $u_{0}<u_{1}<\cdots<u_{k}$).
Recalling the condition $u_{\sigma\left(  m\right)  }-u_{\sigma\left(
m\right)  -1}=\alpha_{m}$ once again, we then recover the $\alpha_{m}$ and
thus the composition $\alpha$.} and therefore are $\mathbf{k}$-linearly
independent. This yields the $\mathbf{k}$-linear independence of the family%
\[
\left(  p_{\alpha,\sigma}\right)  _{\substack{k\in\mathbb{N}\text{;}%
\\\alpha\text{ is a composition of length }k\text{ and size }s\text{;}%
\\\sigma\in\mathfrak{S}_{k}}}
\]
for any given $s\in\left\{  0,1,\ldots,n-1\right\}  $. Since each
$p_{\alpha,\sigma}$ lies in $\operatorname*{End}\nolimits_{\mathbf{k}}\left(
H_{\left\vert \alpha\right\vert }\right)  $, we thus obtain the $\mathbf{k}%
$-linear independence of the entire family%
\[
\left(  p_{\alpha,\sigma}\right)  _{\substack{k\in\mathbb{N}\text{;}%
\\\alpha\text{ is a composition of length }k\text{ and size }<n\text{;}%
\\\sigma\in\mathfrak{S}_{k}}}
\]
(since the sum $\sum_{s=0}^{n-1}\operatorname*{End}\nolimits_{\mathbf{k}%
}\left(  H_{s}\right)  $ is a direct sum). This proves Theorem
\ref{thm.pas-lin-ind} \textbf{(b)}. \medskip

\textbf{(a)} As for part \textbf{(b)}, but remove \textquotedblleft$\leq
n$\textquotedblright\ everywhere.
\end{proof}

Theorem \ref{thm.pas-lin-ind} also has an analogue for graded bialgebras and
Hopf algebras without the connectedness requirement:

\begin{theorem}
\label{thm.pas-lin-ind2}There is a graded Hopf algebra $H$ such that the
family%
\[
\left(  p_{\alpha,\sigma}\right)  _{\substack{k\in\mathbb{N}\text{;}%
\\\alpha\text{ is a weak composition of length }k\text{;}\\\sigma
\in\mathfrak{S}_{k}}}
\]
(of endomorphisms of $H$) is $\mathbf{k}$-linearly independent, and such that
each $H_{n}$ is a free $\mathbf{k}$-module.
\end{theorem}

\begin{proof}
[Proof idea.]This is similar to Theorem \ref{thm.pas-lin-ind} \textbf{(b)},
but there are more indices involved. We shall only give a rough outline of the proof.

\begin{noncompile}
Equip the set $\left[  n\right]  ^{2}$ with the componentwise partial order:
i.e., the partial order given by%
\[
\left(  i,p\right)  \leq\left(  j,q\right)  \ \ \ \ \ \ \ \ \ \ \text{if and
only if}\ \ \ \ \ \ \ \ \ \ \left(  i\leq j\text{ and }p\leq q\right)  .
\]

\end{noncompile}

Let $H$ be the free $\mathbf{k}$-algebra with generators
\[
x_{i,j}^{p,q}\ \ \ \ \ \ \ \ \ \ \text{with }\left(  i,p\right)  ,\left(
j,q\right)  \in\mathbb{Z}^{2}\text{ satisfying }i\leq j\text{ and }p\leq q.
\]
(The \textquotedblleft$p,q$\textquotedblright\ superscript is not an
exponent!) Note that we don't set any of the $x_{i,j}^{p,q}$ equal to $1$,
unlike in the previous proof.

We make the $\mathbf{k}$-algebra $H$ graded by declaring each $x_{i,j}^{p,q}$
to be homogeneous of degree $j-i$. We define a comultiplication $\Delta
:H\rightarrow H\otimes H$ on $H$ to be the $\mathbf{k}$-algebra homomorphism
that satisfies%
\[
\Delta\left(  x_{i,j}^{p,q}\right)  =\sum_{k=i}^{j}\ \ \sum_{r=p}^{q}%
x_{i,k}^{p,r}\otimes x_{k,j}^{r,q}\ \ \ \ \ \ \ \ \ \ \text{for all generators
}x_{i,j}^{p,q}.
\]
We define a counit $\epsilon:H\rightarrow\mathbf{k}$ by%
\[
\epsilon\left(  x_{i,j}^{p,q}\right)  =%
\begin{cases}
1, & \text{if }\left(  i,p\right)  =\left(  j,q\right)  ;\\
0, & \text{if }\left(  i,p\right)  \neq\left(  j,q\right)  .
\end{cases}
\]
It is easy to see that $H$ is a graded $\mathbf{k}$-bialgebra. It is not
connected (with our grading) and not a Hopf algebra (since its grouplike
elements $x_{i,i}^{p,p}$ have no inverses); we will tweak it to become a Hopf
algebra later.

For any generator $x_{i,j}^{p,q}$ of $H$ and any $k\in\mathbb{N}$, we can show
that%
\[
\Delta^{\left[  k\right]  }\left(  x_{i,j}^{p,q}\right)  =\sum
_{\substack{i=u_{0}\leq u_{1}\leq\cdots\leq u_{k}=j;\\p=v_{0}\leq v_{1}%
\leq\cdots\leq v_{k}=q}}x_{u_{0},u_{1}}^{v_{0},v_{1}}\otimes x_{u_{1},u_{2}%
}^{v_{1},v_{2}}\otimes\cdots\otimes x_{u_{k-1},u_{k}}^{v_{k-1},v_{k}}.
\]

Now, fix $n\in\mathbb{N}$, and fix $i\leq j$ in $\mathbb{Z}$. As we just
showed, we have%
\[
\Delta^{\left[  k\right]  }\left(  x_{i,j}^{0,n}\right)  =\sum
_{\substack{i=u_{0}\leq u_{1}\leq\cdots\leq u_{k}=j;\\0=v_{0}\leq v_{1}%
\leq\cdots\leq v_{k}=n}}x_{u_{0},u_{1}}^{v_{0},v_{1}}\otimes x_{u_{1},u_{2}%
}^{v_{1},v_{2}}\otimes\cdots\otimes x_{u_{k-1},u_{k}}^{v_{k-1},v_{k}}%
\]
for any $k\in\mathbb{N}$. Thus, by a similar argument as in the proof of
Theorem \ref{thm.pas-lin-ind} \textbf{(b)}, we can show that for any weak
composition $\alpha=\left(  \alpha_{1},\alpha_{2},\ldots,\alpha_{k}\right)
\in\mathbb{N}^{k}$ and any $\sigma\in\mathfrak{S}_{k}$ and any integers $i\leq
j$ satisfying $j-i=\left\vert \alpha\right\vert $, we have
\[
p_{\alpha,\sigma}\left(  x_{i,j}^{0,n}\right)  =\sum_{\substack{i=u_{0}\leq
u_{1}\leq\cdots\leq u_{k}=j;\\0=v_{0}\leq v_{1}\leq\cdots\leq v_{k}%
=n;\\u_{\sigma\left(  r\right)  }-u_{\sigma\left(  r\right)  -1}=\alpha
_{r}\text{ for each }r\in\left[  k\right]  }}\left(  x_{u_{0},u_{1}}%
^{v_{0},v_{1}}x_{u_{1},u_{2}}^{v_{1},v_{2}}\cdots x_{u_{k-1},u_{k}}%
^{v_{k-1},v_{k}}\right)  \leftharpoonup\sigma,
\]
where the \textquotedblleft$\leftharpoonup$\textquotedblright\ arrow denotes
the action of a permutation on the positions in a monomial (i.e., we set
$\left(  y_{1}y_{2}\cdots y_{k}\right)  \leftharpoonup\sigma:=y_{\sigma\left(
1\right)  }y_{\sigma\left(  2\right)  }\cdots y_{\sigma\left(  k\right)  }$
whenever $\sigma\in\mathfrak{S}_{k}$ and whenever $y_{1},y_{2},\ldots,y_{k}$
are some of the generators of $H$).

We now claim that these images $p_{\alpha,\sigma}\left(  x_{i,j}^{0,n}\right)
$ -- where $k$ ranges over $\left\{  1,2,\ldots,n-1\right\}  $, where $\alpha$
ranges over the weak compositions of length $k$ and size $j-i$, and where
$\sigma$ ranges over $\mathfrak{S}_{k}$ -- are $\mathbf{k}$-linearly
independent. Unfortunately, they no longer are single monomials, but hope is
not lost: Each of them contains a monomial (with coefficient $1$) that is not
contained in any of the others (which ensures that any nontrivial $\mathbf{k}%
$-linear combination of these images will have nonzero coefficients in front
of some of these monomials). To wit, $p_{\alpha,\sigma}\left(  x_{i,j}%
^{0,n}\right)  $ contains (with coefficient $1$) the monomial%
\[
\left(  x_{u_{0},u_{1}}^{0,1}x_{u_{1},u_{2}}^{1,2}\cdots x_{u_{k-2},u_{k-1}%
}^{k-2,k-1}x_{u_{k-1},u_{k}}^{k-1,n}\right)  \leftharpoonup\sigma
\]
for the unique weakly increasing sequence $\left(  u_{0}\leq u_{1}\leq
\cdots\leq u_{k}\right)  $ of integers satisfying $u_{0}=i$ and $u_{k}=j$ and
$u_{\sigma\left(  r\right)  }-u_{\sigma\left(  r\right)  -1}=\alpha_{r}$ for
each $r\in\left[  k\right]  $ (this monomial is obtained by taking $v_{r}=r$
for each $r\in\left\{  0,1,\ldots,k-1\right\}  $ and $v_{k}=n$), and this
monomial is not contained in any other $p_{\beta,\tau}\left(  x_{i,j}%
^{0,n}\right)  $ with $\left(  \beta,\tau\right)  \neq\left(  \alpha
,\sigma\right)  $ (because reading its superscripts in order reveals the
permutation $\sigma$, and then undoing the $\sigma$-action and reading the
subscripts in order reveals the $u_{r}$'s and thus the weak composition
$\alpha$). Thus, the images $p_{\alpha,\sigma}\left(  x_{i,j}^{0,n}\right)  $
are $\mathbf{k}$-linearly independent. Hence, so are the maps $p_{\alpha
,\sigma}$ themselves, at least for $k$ ranging over $\left\{  1,2,\ldots
,n-1\right\}  $. Moreover, we can replace \textquotedblleft$k$ ranging over
$\left\{  1,2,\ldots,n-1\right\}  $\textquotedblright\ by \textquotedblleft$k$
ranging over $\left\{  0,1,\ldots,n-1\right\}  $\textquotedblright\ in the
preceding sentence, because allowing $k=0$ only introduces a single new map
$p_{\alpha,\sigma}$, namely $p_{\left(  {}\right)  ,\operatorname*{id}%
\nolimits_{\left[  {}\right]  }}$, and this map is still linearly independent
from the other $p_{\alpha,\sigma}$'s because it annihilates $x_{i,j}^{0,n}$
while sending $1$ to $1$.

Thus, we have shown that the maps $p_{\alpha,\sigma}$ -- where $k$ ranges over
$\left\{  0,1,\ldots,n-1\right\}  $, where $\alpha$ ranges over the weak
compositions of length $k$ and size $<n$, and where $\sigma$ ranges over
$\mathfrak{S}_{k}$ -- are $\mathbf{k}$-linearly independent (indeed,
compositions $\alpha$ of different sizes lead to maps $p_{\alpha,\sigma}$ that
live in different graded components, and thus cannot be linearly dependent).

By letting $n\rightarrow\infty$, we thus conclude that \textbf{all} the maps
$p_{\alpha,\sigma}$ -- where $k$ ranges over all nonnegative integers, where
$\alpha$ ranges over the weak compositions of length $k$, and where $\sigma$
ranges over $\mathfrak{S}_{k}$ -- are $\mathbf{k}$-linearly independent (since
any $\mathbf{k}$-linear relation between these maps $p_{\alpha,\sigma}$ can
only use finitely many of them, and thus only uses $k\in\left\{
0,1,\ldots,n-1\right\}  $ and size $<n$ for a certain $n$).

From here, we can finish the proof of Theorem \ref{thm.pas-lin-ind2} just as
in the proof of Theorem \ref{thm.pas-lin-ind}, except that our $H$ is just a
bialgebra, not a Hopf algebra.

Now how do we make $H$ into a Hopf algebra? The simplest way is to adjoin an
inverse to each of the grouplike elements $x_{i,i}^{p,p}$. That is, we take
the noncommutative localization $H^{\prime}$ of $H$ at the multiplicative set
generated by the $x_{i,i}^{p,p}$ for all $i,p\in\mathbb{Z}$ (that is, we
adjoin new elements $y_{i,p}$ for all $i,p\in\mathbb{Z}$ to $H$ which are to
satisfy $y_{i,p}x_{i,i}^{p,p}=x_{i,i}^{p,p}y_{i,p}=1$). Since the
$x_{i,i}^{p,p}$ are grouplike, this localization $H^{\prime}$ is still a
bialgebra\footnote{This is easiest to see by viewing the inverses of the
$x_{i,i}^{p,p}$ as extra generators $y_{i,p}$, and extending $\Delta$ and
$\epsilon$ to these extra generators by setting $\Delta\left(  y_{i,p}\right)
=y_{i,p}\otimes y_{i,p}$ and $\epsilon\left(  y_{i,p}\right)  =1$. Once again,
all relevant axioms are easy to check.}, and it contains $H$ as a subbialgebra
(i.e., the canonical $\mathbf{k}$-algebra morphism $H\rightarrow H^{\prime}$
is injective\footnote{Indeed, as $\mathbf{k}$-algebras, both $H$ and
$H^{\prime}$ are monoid algebras: Namely, $H$ is the monoid algebra of the
free monoid on the generators $x_{i,j}^{p,q}$, whereas $H^{\prime}$ is the
monoid algebra of the same free monoid amalgamated with the free group on the
generators $x_{i,i}^{p,p}$. It is easy to see that the former monoid embeds
into the latter (since the latter has a normal form consisting of reduced
words); thus, the algebra $H$ embeds into $H^{\prime}$.}). Moreover, it is
still graded, since the original grading $H$ can be extended to it in an
obvious way (indeed, the elements $x_{i,i}^{p,p}$ whose inverses we have
adjoined were of degree $0$, so the inverses will have degree $0$ as well).

Furthermore, the $\mathbf{k}$-bialgebra $H^{\prime}$ can be assigned with a
\textbf{new} grading, in which each generator $x_{i,j}^{p,q}$ is homogeneous
of degree $\left(  j-i\right)  +\left(  q-p\right)  $ (not $j-i$ any more)
while the new extra generators $\left(  x_{i,i}^{p,p}\right)  ^{-1}$ are
homogeneous of degree $0$. In this new grading, the $0$-th graded component is
generated by the grouplike elements $x_{i,i}^{p,p}$ and their inverses
$\left(  x_{i,i}^{p,p}\right)  ^{-1}$ (in fact, it is the group algebra of a
free group), so it is a Hopf algebra. Hence, by a famous result of Takeuchi
(\cite[Exercise 1.4.34 (b)]{GriRei}), the whole $H^{\prime}$ is a Hopf algebra
as well. Equipped with the original grading, $H^{\prime}$ is thus a graded
Hopf algebra (since any graded bialgebra with an antipode is a graded Hopf
algebra -- i.e., its antipode is a graded map\footnote{This is a folklore
fact; see \cite[Theorem 23.2]{logid} for a detailed proof.}). Of course, the
$p_{\alpha,\sigma}$'s for this graded Hopf algebra $H^{\prime}$ (with the
original grading!) are still $\mathbf{k}$-linearly independent, since they
were $\mathbf{k}$-linearly independent for $H$ (and $H$ embeds into
$H^{\prime}$). Moreover, each $H_{n}^{\prime}$ is a free $\mathbf{k}$-module,
as can be easily shown by constructing its basis as reduced words. Thus,
Theorem \ref{thm.pas-lin-ind2} is proved.
\end{proof}

\subsection{The tensor product formula}

We shall now connect the $p_{\alpha,\sigma}$ operators on different graded
bialgebras. For this, we need a piece of notation:

\begin{convention}
\label{conv.pas-for-H}The endomorphism $p_{\alpha,\sigma}:H\rightarrow H$
defined for a given graded $\mathbf{k}$-bialgebra $H$ shall be denoted by
\textquotedblleft$p_{\alpha,\sigma}$ for $H$\textquotedblright\ when we want
to stress its dependence on $H$. Thus, for instance, if $K$ is another graded
$\mathbf{k}$-bialgebra, then the analogous endomorphism $p_{\alpha,\sigma
}:K\rightarrow K$ will be denoted by \textquotedblleft$p_{\alpha,\sigma}$ for
$K$\textquotedblright.
\end{convention}

Using this convention, we can now describe the endomorphism $p_{\alpha,\sigma
}$ for a tensor product of two $\mathbf{k}$-bialgebras:

\begin{proposition}
\label{prop.pas-on-tensors}Let $H$ and $G$ be two graded bialgebras. Let
$\alpha\in\mathbb{N}^{k}$ be a weak composition of length $k$, and let
$\sigma\in\mathfrak{S}_{k}$ be a permutation. Then,%
\[
\left(  p_{\alpha,\sigma}\text{ for }H\otimes G\right)  =\sum_{\substack{\beta
,\gamma\in\mathbb{N}^{k}\text{;}\\\beta+\gamma=\alpha}}\left(  p_{\beta
,\sigma}\text{ for }H\right)  \otimes\left(  p_{\gamma,\sigma}\text{ for
}G\right)
\]
as endomorphisms of $H\otimes G$. Here, $\beta+\gamma$ denotes the entrywise
sum of $\beta$ and $\gamma$ (that is, if $\beta=\left(  \beta_{1},\beta
_{2},\ldots,\beta_{k}\right)  $ and $\gamma=\left(  \gamma_{1},\gamma
_{2},\ldots,\gamma_{k}\right)  $, then $\beta+\gamma:=\left(  \beta_{1}%
+\gamma_{1},\ \beta_{2}+\gamma_{2},\ \ldots,\ \beta_{k}+\gamma_{k}\right)  $).
\end{proposition}

\begin{proof}
[Proof sketch.]We shall use the Sweedler notation again.

The definition of the grading on $H\otimes G$ shows that $\left(  H\otimes
G\right)  _{a}=\bigoplus_{\substack{b,c\in\mathbb{N};\\b+c=a}}H_{b}\otimes
G_{c}$ for any $a\in\mathbb{N}$. Hence,%
\begin{equation}
p_{a}\left(  u\otimes v\right)  =\sum_{\substack{b,c\in\mathbb{N}%
;\\b+c=a}}p_{b}\left(  u\right)  \otimes p_{c}\left(  v\right)
\label{pf.prop.pas-on-tensors.grading}%
\end{equation}
for any $u\in H$, any $v\in G$ and any $a\in\mathbb{N}$.

Write the weak composition $\alpha\in\mathbb{N}^{k}$ as $\alpha=\left(
\alpha_{1},\alpha_{2},\ldots,\alpha_{k}\right)  $. Consider a pure tensor
$h\otimes g\in H\otimes G$. Then, Remark \ref{rmk.pas.informal-def} yields%
\begin{align*}
&  p_{\alpha,\sigma}\left(  h\otimes g\right) \\
&  =\sum_{\left(  h\otimes g\right)  }p_{\alpha_{1}}\left(  \left(  h\otimes
g\right)  _{\left(  \sigma\left(  1\right)  \right)  }\right)  p_{\alpha_{2}%
}\left(  \left(  h\otimes g\right)  _{\left(  \sigma\left(  2\right)  \right)
}\right)  \cdots p_{\alpha_{k}}\left(  \left(  h\otimes g\right)  _{\left(
\sigma\left(  k\right)  \right)  }\right) \\
&  =\sum_{\left(  h\otimes g\right)  }\ \ \prod_{i=1}^{k}p_{\alpha_{i}}\left(
\underbrace{\left(  h\otimes g\right)  _{\left(  \sigma\left(  i\right)
\right)  }}_{=h_{\left(  \sigma\left(  i\right)  \right)  }\otimes g_{\left(
\sigma\left(  i\right)  \right)  }}\right) \\
&  =\sum_{\left(  h\right)  ,\ \left(  g\right)  }\ \ \prod_{i=1}%
^{k}\underbrace{p_{\alpha_{i}}\left(  h_{\left(  \sigma\left(  i\right)
\right)  }\otimes g_{\left(  \sigma\left(  i\right)  \right)  }\right)
}_{\substack{=\sum_{\substack{b,c\in\mathbb{N};\\b+c=\alpha_{i}}}p_{b}\left(
h_{\left(  \sigma\left(  i\right)  \right)  }\right)  \otimes p_{c}\left(
g_{\left(  \sigma\left(  i\right)  \right)  }\right)  \\\text{(by
(\ref{pf.prop.pas-on-tensors.grading}))}}}\\
&  =\sum_{\left(  h\right)  ,\ \left(  g\right)  }\ \ \prod_{i=1}^{k}%
\ \ \sum_{\substack{b,c\in\mathbb{N};\\b+c=\alpha_{i}}}p_{b}\left(  h_{\left(
\sigma\left(  i\right)  \right)  }\right)  \otimes p_{c}\left(  g_{\left(
\sigma\left(  i\right)  \right)  }\right) \\
&  =\sum_{\left(  h\right)  ,\ \left(  g\right)  }\ \ \sum_{\substack{\beta
_{i},\gamma_{i}\in\mathbb{N}\text{ for all }i\in\left[  k\right]  ;\\\beta
_{i}+\gamma_{i}=\alpha_{i}\text{ for all }i\in\left[  k\right]  }%
}\ \ \prod_{i=1}^{k}\left(  p_{\beta_{i}}\left(  h_{\left(  \sigma\left(
i\right)  \right)  }\right)  \otimes p_{\gamma_{i}}\left(  g_{\left(
\sigma\left(  i\right)  \right)  }\right)  \right)
\ \ \ \ \ \ \ \ \ \ \left(  \text{by the product rule}\right) \\
&  =\underbrace{\sum_{\substack{\beta_{i},\gamma_{i}\in\mathbb{N}\text{ for
all }i\in\left[  k\right]  ;\\\beta_{i}+\gamma_{i}=\alpha_{i}\text{ for all
}i\in\left[  k\right]  }}}_{=\sum_{\substack{\beta=\left(  \beta_{1},\beta
_{2},\ldots,\beta_{k}\right)  \in\mathbb{N}^{k}\text{ and}\\\gamma=\left(
\gamma_{1},\gamma_{2},\ldots,\gamma_{k}\right)  \in\mathbb{N}^{k}%
;\\\beta+\gamma=\alpha}}}\ \ \sum_{\left(  h\right)  ,\ \left(  g\right)
}\ \ \underbrace{\prod_{i=1}^{k}\left(  p_{\beta_{i}}\left(  h_{\left(
\sigma\left(  i\right)  \right)  }\right)  \otimes p_{\gamma_{i}}\left(
g_{\left(  \sigma\left(  i\right)  \right)  }\right)  \right)  }_{=\left(
\prod_{i=1}^{k}p_{\beta_{i}}\left(  h_{\left(  \sigma\left(  i\right)
\right)  }\right)  \right)  \otimes\left(  \prod_{i=1}^{k}p_{\gamma_{i}%
}\left(  g_{\left(  \sigma\left(  i\right)  \right)  }\right)  \right)  }\\
&  =\sum_{\substack{\beta=\left(  \beta_{1},\beta_{2},\ldots,\beta_{k}\right)
\in\mathbb{N}^{k}\text{ and}\\\gamma=\left(  \gamma_{1},\gamma_{2}%
,\ldots,\gamma_{k}\right)  \in\mathbb{N}^{k};\\\beta+\gamma=\alpha}%
}\ \ \sum_{\left(  h\right)  ,\ \left(  g\right)  }\left(  \prod_{i=1}%
^{k}p_{\beta_{i}}\left(  h_{\left(  \sigma\left(  i\right)  \right)  }\right)
\right)  \otimes\left(  \prod_{i=1}^{k}p_{\gamma_{i}}\left(  g_{\left(
\sigma\left(  i\right)  \right)  }\right)  \right) \\
&  =\sum_{\substack{\beta=\left(  \beta_{1},\beta_{2},\ldots,\beta_{k}\right)
\in\mathbb{N}^{k}\text{ and}\\\gamma=\left(  \gamma_{1},\gamma_{2}%
,\ldots,\gamma_{k}\right)  \in\mathbb{N}^{k};\\\beta+\gamma=\alpha
}}\ \ \underbrace{\left(  \sum_{\left(  h\right)  }\ \ \prod_{i=1}^{k}%
p_{\beta_{i}}\left(  h_{\left(  \sigma\left(  i\right)  \right)  }\right)
\right)  }_{\substack{=p_{\beta,\sigma}\left(  h\right)  \\\text{(by Remark
\ref{rmk.pas.informal-def})}}}\otimes\underbrace{\left(  \sum_{\left(
g\right)  }\ \ \prod_{i=1}^{k}p_{\gamma_{i}}\left(  g_{\left(  \sigma\left(
i\right)  \right)  }\right)  \right)  }_{\substack{=p_{\gamma,\sigma}\left(
g\right)  \\\text{(by Remark \ref{rmk.pas.informal-def})}}}
\end{align*}%
\begin{align*}
&  =\sum_{\substack{\beta=\left(  \beta_{1},\beta_{2},\ldots,\beta_{k}\right)
\in\mathbb{N}^{k}\text{ and}\\\gamma=\left(  \gamma_{1},\gamma_{2}%
,\ldots,\gamma_{k}\right)  \in\mathbb{N}^{k};\\\beta+\gamma=\alpha}%
}p_{\beta,\sigma}\left(  h\right)  \otimes p_{\gamma,\sigma}\left(  g\right)
=\sum_{\substack{\beta,\gamma\in\mathbb{N}^{k};\\\beta+\gamma=\alpha}%
}p_{\beta,\sigma}\left(  h\right)  \otimes p_{\gamma,\sigma}\left(  g\right)
\\
&  =\left(  \sum_{\substack{\beta,\gamma\in\mathbb{N}^{k};\\\beta
+\gamma=\alpha}}\left(  p_{\beta,\sigma}\text{ for }H\right)  \otimes\left(
p_{\gamma,\sigma}\text{ for }G\right)  \right)  \left(  h\otimes g\right)  .
\end{align*}
Hence, the two $\mathbf{k}$-linear maps%
\[
\left(  p_{\alpha,\sigma}\text{ for }H\otimes G\right)  \text{ and }%
\sum_{\substack{\beta,\gamma\in\mathbb{N}^{k};\\\beta+\gamma=\alpha}}\left(
p_{\beta,\sigma}\text{ for }H\right)  \otimes\left(  p_{\gamma,\sigma}\text{
for }G\right)
\]
agree on each pure tensor. Therefore, they are identical. This proves
Proposition \ref{prop.pas-on-tensors}.
\end{proof}

\subsection{The dual formula}

Graded bialgebras often (not always) have duals. Indeed, if $H=\bigoplus
\limits_{n\in\mathbb{N}}H_{n}$ is a graded $\mathbf{k}$-bialgebra whose all
graded components $H_{n}$ are finite free\footnote{\textquotedblleft Finite
free\textquotedblright\ means \textquotedblleft free of finite
rank\textquotedblright, i.e., \textquotedblleft has a finite
basis\textquotedblright.} $\mathbf{k}$-modules, then its graded dual
$H^{o}:=\bigoplus\limits_{n\in\mathbb{N}}\left(  H_{n}\right)  ^{\ast}$ itself
becomes a graded $\mathbf{k}$-bialgebra (see \cite[\S 1.6]{GriRei}). Now we
claim (using Convention \ref{conv.pas-for-H} again):

\begin{proposition}
\label{prop.pas-on-dual}Let $H$ be a graded $\mathbf{k}$-bialgebra whose all
graded components are finite free $\mathbf{k}$-modules. Consider its graded
dual $H^{o}$. Let $\alpha=\left(  \alpha_{1},\alpha_{2},\ldots,\alpha
_{k}\right)  $ be a weak composition of length $k$, and let $\sigma
\in\mathfrak{S}_{k}$ be a permutation. Then,%
\[
\left(  p_{\alpha,\sigma}\text{ for }H^{o}\right)  =\underbrace{\left(
p_{\alpha\cdot\sigma^{-1},\ \sigma^{-1}}\text{ for }H\right)  ^{\ast}%
}_{\text{restricted to }H^{o}}%
\]
as endomorphisms of $H^{o}$. Here, \textquotedblleft$\alpha\cdot\sigma^{-1}%
$\textquotedblright\ is understood using the right action of $\mathfrak{S}%
_{k}$ on $\mathbb{N}^{k}$ defined in (\ref{eq.gammapi}).
\end{proposition}

\begin{proof}
[Proof sketch.]By definition, $\left(  p_{\alpha\cdot\sigma^{-1},\ \sigma
^{-1}}\text{ for }H\right)  =m_{H}^{\left[  k\right]  }\circ P_{\alpha
\cdot\sigma^{-1}}\circ\underbrace{\left(  \sigma^{-1}\right)  ^{-1}}_{=\sigma
}\circ\,\Delta_{H}^{\left[  k\right]  }=m_{H}^{\left[  k\right]  }\circ
P_{\alpha\cdot\sigma^{-1}}\circ\sigma\circ\Delta_{H}^{\left[  k\right]  }$, so
that%
\begin{align}
\left(  p_{\alpha\cdot\sigma^{-1},\ \sigma^{-1}}\text{ for }H\right)  ^{\ast}
&  =\left(  m_{H}^{\left[  k\right]  }\circ P_{\alpha\cdot\sigma^{-1}}%
\circ\sigma\circ\Delta_{H}^{\left[  k\right]  }\right)  ^{\ast}\nonumber\\
&  =\left(  \Delta_{H}^{\left[  k\right]  }\right)  ^{\ast}\circ\sigma^{\ast
}\circ\left(  P_{\alpha\cdot\sigma^{-1}}\right)  ^{\ast}\circ\left(
m_{H}^{\left[  k\right]  }\right)  ^{\ast}. \label{pf.prop.pas-on-dual.1}%
\end{align}
However, the definition of the bialgebra structure on the graded dual $H^{o}$
shows that $\left(  \Delta_{H}^{\left[  k\right]  }\right)  ^{\ast}=m_{H^{o}%
}^{\left[  k\right]  }$ and $\left(  m_{H}^{\left[  k\right]  }\right)
^{\ast}=\Delta_{H^{o}}^{\left[  k\right]  }$. Moreover, basic linear algebra
(really just combinatorics of indices in tensor products) shows that
$\sigma^{\ast}=\sigma^{-1}$ and that $\left(  P_{\alpha\cdot\sigma^{-1}%
}\right)  ^{\ast}=P_{\alpha\cdot\sigma^{-1}}$ (since $\left(  P_{\beta
}\right)  ^{\ast}=P_{\beta}$ for any weak composition $\beta$). Using these
four equalities, we can rewrite (\ref{pf.prop.pas-on-dual.1}) as%
\begin{align*}
\left(  p_{\alpha\cdot\sigma^{-1},\ \sigma^{-1}}\text{ for }H\right)  ^{\ast}
&  =m_{H^{o}}^{\left[  k\right]  }\circ\underbrace{\sigma^{-1}\circ
P_{\alpha\cdot\sigma^{-1}}}_{\substack{=P_{\alpha\cdot\sigma^{-1}\cdot\left(
\sigma^{-1}\right)  ^{-1}}\circ\sigma^{-1}\\\text{(by
(\ref{pf.thm.sol-mac-1.piP}))}}}\circ\,\Delta_{H^{o}}^{\left[  k\right]  }\\
&  =m_{H^{o}}^{\left[  k\right]  }\circ\underbrace{P_{\alpha\cdot\sigma
^{-1}\cdot\left(  \sigma^{-1}\right)  ^{-1}}}_{=P_{\alpha}}\circ\,\sigma
^{-1}\circ\Delta_{H^{o}}^{\left[  k\right]  }\\
&  =m_{H^{o}}^{\left[  k\right]  }\circ P_{\alpha}\circ\sigma^{-1}\circ
\Delta_{H^{o}}^{\left[  k\right]  }=\left(  p_{\alpha,\sigma}\text{ for }%
H^{o}\right)
\end{align*}
(again by the definition of $p_{\alpha,\sigma}$). This proves Proposition
\ref{prop.pas-on-dual}.
\end{proof}

\subsection{Have we found them all?}

Now let $H$ be a connected graded bialgebra. As we already mentioned, each
$p_{\alpha,\sigma}$ belongs to the $\mathbf{k}$-module $\mathbf{E}\left(
H\right)  $ of all graded $\mathbf{k}$-module endomorphisms of $H$ that
annihilate all but finitely many degrees of $H$. Thus, the same holds for any
$\mathbf{k}$-linear combination of the $p_{\alpha,\sigma}$. The fact that each
$p_{\alpha,\sigma}$ annihilates all but the $\left\vert \alpha\right\vert $-th
graded component of $H$ allows us to form infinite $\mathbf{k}$-linear
combinations $\sum_{k\in\mathbb{N}}\ \ \sum_{\alpha\in\left\{  1,2,3,\ldots
\right\}  ^{k}}\ \ \sum_{\sigma\in\mathfrak{S}_{k}}\lambda_{\alpha,\sigma
}p_{\alpha,\sigma}$ of these $p_{\alpha,\sigma}$ as well. Each such
combination belongs to $\operatorname*{End}\nolimits_{\operatorname*{gr}}H$
(although not to $\mathbf{E}\left(  H\right)  $ any more) and (since it is
natural in $H$) is therefore a natural graded $\mathbf{k}$-module endomorphism
of $H$ defined for any connected graded bialgebra $H$.

\begin{question}
\label{quest.natural}Are these combinations the only natural graded
$\mathbf{k}$-module endomorphisms of $H$ defined for any connected graded
bialgebra $H$ ?

In other words: Let $g$ be a natural graded $\mathbf{k}$-module endomorphism
on the category of connected graded $\mathbf{k}$-Hopf algebras. (That is, for
each connected graded Hopf algebra $H$, we have a graded $\mathbf{k}$-module
endomorphism $g_{H}$, and each graded Hopf algebra morphism $\varphi
:H\rightarrow H^{\prime}$ gives a commutative diagram.) Is it true that $g$ is
an infinite $\mathbf{k}$-linear combination of $p_{\alpha,\sigma}$'s?
\end{question}

I have so far been unable to answer this question, for lack of convenient free
objects in the relevant category. Nevertheless, I suspect that the answer is
positive (i.e., every natural endomorphism is an infinite $\mathbf{k}$-linear
combination of $p_{\alpha,\sigma}$'s), at least when $\mathbf{k}$ is a field
of characteristic $0$.

Question \ref{quest.natural} can also be asked without the first
\textquotedblleft graded\textquotedblright. That is, we can ask about natural
$\mathbf{k}$-module endomorphisms of $H$ that are not necessarily graded.

A similar question can be asked for general (as opposed to connected) graded
bialgebras. However, it requires some more care in its formulation, as not
every infinite $\mathbf{k}$-linear combination $\sum_{k\in\mathbb{N}}%
\ \ \sum_{\alpha\in\mathbb{N}^{k}}\ \ \sum_{\sigma\in\mathfrak{S}_{k}}%
\lambda_{\alpha,\sigma}p_{\alpha,\sigma}$ is well-defined (and we cannot
restrict the second sum to the $\alpha\in\left\{  1,2,3,\ldots\right\}  ^{k}$ only).

\section{\label{sec.PNSym}The combinatorial Hopf algebra
$\operatorname*{PNSym}$}

Let us now recall the connected graded Hopf algebra $\operatorname*{NSym}$ of
noncommutative symmetric functions (introduced in \cite{ncsf1}, recently
exposed in \cite[\S 5.4]{GriRei} and \cite[Definition 6.1]{Meliot}%
\footnote{These three references give slightly different definitions of this
Hopf algebra $\operatorname*{NSym}$, but all these definitions are easily seen
to be equivalent (e.g., using \cite[Note 3.5, Proposition 3.8, Proposition
3.9]{ncsf1} and \cite[Proposition 6.2, Proposition 6.3]{Meliot}). They also
denote it by different symbols: It is called $\mathbf{Sym}$ in \cite{ncsf1},
called $\operatorname*{NSym}$ in \cite[\S 5.4]{GriRei}, and called
$\operatorname*{NCSym}$ in \cite[Definition 6.1]{Meliot} (a name that means a
different Hopf algebra in most of the literature).}). As a $\mathbf{k}%
$-algebra, it is free with countably many generators $\mathbf{H}%
_{1},\mathbf{H}_{2},\mathbf{H}_{3},\ldots$, and its comultiplication is given
by $\Delta\left(  \mathbf{H}_{m}\right)  =\sum_{i=0}^{m}\mathbf{H}_{i}%
\otimes\mathbf{H}_{m-i}$, where $\mathbf{H}_{0}:=1$. (We are using the
notation $\mathbf{H}_{k}$ for the $k$-th complete homogeneous noncommutative
symmetric function in $\operatorname*{NSym}$, which is denoted $H_{k}$ in
\cite[Theorem 5.4.2]{GriRei}, and denoted $S_{k}$ in \cite[(22)]{ncsf1} and
\cite{Meliot}.)

For any weak composition $\left(  \alpha_{1},\alpha_{2},\ldots,\alpha
_{n}\right)  \in\mathbb{N}^{n}$, we set%
\[
\mathbf{H}_{\left(  \alpha_{1},\alpha_{2},\ldots,\alpha_{n}\right)
}:=\mathbf{H}_{\alpha_{1}}\mathbf{H}_{\alpha_{2}}\cdots\mathbf{H}_{\alpha_{n}%
}\in\operatorname*{NSym}.
\]
Thus, the family $\left(  \mathbf{H}_{\alpha}\right)  _{\alpha\text{ is a
composition}}$ is a basis of the $\mathbf{k}$-module $\operatorname*{NSym}$.
(Note that $\mathbf{H}_{\alpha}$ is called $S^{\alpha}$ in \cite{ncsf1}.)

An \emph{internal product} $\ast$ is defined on $\operatorname*{NSym}$ in
\cite[\S 5.1]{ncsf1}. It is explicitly given by the formula%
\[
\mathbf{H}_{\alpha}\ast\mathbf{H}_{\beta}=\sum_{\substack{\gamma_{i,j}%
\in\mathbb{N}\text{ for all }i\in\left[  k\right]  \text{ and }j\in\left[
\ell\right]  ;\\\gamma_{i,1}+\gamma_{i,2}+\cdots+\gamma_{i,\ell}=\alpha
_{i}\text{ for all }i\in\left[  k\right]  ;\\\gamma_{1,j}+\gamma_{2,j}%
+\cdots+\gamma_{k,j}=\beta_{j}\text{ for all }j\in\left[  \ell\right]
}}\mathbf{H}_{\left(  \gamma_{1,1},\gamma_{1,2},\ldots,\gamma_{k,\ell}\right)
}%
\]
for all compositions $\alpha$ and $\beta$ (see \cite[Proposition 5.1]{ncsf1}).
As mentioned in the introduction, Patras's composition formula
(\ref{eq.intro.pp=sum.cocomm}) is structurally identical to this expression.
This shows that any cocommutative connected graded bialgebra $H$ is a module
over the non-unital algebra $\operatorname*{NSym}\nolimits^{\left(  2\right)
}$ of noncommutative symmetric functions equipped with its internal product.
(Each complete noncommutative symmetric function $\mathbf{H}_{\alpha}$ acts as
the operator $p_{\alpha,\operatorname*{id}}$ on $H$.)

It is natural to ask whether a similar construction can be made for any
connected graded bialgebra $H$ using Theorem \ref{thm.sol-mac-1}. Thus, we are
looking for a non-unital algebra that contains elements $F_{\alpha,\sigma}$
for all compositions $\alpha\in\mathbb{N}^{k}$ and all permutations $\sigma
\in\mathfrak{S}_{k}$, and which acts on any connected graded bialgebra $H$ by
having each $F_{\alpha,\sigma}$ act as the operator $p_{\alpha,\sigma}$.

In this section, we shall construct such an algebra -- which I name
$\operatorname*{PNSym}\nolimits^{\left(  2\right)  }$ for \textquotedblleft
permuted noncommutative symmetric functions\textquotedblright. Besides having
an \textquotedblleft internal product\textquotedblright\ $\ast$, it has an
\textquotedblleft external product\textquotedblright\ $\cdot$ (corresponding
to the convolution of operators on $H$) and a coproduct $\Delta$
(corresponding to acting on a tensor product of bialgebras).

\subsection{Mopiscotions and weak mopiscotions}

To construct this algebra, we need to make some implicit things explicit and
introduce some more notation:

\begin{definition}
\label{def.mopis}A \emph{mopiscotion} (short for \textquotedblleft permuted
composition\textquotedblright) is a pair $\left(  \alpha,\sigma\right)  $,
where $\alpha$ is a composition of length $k$ (for some $k\in\mathbb{N}$) and
$\sigma$ is a permutation in $\mathfrak{S}_{k}$.
\end{definition}

\begin{definition}
\label{def.wmopis}A \emph{weak mopiscotion} is a pair $\left(  \alpha
,\sigma\right)  $, where $\alpha$ is a weak composition of length $k$ (for
some $k\in\mathbb{N}$) and $\sigma$ is a permutation in $\mathfrak{S}_{k}$.
\end{definition}

Clearly, any mopiscotion is a weak mopiscotion. In the reverse direction, we
can transform any weak mopiscotion $\left(  \alpha,\sigma\right)  $ into a
mopiscotion $\left(  \beta,\tau\right)  $ (non-injectively) by removing the
zeroes from the composition $\alpha$ while appropriately reducing the
permutation $\sigma$ as well. To give a precise definition, we need the
concept of \emph{standardization} (\cite[Definition 5.3.3]{GriRei}):

\begin{definition}
\label{def.std}Let $w=\left(  w_{1},w_{2},\ldots,w_{h}\right)  $ be any finite
list of integers (or of elements of any totally ordered set). Then, the
\emph{standardization} of $w$ is defined as the unique permutation $\sigma
\in\mathfrak{S}_{h}$ with the property that for every two elements $a$ and $b$
of $\left[  h\right]  $ satisfying $a<b$, we have the equivalence $\left(
\sigma\left(  a\right)  <\sigma\left(  b\right)  \right)  \Longleftrightarrow
\left(  w_{a}\leq w_{b}\right)  $.
\end{definition}

Roughly speaking, the standardization of the list $w$ is the permutation whose
values are in the same relative order as the entries of $w$; when $w$ has
equal entries, we count entries lying further left as being smaller. For
example, the two lists $\left(  5,7,1,8,2\right)  $ and $\left(
2,2,0,2,1\right)  $ have the same standardization, namely the permutation
$\sigma\in\mathfrak{S}_{5}$ with one-line notation $\left[  3,4,1,5,2\right]
$. (In this paper, we will only use standardizations of lists that consist of
distinct entries. Thus, any subtleties related to equal entries can be
ignored, and we can simply compute the standardization of a list by replacing
its smallest entry by $1$, its second-smallest entry by $2$, and so on, and
finally read the resulting list as the one-line notation of a permutation.)

\begin{definition}
\label{def.red-wmopis}Let $\left(  \alpha,\sigma\right)  $ be a weak
mopiscotion, with $\alpha=\left(  \alpha_{1},\alpha_{2},\ldots,\alpha
_{k}\right)  $ and $\sigma\in\mathfrak{S}_{k}$. Let $\left(  j_{1}%
<j_{2}<\cdots<j_{h}\right)  $ be the list of all elements $i$ of $\left[
k\right]  $ satisfying $\alpha_{i}\neq0$, in increasing order. Let $\tau
\in\mathfrak{S}_{h}$ be the standardization of the list $\left(  \sigma\left(
j_{1}\right)  ,\sigma\left(  j_{2}\right)  ,\ldots,\sigma\left(  j_{h}\right)
\right)  $. Let $\operatorname*{red}\alpha$ denote the composition $\left(
\alpha_{j_{1}},\alpha_{j_{2}},\ldots,\alpha_{j_{h}}\right)  $ (which consists
of all nonzero entries of $\alpha$). Then, we define $\operatorname*{red}%
\left(  \alpha,\sigma\right)  $ to be the mopiscotion $\left(
\operatorname*{red}\alpha,\tau\right)  $. We call $\operatorname*{red}\left(
\alpha,\sigma\right)  $ the \emph{reduction} of $\left(  \alpha,\sigma\right)
$.
\end{definition}

For example,%
\[
\operatorname*{red}\left(  \left(  3,0,1,2,0\right)  ,\ \left[
4,5,1,3,2\right]  \right)  =\left(  \left(  3,1,2\right)  ,\ \left[
3,1,2\right]  \right)  ,
\]
where the square brackets indicate a permutation written in one-line notation.
For another example,%
\[
\operatorname*{red}\left(  \left(  3,0,1,2,0\right)  ,\ \left[
4,1,3,2,5\right]  \right)  =\left(  \left(  3,1,2\right)  ,\ \left[
3,2,1\right]  \right)  .
\]

Clearly, if $\left(  \alpha,\sigma\right)  $ is a mopiscotion (i.e., if all
entries of $\alpha$ are nonzero), then $\operatorname*{red}\left(
\alpha,\sigma\right)  =\left(  \alpha,\sigma\right)  $. \medskip

Mopiscotions have already appeared under the guise of \textquotedblleft
weighted permutations\textquotedblright\ in work by Foissy and Patras
\cite{FoiPat13}. Indeed, the set $\mathcal{S}$ defined in \cite[\S 3]%
{FoiPat13} can be identified with the set of all mopiscotions, as long as we
translate each pair $\left(  \sigma,d\right)  \in\mathfrak{S}_{k}%
\times\operatorname*{Hom}\left(  \left[  k\right]  ,\mathbb{N}^{\ast}\right)
\subseteq\mathcal{S}$ (the notations used here are those of \cite{FoiPat13})
into the mopiscotion $\left(  \delta,\sigma\right)  $, where $\delta:=\left(
d\left(  1\right)  ,d\left(  2\right)  ,\ldots,d\left(  k\right)  \right)  $.
These weighted permutations $\left(  \sigma,d\right)  \in\mathcal{S}$ are used
in \cite{FoiPat13} to define certain linear endomorphisms of a shuffle Hopf
algebra, and might be generalizable to arbitrary graded bialgebras with a
dendriform structure, but do not agree with our maps $p_{\alpha,\sigma}$.

Let us now return to our maps $p_{\alpha,\sigma}$. If $H$ is connected, we can
dispense with the ones in which $\alpha$ contains zeroes, since there is a way
to reduce all such $p_{\alpha,\sigma}$ to the ones where $\alpha$ is a proper
composition (i.e., contains no zeroes):

\begin{proposition}
\label{prop.mopis.reduce}Let $H$ be a connected graded bialgebra. Let $\left(
\alpha,\sigma\right)  $ be a weak mopiscotion, and let $\left(  \beta
,\tau\right)  =\operatorname*{red}\left(  \alpha,\sigma\right)  $. Then,%
\[
p_{\alpha,\sigma}=p_{\beta,\tau}.
\]

\end{proposition}

\begin{proof}
[Proof idea.]This is easy if one does not try to be rigorous. Here is a
handwavy proof using Sweedler notation:

The connectedness of $H$ yields $p_{0}\left(  h\right)  =\epsilon\left(
h\right)  \cdot1_{H}$ for each $h\in H$. Hence, in the formula
\[
p_{\alpha,\sigma}\left(  x\right)  =\sum_{\left(  x\right)  }p_{\alpha_{1}%
}\left(  x_{\left(  \sigma\left(  1\right)  \right)  }\right)  p_{\alpha_{2}%
}\left(  x_{\left(  \sigma\left(  2\right)  \right)  }\right)  \cdots
p_{\alpha_{k}}\left(  x_{\left(  \sigma\left(  k\right)  \right)  }\right)
\]
(from Remark \ref{rmk.pas.informal-def}), all the factors $p_{\alpha_{i}%
}\left(  x_{\left(  \sigma\left(  i\right)  \right)  }\right)  $ with
$\alpha_{i}=0$ can be rewritten as $\epsilon\left(  x_{\left(  \sigma\left(
i\right)  \right)  }\right)  $ (the $1_{H}$ gets swallowed by the product) and
thus can be removed completely (using the $\sum_{\left(  h\right)  }%
\epsilon\left(  h_{\left(  1\right)  }\right)  h_{\left(  2\right)  }%
=\sum_{\left(  h\right)  }h_{\left(  1\right)  }\epsilon\left(  h_{\left(
2\right)  }\right)  =h$ axiom of a coalgebra), as long as we remember to
adjust the permutation $\sigma$ accordingly (removing its value $\sigma\left(
i\right)  $ and decreasing all values larger than $\sigma\left(  i\right)  $
by $1$). At the end of this process, we end up with $p_{\beta,\tau}\left(
x\right)  $. This shows that $p_{\alpha,\sigma}\left(  x\right)
=p_{\beta,\tau}\left(  x\right)  $ for each $x\in H$. Proposition
\ref{prop.mopis.reduce} follows.

A more rigorous version of this proof can be found in the Appendix (Section
\ref{sec.pas-proof}).
\end{proof}

\subsection{$\operatorname*{PNSym}$}

We can now define the combinatorial Hopf algebra that will occupy us for the
rest of this work:

\begin{definition}
\label{def.PNSym.ops}Let $\operatorname*{PNSym}$ be the free $\mathbf{k}%
$-module with basis $\left(  F_{\alpha,\sigma}\right)  _{\left(  \alpha
,\sigma\right)  \text{ is a mopiscotion}}$.

For any weak mopiscotion $\left(  \alpha,\sigma\right)  $, we set%
\[
F_{\alpha,\sigma}:=F_{\beta,\tau},
\]
where $\left(  \beta,\tau\right)  =\operatorname*{red}\left(  \alpha
,\sigma\right)  $.

Define two multiplications on $\operatorname*{PNSym}$: one \textquotedblleft
external multiplication\textquotedblright\ (which mirrors convolution of
$p_{\alpha,\sigma}$'s as expressed in Proposition \ref{prop.sol-mac-0}) given
by%
\[
F_{\alpha,\sigma}\cdot F_{\beta,\tau}=F_{\alpha\beta,\sigma\oplus\tau};
\]
and another \textquotedblleft internal multiplication\textquotedblright%
\ (which mirrors composition of $p_{\alpha,\sigma}$'s as expressed in Theorem
\ref{thm.sol-mac-1}) given by%
\[
F_{\alpha,\sigma}\ast F_{\beta,\tau}=\sum_{\substack{\gamma_{i,j}\in
\mathbb{N}\text{ for all }i\in\left[  k\right]  \text{ and }j\in\left[
\ell\right]  ;\\\gamma_{i,1}+\gamma_{i,2}+\cdots+\gamma_{i,\ell}=\alpha
_{i}\text{ for all }i\in\left[  k\right]  ;\\\gamma_{1,j}+\gamma_{2,j}%
+\cdots+\gamma_{k,j}=\beta_{j}\text{ for all }j\in\left[  \ell\right]
}}F_{\left(  \gamma_{1,1},\gamma_{1,2},\ldots,\gamma_{k,\ell}\right)
,\tau\left[  \sigma\right]  }%
\]
(where $\alpha\in\mathbb{N}^{k}$ and $\beta\in\mathbb{N}^{\ell}$). Also, we
define a comultiplication $\Delta:\operatorname*{PNSym}\rightarrow
\operatorname*{PNSym}\otimes\operatorname*{PNSym}$ on $\operatorname*{PNSym}$
by%
\[
\Delta\left(  F_{\alpha,\sigma}\right)  =\sum_{\substack{\beta,\gamma\text{
weak compositions;}\\\text{entrywise sum }\beta+\gamma=\alpha}}F_{\beta
,\sigma}\otimes F_{\gamma,\sigma}%
\]
(mirroring the formula from Proposition \ref{prop.pas-on-tensors}).

We also equip the $\mathbf{k}$-module $\operatorname*{PNSym}$ with a grading
by letting each $F_{\alpha,\sigma}$ be homogeneous of degree $\left\vert
\alpha\right\vert $.
\end{definition}

These operations on $\operatorname*{PNSym}$ behave as nicely as the analogous
operations on $\operatorname*{NSym}$:

\begin{theorem}
\label{thm.PNSym.alg}The $\mathbf{k}$-module $\operatorname*{PNSym}$ becomes a
connected graded cocommutative Hopf algebra when equipped with the external
multiplication $\cdot$, and a (non-graded) non-unital bialgebra when equipped
with the internal multiplication $\ast$. In particular, both multiplications
are associative.
\end{theorem}

There are two ways to prove this. I shall very briefly outline both:

\begin{proof}
[First proof idea for Theorem \ref{thm.PNSym.alg}.]Most claims can be derived
from properties of the operators $p_{\alpha,\sigma}$, using the $H$ from
Theorem \ref{thm.pas-lin-ind} \textbf{(a)} as a faithful representation.

For an example, let us prove that the internal multiplication $\ast$ on
$\operatorname*{PNSym}$ is associative.

Let $H$ be any connected graded $\mathbf{k}$-bialgebra. Let
$\operatorname*{ev}\nolimits_{H}:\operatorname*{PNSym}\rightarrow
\operatorname*{End}H$ be the $\mathbf{k}$-linear map that sends any
$F_{\alpha,\sigma}$ to the operator $p_{\alpha,\sigma}\in\operatorname*{End}H$
for any mopiscotion $\left(  \alpha,\sigma\right)  $. Note that
\begin{equation}
\operatorname*{ev}\nolimits_{H}\left(  F_{\alpha,\sigma}\right)
=p_{\alpha,\sigma} \label{pf.thm.PNSym.alg.1st.evF=p}%
\end{equation}
is true not only for all mopiscotions $\left(  \alpha,\sigma\right)  $, but
also for all weak mopiscotions $\left(  \alpha,\sigma\right)  $ (because if
$\left(  \alpha,\sigma\right)  $ is any weak mopiscotion, and if $\left(
\beta,\tau\right)  =\operatorname*{red}\left(  \alpha,\sigma\right)  $, then
$p_{\alpha,\sigma}=p_{\beta,\tau}$ and $F_{\alpha,\sigma}=F_{\beta,\tau}$).

Now, let $H$ be the connected graded Hopf algebra $H$ from Theorem
\ref{thm.pas-lin-ind} \textbf{(a)}. Then, Theorem \ref{thm.pas-lin-ind}
\textbf{(a)} says that the family $\left(  p_{\alpha,\sigma}\right)  _{\left(
\alpha,\sigma\right)  \text{ is a mopiscotion}}$ is $\mathbf{k}$-linearly
independent. Hence, the linear map $\operatorname*{ev}\nolimits_{H}$ is injective.

The formula for $F_{\alpha,\sigma}\ast F_{\beta,\tau}$ that we used to define
the internal multiplication $\ast$ is very similar to the formula for
$p_{\alpha,\sigma}\circ p_{\beta,\tau}$ in Theorem \ref{thm.sol-mac-1}. In
view of (\ref{pf.thm.PNSym.alg.1st.evF=p}), this entails that%
\[
\operatorname*{ev}\nolimits_{H}\left(  F_{\alpha,\sigma}\ast F_{\beta,\tau
}\right)  =p_{\alpha,\sigma}\circ p_{\beta,\tau}=\operatorname*{ev}%
\nolimits_{H}\left(  F_{\alpha,\sigma}\right)  \circ\operatorname*{ev}%
\nolimits_{H}\left(  F_{\beta,\tau}\right)
\]
for any two mopiscotions $\left(  \alpha,\sigma\right)  $ and $\left(
\beta,\tau\right)  $. By bilinearity, this entails that%
\[
\operatorname*{ev}\nolimits_{H}\left(  f\ast g\right)  =\left(
\operatorname*{ev}\nolimits_{H}f\right)  \circ\left(  \operatorname*{ev}%
\nolimits_{H}g\right)
\]
for any $f,g\in\operatorname*{PNSym}$. Thus, the injective $\mathbf{k}$-linear
map $\operatorname*{ev}\nolimits_{H}:\operatorname*{PNSym}\rightarrow
\operatorname*{End}H$ embeds the $\mathbf{k}$-module $\operatorname*{PNSym}$
with its binary operation $\ast$ into the algebra $\operatorname*{End}H$ with
its binary operation $\circ$. Since the latter operation $\circ$ is
associative, it thus follows that the former operation $\ast$ is associative
as well.

Similarly, we can show that the operation $\cdot$ on $\operatorname*{PNSym}$
is associative and unital with the unity $1=F_{\varnothing,\varnothing}$
(although this is pretty obvious).

It is very easy to see that the cooperation $\Delta$ is coassociative,
counital and cocommutative. It is also clear that both the multiplication
$\cdot$ and the comultiplication $\Delta$ on $\operatorname*{PNSym}$ are graded.

The next difficulty is to prove that $\Delta$ is a $\mathbf{k}$-algebra
homomorphism, i.e., that $\Delta\left(  fg\right)  =\Delta\left(  f\right)
\cdot\Delta\left(  g\right)  $ for all $f,g\in\operatorname*{PNSym}$ (where
the \textquotedblleft$\cdot$\textquotedblright\ on the right-hand side is the
extension of the external multiplication $\cdot$ to $\operatorname*{PNSym}%
\otimes\operatorname*{PNSym}$). Here, we can argue as above, using the fact (a
consequence of Proposition \ref{prop.pas-on-tensors}) that%
\begin{align*}
\operatorname*{ev}\nolimits_{H\otimes H}\left(  f\right)   &  =\left(
\operatorname*{ev}\nolimits_{H}\otimes\operatorname*{ev}\nolimits_{H}\right)
\left(  \Delta\left(  f\right)  \right)  \in\operatorname*{End}\left(
H\otimes H\right) \\
&  \ \ \ \ \ \ \ \ \ \ \text{for every }f\in\operatorname*{PNSym}%
\end{align*}
to make sense of $\Delta$), and using the fact that the map
\[
\operatorname*{ev}\nolimits_{H}\otimes\operatorname*{ev}\nolimits_{H}%
:\operatorname*{PNSym}\otimes\operatorname*{PNSym}\rightarrow
\operatorname*{End}H\otimes\operatorname*{End}H\rightarrow\operatorname*{End}%
\left(  H\otimes H\right)
\]
is injective (this is not hard to show using the argument used in the proof of
Theorem \ref{thm.pas-lin-ind}).

What we have shown so far yields that $\operatorname*{PNSym}$ (equipped with
$\cdot$ and $\Delta$) is a connected graded $\mathbf{k}$-bialgebra. Thus,
$\operatorname*{PNSym}$ is a Hopf algebra (since any connected graded
$\mathbf{k}$-bialgebra is a Hopf algebra).

It remains to show that $\operatorname*{PNSym}$ (equipped with $\ast$ and
$\Delta$) is a non-unital $\mathbf{k}$-bialgebra. Having already verified that
$\ast$ is associative, we only need to show that $\Delta\left(  f\ast
g\right)  =\Delta\left(  f\right)  \ast\Delta\left(  g\right)  $ for all
$f,g\in\operatorname*{PNSym}$. But this is similar to the proof of
$\Delta\left(  fg\right)  =\Delta\left(  f\right)  \cdot\Delta\left(
g\right)  $ above. Thus, the proof of Theorem \ref{thm.PNSym.alg} is complete.
\end{proof}

\begin{proof}
[Second proof idea for Theorem \ref{thm.PNSym.alg}.]There is also a more
direct combinatorial approach to this theorem. First, we shall define two
simpler bialgebras $\operatorname*{WNSym}$ and $\operatorname*{Perm}$, and
then present $\operatorname*{PNSym}$ as a subquotient of their tensor product
$\operatorname*{WNSym}\otimes\operatorname*{Perm}$.

Here are some details:

We define $\operatorname*{WNSym}$ to be the free $\mathbf{k}$-module with
basis $\left(  C_{\alpha}\right)  _{\alpha\text{ is a weak composition}}$. We
equip this $\mathbf{k}$-module $\operatorname*{WNSym}$ with an
\textquotedblleft external multiplication\textquotedblright\ defined by%
\[
C_{\alpha}\cdot C_{\beta}=C_{\alpha\beta}%
\]
(where $\alpha\beta$ is the concatenation of $\alpha$ and $\beta$), and an
\textquotedblleft internal multiplication\textquotedblright\ defined by%
\[
C_{\alpha}\ast C_{\beta}=\sum_{\substack{\gamma_{i,j}\in\mathbb{N}\text{ for
all }i\in\left[  k\right]  \text{ and }j\in\left[  \ell\right]  ;\\\gamma
_{i,1}+\gamma_{i,2}+\cdots+\gamma_{i,\ell}=\alpha_{i}\text{ for all }%
i\in\left[  k\right]  ;\\\gamma_{1,j}+\gamma_{2,j}+\cdots+\gamma_{k,j}%
=\beta_{j}\text{ for all }j\in\left[  \ell\right]  }}C_{\left(  \gamma
_{1,1},\gamma_{1,2},\ldots,\gamma_{k,\ell}\right)  }%
\]
(where $\alpha\in\mathbb{N}^{k}$ and $\beta\in\mathbb{N}^{\ell}$), and a
comultiplication $\Delta:\operatorname*{WNSym}\rightarrow\operatorname*{WNSym}%
\otimes\operatorname*{WNSym}$ defined by%
\[
\Delta\left(  C_{\alpha}\right)  =\sum_{\substack{\beta,\gamma\in
\mathbb{N}^{k}\text{;}\\\text{entrywise sum }\beta+\gamma=\alpha}}C_{\beta
}\otimes C_{\gamma}\ \ \ \ \ \ \ \ \ \ \text{for any }\alpha\in\mathbb{N}%
^{k}.
\]
It is not too hard to show that $\operatorname*{WNSym}$ thus becomes a graded
(but not connected!) cocommutative $\mathbf{k}$-bialgebra when equipped with
the external multiplication $\cdot$, and a (non-graded) non-unital bialgebra
when equipped with the internal multiplication $\ast$. (Indeed, this
$\operatorname*{WNSym}$ is a mild variation on the Hopf algebra
$\operatorname*{NSym}$ of noncommutative symmetric functions, which is studied
(e.g.) in \cite{ncsf1} or \cite[\S 5.4]{GriRei}; the only difference is that
compositions have been replaced by weak compositions. The first letter
\textquotedblleft W\textquotedblright\ in $\operatorname*{WNSym}$ refers to
this weakness.)

We define $\mathfrak{S}$ to be the disjoint union $\bigsqcup\limits_{k\in
\mathbb{N}}\mathfrak{S}_{k}$ of all symmetric groups $\mathfrak{S}_{k}$ for
all $k\in\mathbb{N}$. We define $\operatorname*{Perm}$ to be the free
$\mathbf{k}$-module with basis $\left(  P_{\sigma}\right)  _{\sigma
\in\mathfrak{S}}$. We equip this $\mathbf{k}$-module $\operatorname*{Perm}$
with an \textquotedblleft external multiplication\textquotedblright\ $\cdot$
defined by%
\[
P_{\sigma}\cdot P_{\tau}=P_{\sigma\oplus\tau},
\]
and an \textquotedblleft internal multiplication\textquotedblright\ $\ast$
defined by%
\[
P_{\sigma}\ast P_{\tau}=P_{\tau\left[  \sigma\right]  },
\]
and a comultiplication $\Delta:\operatorname*{Perm}\rightarrow
\operatorname*{Perm}\otimes\operatorname*{Perm}$ defined by%
\[
\Delta\left(  P_{\sigma}\right)  =P_{\sigma}\otimes P_{\sigma}.
\]
It is not too hard to show that $\operatorname*{Perm}$ becomes a (non-graded)
cocommutative $\mathbf{k}$-bialgebra when equipped with either of the two
multiplications. Indeed, in both cases, it becomes the monoid algebra of an
appropriate monoid on the set $\mathfrak{S}$. The hardest part of the proof is
to check the associativity of the internal multiplication; this can be done as follows:

\begin{statement}
\textit{Claim 1:} We have $\tau\left[  \sigma\left[  \rho\right]  \right]
=\left(  \tau\left[  \sigma\right]  \right)  \left[  \rho\right]  $ for any
three permutations $\tau,\sigma,\rho\in\mathfrak{S}$.
\end{statement}

\begin{proof}
[Proof of Claim 1.]We can regard the set $\mathfrak{S}$ as a skeletal groupoid
with objects $0,1,2,\ldots$ and morphism sets $\mathfrak{S}\left(  k,k\right)
=\mathfrak{S}_{k}$ and $\mathfrak{S}\left(  k,\ell\right)  =\varnothing$ for
all $k\neq\ell$. However, the definition of $\tau\left[  \sigma\right]  $
becomes cleaner if we \textquotedblleft de-skeletize\textquotedblright%
\ $\mathfrak{S}$ to a larger category. Namely, we define a \emph{tormutation}
to be a bijection (not necessarily order-preserving) between two finite
totally ordered sets. Clearly, each permutation $\sigma\in\mathfrak{S}_{k}$ is
a tormutation $\left[  k\right]  \rightarrow\left[  k\right]  $. Conversely,
any tormutation $\phi:A\rightarrow B$ induces a canonical permutation
$\overline{\phi}\in\mathfrak{S}_{\left\vert A\right\vert }$ by the rule%
\[
\overline{\phi}:=\operatorname*{inc}\nolimits_{B\rightarrow\left[  \left\vert
B\right\vert \right]  }\circ\,\phi\circ\operatorname*{inc}\nolimits_{\left[
\left\vert A\right\vert \right]  \rightarrow A}:\left[  \left\vert
A\right\vert \right]  \rightarrow\left[  \left\vert B\right\vert \right]  ,
\]
where $\operatorname*{inc}\nolimits_{X\rightarrow Y}$ denotes the unique order
isomorphism between two given finite totally ordered sets $X$ and $Y$. We may
call $\overline{\phi}$ the \emph{standardization} of $\phi$. Thus, the
skeletal groupoid $\mathfrak{S}$ is a skeleton of the groupoid
$\widetilde{\mathfrak{S}}$ whose objects are the finite totally ordered sets
and whose morphisms are the tormutations.

Given two tormutations $\phi:A\rightarrow B$ and $\phi^{\prime}:A^{\prime
}\rightarrow B^{\prime}$, we now define a tormutation $\phi^{\prime
}\left\langle \phi\right\rangle :A\times A^{\prime}\rightarrow B^{\prime
}\times B$ by
\[
\left(  \phi^{\prime}\left\langle \phi\right\rangle \right)  \left(
a,a^{\prime}\right)  =\left(  \phi^{\prime}\left(  a^{\prime}\right)
,\phi\left(  a\right)  \right)  \ \ \ \ \ \ \ \ \ \ \text{for all }\left(
a,a^{\prime}\right)  \in A\times A^{\prime}.
\]
In other words, $\phi^{\prime}\left\langle \phi\right\rangle $ applies $\phi$
and $\phi^{\prime}$ to the respective entries of the input, then swaps the outputs.

It is now easy to see that $\overline{\phi^{\prime}\left\langle \phi
\right\rangle }=\overline{\phi^{\prime}}\left[  \overline{\phi}\right]  $ for
any two tormutations $\phi$ and $\phi^{\prime}$. Thus, in order to prove that
$\tau\left[  \sigma\left[  \rho\right]  \right]  =\left(  \tau\left[
\sigma\right]  \right)  \left[  \rho\right]  $ for any three permutations
$\tau,\sigma,\rho\in\mathfrak{S}$, it suffices to show that $\overline
{\phi^{\prime\prime}\left\langle \phi^{\prime}\left\langle \phi\right\rangle
\right\rangle }=\overline{\left(  \phi^{\prime\prime}\left\langle \phi
^{\prime}\right\rangle \right)  \left\langle \phi\right\rangle }$ for any
three tormutations $\phi,\phi^{\prime},\phi^{\prime\prime}$. But the latter is
easy: The map $\phi^{\prime\prime}\left\langle \phi^{\prime}\left\langle
\phi\right\rangle \right\rangle $ sends each $\left(  \left(  a,a^{\prime
}\right)  ,a^{\prime\prime}\right)  $ to $\left(  \phi^{\prime\prime}\left(
a^{\prime\prime}\right)  ,\left(  \phi^{\prime}\left(  a^{\prime}\right)
,\phi\left(  a\right)  \right)  \right)  $, whereas the map $\left(
\phi^{\prime\prime}\left\langle \phi^{\prime}\right\rangle \right)
\left\langle \phi\right\rangle $ sends each $\left(  a,\left(  a^{\prime
},a^{\prime\prime}\right)  \right)  $ to $\left(  \left(  \phi^{\prime\prime
}\left(  a^{\prime\prime}\right)  ,\phi^{\prime}\left(  a^{\prime}\right)
\right)  ,\phi\left(  a\right)  \right)  $. Because of the canonical order
isomorphism $A\times\left(  B\times C\right)  \cong\left(  A\times B\right)
\times C$ for any three totally ordered sets $A,B,C$, these two maps are
therefore equivalent, i.e., have the same canonical permutation. This proves
Claim 1.
\end{proof}

We now know that $\operatorname*{Perm}$ becomes a (non-graded) cocommutative
$\mathbf{k}$-bialgebra when equipped with either of the two multiplications.

Now, let $\operatorname*{WPNSym}$ be the tensor product $\operatorname*{WNSym}%
\otimes\operatorname*{Perm}$. We equip this tensor product
$\operatorname*{WPNSym}$ with an \textquotedblleft external
multiplication\textquotedblright\ (obtained by tensoring the external
multiplications of $\operatorname*{WNSym}$ and of $\operatorname*{Perm}$), an
\textquotedblleft internal multiplication\textquotedblright\ (obtained
similarly) and a comultiplication (also obtained similarly). Thus,
$\operatorname*{WPNSym}$ becomes a graded (but not connected) cocommutative
$\mathbf{k}$-bialgebra when equipped with $\cdot$, and a (non-graded)
non-unital $\mathbf{k}$-bialgebra when equipped with $\ast$. (Here we use the
facts that the tensor product of two cocommutative bialgebras is a
cocommutative bialgebra, and that the tensor product of two non-unital
bialgebras is a non-unital bialgebra.)

We furthermore set%
\begin{align*}
\widehat{F}_{\alpha,\sigma}  &  :=C_{\alpha}\otimes P_{\sigma}\in
\operatorname*{WPNSym}\ \ \ \ \ \ \ \ \ \ \text{for any pair }\left(
\alpha,\sigma\right)  \text{ of}\\
&  \ \ \ \ \ \ \ \ \ \ \text{a weak composition }\alpha\text{ and a
permutation }\sigma\in\mathfrak{S}.
\end{align*}
Then, $\left(  \widehat{F}_{\alpha,\sigma}\right)  _{\left(  \alpha
,\sigma\right)  \text{ is a weak mopiscotion}}$ is a basis of a certain
$\mathbf{k}$-submodule $\operatorname*{WPNSym}\nolimits^{\prime}$ of the
$\mathbf{k}$-module $\operatorname*{WPNSym}$, and our operations $\cdot$,
$\ast$ and $\Delta$ on $\operatorname*{WPNSym}$ preserve this submodule
$\operatorname*{WPNSym}\nolimits^{\prime}$ and satisfy the same relations for
this basis as the analogous operations on $\operatorname*{PNSym}$ do for the
basis $\left(  F_{\alpha,\sigma}\right)  _{\left(  \alpha,\sigma\right)
\text{ is a mopiscotion}}$ (we just have to replace each \textquotedblleft%
$F$\textquotedblright\ by \textquotedblleft$\widehat{F}$\textquotedblright).
Thus, $\operatorname*{WPNSym}\nolimits^{\prime}$ is a graded (but not
connected) cocommutative $\mathbf{k}$-subbialgebra of $\operatorname*{WPNSym}$
when equipped with $\cdot$, and a (non-graded) non-unital $\mathbf{k}%
$-subbialgebra of $\operatorname*{WPNSym}$ when equipped with $\ast$.

As we said, $\operatorname*{WPNSym}\nolimits^{\prime}$ is almost the
$\operatorname*{PNSym}$ that we care about. But $\operatorname*{WPNSym}%
\nolimits^{\prime}$ does not satisfy the rule%
\[
F_{\alpha,\sigma}=F_{\beta,\tau}\ \ \ \ \ \ \ \ \ \ \text{for }\left(
\beta,\tau\right)  =\operatorname*{red}\left(  \alpha,\sigma\right)
\]
that is fundamental to the definition of $\operatorname*{PNSym}$. Hence,
$\operatorname*{PNSym}$ is not quite $\operatorname*{WPNSym}\nolimits^{\prime
}$ but rather a quotient of $\operatorname*{WPNSym}\nolimits^{\prime}$. To be
specific, we define a $\mathbf{k}$-submodule $I_{\operatorname*{red}}$ of
$\operatorname*{WPNSym}\nolimits^{\prime}$ by
\begin{align*}
I_{\operatorname*{red}}:=  &  \ \operatorname*{span}\nolimits_{\mathbf{k}%
}\left(  \widehat{F}_{\alpha,\sigma}-\widehat{F}_{\beta,\tau}\ \mid\ \left(
\beta,\tau\right)  =\operatorname*{red}\left(  \alpha,\sigma\right)  \right)
\\
=  &  \ \operatorname*{span}\nolimits_{\mathbf{k}}\left(  \widehat{F}%
_{\alpha,\sigma}-\widehat{F}_{\beta,\tau}\ \mid\ \operatorname*{red}\left(
\beta,\tau\right)  =\operatorname*{red}\left(  \alpha,\sigma\right)  \right)
.
\end{align*}
It is not too hard to show that this $I_{\operatorname*{red}}$ is an ideal of
$\operatorname*{WPNSym}\nolimits^{\prime}$ with respect to both $\cdot$ and
$\ast$ and a coideal with respect to $\Delta$. Hence, the quotient
$\operatorname*{WPNSym}\nolimits^{\prime}/I_{\operatorname*{red}}$ inherits
all operations of $\operatorname*{WPNSym}\nolimits^{\prime}$, thus becoming a
graded bialgebra under $\cdot$ and $\Delta$ and a non-unital bialgebra under
$\ast$ and $\Delta$. Moreover, the graded bialgebra $\operatorname*{WPNSym}%
\nolimits^{\prime}/I_{\operatorname*{red}}$ is connected (since
$\operatorname*{red}\left(  \alpha,\sigma\right)  =\left(  \varnothing
,\varnothing\right)  $ whenever $\left\vert \alpha\right\vert =0$), and thus
is a Hopf algebra. As we recall, this means that $\operatorname*{PNSym}$ is a
well-defined connected graded Hopf algebra (since $\operatorname*{PNSym}%
\cong\operatorname*{WPNSym}\nolimits^{\prime}/I_{\operatorname*{red}}$). This
completes the proof of Theorem \ref{thm.PNSym.alg} again.
\end{proof}

\begin{proposition}
\label{prop.PNSym.rank}Let $n\in\mathbb{N}$. Then, the $n$-th graded component
$\operatorname*{PNSym}\nolimits_{n}$ of $\operatorname*{PNSym}$ is a free
$\mathbf{k}$-module of rank
\[
\sum_{k=0}^{n}\dbinom{n-1}{n-k}k!.
\]

\end{proposition}

\begin{proof}
[Proof idea.]Clearly, $\operatorname*{PNSym}\nolimits_{n}$ is a free
$\mathbf{k}$-module with a basis consisting of all $F_{\alpha,\sigma}$ where
$\left(  \alpha,\sigma\right)  $ ranges over all mopiscotions satisfying
$\left\vert \alpha\right\vert =n$. It remains to show that the number of such
mopiscotions is $\sum_{k=0}^{n}\dbinom{n-1}{n-k}k!$. But this is easy: For any
$k$, the number of such mopiscotions in which $\alpha$ has length $k$ is
$\dbinom{n-1}{n-k}k!$ (since there are $\dbinom{n-1}{n-k}$ compositions of $n$
into $k$ parts, and $k!$ permutations $\sigma\in\mathfrak{S}_{k}$).
\end{proof}

We note that the addend $\dbinom{n-1}{n-k}k!$ in Proposition
\ref{prop.PNSym.rank} can also be rewritten as $k\cdot\left(  n-1\right)
\left(  n-2\right)  \cdots\left(  n-k+1\right)  $ when $n$ is positive (but
not when $n=0$).

Here are the ranks of the free $\mathbf{k}$-modules $\operatorname*{PNSym}%
\nolimits_{n}$ for the first few values of $n$:%
\[%
\begin{tabular}
[c]{|c||c|c|c|c|c|c|c|c|}\hline
$n$ & $0$ & $1$ & $2$ & $3$ & $4$ & $5$ & $6$ & $7$\\\hline
$\operatorname*{rank}\left(  \operatorname*{PNSym}\nolimits_{n}\right)  $ &
$1$ & $1$ & $3$ & $11$ & $49$ & $261$ & $1631$ & $11743$\\\hline
\end{tabular}
\ \ \ .
\]
Starting at $n=1$, this sequence of ranks is known to the OEIS as
\href{https://oeis.org/A001339}{Sequence A001339}, and has appeared in the
theory of combinatorial Hopf algebras before (\cite[\S 3.8.3]{HiNoTh06}),
although the Hopf algebra considered there appears to be different (perhaps
dual to ours?).\footnote{Some words about the connection to \cite[\S 3.8.3]%
{HiNoTh06} are in order. One of the many concepts studied in \cite{HiNoTh06}
is the \emph{stalactic equivalence}, an equivalence relation on the words over
a given alphabet $A$. It is the equivalence relation on $A^{\ast}$ defined by
the single axiom $uawav\equiv uaawv$ for any words $u,v,w\in A^{\ast}$ and any
letter $a\in A$. It is easy to see that two words $u,v\in A^{\ast}$ are
equivalent if and only if they contain each letter the same number of times,
and the order in which the letters \textbf{first} appear in each word is the
same for both words. Thus, each stalactic equivalence class is uniquely
determined by the multiplicities of its letters and by the order in which they
first appear. For \emph{initial} words (i.e., for words over the alphabet
$\left\{  1,2,3,\ldots\right\}  $ whose set of letters is $\left[  k\right]  $
for some $k\in\mathbb{N}$), this data is equivalent to a mopiscotion $\left(
\alpha,\sigma\right)  $ (where $\alpha$ determines the multiplicities of
letters, and $\sigma$ determines their order of first appearance). The number
of stalactic equivalence classes of initial words of length $n$ thus equals
the number of mopiscotions $\left(  \alpha,\sigma\right)  $ with $\left\vert
\alpha\right\vert =n$.}

\begin{theorem}
\label{thm.PNSym.mods}Let $\operatorname*{PNSym}\nolimits^{\left(  2\right)
}$ be the non-unital algebra $\operatorname*{PNSym}$ with multiplication
$\ast$. Then, every connected graded bialgebra $H$ becomes a
$\operatorname*{PNSym}\nolimits^{\left(  2\right)  }$-module, with
$F_{\alpha,\sigma}$ acting as $p_{\alpha,\sigma}$. Moreover, the action of an
external product $uv$ of two elements $u,v\in\operatorname*{PNSym}$ on $H$ is
the convolution of the action of $u$ with the action of $v$.
\end{theorem}

\begin{proof}
[Proof idea.]The first claim follows from Theorem \ref{thm.sol-mac-1}, the
second from Proposition \ref{prop.sol-mac-0}.
\end{proof}

\begin{remark}
\label{rmk.splitting}It is tempting to look for an analogue of the
\textquotedblleft splitting formula\textquotedblright\ (\cite[Proposition
5.2]{ncsf1}) for $\operatorname*{PNSym}$, connecting the two products
(internal and external) with the comultiplication. However, the most
natural-looking version is false: It is \textbf{not} true that all
$f,g,h\in\operatorname*{PNSym}$ satisfy $\left(  fg\right)  \ast
h=\sum_{\left(  h\right)  }\left(  f\ast h_{\left(  1\right)  }\right)
\left(  g\ast h_{\left(  2\right)  }\right)  $, where we write $\Delta\left(
h\right)  $ as $\Delta\left(  h\right)  =\sum_{\left(  h\right)  }h_{\left(
1\right)  }\otimes h_{\left(  2\right)  }$ (using Sweedler notation).
\end{remark}

\begin{remark}
\label{rmk.PNSym.p-i}The $\mathbf{k}$-linear map%
\begin{align*}
\mathfrak{p}:\operatorname*{PNSym}  &  \rightarrow\operatorname*{NSym},\\
F_{\alpha,\sigma}  &  \mapsto\mathbf{H}_{\alpha}\ \ \ \ \ \ \ \ \ \ \text{for
any mopiscotion }\left(  \alpha,\sigma\right)
\end{align*}
is a surjection that respects all structures (external and internal
multiplication, comultiplication and grading).

This surjection is furthermore split: The $\mathbf{k}$-linear map%
\begin{align*}
\mathfrak{i}:\operatorname*{NSym}  &  \rightarrow\operatorname*{PNSym},\\
\mathbf{H}_{\alpha}  &  \mapsto F_{\alpha,\operatorname*{id}}%
\ \ \ \ \ \ \ \ \ \ \text{for any composition }\alpha
\end{align*}
(where $\operatorname*{id}$ denotes the identity permutation in $\mathfrak{S}%
_{k}$ where $\alpha$ has length $k$) is an injection that is right-inverse to
$\mathfrak{p}$ and respects external multiplication, comultiplication and
grading (but does not respect internal multiplication).
\end{remark}

\begin{proposition}
The $\mathbf{k}$-algebra $\operatorname*{PNSym}$ (equipped with the external
multiplication) is free.
\end{proposition}

\begin{proof}
[Proof idea.]This algebra is just the monoid algebra of the monoid of
mopiscotions under the operation $\left(  \alpha,\sigma\right)  \cdot\left(
\beta,\tau\right)  =\left(  \alpha\beta,\sigma\oplus\tau\right)  $. But this
monoid is free, with the generators being the mopiscotions $\left(
\alpha,\sigma\right)  $ for which the permutation $\sigma$ is
connected\footnote{See \cite[Exercise 8.1.10]{GriRei} for the definition of a
connected permutation.} (as can be easily verified).
\end{proof}

\subsection{The $\Omega$ involution}

If $A$ is any $\mathbf{k}$-algebra, then its \emph{opposite algebra}
$A^{\operatorname*{op}}$ is defined as the $\mathbf{k}$-algebra obtained from
$A$ by \textquotedblleft reversing the order of factors in a
product\textquotedblright: i.e., as a $\mathbf{k}$-module,
$A^{\operatorname*{op}}=A$, but the product $ab$ of two elements of
$A^{\operatorname*{op}}$ equals their product $ba$ when viewed as elements of
$A$. This definition makes sense for both unital and non-unital $\mathbf{k}$-algebras.

A $\mathbf{k}$-algebra $A$ (unital or not) is said to be \emph{self-opposite}
if it is isomorphic to its opposite algebra $A^{\operatorname*{op}}$. The
non-unital $\mathbf{k}$-algebra $\operatorname*{NSym}\nolimits^{\left(
2\right)  }$ -- that is, $\operatorname*{NSym}$ equipped with the internal
multiplication $\ast$ -- is not self-opposite (at least not in a graded way)
in general\footnote{Namely, its $5$-th graded component $\operatorname*{NSym}%
\nolimits_{5}^{\left(  2\right)  }$ is not self-opposite when $\mathbf{k}$ is
a field of characteristic $0$. Indeed, this is implicit in \cite{Saliol08}: By
\cite[Theorem 2.1]{Saliol08}, it suffices to prove that the $S_{5}$-invariant
subalgebra $\left(  \mathbf{k}\mathcal{F}\right)  ^{S_{5}}$ is not
self-opposite. By \cite[Proposition 4.1]{Saliol08}, this subalgebra is a split
basic algebra. If it was self-opposite, its quiver (see \cite[\S 3]{Saliol08})
would thus be isomorphic to the quiver obtained from it by reversing all
arrows. But the quiver of $\left(  \mathbf{k}\mathcal{F}\right)  ^{S_{5}}$ is
described explicitly in \cite[Theorem 8.1]{Saliol08}, and has three sources
(i.e., vertices with no incoming arcs) but only two sinks (i.e., vertices with
no outgoing arcs). Thus, this quiver cannot be isomorphic to the quiver
obtained from it by reversing all arrows. This shows that the $\mathbf{k}%
$-algebra $\operatorname*{NSym}\nolimits_{5}^{\left(  2\right)  }$ is not
self-opposite when $\mathbf{k}$ is a field of characteristic $0$.
\par
Hence, $\operatorname*{NSym}\nolimits^{\left(  2\right)  }$ is not
self-opposite as a graded non-unital $\mathbf{k}$-algebra under this
condition. Theoretically, this does not preclude it being self-opposite
without the grading, but I don't consider this to be very likely.}. However,
the non-unital $\mathbf{k}$-algebra $\operatorname*{PNSym}\nolimits^{\left(
2\right)  }$ (as defined in Theorem \ref{thm.PNSym.mods}) is self-opposite.
Even better, the following holds:

\begin{proposition}
\label{prop.PNSym.self-opp}The $\mathbf{k}$-linear map%
\begin{align*}
\Omega:\operatorname*{PNSym}  &  \rightarrow\operatorname*{PNSym},\\
F_{\alpha,\sigma}  &  \mapsto F_{\alpha\cdot\sigma^{-1},\ \sigma^{-1}}%
\end{align*}
is a graded $\mathbf{k}$-bialgebra morphism from $\operatorname*{PNSym}$ to
$\operatorname*{PNSym}$ (in particular, it respects the external
multiplication $\cdot$ and the comultiplication $\Delta$) and is furthermore a
non-unital $\mathbf{k}$-algebra morphism from $\operatorname*{PNSym}%
\nolimits^{\left(  2\right)  }$ to $\left(  \operatorname*{PNSym}%
\nolimits^{\left(  2\right)  }\right)  ^{\operatorname*{op}}$ (that is, it
satisfies $\Omega\left(  a\ast b\right)  =\Omega\left(  b\right)  \ast
\Omega\left(  a\right)  $ for all $a,b\in\operatorname*{PNSym}$). Moreover, it
is an involution (i.e., satisfies $\Omega\circ\Omega=\operatorname*{id}$).
\end{proposition}

\begin{proof}
This can be checked by hand. Alternatively, use Proposition
\ref{prop.pas-on-dual} in a faithful representation argument along the lines
of the First proof idea for Theorem \ref{thm.PNSym.alg}.
\end{proof}

\subsection{The Janus monoid and the structure of internal multiplication}

We shall now take a closer look at the internal multiplication $\ast$ on
$\operatorname*{PNSym}$.

For this entire section, we fix an $n\in\mathbb{N}$. Consider the $n$-th
graded component $\operatorname*{PNSym}\nolimits_{n}$ of
$\operatorname*{PNSym}$. As we know, this component $\operatorname*{PNSym}%
\nolimits_{n}$ is closed under the internal multiplication $\ast$ (indeed,
this is clear from the definition of $\ast$). Moreover, the element
$F_{\left(  \left(  n\right)  ,\operatorname*{id}\right)  }$ of
$\operatorname*{PNSym}\nolimits_{n}$ is a neutral element for $\ast$. Thus,
$\operatorname*{PNSym}\nolimits_{n}$ becomes a (unital) $\mathbf{k}$-algebra
under the internal multiplication $\ast$. Let us denote this $\mathbf{k}%
$-algebra by $\operatorname*{PNSym}\nolimits_{n}^{\left(  2\right)  }$.

As we saw in Remark \ref{rmk.PNSym.p-i}, the surjection $\mathfrak{p}%
:\operatorname*{PNSym}\rightarrow\operatorname*{NSym}$ respects $\ast$. Thus,
$\operatorname*{PNSym}\nolimits_{n}^{\left(  2\right)  }$ surjects onto the
$\mathbf{k}$-algebra $\operatorname*{NSym}\nolimits_{n}^{\left(  2\right)  }$,
which is known to be isomorphic to the descent algebra of the symmetric group
$\mathfrak{S}_{n}$ (see \cite[\S 5.1]{ncsf1}). Hence, $\operatorname*{PNSym}%
\nolimits_{n}^{\left(  2\right)  }$ can be viewed as a \textquotedblleft
twisted\textquotedblright\ version of the descent algebra (though not in the
same sense as in \cite{PatSch06}). Therefore, we can try to extend to
$\operatorname*{PNSym}\nolimits_{n}^{\left(  2\right)  }$ the rich theory that
has been developed for the descent algebra (\cite{GarReu89}, \cite{Bidiga97},
\cite{Schocker}, etc.).

An important tool for understanding the descent algebra is the \emph{face
monoid} (aka \emph{Tits monoid}) of the braid arrangement. Its use has been
pioneered by Bidigare in \cite{Bidiga97}; a more recent exposition is found in
\cite{Saliol06}. As we will soon see, the analogous role for
$\operatorname*{PNSym}\nolimits_{n}^{\left(  2\right)  }$ is played by the
\emph{Janus algebra} of the braid arrangement, defined by Aguiar and Mahajan
in \cite[\S 1.9.3]{AguMah20}. We will introduce this algebra in a
combinatorial language, avoiding any mention of hyperplane arrangements, as
the combinatorial arguments stand well on their own legs here and the
geometric viewpoint would be a distraction.

\subsubsection{Set compositions and their friends}

We begin by defining a number of combinatorial structures related to \emph{set
compositions} (also known as \emph{ordered set partitions}):

\begin{itemize}
\item A \emph{weak set composition} of $\left[  n\right]  $ means a (finite)
tuple $\left(  A_{1},A_{2},\ldots,A_{k}\right)  $ of disjoint subsets of
$\left[  n\right]  $ such that $A_{1}\cup A_{2}\cup\cdots\cup A_{k}=\left[
n\right]  $. The entries $A_{i}$ of such a tuple are called its \emph{blocks}.
Note that some of these blocks $A_{i}$ can be empty, so $k$ can be arbitrarily
large. We let $\operatorname*{WSC}$ be the set of all weak set compositions of
$\left[  n\right]  $. This is an infinite set.

\item A \emph{set composition} of $\left[  n\right]  $ means a weak set
composition of $\left[  n\right]  $ whose blocks are all nonempty. For
instance, $\left(  \left\{  2,5\right\}  ,\ \left\{  1\right\}  ,\ \left\{
3,4\right\}  \right)  $ is a set composition of $\left[  5\right]  $, but
$\left(  \left\{  2,5\right\}  ,\ \varnothing,\ \left\{  1\right\}
,\ \left\{  3,4\right\}  \right)  $ is only a weak set composition of $\left[
5\right]  $. We let $\operatorname*{SC}$ be the set of all set compositions of
$\left[  n\right]  $. This is a finite set, whose size is known as the $n$-th
\emph{ordered Bell number}.

\item A \emph{weak set mopiscotion} of $\left[  n\right]  $ means a pair
$\left(  \mathbf{A},\sigma\right)  $ consisting of a weak set composition
$\mathbf{A}$ of $\left[  n\right]  $ that has $k$ blocks and a permutation
$\sigma\in\mathfrak{S}_{k}$. Likewise we define \emph{set mopiscotions}. We
let $\operatorname*{WSM}$ be the set of all weak set mopiscotions of $\left[
n\right]  $, and we let $\operatorname*{SM}$ be the set of all set
mopiscotions of $\left[  n\right]  $.

\item If $\mathbf{A}=\left(  A_{1},A_{2},\ldots,A_{k}\right)  $ is a weak set
composition of $\left[  n\right]  $, then $\operatorname*{comp}\mathbf{A}$ is
the weak composition of $n$ defined by%
\[
\operatorname*{comp}\mathbf{A}:=\left(  \left\vert A_{1}\right\vert
,\left\vert A_{2}\right\vert ,\ldots,\left\vert A_{k}\right\vert \right)  .
\]
If $\mathbf{A}$ is a set composition, then $\operatorname*{comp}\mathbf{A}$ is
a composition.

\item If $\mathbf{A}=\left(  A_{1},A_{2},\ldots,A_{k}\right)  \in
\operatorname*{WSC}$ is a weak set composition with $k$ blocks, and if
$\sigma\in\mathfrak{S}_{k}$, then $\mathbf{A}\cdot\sigma$ is the weak set
composition $\left(  A_{\sigma\left(  1\right)  },A_{\sigma\left(  2\right)
},\ldots,A_{\sigma\left(  k\right)  }\right)  $. This gives a right action of
$\mathfrak{S}_{k}$ on the set of all weak set compositions of $\left[
n\right]  $ with $k$ blocks. Two weak set compositions $\mathbf{A}$ and
$\mathbf{B}$ are said to be \emph{support-equivalent} if $\mathbf{A}%
=\mathbf{B}\cdot\sigma$ for some $\sigma\in\mathfrak{S}_{k}$ (that is, if
$\mathbf{A}$ can be obtained from $\mathbf{B}$ by reordering the blocks). This
is an equivalence relation; its equivalence classes are called the \emph{weak
set partitions} of $\left[  n\right]  $. (They can be viewed as set partitions
that allow empty blocks.) Likewise, we define \emph{set partitions} of
$\left[  n\right]  $ as support-equivalence classes of set compositions.
(These can be identified with the usual set partitions from enumerative combinatorics.)

If $\mathbf{A}$ is a weak set composition, then its support-equivalence class
is called its \emph{support} and is denoted $\operatorname*{supp}\mathbf{A}$.

For instance, the two set compositions $\left(  \left\{  2,5\right\}
,\ \left\{  1\right\}  ,\ \left\{  3,4\right\}  \right)  $ and $\left(
\left\{  1\right\}  ,\ \left\{  3,4\right\}  ,\ \left\{  2,5\right\}  \right)
$ are equivalent.

\item A \emph{weak set bicomposition} of $\left[  n\right]  $ means a pair
$\left(  \mathbf{A},\mathbf{B}\right)  $ of two weak set compositions
$\mathbf{A}$ and $\mathbf{B}$ that have the same support (i.e., satisfy
$\mathbf{A}=\mathbf{B}\cdot\sigma$ for some $\sigma\in\mathfrak{S}_{k}$, where
$k$ is the number of blocks of $\mathbf{B}$). We let $\operatorname*{WJ}$ be
the set of all weak set bicompositions of $\left[  n\right]  $. Likewise, we
define \emph{set bicompositions}, and we call their set $\operatorname*{J}$.
The letter $\operatorname*{J}$ is for \textquotedblleft
Janus\textquotedblright, as we regard $\mathbf{A}$ and $\mathbf{B}$ as the two
\textquotedblleft faces\textquotedblright\ of a weak set bicomposition
$\left(  \mathbf{A},\mathbf{B}\right)  $; this terminology goes back to
\cite[\S 1.9.3]{AguMah20}.
\end{itemize}

Between all these sets we have a number of maps:

\begin{itemize}
\item We have the so-called \emph{reduction maps}, whose purpose is to remove
empty blocks from weak set compositions. Specifically, these are the maps%
\begin{align*}
\operatorname*{red}  &  :\operatorname*{WSC}\rightarrow\operatorname*{SC},\\
\operatorname*{red}  &  :\operatorname*{WSM}\rightarrow\operatorname*{SM},\\
\operatorname*{red}  &  :\operatorname*{WJ}\rightarrow\operatorname*{J}%
\end{align*}
defined as follows:

\begin{itemize}
\item If $\mathbf{A}=\left(  A_{1},A_{2},\ldots,A_{k}\right)  \in
\operatorname*{WSC}$ is a weak set composition, then $\operatorname*{red}%
\mathbf{A}:=\left(  A_{j_{1}},A_{j_{2}},\ldots,A_{j_{h}}\right)
\in\operatorname*{SC}$, where $\left(  j_{1}<j_{2}<\cdots<j_{h}\right)  $ is
the list of all elements $i$ of $\left[  k\right]  $ satisfying $A_{i}%
\neq\varnothing$.

\item If $\left(  \mathbf{A},\sigma\right)  \in\operatorname*{WSM}$ is a weak
set mopiscotion with $\mathbf{A}=\left(  A_{1},A_{2},\ldots,A_{k}\right)  $
and $\sigma\in\mathfrak{S}_{k}$, then
\[
\operatorname*{red}\left(  \mathbf{A},\sigma\right)  :=\left(
\operatorname*{red}\mathbf{A},\tau\right)  \in\operatorname*{SM},
\]
where $\left(  j_{1}<j_{2}<\cdots<j_{h}\right)  $ is the list of all elements
$i$ of $\left[  k\right]  $ satisfying $A_{i}\neq\varnothing$, and where
$\tau\in\mathfrak{S}_{h}$ is the standardization of the list $\left(
\sigma\left(  j_{1}\right)  ,\sigma\left(  j_{2}\right)  ,\ldots,\sigma\left(
j_{h}\right)  \right)  $ (as in Definition \ref{def.red-wmopis}).

\item If $\left(  \mathbf{A},\mathbf{B}\right)  \in\operatorname*{WJ}$ is a
weak set bicomposition, then $\operatorname*{red}\left(  \mathbf{A}%
,\mathbf{B}\right)  :=\left(  \operatorname*{red}\mathbf{A}%
,\operatorname*{red}\mathbf{B}\right)  \in\operatorname*{J}$.
\end{itemize}

\item We have a map
\begin{align*}
g_{\operatorname*{W}}:\operatorname*{WSM}  &  \rightarrow\operatorname*{WJ},\\
\left(  \mathbf{A},\sigma\right)   &  \mapsto\left(  \mathbf{A},\ \mathbf{A}%
\cdot\sigma^{-1}\right)
\end{align*}
and an analogously defined map $g:\operatorname*{SM}\rightarrow
\operatorname*{J}$ (which is simply the restriction of $g_{\operatorname*{W}}$
to the subset $\operatorname*{SM}$).
\end{itemize}

Most importantly, we have monoid structures on some of the above sets:

\begin{itemize}
\item The \emph{lexicographic meet} $\mathbf{A}\wedge\mathbf{B}$ of two weak
set compositions $\mathbf{A}=\left(  A_{1},A_{2},\ldots,A_{k}\right)  $ and
$\mathbf{B}=\left(  B_{1},B_{2},\ldots,B_{\ell}\right)  $ of $\left[
n\right]  $ is defined to be the weak set composition
\begin{align*}
\mathbf{A}\wedge\mathbf{B}  &  :=(A_{1}\cap B_{1},\ A_{1}\cap B_{2}%
,\ \ldots,\ A_{1}\cap B_{\ell},\\
&  \ \ \ \ \ \ \ A_{2}\cap B_{1},\ A_{2}\cap B_{2},\ \ldots,\ A_{2}\cap
B_{\ell},\\
&  \ \ \ \ \ \ \ \ldots,\\
&  \ \ \ \ \ \ \ A_{k}\cap B_{1},\ A_{k}\cap B_{2},\ \ldots,\ A_{k}\cap
B_{\ell})
\end{align*}
of $\left[  n\right]  $. (This is just the $k\ell$-tuple consisting of all the
intersections $A_{i}\cap B_{j}$, listed in the order of lexicographically
increasing $\left(  i,j\right)  $.) This operation $\left(  \mathbf{A}%
,\mathbf{B}\right)  \mapsto\mathbf{A}\wedge\mathbf{B}$ makes
$\operatorname*{WSC}$ into a monoid (the associativity is easy to check); its
neutral element is the $1$-block weak set composition $\left(  \left[
n\right]  \right)  $.

\item The \emph{reduced lexicographic meet} $\mathbf{AB}$ of two set
compositions $\mathbf{A}$ and $\mathbf{B}$ of $\left[  n\right]  $ is defined
to be the set composition $\operatorname*{red}\left(  \mathbf{A}%
\wedge\mathbf{B}\right)  $. This operation $\left(  \mathbf{A},\mathbf{B}%
\right)  \mapsto\mathbf{AB}$ makes $\operatorname*{SC}$ into a monoid (but not
into a submonoid of $\operatorname*{WSC}$). This monoid is a \emph{band} --
i.e., a monoid in which every element is idempotent. It is known as the
\emph{face monoid} (or \emph{Tits monoid}) of the type-A braid arrangement
(due to a geometric interpretation of set compositions as faces of the
arrangement -- see, e.g., \cite[\S 1.2.1]{Saliol06}). The map
$\operatorname*{red}:\operatorname*{WSC}\rightarrow\operatorname*{SC}$ is a
surjective monoid morphism.

\item The \emph{lexicographic meet} $\left(  \mathbf{A},\sigma\right)
\wedge\left(  \mathbf{B},\tau\right)  $ of two weak set mopiscotions $\left(
\mathbf{A},\sigma\right)  $ and $\left(  \mathbf{B},\tau\right)  $ is defined
to be the weak set mopiscotion $\left(  \mathbf{A}\wedge\mathbf{B}%
,\ \tau\left[  \sigma\right]  \right)  $. This makes $\operatorname*{WSM}$
into a monoid (actually a submonoid of $\operatorname*{WSC}\times\mathfrak{S}%
$, where $\mathfrak{S}$ is as defined in the Second proof idea for Theorem
\ref{thm.PNSym.alg}). Its neutral element is $\left(  \left(  \left[
n\right]  \right)  ,\ \operatorname*{id}\nolimits_{\left[  1\right]  }\right)
$.

\item The \emph{reduced lexicographic meet} $\left(  \mathbf{A},\sigma\right)
\left(  \mathbf{B},\tau\right)  $ of two set mopiscotions $\left(
\mathbf{A},\sigma\right)  $ and $\left(  \mathbf{B},\tau\right)  $ is defined
to be the set mopiscotion $\operatorname*{red}\left(  \left(  \mathbf{A}%
,\sigma\right)  \wedge\left(  \mathbf{B},\tau\right)  \right)  $. This makes
$\operatorname*{SM}$ into a monoid. The map $\operatorname*{red}%
:\operatorname*{WSM}\rightarrow\operatorname*{SM}$ is a surjective monoid morphism.

\item The \emph{lexicographic meet} $\left(  \mathbf{A},\mathbf{A}^{\prime
}\right)  \wedge\left(  \mathbf{B},\mathbf{B}^{\prime}\right)  $ of two weak
set bicompositions $\left(  \mathbf{A},\mathbf{A}^{\prime}\right)  $ and
$\left(  \mathbf{B},\mathbf{B}^{\prime}\right)  $ is defined to be the weak
set bicomposition $\left(  \mathbf{A}\wedge\mathbf{B},\ \mathbf{B}^{\prime
}\wedge\mathbf{A}^{\prime}\right)  $ (note the reverse order of factors in the
second argument!). This makes $\operatorname*{WJ}$ into a monoid, known as the
(type-A) \emph{weak Janus monoid}.

\item The \emph{reduced lexicographic meet} $\left(  \mathbf{A},\mathbf{A}%
^{\prime}\right)  \left(  \mathbf{B},\mathbf{B}^{\prime}\right)  $ of two set
bicompositions $\left(  \mathbf{A},\mathbf{A}^{\prime}\right)  $ and $\left(
\mathbf{B},\mathbf{B}^{\prime}\right)  $ is defined to be the set
bicomposition $\operatorname*{red}\left(  \left(  \mathbf{A},\mathbf{A}%
^{\prime}\right)  \wedge\left(  \mathbf{B},\mathbf{B}^{\prime}\right)
\right)  =\left(  \operatorname*{red}\left(  \mathbf{A}\wedge\mathbf{B}%
\right)  ,\ \operatorname*{red}\left(  \mathbf{B}^{\prime}\wedge
\mathbf{A}^{\prime}\right)  \right)  $. This makes $\operatorname*{J}$ into a
monoid, known as the (type-A) \emph{Janus monoid}. The map
$\operatorname*{red}:\operatorname*{WJ}\rightarrow\operatorname*{J}$ is a
surjective monoid morphism.
\end{itemize}

\begin{proposition}
\label{prop.PNSym.g-square}The diagram%
\[
\xymatrix{
\operatorname{WSM} \ar[r]^{g_{\operatorname{W}}} \ar[d]_{\operatorname{red}} & \operatorname{WJ} \ar[d]^{\operatorname{red}} \\
\operatorname{SM} \ar[r]_g & \operatorname{J}
}
\]
is commutative.
\end{proposition}

\begin{proof}
[Proof idea.]A refreshing exercise in the definition of standardization.

\begin{fineprint}
Let $\left(  \mathbf{A},\sigma\right)  \in\operatorname*{WSM}$. We must show
that $\operatorname*{red}\left(  g_{\operatorname*{W}}\left(  \mathbf{A}%
,\sigma\right)  \right)  =g\left(  \operatorname*{red}\left(  \mathbf{A}%
,\sigma\right)  \right)  $.

Write the weak set composition $\mathbf{A}$ as $\mathbf{A}=\left(  A_{1}%
,A_{2},\ldots,A_{k}\right)  $. Then, the definition of $\operatorname*{red}%
:\operatorname*{WSM}\rightarrow\operatorname*{SM}$ yields
\begin{equation}
\operatorname*{red}\left(  \mathbf{A},\sigma\right)  =\left(
\operatorname*{red}\mathbf{A},\tau\right)  , \label{pf.prop.PNSym.g-square.1}%
\end{equation}
where $\left(  j_{1}<j_{2}<\cdots<j_{h}\right)  $ is the list of all elements
$i$ of $\left[  k\right]  $ satisfying $A_{i}\neq\varnothing$, and where
$\tau\in\mathfrak{S}_{h}$ is the standardization of the list $\left(
\sigma\left(  j_{1}\right)  ,\sigma\left(  j_{2}\right)  ,\ldots,\sigma\left(
j_{h}\right)  \right)  $ (as in Definition \ref{def.red-wmopis}). Thus,%
\begin{equation}
g\left(  \operatorname*{red}\left(  \mathbf{A},\sigma\right)  \right)
=g\left(  \operatorname*{red}\mathbf{A},\tau\right)  =\left(
\operatorname*{red}\mathbf{A},\ \left(  \operatorname*{red}\mathbf{A}\right)
\cdot\tau^{-1}\right)  \label{pf.prop.PNSym.g-square.2}%
\end{equation}
by the definition of $g$.

By the definition of $g_{\operatorname*{W}}$, we have $g_{\operatorname*{W}%
}\left(  \mathbf{A},\sigma\right)  =\left(  \mathbf{A},\ \mathbf{A}\cdot
\sigma^{-1}\right)  $, so that
\begin{equation}
\operatorname*{red}\left(  g_{\operatorname*{W}}\left(  \mathbf{A}%
,\sigma\right)  \right)  =\operatorname*{red}\left(  \mathbf{A},\ \mathbf{A}%
\cdot\sigma^{-1}\right)  =\left(  \operatorname*{red}\mathbf{A}%
,\ \operatorname*{red}\left(  \mathbf{A}\cdot\sigma^{-1}\right)  \right)  .
\label{pf.prop.PNSym.g-square.3}%
\end{equation}

We must prove that the left-hand sides of the equalities
(\ref{pf.prop.PNSym.g-square.3}) and (\ref{pf.prop.PNSym.g-square.2}) are
equal. Clearly, it suffices to show that the right-hand sides are equal. In
other words, it suffices to show that $\operatorname*{red}\left(
\mathbf{A}\cdot\sigma^{-1}\right)  =\left(  \operatorname*{red}\mathbf{A}%
\right)  \cdot\tau^{-1}$.

Both set compositions $\operatorname*{red}\left(  \mathbf{A}\cdot\sigma
^{-1}\right)  $ and $\left(  \operatorname*{red}\mathbf{A}\right)  \cdot
\tau^{-1}$ are obtained from $\mathbf{A}$ by removing all empty blocks and
permuting the blocks (either before or after the removal of empty blocks).
Thus, they both consist of the nonempty blocks of $\mathbf{A}$ in some order.
In order to show that $\operatorname*{red}\left(  \mathbf{A}\cdot\sigma
^{-1}\right)  =\left(  \operatorname*{red}\mathbf{A}\right)  \cdot\tau^{-1}$,
we only need to prove that the orders in which they contain these nonempty
blocks are the same. In other words, we must prove that if $p$ and $q$ are two
distinct elements of $\left[  k\right]  $ such that $A_{p}\neq\varnothing$ and
$A_{q}\neq\varnothing$, then the equivalence%
\begin{align}
&  \ \left(  A_{p}\text{ occurs before }A_{q}\text{ in }\operatorname*{red}%
\left(  \mathbf{A}\cdot\sigma^{-1}\right)  \right) \nonumber\\
&  \Longleftrightarrow\ \left(  A_{p}\text{ occurs before }A_{q}\text{ in
}\left(  \operatorname*{red}\mathbf{A}\right)  \cdot\tau^{-1}\right)
\label{pf.prop.PNSym.g-square.goal}%
\end{align}
holds. So let us prove this. Let $p$ and $q$ be two distinct elements of
$\left[  k\right]  $ such that $A_{p}\neq\varnothing$ and $A_{q}%
\neq\varnothing$. Then, we have $p=j_{x}$ and $q=j_{y}$ for some
$x,y\in\left[  h\right]  $. But we have the following chain of equivalences:%
\begin{align*}
&  \ \left(  A_{p}\text{ occurs before }A_{q}\text{ in }\operatorname*{red}%
\left(  \mathbf{A}\cdot\sigma^{-1}\right)  \right) \\
&  \Longleftrightarrow\ \left(  \sigma\left(  p\right)  <\sigma\left(
q\right)  \right) \\
&  \Longleftrightarrow\ \left(  \sigma\left(  j_{x}\right)  <\sigma\left(
j_{y}\right)  \right)  \ \ \ \ \ \ \ \ \ \ \left(  \text{since }p=j_{x}\text{
and }q=j_{y}\right) \\
&  \Longleftrightarrow\ \left(  \tau\left(  x\right)  <\tau\left(  y\right)
\right)
\end{align*}
(since $\tau$ is the standardization of the list $\left(  \sigma\left(
j_{1}\right)  ,\sigma\left(  j_{2}\right)  ,\ldots,\sigma\left(  j_{h}\right)
\right)  $, so that the relative order of the values of $\tau$ agrees with the
relative order of the entries of $\left(  \sigma\left(  j_{1}\right)
,\sigma\left(  j_{2}\right)  ,\ldots,\sigma\left(  j_{h}\right)  \right)  $).
We also have the following chain of equivalences:%
\begin{align*}
&  \ \left(  A_{p}\text{ occurs before }A_{q}\text{ in }\left(
\operatorname*{red}\mathbf{A}\right)  \cdot\tau^{-1}\right) \\
&  \Longleftrightarrow\ \left(  A_{j_{x}}\text{ occurs before }A_{j_{y}}\text{
in }\left(  \operatorname*{red}\mathbf{A}\right)  \cdot\tau^{-1}\right)
\ \ \ \ \ \ \ \ \ \ \left(  \text{since }p=j_{x}\text{ and }q=j_{y}\right) \\
&  \Longleftrightarrow\ \left(  A_{j_{x}}\text{ occurs before }A_{j_{y}}\text{
in }\left(  A_{j_{\tau^{-1}\left(  1\right)  }},A_{j_{\tau^{-1}\left(
2\right)  }},\ldots,A_{j_{\tau^{-1}\left(  h\right)  }}\right)  \right) \\
&  \ \ \ \ \ \ \ \ \ \ \ \ \ \ \ \ \ \ \ \ \left(
\begin{array}
[c]{c}%
\text{since the definition of }\operatorname*{red}\mathbf{A}\text{ shows}\\
\text{that }\operatorname*{red}\mathbf{A}=\left(  A_{j_{1}},A_{j_{2}}%
,\ldots,A_{j_{h}}\right)  \text{ and}\\
\text{thus }\left(  \operatorname*{red}\mathbf{A}\right)  \cdot\tau
^{-1}=\left(  A_{j_{\tau^{-1}\left(  1\right)  }},A_{j_{\tau^{-1}\left(
2\right)  }},\ldots,A_{j_{\tau^{-1}\left(  h\right)  }}\right)
\end{array}
\right) \\
&  \Longleftrightarrow\ \left(  x\text{ occurs before }y\text{ in }\left(
\tau^{-1}\left(  1\right)  ,\ \tau^{-1}\left(  2\right)  ,\ \ldots,\ \tau
^{-1}\left(  h\right)  \right)  \right) \\
&  \Longleftrightarrow\ \left(  \tau\left(  x\right)  <\tau\left(  y\right)
\right)  .
\end{align*}
Comparing these two chains of equivalences, we obtain the desired equivalence
(\ref{pf.prop.PNSym.g-square.goal}), which is all that stood between us and
the proof of Proposition \ref{prop.PNSym.g-square}.
\end{fineprint}
\end{proof}

\begin{proposition}
\label{prop.PNSym.g-mon}The maps $g_{\operatorname*{W}}:\operatorname*{WSM}%
\rightarrow\operatorname*{WJ}$ and $g:\operatorname*{SM}\rightarrow
\operatorname*{J}$ are monoid morphisms.
\end{proposition}

\begin{proof}
[Proof idea.]An easy exercise about $\tau\left[  \sigma\right]  $.

\begin{fineprint}
In more detail: It suffices to show that $g_{\operatorname*{W}}%
:\operatorname*{WSM}\rightarrow\operatorname*{WJ}$ is a monoid morphism, since
then the commutative diagram in Proposition \ref{prop.PNSym.g-square} (and the
surjectivity of $\operatorname*{red}:\operatorname*{WSM}\rightarrow
\operatorname*{SM}$) will yield the same claim about $g:\operatorname*{SM}%
\rightarrow\operatorname*{J}$.

So we must prove the equality $g_{\operatorname*{W}}\left(  \left(
\mathbf{A},\sigma\right)  \wedge\left(  \mathbf{B},\tau\right)  \right)
=g_{\operatorname*{W}}\left(  \mathbf{A},\sigma\right)  \wedge
g_{\operatorname*{W}}\left(  \mathbf{B},\tau\right)  $ for any two weak set
mopiscotions $\left(  \mathbf{A},\sigma\right)  $ and $\left(  \mathbf{B}%
,\tau\right)  $. Fix two weak set mopiscotions $\left(  \mathbf{A}%
,\sigma\right)  $ and $\left(  \mathbf{B},\tau\right)  $. The definition of
$g_{\operatorname*{W}}$ yields $g_{\operatorname*{W}}\left(  \mathbf{A}%
,\sigma\right)  =\left(  \mathbf{A},\ \mathbf{A}\cdot\sigma^{-1}\right)  $ and
$g_{\operatorname*{W}}\left(  \mathbf{B},\tau\right)  =\left(  \mathbf{B}%
,\ \mathbf{B}\cdot\tau^{-1}\right)  $. Thus,%
\begin{align}
g_{\operatorname*{W}}\left(  \mathbf{A},\sigma\right)  \wedge
g_{\operatorname*{W}}\left(  \mathbf{B},\tau\right)   &  =\left(
\mathbf{A},\ \mathbf{A}\cdot\sigma^{-1}\right)  \wedge\left(  \mathbf{B}%
,\ \mathbf{B}\cdot\tau^{-1}\right) \nonumber\\
&  =\left(  \mathbf{A}\wedge\mathbf{B},\ \left(  \mathbf{B}\cdot\tau
^{-1}\right)  \wedge\left(  \mathbf{A}\cdot\sigma^{-1}\right)  \right)  .
\label{pf.prop.PNSym.g-mon.1}%
\end{align}
On the other hand, we have $\left(  \mathbf{A},\sigma\right)  \wedge\left(
\mathbf{B},\tau\right)  =\left(  \mathbf{A}\wedge\mathbf{B},\ \tau\left[
\sigma\right]  \right)  $, so that%
\begin{align}
g_{\operatorname*{W}}\left(  \left(  \mathbf{A},\sigma\right)  \wedge\left(
\mathbf{B},\tau\right)  \right)   &  =g_{\operatorname*{W}}\left(  \left(
\mathbf{A}\wedge\mathbf{B},\ \tau\left[  \sigma\right]  \right)  \right)
\nonumber\\
&  =\left(  \mathbf{A}\wedge\mathbf{B},\ \left(  \mathbf{A}\wedge
\mathbf{B}\right)  \cdot\left(  \tau\left[  \sigma\right]  \right)
^{-1}\right)  . \label{pf.prop.PNSym.g-mon.2}%
\end{align}
We must prove that $g_{\operatorname*{W}}\left(  \left(  \mathbf{A}%
,\sigma\right)  \wedge\left(  \mathbf{B},\tau\right)  \right)
=g_{\operatorname*{W}}\left(  \mathbf{A},\sigma\right)  \wedge
g_{\operatorname*{W}}\left(  \mathbf{B},\tau\right)  $. In other words, we
must prove that the left-hand sides of the equalities
(\ref{pf.prop.PNSym.g-mon.1}) and (\ref{pf.prop.PNSym.g-mon.2}) are equal.
Clearly, it suffices to show that the right-hand sides are equal. This boils
down to showing that
\begin{equation}
\left(  \mathbf{B}\cdot\tau^{-1}\right)  \wedge\left(  \mathbf{A}\cdot
\sigma^{-1}\right)  =\left(  \mathbf{A}\wedge\mathbf{B}\right)  \cdot\left(
\tau\left[  \sigma\right]  \right)  ^{-1}. \label{pf.prop.PNSym.g-mon.3}%
\end{equation}

But this is all about permutations: Both weak set compositions $\left(
\mathbf{A}\wedge\mathbf{B}\right)  \cdot\left(  \tau\left[  \sigma\right]
\right)  ^{-1}$ and $\left(  \mathbf{B}\cdot\tau^{-1}\right)  \wedge\left(
\mathbf{A}\cdot\sigma^{-1}\right)  $ contain the same blocks (namely, the
pairwise intersections $A_{i}\cap B_{j}$), but we must show that these blocks
also appear in the same order. Both of our weak set compositions have $k\ell$
entries each, but it is best to index them not by the numbers $i\in\left[
k\ell\right]  $ but rather by the pairs $\left(  j,i\right)  \in\left[
\ell\right]  \times\left[  k\right]  $ (in lexicographically increasing
order). Upon this reindexing, the $\left(  j,i\right)  $-th entry of $\left(
\mathbf{B}\cdot\tau^{-1}\right)  \wedge\left(  \mathbf{A}\cdot\sigma
^{-1}\right)  $ becomes $B_{\tau^{-1}\left(  j\right)  }\cap A_{\sigma
^{-1}\left(  i\right)  }=A_{\sigma^{-1}\left(  i\right)  }\cap B_{\tau
^{-1}\left(  j\right)  }$. What about the $\left(  j,i\right)  $-th entry of
$\left(  \mathbf{A}\wedge\mathbf{B}\right)  \cdot\left(  \tau\left[
\sigma\right]  \right)  ^{-1}$ ? Recall the tormutation $\phi^{\prime
}\left\langle \phi\right\rangle $ defined in the second proof idea for Theorem
\ref{thm.PNSym.alg}. As we know, $\tau\left[  \sigma\right]  $ is really just
the standardization of the tormutation $\tau\left\langle \sigma\right\rangle
:\left[  k\right]  \times\left[  \ell\right]  \rightarrow\left[  \ell\right]
\times\left[  k\right]  $ that sends each $\left(  u,v\right)  $ to $\left(
\tau\left(  v\right)  ,\sigma\left(  u\right)  \right)  $. Our reindexing has
turned $\left(  \mathbf{A}\wedge\mathbf{B}\right)  \cdot\left(  \tau\left[
\sigma\right]  \right)  ^{-1}$ into $\left(  \mathbf{A}\wedge\mathbf{B}%
\right)  \cdot\left(  \tau\left\langle \sigma\right\rangle \right)  ^{-1}$.
But $\left(  \tau\left\langle \sigma\right\rangle \right)  ^{-1}\left(
j,i\right)  =\left(  \sigma^{-1}\left(  i\right)  ,\tau^{-1}\left(  j\right)
\right)  $. Hence, the $\left(  j,i\right)  $-th entry of $\left(
\mathbf{A}\wedge\mathbf{B}\right)  \cdot\left(  \tau\left[  \sigma\right]
\right)  ^{-1}$ is the $\left(  \sigma^{-1}\left(  i\right)  ,\tau^{-1}\left(
j\right)  \right)  $-th entry of $\mathbf{A}\wedge\mathbf{B}$; but the latter
entry is $A_{\sigma^{-1}\left(  i\right)  }\cap B_{\tau^{-1}\left(  j\right)
}$. This is the same answer that we obtained for the $\left(  j,i\right)  $-th
entry of $\left(  \mathbf{B}\cdot\tau^{-1}\right)  \wedge\left(
\mathbf{A}\cdot\sigma^{-1}\right)  $. Hence, we have shown that respective
entries of $\left(  \mathbf{A}\wedge\mathbf{B}\right)  \cdot\left(
\tau\left[  \sigma\right]  \right)  ^{-1}$ and $\left(  \mathbf{B}\cdot
\tau^{-1}\right)  \wedge\left(  \mathbf{A}\cdot\sigma^{-1}\right)  $ are
equal. Thus, (\ref{pf.prop.PNSym.g-mon.3}) follows, and we are done.
\end{fineprint}
\end{proof}

\begin{proposition}
\label{prop.PNSym.g-inj}The map $g:\operatorname*{SM}\rightarrow
\operatorname*{J}$ is a monoid isomorphism.
\end{proposition}

\begin{proof}
[Proof idea.]Also easy (but note that this does not apply to
$g_{\operatorname*{W}}$, since $g_{\operatorname*{W}}$ is not injective!).

\begin{fineprint}
In more detail: We already know from Proposition \ref{prop.PNSym.g-mon} that
$g$ is a monoid morphism. It remains to prove that $g$ is bijective.
Surjectivity is clear, since each $\left(  \mathbf{A},\mathbf{B}\right)
\in\operatorname*{J}$ satisfies $\mathbf{A}=\mathbf{B}\cdot\sigma$ for some
$\sigma\in\mathfrak{S}_{k}$. To prove injectivity, we must show that this
$\sigma$ is unique. But this is clear, because the blocks of $\mathbf{B}$ are
distinct (being nonempty disjoint sets) and thus cannot be nontrivially
permuted without changing $\mathbf{B}$.
\end{fineprint}
\end{proof}

We could also define a monoid structure on the sets of (weak) set partitions,
i.e., on the images of the support maps. Then, the support maps would be
monoid morphisms.

All the above-defined monoids come with a left action of the symmetric group
$\mathfrak{S}_{n}$ (unrelated to the right action of $\mathfrak{S}_{k}$ that
was used in defining supports). Namely:

\begin{itemize}
\item The symmetric group $\mathfrak{S}_{n}$ acts on $\operatorname*{WSC}$ and
on $\operatorname*{SC}$ by the rule%
\[
\sigma\cdot\left(  A_{1},A_{2},\ldots,A_{k}\right)  :=\left(  \sigma\left(
A_{1}\right)  ,\ \sigma\left(  A_{2}\right)  ,\ \ldots,\ \sigma\left(
A_{k}\right)  \right)  .
\]
(Here, $\sigma\left(  A_{i}\right)  $ is understood as usual -- i.e., as the
image of the subset $A_{i}$ of $\left[  n\right]  $ under $\sigma$.)

\item The symmetric group $\mathfrak{S}_{n}$ acts on $\operatorname*{WSM}$ and
on $\operatorname*{SM}$ by the rule%
\[
\sigma\cdot\left(  \mathbf{A},\tau\right)  :=\left(  \sigma\cdot
\mathbf{A},\ \tau\right)  .
\]
(Thus, $\sigma$ only affects the first argument of a weak set mopiscotion.)

\item The symmetric group $\mathfrak{S}_{n}$ acts on $\operatorname*{WJ}$ and
on $\operatorname*{J}$ by the rule%
\[
\sigma\cdot\left(  \mathbf{A},\mathbf{B}\right)  :=\left(  \sigma
\cdot\mathbf{A},\ \sigma\cdot\mathbf{B}\right)  .
\]
(That is, $\sigma$ acts diagonally on the two arguments.)
\end{itemize}

All these actions are actions by monoid automorphisms.

\begin{remark}
\label{rmk.PNSym.AM1}Let us briefly discuss the connection of our Janus monoid
$\operatorname*{J}$ to the Janus monoid of Aguiar and Mahajan.

In \cite[\S 1.9.3]{AguMah20}, Aguiar and Mahajan define the \emph{Janus
monoid} $\mathrm{J}\left[  \mathcal{A}\right]  $ of a hyperplane arrangement
$\mathcal{A}$. It consists of the so-called \emph{bifaces}, which are the
pairs $\left(  U,V\right)  $ of two faces of $\mathcal{A}$ such that $U$ and
$V$ have the same support (i.e., supporting hyperplane). When $\mathcal{A}$ is
the braid arrangement in $\mathbb{R}^{n}$, the faces of $\mathcal{A}$ are in
canonical bijection with the set compositions of $\left[  n\right]  $, and the
support of a face corresponds to the (unordered) set partition obtained by
forgetting the order of the blocks in the corresponding set composition. Thus,
the bifaces of the braid arrangement are what we call set bicompositions. The
multiplication of bifaces that defines the product in $\mathrm{J}\left[
\mathcal{A}\right]  $ turns out to be precisely our reduced lexicographic meet
operation on set bicompositions.
\end{remark}

\subsubsection{Algebras and the master diagram}

Each monoid $M$ induces a monoid algebra $\mathbf{k}\left[  M\right]  $, which
is a $\mathbf{k}$-algebra that is spanned freely (as a $\mathbf{k}$-module) by
the elements of $M$ and whose multiplication is just the multiplication of $M$
extended by linearity. In particular, all the above-defined monoids
$\operatorname*{WSC}$, $\operatorname*{SC}$, $\operatorname*{WSM}$,
$\operatorname*{SM}$, $\operatorname*{WJ}$ and $\operatorname*{J}$ have monoid
algebras $\mathbf{k}\left[  \operatorname*{WSC}\right]  $, $\mathbf{k}\left[
\operatorname*{SC}\right]  $, $\mathbf{k}\left[  \operatorname*{WSM}\right]
$, $\mathbf{k}\left[  \operatorname*{SM}\right]  $, $\mathbf{k}\left[
\operatorname*{WJ}\right]  $ and $\mathbf{k}\left[  \operatorname*{J}\right]
$. Some of these algebras have names: $\mathbf{k}\left[  \operatorname*{J}%
\right]  $ is known as the \emph{Janus algebra}, while $\mathbf{k}\left[
\operatorname*{SC}\right]  $ is known as the \emph{face algebra} or the
\emph{Tits algebra}. The symmetric group $\mathfrak{S}_{n}$ acts on each of
the above-listed monoids by monoid automorphisms, and thus acts on all their
monoid algebras by $\mathbf{k}$-algebra automorphisms.

Any monoid morphism $f:M\rightarrow N$ induces a $\mathbf{k}$-algebra morphism
$\mathbf{k}\left[  M\right]  \rightarrow\mathbf{k}\left[  N\right]  $, which
we shall again denote by $f$ by abuse of notation. In particular, the
commutative diagram in Proposition \ref{prop.PNSym.g-square} thus induces a
commutative diagram of $\mathbf{k}$-algebras and $\mathbf{k}$-algebra
morphisms%
\[
\xymatrix{
\kk\left[\operatorname{WSM}\right] \ar[r]^{g_{\operatorname{W}}} \ar[d]_{\operatorname{red}} & \kk\left[\operatorname{WJ}\right] \ar[d]^{\operatorname{red}} \\
\kk\left[\operatorname{SM}\right] \ar[r]_g & \kk\left[\operatorname{J}\right]
}
\]
(by Proposition \ref{prop.PNSym.g-mon}).

However, we can do more with the algebras than we could with the monoids. For
this purpose, we recall the $\mathbf{k}$-algebra $\operatorname*{PNSym}%
\nolimits_{n}^{\left(  2\right)  }$, which has basis
\[
\left(  F_{\alpha,\sigma}\right)  _{\left(  \alpha,\sigma\right)  \text{ is a
mopiscotion with }\left\vert \alpha\right\vert =n}.
\]
We also consider the analogously defined $\mathbf{k}$-algebra
$\operatorname*{WPNSym}\nolimits_{n}^{\prime\left(  2\right)  }$, which has
basis
\[
\left(  \widehat{F}_{\alpha,\sigma}\right)  _{\left(  \alpha,\sigma\right)
\text{ is a weak mopiscotion with }\left\vert \alpha\right\vert =n}%
\]
(see the Second proof idea for Theorem \ref{thm.PNSym.alg} for its precise
definition). There is a surjective $\mathbf{k}$-algebra morphism
\begin{align*}
\operatorname*{red}:\operatorname*{WPNSym}\nolimits_{n}^{\prime\left(
2\right)  }  &  \rightarrow\operatorname*{PNSym}\nolimits_{n}^{\left(
2\right)  },\\
\widehat{F}_{\alpha,\sigma}  &  \mapsto F_{\alpha,\sigma}%
\end{align*}
(since the rule for multiplying $F_{\alpha,\sigma}$'s in
$\operatorname*{PNSym}\nolimits_{n}^{\left(  2\right)  }$ is exactly the rule
for multiplying $\widehat{F}_{\alpha,\sigma}$'s in $\operatorname*{WPNSym}%
\nolimits_{n}^{\prime\left(  2\right)  }$). Now we claim the following.

\begin{theorem}
\label{thm.PNSym.f-diagram}Consider the two $\mathbf{k}$-linear maps
\begin{align*}
f_{\operatorname*{W}}:\operatorname*{WPNSym}\nolimits_{n}^{\prime\left(
2\right)  }  &  \rightarrow\mathbf{k}\left[  \operatorname*{WSM}\right]  ,\\
\widehat{F}_{\alpha,\sigma}  &  \mapsto\sum_{\substack{\mathbf{A}%
\in\operatorname*{WSC};\\\operatorname*{comp}\mathbf{A}=\alpha}}\left(
\mathbf{A},\sigma\right)
\end{align*}
and%
\begin{align*}
f:\operatorname*{PNSym}\nolimits_{n}^{\left(  2\right)  }  &  \rightarrow
\mathbf{k}\left[  \operatorname*{SM}\right]  ,\\
F_{\alpha,\sigma}  &  \mapsto\sum_{\substack{\mathbf{A}\in\operatorname*{SC}%
;\\\operatorname*{comp}\mathbf{A}=\alpha}}\left(  \mathbf{A},\sigma\right)  .
\end{align*}
Then:

\begin{enumerate}
\item[\textbf{(a)}] Both maps $f_{\operatorname*{W}}$ and $f$ are injective
$\mathbf{k}$-algebra morphisms.

\item[\textbf{(b)}] The diagram%
\[
\xymatrix{
\operatorname{WPNSym}_n^{\prime\left(2\right)} \ar[r]^-{f_{\operatorname{W}}} \ar[d]_{\operatorname{red}} & \kk\left[\operatorname{WSM}\right] \ar[r]^-{g_{\operatorname{W}}} \ar[d]_{\operatorname{red}} & \kk\left[\operatorname{WJ}\right] \ar[d]^{\operatorname{red}} \\
\operatorname{PNSym}_n^{\left(2\right)} \ar[r]_f & \kk\left[\operatorname{SM}\right] \ar[r]_g & \kk\left[\operatorname{J}\right]
}
\]
is commutative.

\item[\textbf{(c)}] Recall the $\mathfrak{S}_{n}$-actions on $\mathbf{k}%
\left[  \operatorname*{WSM}\right]  $ and $\mathbf{k}\left[
\operatorname*{SM}\right]  $. Then, the image of $f_{\operatorname*{W}}$ is
the $\mathfrak{S}_{n}$-invariant part of $\mathbf{k}\left[
\operatorname*{WSM}\right]  $ (that is, the $\mathbf{k}$-subalgebra of
$\mathbf{k}\left[  \operatorname*{WSM}\right]  $ consisting of all vectors
$v\in\mathbf{k}\left[  \operatorname*{WSM}\right]  $ that satisfy $\sigma\cdot
v=v$ for all $\sigma\in\mathfrak{S}_{n}$). Likewise, the image of $f$ is the
$\mathfrak{S}_{n}$-invariant part of $\mathbf{k}\left[  \operatorname*{SM}%
\right]  $.
\end{enumerate}
\end{theorem}

\begin{proof}
[Proof idea.]\textbf{(b)} Thanks to Proposition \ref{prop.PNSym.g-square}, we
only need to show the commutativity of the left square. But this is fairly
easy. \medskip

\textbf{(a)} Injectivity is easy for both $f_{\operatorname*{W}}$ and $f$.
Next, it is easy to prove that $f_{\operatorname*{W}}$ is a $\mathbf{k}%
$-algebra morphism. Since $\operatorname*{red}:\operatorname*{WPNSym}%
\nolimits_{n}^{\prime\left(  2\right)  }\rightarrow\operatorname*{PNSym}%
\nolimits_{n}^{\left(  2\right)  }$ is a surjective $\mathbf{k}$-algebra
morphism, we conclude that $f$ is a $\mathbf{k}$-algebra morphism as well.
\medskip

\textbf{(c)} This is pretty obvious, since the images of $\widehat{F}%
_{\alpha,\sigma}$ under $f_{\operatorname*{W}}$ (resp., the images of
$F_{\alpha,\sigma}$ under $f$) are precisely the orbit sums of the
$\mathfrak{S}_{n}$-action on $\operatorname*{WSM}$ (resp. $\operatorname*{SM}$).
\end{proof}

\begin{corollary}
\label{cor.PNSym.as-invariant-ring}The $\mathbf{k}$-algebra
$\operatorname*{PNSym}\nolimits_{n}^{\left(  2\right)  }$ is isomorphic to the
$\mathfrak{S}_{n}$-invariant part of $\mathbf{k}\left[  \operatorname*{J}%
\right]  $.
\end{corollary}

\begin{proof}
[Proof idea.]Consider the following chain of $\mathbf{k}$-algebra
isomorphisms:
\begin{align*}
\operatorname*{PNSym}\nolimits_{n}^{\left(  2\right)  }  &  \cong f\left(
\operatorname*{PNSym}\nolimits_{n}^{\left(  2\right)  }\right)
\ \ \ \ \ \ \ \ \ \ \left(  \text{by Theorem \ref{thm.PNSym.f-diagram}
\textbf{(a)}}\right) \\
&  =\left(  \mathfrak{S}_{n}\text{-invariant part of }\mathbf{k}\left[
\operatorname*{SM}\right]  \right)  \ \ \ \ \ \ \ \ \ \ \left(  \text{by
Theorem \ref{thm.PNSym.f-diagram} \textbf{(c)}}\right) \\
&  \cong\left(  \mathfrak{S}_{n}\text{-invariant part of }\mathbf{k}\left[
\operatorname*{J}\right]  \right)
\end{align*}
(since $g:\operatorname*{SM}\rightarrow\operatorname*{J}$ is an $\mathfrak{S}%
_{n}$-equivariant monoid isomorphism by Proposition \ref{prop.PNSym.g-inj},
and thus gives rise to an $\mathfrak{S}_{n}$-equivariant $\mathbf{k}$-algebra
isomorphism $\mathbf{k}\left[  \operatorname*{SM}\right]  \rightarrow
\mathbf{k}\left[  \operatorname*{J}\right]  $).
\end{proof}

\begin{fineprint}
Note that we have not used the associativity of the operation $\ast$ in our
proof of Corollary \ref{cor.PNSym.as-invariant-ring} (except in order to call
$\operatorname*{PNSym}\nolimits_{n}^{\left(  2\right)  }$ a $\mathbf{k}%
$-algebra; but we could have just as well called it a nonassociative
$\mathbf{k}$-algebra until we knew that it is associative). Thus, Corollary
\ref{cor.PNSym.as-invariant-ring} provides a new proof of the associativity of
the operation $\ast$ on $\operatorname*{PNSym}\nolimits_{n}$ (and thus on all
of $\operatorname*{PNSym}$).
\end{fineprint}

\subsubsection{The structure of $\operatorname*{PNSym}\nolimits_{n}^{\left(
2\right)  }$}

We shall next combine Theorem \ref{thm.PNSym.f-diagram} with some known
results about bands (\cite[Appendices A and B]{Brown04}) to uncover the
structure of the $\mathbf{k}$-algebra $\operatorname*{PNSym}\nolimits_{n}%
^{\left(  2\right)  }$.

We recall that a \emph{band} means a semigroup in which every element is
idempotent. It is easy to see that $\operatorname*{SC}$ and $\operatorname*{J}%
$ are bands.

Let us say some general words about finite bands first.

Let $S$ be any finite band. Then, a partial preorder (i.e., reflexive and
transitive relation) $\succcurlyeq$ is defined on $S$ by setting%
\[
\left(  x\succcurlyeq y\right)  \ \Longleftrightarrow\ \left(  x=xyx\right)
.
\]
Like any partial preorder, this preorder $\succcurlyeq$ induces an equivalence
relation $\sim$ defined by%
\[
\left(  x\sim y\right)  \ \Longleftrightarrow\ \left(  x\succcurlyeq y\text{
and }y\succcurlyeq x\right)  .
\]
The equivalence class of an element $x\in S$ is called the \emph{support} of
$x$, denoted $\operatorname*{supp}x$. The set of all $\operatorname*{supp}x$
with $x\in S$ is called the \emph{semilattice of supports} of $S$, and is
denoted $L$. It is a partially ordered set (since the preorder $\succcurlyeq$
descends to a partial order on $L$), and in fact a semilattice (by
\cite[Theorem A.11]{Brown04}). The join operation of this semilattice does in
fact descend from the multiplication on $S$: that is, we have $\left(
\operatorname*{supp}x\right)  \vee\left(  \operatorname*{supp}y\right)
=\operatorname*{supp}\left(  xy\right)  $ for all $x,y\in S$ (by
\cite[(A.3)]{Brown04}). Hence, $L$ (equipped with the join operation) is a
semigroup, and the map $\operatorname*{supp}:S\rightarrow L$ is a surjective
semigroup morphism. By linearizing $\operatorname*{supp}$, we obtain a
surjective non-unital $\mathbf{k}$-algebra morphism $\operatorname*{supp}%
:\mathbf{k}\left[  S\right]  \rightarrow\mathbf{k}\left[  L\right]  $. (In all
cases we are interested in, $S$ is actually a monoid, so all these
$\mathbf{k}$-algebras and morphisms are unital, but the theory works just as
well for proper semigroups.)

The monoid algebra $\mathbf{k}\left[  L\right]  $ is isomorphic to the
direct-power algebra $\mathbf{k}^{L}$ via a $\mathbf{k}$-algebra isomorphism
$\phi:\mathbf{k}\left[  L\right]  \rightarrow\mathbf{k}^{L}$ (see
\cite[\S B.1]{Brown04}). Composing this isomorphism $\phi$ with
$\operatorname*{supp}:\mathbf{k}\left[  S\right]  \rightarrow\mathbf{k}\left[
L\right]  $, we obtain a surjective non-unital $\mathbf{k}$-algebra morphism
$\psi=\phi\circ\operatorname*{supp}:\mathbf{k}\left[  S\right]  \rightarrow
\mathbf{k}^{L}$. It is shown in \cite[Theorem B.1]{Brown04} that the kernel
$\operatorname*{Ker}\psi$ is a nilpotent ideal of $\mathbf{k}\left[  S\right]
$. If $\mathbf{k}$ is a field, this entails that the $\mathbf{k}$-algebra
$\mathbf{k}\left[  S\right]  $ has Jacobson radical (and nilradical) equal to
$\operatorname*{Ker}\psi$ and semisimple quotient isomorphic to $\mathbf{k}%
^{L}$ (since $\mathbf{k}\left[  S\right]  /\operatorname*{Ker}\psi
\cong\operatorname{Im}\psi=\mathbf{k}^{L}$ is clearly semisimple). When $S$ is
a monoid, all these algebras and morphisms are unital.

Applying all this to $S=\operatorname*{SC}$, we recover classical properties
of the face algebra $\mathbf{k}\left[  \operatorname*{SC}\right]  $
(\cite[\S 2.3.3]{Bidiga97}), once we understand what the semilattice of
supports $L$ is. Namely, if $S=\operatorname*{SC}$, then the equivalence
relation $\sim$ on $S$ we have defined above turns out to be precisely the
support-equivalence relation for set compositions. Thus, the semilattice $L$
in this case is naturally isomorphic to the (reversed)\footnote{The reversal
is annoying but necessary: The join operation $\vee$ on the semilattice $S$ is
the meet operation $\wedge$ in our above terminology.} lattice of partitions
of the set $\left[  n\right]  $; the support map $\operatorname*{supp}%
:\operatorname*{SC}\rightarrow L$ essentially turns a set composition into a
set partition by forgetting the order of the blocks. Using the identification
of the descent algebra with the $\mathfrak{S}_{n}$-invariant part of
$\mathbf{k}\left[  \operatorname*{SC}\right]  $, we can furthermore derive
structural properties of the descent algebra from this (\cite[Theorem 3 in the
type-A case]{Solomo76}, see also \cite[\S 3.8.5]{Bidiga97} and \cite[Theorem
1.1]{GarReu89}).

It is not much harder to analyze the case of $S=\operatorname*{J}$ in the same
way. The relevant lemma is the following (easy proof left to the reader):

\begin{lemma}
Let $\left(  \mathbf{A},\mathbf{A}^{\prime}\right)  $ and $\left(
\mathbf{B},\mathbf{B}^{\prime}\right)  $ be two set bicompositions. Then,
$\operatorname*{supp}\left(  \mathbf{A},\mathbf{A}^{\prime}\right)
=\operatorname*{supp}\left(  \mathbf{B},\mathbf{B}^{\prime}\right)  $ if and
only if $\operatorname*{supp}\mathbf{A}=\operatorname*{supp}\mathbf{B}%
=\operatorname*{supp}\mathbf{A}^{\prime}=\operatorname*{supp}\mathbf{B}%
^{\prime}$.
\end{lemma}

Thus, the semilattice of supports $L$ for $S=\operatorname*{J}$ is the same as
for $S=\operatorname*{SC}$, namely the (reversed) lattice of partitions of
$\left[  n\right]  $. The above-mentioned general results about bands then
apply again. We obtain the following:

\begin{theorem}
\label{thm.PNSym.nil}\textbf{(a)} There is a surjective $\mathbf{k}$-algebra
morphism $\psi:\mathbf{k}\left[  \operatorname*{J}\right]  \rightarrow
\mathbf{k}^{L}$, where $L$ is the set of all set partitions of $\left[
n\right]  $. Its kernel $\operatorname*{Ker}\psi$ is nilpotent. Furthermore,
this morphism $\psi$ is $\mathfrak{S}_{n}$-equivariant (i.e., respects the
$\mathfrak{S}_{n}$-module structures on $\mathbf{k}\left[  \operatorname*{J}%
\right]  $ and $\mathbf{k}^{L}$). \medskip

\textbf{(b)} Now assume that $n!$ is invertible in $\mathbf{k}$. Then, there
is a surjective $\mathbf{k}$-algebra morphism $\psi^{\prime}%
:\operatorname*{PNSym}\nolimits_{n}^{\left(  2\right)  }\rightarrow
\mathbf{k}^{P}$, where $P$ is the set of all partitions of $n$. Its kernel
$\operatorname*{Ker}\left(  \psi^{\prime}\right)  $ is nilpotent.
\end{theorem}

\begin{proof}
[Proof idea.]\textbf{(a)} All of this follows from the discussion above the
theorem, except for the last claim about $\mathfrak{S}_{n}$-equivariance,
which is obvious. \medskip

\textbf{(b)} We note the following general fact:

\begin{statement}
\textit{Claim 1:} Let $f:V\rightarrow W$ be a surjective morphism of
representations of $\mathfrak{S}_{n}$. Then, the restriction of $f$ to the
$\mathfrak{S}_{n}$-invariant part of $V$ is a surjective map onto the
$\mathfrak{S}_{n}$-invariant part of $W$.
\end{statement}

\begin{proof}
[Proof of Claim 1.]Let $w$ belong to the $\mathfrak{S}_{n}$-invariant part of
$W$. Then, $w=f\left(  v\right)  $ for some $v\in V$ (since $f$ is
surjective). Consider this $v$, and set $v^{\prime}:=\dfrac{1}{n!}%
\sum\limits_{\sigma\in\mathfrak{S}_{n}}\sigma v$ (which is well-defined, since
$n!$ is invertible in $\mathbf{k}$). Then, $v^{\prime}$ lies in the
$\mathfrak{S}_{n}$-invariant part of $V$, and satisfies
\begin{align*}
f\left(  v^{\prime}\right)   &  =\dfrac{1}{n!}\sum\limits_{\sigma
\in\mathfrak{S}_{n}}\underbrace{f\left(  \sigma v\right)  }_{=\sigma f\left(
v\right)  }=\dfrac{1}{n!}\sum\limits_{\sigma\in\mathfrak{S}_{n}}%
\sigma\underbrace{f\left(  v\right)  }_{=w}\\
&  =\dfrac{1}{n!}\sum\limits_{\sigma\in\mathfrak{S}_{n}}\underbrace{\sigma
w}_{\substack{=w\\\text{(since }w\text{ is }\mathfrak{S}_{n}\text{-invariant)}%
}}=\dfrac{1}{n!}\underbrace{\sum\limits_{\sigma\in\mathfrak{S}_{n}}%
w}_{\substack{=\left\vert \mathfrak{S}_{n}\right\vert w\\=n!w}}=\dfrac{1}%
{n!}n!w=w.
\end{align*}
This shows that $w$ is an image under the restriction of $f$ to the
$\mathfrak{S}_{n}$-invariant part of $V$. This proves the desired
surjectivity, i.e., Claim 1.
\end{proof}

Now, the morphism $\psi$ from part \textbf{(a)} is $\mathfrak{S}_{n}%
$-equivariant, hence is a morphism of representations of $\mathfrak{S}_{n}$.
Thus, it restricts to a morphism $\psi^{\prime}$ from the $\mathfrak{S}_{n}%
$-invariant part of $\mathbf{k}\left[  \operatorname*{J}\right]  $ to the
$\mathfrak{S}_{n}$-invariant part of $\mathbf{k}^{L}$. This restriction
$\psi^{\prime}$ must again be surjective (by Claim 1, since $\psi$ is
surjective). Thus, $\psi^{\prime}$ is a surjective $\mathbf{k}$-algebra
morphism from the $\mathfrak{S}_{n}$-invariant part of $\mathbf{k}\left[
\operatorname*{J}\right]  $ to the $\mathfrak{S}_{n}$-invariant part of
$\mathbf{k}^{L}$.

But the former part is isomorphic to $\operatorname*{PNSym}\nolimits_{n}%
^{\left(  2\right)  }$ (by Corollary \ref{cor.PNSym.as-invariant-ring}),
whereas the latter part is isomorphic to $\mathbf{k}^{P}$ (since it is spanned
by the orbit sums in $\mathbf{k}^{L}$, but there are clearly $\left\vert
P\right\vert $ such orbit sums -- one for each partition of $n$ -- and these
orbit sums are orthogonal idempotents). Thus, Theorem \ref{thm.PNSym.nil}
\textbf{(b)} follows.
\end{proof}

When $\mathbf{k}$ is a field of characteristic $0$, Theorem
\ref{thm.PNSym.nil} \textbf{(b)} shows that the kernel $\operatorname*{Ker}%
\left(  \psi^{\prime}\right)  $ is the Jacobson radical (and the nilradical)
of $\operatorname*{PNSym}\nolimits_{n}^{\left(  2\right)  }$, while
$\mathbf{k}^{P}$ is the semisimple quotient of $\operatorname*{PNSym}%
\nolimits_{n}^{\left(  2\right)  }$. Note that $\mathbf{k}^{P}$ can also be
interpreted as the center of the symmetric group algebra $\mathbf{k}\left[
\mathfrak{S}_{n}\right]  $ when $n!$ is invertible in $\mathbf{k}$.

\begin{remark}
As we pointed out in Remark \ref{rmk.PNSym.AM1}, our Janus monoid
$\operatorname*{J}$ is isomorphic to the Janus monoid $\mathrm{J}\left[
\mathcal{A}\right]  $ of the type-A braid arrangement, as defined in
\cite[\S 1.9.3]{AguMah20}. Thus, our Janus algebra $\mathbf{k}\left[
\operatorname*{J}\right]  $ is isomorphic to the Janus algebra $\mathsf{J}%
\left[  \mathcal{A}\right]  $ defined loc. cit.. Hence, our $\mathbf{k}%
$-algebra $\operatorname*{PNSym}\nolimits_{n}^{\left(  2\right)  }$ is
isomorphic to the $\mathfrak{S}_{n}$-invariant part of the latter. This
invariant part has also been studied by Aguiar and Mahajan \cite[\S 1.12.3]%
{AguMah22} (who have even introduced a $q$-deformation thereof). While their
focus is different from ours, there are some overlaps in the results obtained.
In particular, the $\operatorname*{PNSym}\nolimits_{n}^{\left(  2\right)
}\cong\left(  \operatorname*{PNSym}\nolimits_{n}^{\left(  2\right)  }\right)
^{\operatorname*{op}}$ part of Proposition \ref{prop.PNSym.self-opp} appears
in \cite{AguMah22}, being entirely obvious from the Janus viewpoint (the
isomorphism corresponds to sending $\left(  \mathbf{A},\mathbf{B}\right)
\mapsto\left(  \mathbf{B},\mathbf{A}\right)  $ in the Janus monoid), and parts
of Theorem \ref{thm.PNSym.nil} \textbf{(a)} appear in \cite[\S 1.12.4]%
{AguMah22}. Finally, \cite[Proposition 11.6]{AguMah20} is an analogue of our
Theorem \ref{thm.PNSym.mods} for $\mathcal{A}$-bimonoids instead of connected
graded bialgebras.

Aguiar and Mahajan mostly study a single hyperplane arrangement $\mathcal{A}$
in isolation, so they do not appear to get any results about the external
multiplication and the coproduct on $\operatorname*{PNSym}$. However, it might
be possible to extend their theory to towers of arrangements, and interpret
these operations in geometric terms as well.
\end{remark}

\begin{noncompile}
OLD version:

We finish with some cryptic remarks connecting $\operatorname*{PNSym}$ with
Aguiar's and Mahajan's theory of Hopf monoids.

In \cite[\S 1.9.3]{AguMah20}, Aguiar and Mahajan define the \emph{Janus
algebra} $\mathrm{J}\left[  \mathcal{A}\right]  $ of a hyperplane arrangement
$\mathcal{A}$. When $\mathcal{A}$ is the braid arrangement in $\mathbb{R}^{n}%
$, this Janus algebra $\mathrm{J}\left[  \mathcal{A}\right]  $ has a basis
indexed by pairs $\left(  U,V\right)  $, where $U$ and $V$ are two set
compositions (i.e., ordered set partitions) of $\left[  n\right]  $ that
define the same (unordered) set partition (i.e., the blocks of $U$ are the
blocks of $V$, perhaps in a different order)\footnote{In the language of
hyperplane arrangements, set compositions are called \emph{faces}, and our
pairs $\left(  U,V\right)  $ of set compositions are called \emph{bifaces}.
Thus the name \textquotedblleft Janus algebra\textquotedblright.}. Such pairs
$\left(  U,V\right)  $ can equivalently be viewed as pairs $\left(
U,\sigma\right)  $, where $U$ is a set composition of $\left[  n\right]  $
into $k$ blocks (for some $k\in\mathbb{N}$), and where $\sigma\in
\mathfrak{S}_{k}$ is a permutation (the one that permutes the blocks of $U$ to
obtain $V$). However, viewing them as pairs $\left(  U,V\right)  $ exposes a
symmetry that would be hidden in the $\left(  U,\sigma\right)  $ formulation.
Note that \cite[Proposition 11.6]{AguMah20} is an analogue of our Theorem
\ref{thm.PNSym.mods} for $\mathcal{A}$-bimonoids instead of connected graded bialgebras.

Let us continue with the case when $\mathcal{A}$ is the braid arrangement in
$\mathbb{R}^{n}$. The symmetric group $\mathfrak{S}_{n}$ acts on the Janus
algebra $\mathrm{J}\left[  \mathcal{A}\right]  $ by the rule $\tau\cdot\left(
U,V\right)  =\left(  \tau\left(  U\right)  ,\tau\left(  V\right)  \right)  $
(for any $\tau\in\mathfrak{S}_{n}$), or (in the $\left(  U,\sigma\right)  $
formulation) by the rule $\tau\cdot\left(  U,\sigma\right)  =\left(
\tau\left(  U\right)  ,\sigma\right)  $. The invariant space (more precisely,
$\mathbf{k}$-module) $\mathrm{J}\left[  \mathcal{A}\right]  ^{\mathfrak{S}%
_{n}}$ under this action is a subalgebra of the Janus algebra $\mathrm{J}%
\left[  \mathcal{A}\right]  $, called the \emph{invariant Janus algebra}. It
has a basis formed by the orbit sums, which are in bijection with the
mopiscotions $\left(  \alpha,\sigma\right)  $ satisfying $\left\vert
\alpha\right\vert =n$ (since an orbit of a set composition of $\left[
n\right]  $ under the symmetric group $\mathfrak{S}_{n}$ is essentially an
integer composition of $n$, recording the sizes of the blocks).

Is $\mathrm{J}\left[  \mathcal{A}\right]  ^{\mathfrak{S}_{n}}$ isomorphic (as
an algebra) to the $n$-th graded component of $\operatorname*{PNSym}$ equipped
with the internal multiplication $\ast$ ? Yes! If so, then it stands to reason
that results by Aguiar and Mahajan should be translatable into properties of
$\operatorname*{PNSym}$, at least as far as the internal multiplication is
concerned\footnote{The external multiplication and the coproduct are less
likely to arise in this way, since they relate different graded components and
thus should correspond not to a single braid arrangement but rather a whole
family thereof.}. The invariant Janus algebra $\mathrm{J}\left[
\mathcal{A}\right]  ^{\mathfrak{S}_{n}}$ (and even a $q$-deformation thereof)
is studied by Aguiar and Mahajan in \cite[\S 1.12.3]{AguMah22}, although
establishing a dictionary between the geometric language of \cite{AguMah22}
and our combinatorial one is a rather daunting task.
\end{noncompile}

\subsection{Questions}

Much remains to be understood about $\operatorname*{PNSym}$. Many
combinatorial Hopf algebras can be embedded into algebras of noncommutative
formal power series (i.e., completions of free algebras) such that the
comultiplication $\Delta$ is an \textquotedblleft alphabet
doubling\textquotedblright\ operation. For instance, $\operatorname*{NSym}$
has such an embedding\footnote{This is explained, e.g., in \cite[\S 8.1]%
{GriRei}: Namely, \cite[Corollary 8.1.14(b)]{GriRei} embeds
$\operatorname*{NSym}$ in the Hopf algebra $\operatorname*{FQSym}$, whereas
\cite[(8.1.3)]{GriRei} embeds $\operatorname*{FQSym}$ in the algebra
$\mathbf{k}\left\langle \left\langle X_{1},X_{2},X_{3},\ldots\right\rangle
\right\rangle $ of noncommutative formal power series.}.

\begin{question}
Does $\operatorname*{PNSym}$ have such an embedding as well?
\end{question}

We can ask about some other features that certain combinatorial Hopf algebras have:

\begin{question}
Does $\operatorname*{PNSym}$ have a categorification (e.g., a presentation as
$K_{0}$ of some category)?
\end{question}

\begin{question}
Is there a cancellation-free formula for the antipode of
$\operatorname*{PNSym}$ ?
\end{question}

\begin{question}
Is $\operatorname*{PNSym}$ isomorphic to the dual of the commutative
combinatorial Hopf algebra arising from \cite[\S 3.8.3]{HiNoTh06}?
\end{question}

\begin{question}
What are the primitive elements of $\operatorname*{PNSym}$ ? We recall that
the primitive elements of $\operatorname*{NSym}$ correspond to the \emph{Lie
quasi-idempotents} in the descent algebra (\cite[Corollary 5.17]{ncsf1}), and
include several renowned families, such as the noncommutative power sums of
the first two kinds (\cite[\S 5.2]{ncsf1}, \cite[Exercises 5.4.5 and
5.4.12]{GriRei}) and the third kind (\cite[Definition 5.26]{ncsf2}) and the
Klyachko elements (\cite[\S 6.2]{ncsf2}). What is the meaning of primitive
elements of $\operatorname*{PNSym}$ ? Do some of them act as quasi-idempotents
on every connected graded bialgebra $H$ ?
\end{question}

\begin{question}
Is the dual of $\operatorname*{PNSym}$ a polynomial ring? (For comparison, the
dual of $\operatorname*{NSym}$ is $\operatorname*{QSym}$, which is a
polynomial ring by a result of Hazewinkel \cite[Theorem 6.4.3]{GriRei}. Note
that the dual of $\operatorname*{PNSym}$ is a commutative connected graded
$\mathbf{k}$-bialgebra, and thus a polynomial ring whenever $\mathbf{k}$ is a
field of characteristic $0$ by Leray's theorem \cite[Remark 1.7.30
(a)]{GriRei}. The interesting question is what happens for arbitrary
$\mathbf{k}$.)
\end{question}

\begin{question}
Let us rename the Janus monoid $\operatorname*{J}$ as $\operatorname*{J}%
\nolimits_{n}$ to stress its dependence on $n$. Can the Janus algebras
$\mathbf{k}\left[  \operatorname*{J}\nolimits_{n}\right]  $ for varying
$n\in\mathbb{N}$ be combined into a combinatorial Hopf algebra on the
$\mathbf{k}$-module $\bigoplus_{n\in\mathbb{N}}\mathbf{k}\left[
\operatorname*{J}\nolimits_{n}\right]  $ equipped with a second
multiplication, extending $\operatorname*{PNSym}$? (The second multiplication
would be coming from the $\mathbf{k}\left[  \operatorname*{J}\nolimits_{n}%
\right]  $'s, while the first multiplication and the comultiplication would
extend those on $\operatorname*{PNSym}$.) Such a Hopf algebra would probably
be a Janus analogue of the $NCQSym^{\ast}$ from \cite[\S 5.2]{BerZab09}.
\end{question}

Sarah Brauner suggests another question: It is known that for any
$n\in\mathbb{N}$, the span of the elements $\sum_{\substack{\alpha\text{ is a
composition of }n;\\\ell\left(  \alpha\right)  =k}}\mathbf{H}_{\alpha}$ over
all $k\in\left\{  0,1,\ldots,n\right\}  $ is a commutative subalgebra of the
$\mathbf{k}$-algebra $\operatorname*{NSym}\nolimits_{n}^{\left(  2\right)  }$;
this subalgebra is known as the \emph{Eulerian subalgebra} (this appears,
e.g., in \cite[\S 4.2]{Schocker} in the language of the descent algebra). Is
there an analogous \textquotedblleft Eulerian\textquotedblright\ subalgebra of
$\operatorname*{PNSym}\nolimits_{n}^{\left(  2\right)  }$ ? The following
appears to work for small $n$:

\begin{question}
Fix $n\in\mathbb{N}$, and set $\mathfrak{e}_{k,\sigma}:=\sum_{\substack{\alpha
\text{ is a composition of }n;\\\ell\left(  \alpha\right)  =k}}F_{\alpha
,\sigma}\in\operatorname*{PNSym}\nolimits_{n}$ for each $k\in\mathbb{N}$ and
each $\sigma\in\mathfrak{S}_{k}$. Is the span of these $\mathfrak{e}%
_{k,\sigma}$ a $\mathbf{k}$-subalgebra of $\operatorname*{PNSym}%
\nolimits_{n}^{\left(  2\right)  }$ (that is, closed under $\ast$)?
\end{question}

The answer is positive for all $n\leq4$, but this subalgebra is not
commutative for $n=3$ already.

\section{\label{sec.idcheck}Application to identity checking}

In this last section, we shall outline a way in which the above results can be
used. For an example, let us prove that every connected graded bialgebra $H$
satisfies%
\begin{equation}
\left(  p_{1}\star p_{2}-p_{2}\star p_{1}\right)  ^{5}=0
\label{eq.idcheck.p12}%
\end{equation}
(where the $5$-th power is taken with respect to the composition $\circ$). To
prove this equality, we can rewrite the left-hand side as $\left(  p_{\left(
1,2\right)  ,\operatorname*{id}}-p_{\left(  2,1\right)  ,\operatorname*{id}%
}\right)  ^{5}$ (where $\operatorname*{id}$ is the identity permutation in
$\mathfrak{S}_{2}$), and expand this (using Theorem \ref{thm.sol-mac-1} and
Proposition \ref{prop.sol-mac-0}) as a $\mathbb{Z}$-linear combination of
$p_{\alpha,\sigma}$'s. According to Theorem \ref{thm.pas-lin-ind}, for the
above equality (\ref{eq.idcheck.p12}) to hold, this $\mathbb{Z}$-linear
combination has to be trivial (i.e., all coefficients must be zeroes), which
we can directly check. Alternatively, using Theorem \ref{thm.PNSym.mods}, we
can reduce the equality (\ref{eq.idcheck.p12}) to the equality%
\[
\left(  F_{\left(  1,2\right)  ,\operatorname*{id}}-F_{\left(  2,1\right)
,\operatorname*{id}}\right)  ^{\ast5}=0\ \ \ \ \ \ \ \ \ \ \text{in
}\operatorname*{PNSym}%
\]
(where the \textquotedblleft$\ast5$\textquotedblright\ exponent means a $5$-th
power with respect to the internal product $\ast$). This equality is
straightforward to check using the definition of $\ast$.

Likewise, any identity such as (\ref{eq.idcheck.p12}) (that is, any identity
whose both sides are formed from the maps $p_{\alpha,\sigma}$ by addition,
convolution and composition) can be proved mechanically using computations
inside $\operatorname*{PNSym}$. This approach can be useful even if we don't
end up making the computations. For example, it shows that when proving an
identity such as (\ref{eq.idcheck.p12}) (with integer coefficients), it
suffices to prove it for $\mathbf{k}=\mathbb{Q}$ (since the Hopf algebra
$\operatorname*{PNSym}$ with all its structures is defined over $\mathbb{Z}$
and is free as a $\mathbf{k}$-module, so that its $\mathbb{Z}$-version embeds
naturally into its $\mathbb{Q}$-version). Moreover, it suffices to prove it in
the case when all graded components $H_{n}$ of $H$ are finite-dimensional
$\mathbf{k}$-vector spaces (thanks to Theorem \ref{thm.pas-lin-ind}
\textbf{(b)}). This makes a number of methods available (graded duals,
coradical filtration\footnote{Both of these methods have been used in proving
such identities (see, e.g., \cite{AguLau14}, \cite[\S 3]{Pang18}).}) that
could not be used in the a-priori generality of a connected graded bialgebra
over an arbitrary commutative ring. While this is perhaps not very unexpected,
it is reassuring and helpful.

This all can be applied to a slightly wider class of identities. Indeed, while
in (\ref{eq.idcheck.p12}) we don't allow the use of the identity map
$\operatorname*{id}\nolimits_{H}:H\rightarrow H$, we can actually express
$\operatorname*{id}\nolimits_{H}$ as the infinite sum $p_{0}+p_{1}%
+p_{2}+\cdots$. Infinite sums are not allowed in (\ref{eq.idcheck.p12}), but
we can replace them by finite partial sums if we restrict ourselves to the
submodule $H_{0}+H_{1}+\cdots+H_{k}$ of $H$. For example, if we want to prove
the equality%
\[
\left(  p_{1}\star\operatorname*{id}\nolimits_{H}-\,2\operatorname*{id}%
\nolimits_{H}\right)  \circ\left(  p_{1}\star\operatorname*{id}\nolimits_{H}%
\right)  ^{2}=0\ \ \ \ \ \ \ \ \ \ \text{on }H_{2}%
\]
for any connected graded bialgebra $H$ (this identity is part of \cite[Theorem
1.4 \textbf{(b)}]{aphae}), then we can replace each $\operatorname*{id}%
\nolimits_{H}$ by $p_{0}+p_{1}+p_{2}$, which renders the equality amenable to
our $\operatorname*{PNSym}$-method. Likewise, the antipode $S$ of a connected
graded Hopf algebra $H$ is not directly of a form supported by our method, but
can be written as an infinite sum%
\[
\sum_{k\in\mathbb{N}}\left(  -1\right)  ^{k}\left(  p_{1}+p_{2}+p_{3}%
+\cdots\right)  ^{\star k}=\sum_{k\in\mathbb{N}}\left(  -1\right)  ^{k}%
\sum_{\alpha\in\mathbb{N}^{k}\text{ is a composition}}p_{\alpha
,\operatorname*{id}},
\]
which -- when restricted to a submodule $H_{0}+H_{1}+\cdots+H_{k}$ -- can be
replaced by a finite partial sum. Thus, claims such as Aguiar's and Lauve's
result \textquotedblleft$\left(  S^{2}-\operatorname*{id}\right)  ^{k}=0$ on
$H_{k}$ for any connected graded Hopf algebra $H$\textquotedblright\ (see
\cite[Proposition 7]{AguLau14}) can (for a fixed $k\in\mathbb{N}$) be checked
using the $\operatorname*{PNSym}$-method (although this particular claim has
been extended and proved in a much more elementary way in
\cite{antipode-squared}). I believe that the field of Hopf algebra identities
is vast and mostly unexplored so far. \medskip

Using Theorem \ref{thm.pas-lin-ind2}, we can obtain a similar method for
checking identities for general (not just connected) graded bialgebras. Such
identities can be viewed as equalities in the bialgebra
$\operatorname*{WPNSym}\nolimits^{\prime}$ (constructed in our second proof of
Theorem \ref{thm.PNSym.alg}) instead of $\operatorname*{PNSym}$. Incidentally,
this shows that any such identity that holds for all graded $\mathbf{k}$-Hopf
algebras will necessarily hold for all graded $\mathbf{k}$-bialgebras as well.

\begin{question}
Is there a similar method for checking identities between maps $H^{\otimes
k}\rightarrow H^{\otimes\ell}$ as opposed to only maps $H\rightarrow H$ ?
\end{question}

The same problem without the grading has been solved by categorical methods in
Pirashvili's 2002 paper \cite{Pirash02}. \medskip

We end this section with an open question that generalizes
(\ref{eq.idcheck.p12}):

\begin{question}
Let $i,j\in\mathbb{N}$. What is the smallest integer $k\left(  i,j\right)
\in\mathbb{N}$ for which every connected graded bialgebra $H$ satisfies%
\[
\left(  p_{i}\star p_{j}-p_{j}\star p_{i}\right)  ^{k\left(  i,j\right)  }=0
\]
(where the power is with respect to composition)? In other words, what is the
smallest integer $k\left(  i,j\right)  \in\mathbb{N}$ for which
\[
\left(  F_{\left(  i,j\right)  ,\operatorname*{id}}-F_{\left(  j,i\right)
,\operatorname*{id}}\right)  ^{\ast k\left(  i,j\right)  }%
=0\ \ \ \ \ \ \ \ \ \ \text{in }\operatorname*{PNSym}\nolimits^{\left(
2\right)  }\text{ ?}%
\]

\end{question}

Some known values (computed using SageMath) are%
\begin{align*}
k\left(  0,j\right)   &  =1,\ \ \ \ \ \ \ \ \ \ k\left(  i,i\right)
=1,\ \ \ \ \ \ \ \ \ \ k\left(  1,2\right)  =5,\ \ \ \ \ \ \ \ \ \ k\left(
1,3\right)  =7,\\
k\left(  1,4\right)   &  =9,\ \ \ \ \ \ \ \ \ \ k\left(  1,5\right)
=11,\ \ \ \ \ \ \ \ \ \ k\left(  1,6\right)  =13,\ \ \ \ \ \ \ \ \ \ k\left(
2,3\right)  =9,\\
k\left(  2,4\right)   &  =9.
\end{align*}
(Of course, $k\left(  i,j\right)  =k\left(  j,i\right)  $ for all $i$ and
$j$.) It is striking that all these known values are odd.

\begin{noncompile}
In this last section, we shall outline two ways in which the above results can
be used. The first application -- known at least since Reutenauer's
\cite{Reuten93} -- is a proof of Solomon's Mackey formula for the symmetric
group; we briefly recall it here not to falsely claim any novelty, but to
suggest possible variants. The second -- which motivated most of this work --
is a method of checking identities in Hopf algebras.
\end{noncompile}

\appendix

\section{\label{sec.pas-proof}Appendix: Rigorous proof of Proposition
\ref{prop.mopis.reduce}}

\begin{fineprint}
We proved Proposition \ref{prop.mopis.reduce} using combinatorial handwaving;
let us now give a rigorous (if rather cumbersome and tiring) proof. Arguably,
this is just the above informal proof, rewritten without Sweedler notation and
handwaving, but the rewrite has turned out to be surprisingly nontrivial.

\begin{proof}
[Rigorous proof of Proposition \ref{prop.mopis.reduce}.]Forget that we fixed
$\left(  \alpha,\sigma\right)  $ and $\left(  \beta,\tau\right)  $.

In the following, the \textquotedblleft hat\textquotedblright%
\ symbol\ $\ \widehat{}$\ \ over an entry of a list will signify that this
entry must be omitted from the list. For instance, \textquotedblleft$\left(
\alpha_{1},\alpha_{2},\ldots,\widehat{\alpha_{q}},\ldots,\alpha_{k}\right)
$\textquotedblright\ denotes the list $\left(  \alpha_{1},\alpha_{2}%
,\ldots,\alpha_{k}\right)  $ without its $q$-th entry.

For each weak mopiscotion $\left(  \alpha,\sigma\right)  $ with $\alpha
=\left(  \alpha_{1},\alpha_{2},\ldots,\alpha_{k}\right)  \in\mathbb{N}^{k}$
and $\sigma\in\mathfrak{S}_{k}$, and for each index $q\in\left[  k\right]  $
that satisfies $\alpha_{q}=0$, we define a weak mopiscotion
$\operatorname*{del}\left(  \alpha,\sigma,q\right)  $ as follows:

\begin{itemize}
\item Let $\alpha^{\prime}$ denote the weak composition $\left(  \alpha
_{1},\alpha_{2},\ldots,\widehat{\alpha_{q}},\ldots,\alpha_{k}\right)  $. This
is the $\left(  k-1\right)  $-tuple obtained from $\alpha$ by removing the
$q$-th entry, which is $\alpha_{q}=0$. Note that any $0$'s that appear in
$\alpha$ in positions other than the $q$-th are retained in $\alpha^{\prime}$.

\item Let $\sigma^{\prime}\in\mathfrak{S}_{k-1}$ be the standardization of the
list $\left(  \sigma\left(  1\right)  ,\sigma\left(  2\right)  ,\ldots
,\widehat{\sigma\left(  q\right)  },\ldots,\sigma\left(  k\right)  \right)  $.

\item Then, we define $\operatorname*{del}\left(  \alpha,\sigma,q\right)  $ to
be the weak mopiscotion $\left(  \alpha^{\prime},\sigma^{\prime}\right)  $.
\end{itemize}

We call $\operatorname*{del}\left(  \alpha,\sigma,q\right)  $ the
\emph{deletion} of $q$ from $\left(  \alpha,\sigma\right)  $.

We can view this deletion as a single step from the weak mopiscotion $\left(
\alpha,\sigma\right)  $ to its reduction $\operatorname*{red}\left(
\alpha,\sigma\right)  $. The formal reason for this is the following lemma:

\begin{statement}
\textit{Claim 1:} Let $\left(  \alpha,\sigma\right)  $ be a weak mopiscotion
with $\alpha=\left(  \alpha_{1},\alpha_{2},\ldots,\alpha_{k}\right)
\in\mathbb{N}^{k}$. Let $q\in\left[  k\right]  $ be such that $\alpha_{q}=0$.
Let $\left(  \alpha^{\prime},\sigma^{\prime}\right)  :=\operatorname*{del}%
\left(  \alpha,\sigma,q\right)  $. Then,%
\[
\operatorname*{red}\left(  \alpha,\sigma\right)  =\operatorname*{red}\left(
\alpha^{\prime},\sigma^{\prime}\right)  .
\]
In other words, deleting $q$ from $\left(  \alpha,\sigma\right)  $ and then
reducing the resulting weak mopiscotion produces the same result as reducing
$\left(  \alpha,\sigma\right)  $ right away.
\end{statement}

\begin{proof}
[Proof of Claim 1.]This is one of the kind of proofs that give combinatorics a
bad name; but we have little choice. Set
\[
\left(  \beta,\tau\right)  :=\operatorname*{red}\left(  \alpha,\sigma\right)
\ \ \ \ \ \ \ \ \ \ \text{and}\ \ \ \ \ \ \ \ \ \ \left(  \gamma
,\omega\right)  :=\operatorname*{red}\left(  \alpha^{\prime},\sigma^{\prime
}\right)  .
\]
Thus, we must show that $\left(  \beta,\tau\right)  =\left(  \gamma
,\omega\right)  $.

Both compositions $\beta$ and $\gamma$ are obtained from $\alpha$ by removing
all zeroes (either all at once in the case of $\beta$, or starting with
$\alpha_{q}$ and then the rest in the case of $\gamma$). Hence, $\beta=\gamma
$. It remains to show that $\tau=\omega$.

\begin{noncompile}
Let $h\in\mathbb{N}$ be such that $\beta=\gamma\in\mathbb{N}^{h}$. Then,
$\tau,\omega\in\mathfrak{S}_{h}$ (since $\left(  \beta,\tau\right)  $ and
$\left(  \gamma,\omega\right)  $ are mopiscotions).
\end{noncompile}

Let $\left(  r_{1}<r_{2}<\cdots<r_{k-1}\right)  $ be the list of all elements
of $\left[  k\right]  \setminus\left\{  q\right\}  $, in increasing order.
Thus,%
\[
\left[  k\right]  \setminus\left\{  q\right\}  =\left\{  r_{1}<r_{2}%
<\cdots<r_{k-1}\right\}
\]
and%
\begin{equation}
\left(  1<2<\cdots<\widehat{q}<\cdots<k\right)  =\left(  r_{1}<r_{2}%
<\cdots<r_{k-1}\right)  \label{pf.prop.mopis.reduce.rilist}%
\end{equation}
(an equality of strictly increasing $\left(  k-1\right)  $-tuples).
Explicitly, we have $r_{i}=i+\left[  i\geq q\right]  $ for each $i\in\left[
k-1\right]  $, where we use the Iverson bracket notation ($\left[
\mathcal{A}\right]  $ denotes the truth value of a statement $\mathcal{A}$).

By the definition of $\operatorname*{del}\left(  \alpha,\sigma,q\right)  $, we
have $\alpha^{\prime}=\left(  \alpha_{1},\alpha_{2},\ldots,\widehat{\alpha
_{q}},\ldots,\alpha_{k}\right)  $. Let us also denote this $\left(
k-1\right)  $-tuple $\alpha^{\prime}$ as $\alpha^{\prime}=\left(  \alpha
_{1}^{\prime},\alpha_{2}^{\prime},\ldots,\alpha_{k-1}^{\prime}\right)  $.
Thus,%
\begin{align*}
\left(  \alpha_{1}^{\prime},\alpha_{2}^{\prime},\ldots,\alpha_{k-1}^{\prime
}\right)   &  =\alpha^{\prime}=\left(  \alpha_{1},\alpha_{2},\ldots
,\widehat{\alpha_{q}},\ldots,\alpha_{k}\right) \\
&  =\left(  \alpha_{r_{1}},\alpha_{r_{2}},\ldots,\alpha_{r_{k-1}}\right)
\ \ \ \ \ \ \ \ \ \ \left(  \text{by (\ref{pf.prop.mopis.reduce.rilist}%
)}\right)  .
\end{align*}
In other words,%
\begin{equation}
\alpha_{i}^{\prime}=\alpha_{r_{i}}\ \ \ \ \ \ \ \ \ \ \text{for each }%
i\in\left[  k-1\right]  . \label{pf.prop.mopis.reduce.aap}%
\end{equation}

Let $\left(  j_{1}<j_{2}<\cdots<j_{h}\right)  $ be the list of all elements
$i$ of $\left[  k-1\right]  $ satisfying $\alpha_{i}^{\prime}\neq0$, in
increasing order. Then, by the definition of a reduction, $\omega$ is the
standardization of $\left(  \sigma^{\prime}\left(  j_{1}\right)
,\sigma^{\prime}\left(  j_{2}\right)  ,\ldots,\sigma^{\prime}\left(
j_{h}\right)  \right)  $ (since $\left(  \gamma,\omega\right)
=\operatorname*{red}\left(  \alpha^{\prime},\sigma^{\prime}\right)  $). Hence,
by the definition of standardization, we know that for every two elements $a$
and $b$ of $\left[  h\right]  $ satisfying $a<b$, we have the equivalence%
\begin{equation}
\left(  \omega\left(  a\right)  <\omega\left(  b\right)  \right)
\ \Longleftrightarrow\ \left(  \sigma^{\prime}\left(  j_{a}\right)  \leq
\sigma^{\prime}\left(  j_{b}\right)  \right)  .
\label{pf.prop.mopis.reduce.equiv1}%
\end{equation}

Recall that $\left(  j_{1}<j_{2}<\cdots<j_{h}\right)  $ is the list of all
elements $i$ of $\left[  k-1\right]  $ satisfying $\alpha_{i}^{\prime}\neq0$.
Thus, the number of these elements is $h$. In other words, the $\left(
k-1\right)  $-tuple $\alpha^{\prime}$ has exactly $h$ nonzero entries. Hence,
the $k$-tuple $\alpha$ has exactly $h$ nonzero entries, too (since the
$\left(  k-1\right)  $-tuple $\alpha^{\prime}$ is obtained from $\alpha$ by
removing the zero entry $\alpha_{q}=0$, and thus has the same nonzero entries
as $\alpha$).

Next, we observe that the chain of inequalities $r_{j_{1}}<r_{j_{2}}%
<\cdots<r_{j_{h}}$ holds, since it is a sub-chain of $r_{1}<r_{2}%
<\cdots<r_{k-1}$ (because the subscripts satisfy $j_{1}<j_{2}<\cdots<j_{h}$).

Recall that $\left(  j_{1}<j_{2}<\cdots<j_{h}\right)  $ is the list of all
elements $i$ of $\left[  k-1\right]  $ satisfying $\alpha_{i}^{\prime}\neq0$.
Thus, for each $i\in\left\{  j_{1}<j_{2}<\cdots<j_{h}\right\}  $, we have
$\alpha_{i}^{\prime}\neq0$ and therefore $\alpha_{r_{i}}\neq0$ (since
(\ref{pf.prop.mopis.reduce.aap}) yields $\alpha_{r_{i}}=\alpha_{i}^{\prime
}\neq0$). In other words, the numbers $\alpha_{r_{j_{1}}},\alpha_{r_{j_{2}}%
},\ldots,\alpha_{r_{j_{h}}}$ are all nonzero. Thus, $\alpha_{r_{j_{1}}}%
,\alpha_{r_{j_{2}}},\ldots,\alpha_{r_{j_{h}}}$ are $h$ nonzero entries of
$\alpha$ (all occupying different positions in $\alpha$, because $r_{j_{1}%
}<r_{j_{2}}<\cdots<r_{j_{h}}$). There cannot be any further nonzero entries of
$\alpha$ besides these (since $\alpha$ has exactly $h$ nonzero entries). Thus,
$\alpha_{r_{j_{1}}},\alpha_{r_{j_{2}}},\ldots,\alpha_{r_{j_{h}}}$ are
\textbf{all} the nonzero entries of $\alpha$, listed from left to right (since
$r_{j_{1}}<r_{j_{2}}<\cdots<r_{j_{h}}$).

In other words, $\left(  r_{j_{1}}<r_{j_{2}}<\cdots<r_{j_{h}}\right)  $ is the
list of all elements $i$ of $\left[  k\right]  $ satisfying $\alpha_{i}\neq0$,
in increasing order.

Therefore, by the definition of reduction, $\tau$ is the standardization of
$\left(  \sigma\left(  r_{j_{1}}\right)  ,\sigma\left(  r_{j_{2}}\right)
,\ldots,\sigma\left(  r_{j_{h}}\right)  \right)  $ (since $\left(  \beta
,\tau\right)  =\operatorname*{red}\left(  \alpha,\sigma\right)  $). Hence, by
the definition of standardization, we know that for every two elements $a$ and
$b$ of $\left[  h\right]  $ satisfying $a<b$, we have the equivalence%
\begin{equation}
\left(  \tau\left(  a\right)  <\tau\left(  b\right)  \right)
\ \Longleftrightarrow\ \left(  \sigma\left(  r_{j_{a}}\right)  \leq
\sigma\left(  r_{j_{b}}\right)  \right)  . \label{pf.prop.mopis.reduce.equiv2}%
\end{equation}

Finally, recall that $\left(  \alpha^{\prime},\sigma^{\prime}\right)
=\operatorname*{del}\left(  \alpha,\sigma,q\right)  $. Hence, the definition
of deletion yields that $\sigma^{\prime}$ is the standardization of $\left(
\sigma\left(  1\right)  ,\sigma\left(  2\right)  ,\ldots,\widehat{\sigma
\left(  q\right)  },\ldots,\sigma\left(  k\right)  \right)  $. In other words,
$\sigma^{\prime}$ is the standardization of $\left(  \sigma\left(
r_{1}\right)  ,\sigma\left(  r_{2}\right)  ,\ldots,\sigma\left(
r_{k-1}\right)  \right)  $ (since (\ref{pf.prop.mopis.reduce.rilist}) shows
that $\left(  \sigma\left(  1\right)  ,\sigma\left(  2\right)  ,\ldots
,\widehat{\sigma\left(  q\right)  },\ldots,\sigma\left(  k\right)  \right)
=\left(  \sigma\left(  r_{1}\right)  ,\sigma\left(  r_{2}\right)
,\ldots,\sigma\left(  r_{k-1}\right)  \right)  $). Hence, for every two
elements $a$ and $b$ of $\left[  k-1\right]  $ satisfying $a<b$, we have the
equivalence%
\begin{equation}
\left(  \sigma^{\prime}\left(  a\right)  <\sigma^{\prime}\left(  b\right)
\right)  \ \Longleftrightarrow\ \left(  \sigma\left(  r_{a}\right)  \leq
\sigma\left(  r_{b}\right)  \right)  \label{pf.prop.mopis.reduce.equiv3}%
\end{equation}
(by the definition of standardization).

Now, for any two elements $a$ and $b$ of $\left[  h\right]  $ satisfying
$a<b$, we have the following chain of equivalences:%
\begin{align*}
\left(  \omega\left(  a\right)  <\omega\left(  b\right)  \right)  \  &
\Longleftrightarrow\ \left(  \sigma^{\prime}\left(  j_{a}\right)  \leq
\sigma^{\prime}\left(  j_{b}\right)  \right)  \ \ \ \ \ \ \ \ \ \ \left(
\text{by (\ref{pf.prop.mopis.reduce.equiv1})}\right) \\
&  \Longleftrightarrow\ \left(  \sigma^{\prime}\left(  j_{a}\right)
<\sigma^{\prime}\left(  j_{b}\right)  \right) \\
&  \ \ \ \ \ \ \ \ \ \ \ \ \ \ \ \ \ \ \ \ \left(
\begin{array}
[c]{c}%
\text{since }a<b\text{ entails }j_{a}<j_{b}\\
\text{(because }j_{1}<j_{2}<\cdots<j_{h}\text{) and}\\
\text{thus }j_{a}\neq j_{b}\text{ and thus }\sigma^{\prime}\left(
j_{a}\right)  \neq\sigma^{\prime}\left(  j_{b}\right) \\
\text{(since }\sigma^{\prime}\text{ is a permutation)}%
\end{array}
\right) \\
&  \Longleftrightarrow\ \left(  \sigma\left(  r_{j_{a}}\right)  \leq
\sigma\left(  r_{j_{b}}\right)  \right)  \ \ \ \ \ \ \ \ \ \ \left(
\begin{array}
[c]{c}%
\text{by (\ref{pf.prop.mopis.reduce.equiv3})}\\
\text{(applied to }j_{a}\text{ and }j_{b}\text{ instead of }a\text{ and
}b\text{),}\\
\text{since }a<b\text{ entails }j_{a}<j_{b}\\
\text{(because }j_{1}<j_{2}<\cdots<j_{h}\text{)}%
\end{array}
\right) \\
&  \Longleftrightarrow\ \left(  \tau\left(  a\right)  <\tau\left(  b\right)
\right)  \ \ \ \ \ \ \ \ \ \ \left(  \text{by
(\ref{pf.prop.mopis.reduce.equiv2})}\right) \\
&  \Longleftrightarrow\ \left(  \tau\left(  a\right)  \leq\tau\left(
b\right)  \right)  \ \ \ \ \ \ \ \ \ \ \left(
\begin{array}
[c]{c}%
\text{since }a<b\text{ entails }a\neq b\text{ and}\\
\text{thus }\tau\left(  a\right)  \neq\tau\left(  b\right)  \text{ (since
}\tau\text{ is a permutation)}%
\end{array}
\right)  .
\end{align*}
This entails that $\omega$ is the standardization of the list $\left(
\tau\left(  1\right)  ,\ \tau\left(  2\right)  ,\ \ldots,\ \tau\left(
h\right)  \right)  $. But the latter standardization is just $\tau$ itself
(since $\tau$ is a permutation of $\left[  h\right]  $). Hence, we have proved
that $\omega$ is just $\tau$. In other words, $\tau=\omega$. Combining this
with $\beta=\gamma$, we obtain $\left(  \beta,\tau\right)  =\left(
\gamma,\omega\right)  $. As explained, this proves Claim 1.
\end{proof}

We will need one further combinatorial lemma about deletion:

\begin{statement}
\textit{Claim 2:} Let $\left(  \alpha,\sigma\right)  $ be a weak mopiscotion
with $\alpha=\left(  \alpha_{1},\alpha_{2},\ldots,\alpha_{k}\right)
\in\mathbb{N}^{k}$. Let $q\in\left[  k\right]  $ be such that $\alpha_{q}=0$.
Let $\left(  \beta,\tau\right)  :=\operatorname*{del}\left(  \alpha
,\sigma,q\right)  $ and $r:=\sigma\left(  q\right)  $. Write the $\left(
k-1\right)  $-tuple $\beta$ as $\beta=\left(  \beta_{1},\beta_{2},\ldots
,\beta_{k-1}\right)  $.

Let $\left(  x_{1},x_{2},\ldots,x_{k}\right)  $ be a $k$-tuple of arbitrary
objects. Let $\left(  y_{1},y_{2},\ldots,y_{k-1}\right)  :=\left(  x_{1}%
,x_{2},\ldots,\widehat{x_{r}},\ldots,x_{k}\right)  $. Then,%
\begin{equation}
\left(  x_{\sigma\left(  1\right)  },x_{\sigma\left(  2\right)  }%
,\ldots,x_{\sigma\left(  q-1\right)  }\right)  =\left(  y_{\tau\left(
1\right)  },y_{\tau\left(  2\right)  },\ldots,y_{\tau\left(  q-1\right)
}\right)  \label{pf.prop.mopis.reduce.c2.x1}%
\end{equation}
and%
\begin{equation}
\left(  x_{\sigma\left(  q+1\right)  },x_{\sigma\left(  q+2\right)  }%
,\ldots,x_{\sigma\left(  k\right)  }\right)  =\left(  y_{\tau\left(  q\right)
},y_{\tau\left(  q+1\right)  },\ldots,y_{\tau\left(  k-1\right)  }\right)
\label{pf.prop.mopis.reduce.c2.x2}%
\end{equation}
and%
\begin{equation}
\left(  \alpha_{1},\alpha_{2},\ldots,\alpha_{q-1}\right)  =\left(  \beta
_{1},\beta_{2},\ldots,\beta_{q-1}\right)  \label{pf.prop.mopis.reduce.c2.a1}%
\end{equation}
and%
\begin{equation}
\left(  \alpha_{q+1},\alpha_{q+2},\ldots,\alpha_{k}\right)  =\left(  \beta
_{q},\beta_{q+1},\ldots,\beta_{k-1}\right)  .
\label{pf.prop.mopis.reduce.c2.a2}%
\end{equation}

\end{statement}

\begin{proof}
[Proof of Claim 2.]By assumption, we have $\left(  \beta,\tau\right)
=\operatorname*{del}\left(  \alpha,\sigma,q\right)  $. By the definition of
deletion, this means that $\beta=\left(  \alpha_{1},\alpha_{2},\ldots
,\widehat{\alpha_{q}},\ldots,\alpha_{k}\right)  $, whereas $\tau$ is the
standardization of the list $\left(  \sigma\left(  1\right)  ,\sigma\left(
2\right)  ,\ldots,\widehat{\sigma\left(  q\right)  },\ldots,\sigma\left(
k\right)  \right)  $.

We have
\begin{align*}
\left(  \alpha_{1},\alpha_{2},\ldots,\alpha_{q-1},\alpha_{q+1},\alpha
_{q+2},\ldots,\alpha_{k}\right)   &  =\left(  \alpha_{1},\alpha_{2}%
,\ldots,\widehat{\alpha_{q}},\ldots,\alpha_{k}\right) \\
&  =\beta=\left(  \beta_{1},\beta_{2},\ldots,\beta_{k-1}\right)  .
\end{align*}
This is an equality between $\left(  k-1\right)  $-tuples. Splitting it into
two parts (between the $\left(  q-1\right)  $-th and $q$-th entries), we
obtain%
\[
\left(  \alpha_{1},\alpha_{2},\ldots,\alpha_{q-1}\right)  =\left(  \beta
_{1},\beta_{2},\ldots,\beta_{q-1}\right)
\]
and%
\[
\left(  \alpha_{q+1},\alpha_{q+2},\ldots,\alpha_{k}\right)  =\left(  \beta
_{q},\beta_{q+1},\ldots,\beta_{k-1}\right)  .
\]
Thus, we have proved (\ref{pf.prop.mopis.reduce.c2.a1}) and
(\ref{pf.prop.mopis.reduce.c2.a2}). It remains to prove
(\ref{pf.prop.mopis.reduce.c2.x1}) and (\ref{pf.prop.mopis.reduce.c2.x2}).

Let $\partial_{r}$ be the unique strictly increasing bijection from $\left[
k-1\right]  $ to $\left[  k\right]  \setminus\left\{  r\right\}  $.
Explicitly, the map $\partial_{r}$ sends the numbers $1,2,\ldots,k-1$ to
$1,2,\ldots,\widehat{r},\ldots,k$, respectively. Similarly, let $\partial_{q}$
be the unique strictly increasing bijection from $\left[  k-1\right]  $ to
$\left[  k\right]  \setminus\left\{  q\right\}  $. We have%
\[
\left(  y_{1},y_{2},\ldots,y_{k-1}\right)  =\left(  x_{1},x_{2},\ldots
,\widehat{x_{r}},\ldots,x_{k}\right)  =\left(  x_{\partial_{r}\left(
1\right)  },x_{\partial_{r}\left(  2\right)  },\ldots,x_{\partial_{r}\left(
k-1\right)  }\right)
\]
(since the definition of $\partial_{r}$ yields $\left(  1,2,\ldots
,\widehat{r},\ldots,k\right)  =\left(  \partial_{r}\left(  1\right)
,\ \partial_{r}\left(  2\right)  ,\ \ldots,\ \partial_{r}\left(  k-1\right)
\right)  $). In other words,%
\begin{equation}
y_{i}=x_{\partial_{r}\left(  i\right)  }\ \ \ \ \ \ \ \ \ \ \text{for each
}i\in\left[  k-1\right]  . \label{pf.prop.mopis.reduce.c2.pf.yi=}%
\end{equation}

Recall that $\left(  \alpha,\sigma\right)  $ is a weak mopiscotion and
$\alpha\in\mathbb{N}^{k}$. Hence, $\sigma\in\mathfrak{S}_{k}$.

The map $\sigma$ is a bijection from $\left[  k\right]  $ to $\left[
k\right]  $ (since $\sigma\in\mathfrak{S}_{k}$) and sends $q$ to $r$ (since
$r=\sigma\left(  q\right)  $). Hence, it restricts to a bijection
$\overline{\sigma}$ from $\left[  k\right]  \setminus\left\{  q\right\}  $ to
$\left[  k\right]  \setminus\left\{  r\right\}  $. Consider this bijection
$\overline{\sigma}$. The composition $\partial_{r}^{-1}\circ\overline{\sigma
}\circ\partial_{q}$ is thus a bijection from $\left[  k-1\right]  $ to
$\left[  k-1\right]  $, that is, a permutation of $\left[  k-1\right]  $. It
is easy to see that this composition $\partial_{r}^{-1}\circ\overline{\sigma
}\circ\partial_{q}$ is precisely the permutation $\tau$%
\ \ \ \ \footnote{\textit{Proof.} Let $\psi:=\partial_{r}^{-1}\circ
\overline{\sigma}\circ\partial_{q}$. As we know, this $\psi$ is a permutation
of $\left[  k-1\right]  $. Hence, the standardization of the list $\left(
\psi\left(  1\right)  ,\ \psi\left(  2\right)  ,\ \ldots,\ \psi\left(
k-1\right)  \right)  $ is $\psi$ itself.
\par
Recall that $\tau$ is the standardization of the list
\[
\left(  \sigma\left(  1\right)  ,\sigma\left(  2\right)  ,\ldots
,\widehat{\sigma\left(  q\right)  },\ldots,\sigma\left(  k\right)  \right)
=\left(  \sigma\left(  \partial_{q}\left(  1\right)  \right)  ,\ \sigma\left(
\partial_{q}\left(  2\right)  \right)  ,\ \ldots,\ \sigma\left(  \partial
_{q}\left(  k-1\right)  \right)  \right)
\]
(since the definition of $\partial_{q}$ yields $\left(  1,2,\ldots
,\widehat{q},\ldots,k\right)  =\left(  \partial_{q}\left(  1\right)
,\ \partial_{q}\left(  2\right)  ,\ \ldots,\ \partial_{q}\left(  k-1\right)
\right)  $). By the definition of standardization, this shows that for every
two elements $a$ and $b$ of $\left[  k-1\right]  $ satisfying $a<b$, we have
the equivalence $\left(  \tau\left(  a\right)  <\tau\left(  b\right)  \right)
\ \Longleftrightarrow\ \left(  \sigma\left(  \partial_{q}\left(  a\right)
\right)  \leq\sigma\left(  \partial_{q}\left(  b\right)  \right)  \right)  $.
Thus, for every two elements $a$ and $b$ of $\left[  k-1\right]  $ satisfying
$a<b$, we have the chain of equivalences%
\begin{align*}
\left(  \tau\left(  a\right)  <\tau\left(  b\right)  \right)  \  &
\Longleftrightarrow\ \left(  \sigma\left(  \partial_{q}\left(  a\right)
\right)  \leq\sigma\left(  \partial_{q}\left(  b\right)  \right)  \right) \\
&  \Longleftrightarrow\ \left(  \overline{\sigma}\left(  \partial_{q}\left(
a\right)  \right)  \leq\overline{\sigma}\left(  \partial_{q}\left(  b\right)
\right)  \right) \\
&  \ \ \ \ \ \ \ \ \ \ \ \ \ \ \ \ \ \ \ \ \left(
\begin{array}
[c]{c}%
\text{since }\overline{\sigma}\text{ is a restriction of the map }%
\sigma\text{,}\\
\text{and thus we have }\sigma\left(  \partial_{q}\left(  a\right)  \right)
=\overline{\sigma}\left(  \partial_{q}\left(  a\right)  \right) \\
\text{and }\sigma\left(  \partial_{q}\left(  b\right)  \right)  =\overline
{\sigma}\left(  \partial_{q}\left(  b\right)  \right)
\end{array}
\right) \\
&  \Longleftrightarrow\ \left(  \partial_{r}^{-1}\left(  \overline{\sigma
}\left(  \partial_{q}\left(  a\right)  \right)  \right)  \leq\partial_{r}%
^{-1}\left(  \overline{\sigma}\left(  \partial_{q}\left(  b\right)  \right)
\right)  \right) \\
&  \ \ \ \ \ \ \ \ \ \ \ \ \ \ \ \ \ \ \ \ \left(
\begin{array}
[c]{c}%
\text{since the map }\partial_{r}^{-1}\text{ is strictly increasing}\\
\text{(being the inverse of the strictly increasing}\\
\text{bijection }\partial_{r}\text{), and thus preserves inequalities}%
\end{array}
\right) \\
&  \Longleftrightarrow\ \left(  \underbrace{\left(  \partial_{r}^{-1}%
\circ\overline{\sigma}\circ\partial_{q}\right)  }_{=\psi}\left(  a\right)
\leq\underbrace{\left(  \partial_{r}^{-1}\circ\overline{\sigma}\circ
\partial_{q}\right)  }_{=\psi}\left(  b\right)  \right) \\
&  \Longleftrightarrow\ \left(  \psi\left(  a\right)  \leq\psi\left(
b\right)  \right)  .
\end{align*}
Hence, $\tau$ is the standardization of the list $\left(  \psi\left(
1\right)  ,\ \psi\left(  2\right)  ,\ \ldots,\ \psi\left(  k-1\right)
\right)  $. In other words, $\tau$ is $\psi$ (since the standardization of the
list $\left(  \psi\left(  1\right)  ,\ \psi\left(  2\right)  ,\ \ldots
,\ \psi\left(  k-1\right)  \right)  $ is $\psi$ itself). Thus, $\tau
=\psi=\partial_{r}^{-1}\circ\overline{\sigma}\circ\partial_{q}$. That is, the
permutation $\partial_{r}^{-1}\circ\overline{\sigma}\circ\partial_{q}$ is just
$\tau$.}. In other words, $\partial_{r}^{-1}\circ\overline{\sigma}%
\circ\partial_{q}=\tau$. Hence, $\overline{\sigma}\circ\partial_{q}%
=\partial_{r}\circ\tau$. Therefore, each $i\in\left[  k-1\right]  $ satisfies%
\begin{align}
\sigma\left(  \partial_{q}\left(  i\right)  \right)   &  =\overline{\sigma
}\left(  \partial_{q}\left(  i\right)  \right)  \ \ \ \ \ \ \ \ \ \ \left(
\text{since }\overline{\sigma}\text{ is a restriction of }\sigma\right)
\nonumber\\
&  =\left(  \overline{\sigma}\circ\partial_{q}\right)  \left(  i\right)
=\left(  \partial_{r}\circ\tau\right)  \left(  i\right)
\ \ \ \ \ \ \ \ \ \ \left(  \text{since }\overline{\sigma}\circ\partial
_{q}=\partial_{r}\circ\tau\right) \nonumber\\
&  =\partial_{r}\left(  \tau\left(  i\right)  \right)  .
\label{pf.prop.mopis.reduce.c2.pf.6}%
\end{align}

Now, for each $i\in\left[  k-1\right]  $, we have%
\begin{align*}
y_{\tau\left(  i\right)  }  &  =x_{\partial_{r}\left(  \tau\left(  i\right)
\right)  }\ \ \ \ \ \ \ \ \ \ \left(  \text{by
(\ref{pf.prop.mopis.reduce.c2.pf.yi=}), applied to }\tau\left(  i\right)
\text{ instead of }i\right) \\
&  =x_{\sigma\left(  \partial_{q}\left(  i\right)  \right)  }%
\ \ \ \ \ \ \ \ \ \ \left(  \text{since (\ref{pf.prop.mopis.reduce.c2.pf.6})
yields }\partial_{r}\left(  \tau\left(  i\right)  \right)  =\sigma\left(
\partial_{q}\left(  i\right)  \right)  \right)  .
\end{align*}
In other words,%
\begin{align*}
\left(  y_{\tau\left(  1\right)  },y_{\tau\left(  2\right)  },\ldots
,y_{\tau\left(  k-1\right)  }\right)   &  =\left(  x_{\sigma\left(
\partial_{q}\left(  1\right)  \right)  },x_{\sigma\left(  \partial_{q}\left(
2\right)  \right)  },\ldots,x_{\sigma\left(  \partial_{q}\left(  k-1\right)
\right)  }\right) \\
&  =\left(  x_{\sigma\left(  1\right)  },x_{\sigma\left(  2\right)  }%
,\ldots,\widehat{x_{\sigma\left(  q\right)  }},\ldots,x_{\sigma\left(
k\right)  }\right)
\end{align*}
(since the definition of $\partial_{q}$ yields $\left(  \partial_{q}\left(
1\right)  ,\ \partial_{q}\left(  2\right)  ,\ \ldots,\ \partial_{q}\left(
k-1\right)  \right)  =\left(  1,2,\ldots,\widehat{q},\ldots,k\right)  $). In
other words,%
\[
\left(  x_{\sigma\left(  1\right)  },x_{\sigma\left(  2\right)  }%
,\ldots,\widehat{x_{\sigma\left(  q\right)  }},\ldots,x_{\sigma\left(
k\right)  }\right)  =\left(  y_{\tau\left(  1\right)  },y_{\tau\left(
2\right)  },\ldots,y_{\tau\left(  k-1\right)  }\right)  .
\]
This is an equality between two $\left(  k-1\right)  $-tuples. Splitting it
into two parts (between the $\left(  q-1\right)  $-th and $q$-th entries), we
obtain%
\[
\left(  x_{\sigma\left(  1\right)  },x_{\sigma\left(  2\right)  }%
,\ldots,x_{\sigma\left(  q-1\right)  }\right)  =\left(  y_{\tau\left(
1\right)  },y_{\tau\left(  2\right)  },\ldots,y_{\tau\left(  q-1\right)
}\right)
\]
and%
\[
\left(  x_{\sigma\left(  q+1\right)  },x_{\sigma\left(  q+2\right)  }%
,\ldots,x_{\sigma\left(  k\right)  }\right)  =\left(  y_{\tau\left(  q\right)
},y_{\tau\left(  q+1\right)  },\ldots,y_{\tau\left(  k-1\right)  }\right)  .
\]
Thus, (\ref{pf.prop.mopis.reduce.c2.x1}) and (\ref{pf.prop.mopis.reduce.c2.x2}%
) are proved. This completes the proof of Claim 2.
\end{proof}

Thanks to Claim 1, we can easily reduce Proposition \ref{prop.mopis.reduce} to
the following claim about deletion:

\begin{statement}
\textit{Claim 3:} Let $\left(  \alpha,\sigma\right)  $ be a weak mopiscotion
with $\alpha=\left(  \alpha_{1},\alpha_{2},\ldots,\alpha_{k}\right)
\in\mathbb{N}^{k}$. Let $q\in\left[  k\right]  $ be such that $\alpha_{q}=0$.
Let $\left(  \beta,\tau\right)  :=\operatorname*{del}\left(  \alpha
,\sigma,q\right)  $. Then, $p_{\alpha,\sigma}=p_{\beta,\tau}$.
\end{statement}

\begin{proof}
[Proof of Claim 3.]Let $r:=\sigma\left(  q\right)  $. Write the $\left(
k-1\right)  $-tuple $\beta$ as $\beta=\left(  \beta_{1},\beta_{2},\ldots
,\beta_{k-1}\right)  $.

Define the $\mathbf{k}$-linear maps%
\begin{align*}
f:H^{\otimes k}  &  \rightarrow H^{\otimes\left(  k-1\right)  },\\
x_{1}\otimes x_{2}\otimes\cdots\otimes x_{k}  &  \mapsto\epsilon\left(
x_{r}\right)  \cdot x_{1}\otimes x_{2}\otimes\cdots\otimes\widehat{x_{r}%
}\otimes\cdots\otimes x_{k}%
\end{align*}
and
\begin{align*}
g:H^{\otimes\left(  k-1\right)  }  &  \rightarrow H^{\otimes k},\\
x_{1}\otimes x_{2}\otimes\cdots\otimes x_{k-1}  &  \mapsto x_{1}\otimes
x_{2}\otimes\cdots\otimes x_{q-1}\otimes1_{H}\otimes x_{q}\otimes
x_{q+1}\otimes\cdots\otimes x_{k-1}.
\end{align*}
It is well-known and easy to prove that $f\circ\Delta^{\left[  k\right]
}=\Delta^{\left[  k-1\right]  }$ and $m^{\left[  k\right]  }\circ g=m^{\left[
k-1\right]  }$. (Indeed, the latter equality says that a product of the form
$x_{1}x_{2}\cdots x_{q-1}1_{H}x_{q}x_{q+1}\cdots x_{k-1}$ can be simplified by
removing the factor $1_{H}$; the former equality is just the dual result with
$q$ replaced by $r$.) Now, if we can show that%
\begin{equation}
P_{\alpha}\circ\sigma^{-1}=g\circ P_{\beta}\circ\tau^{-1}\circ f
\label{pf.prop.mopis.reduce.c2-mid}%
\end{equation}
as maps from $H^{\otimes k}$ to $H^{\otimes k}$, then we will be able to
conclude that
\begin{align*}
p_{\alpha,\sigma}  &  =m^{\left[  k\right]  }\circ\underbrace{P_{\alpha}%
\circ\sigma^{-1}}_{=g\circ P_{\beta}\circ\tau^{-1}\circ f}\circ\,\Delta
^{\left[  k\right]  }\ \ \ \ \ \ \ \ \ \ \left(  \text{by the definition of
}p_{\alpha,\sigma}\right) \\
&  =\underbrace{m^{\left[  k\right]  }\circ g}_{=m^{\left[  k-1\right]  }%
}\circ\,P_{\beta}\circ\tau^{-1}\circ\underbrace{f\circ\Delta^{\left[
k\right]  }}_{=\Delta^{\left[  k-1\right]  }}\\
&  =m^{\left[  k-1\right]  }\circ P_{\beta}\circ\tau^{-1}\circ\Delta^{\left[
k-1\right]  }=p_{\beta,\tau}\ \ \ \ \ \ \ \ \ \ \left(  \text{by the
definition of }p_{\beta,\tau}\right)  ,
\end{align*}
and so Claim 3 will be proved.

Hence, it remains to prove (\ref{pf.prop.mopis.reduce.c2-mid}). Let us compare
what the maps $P_{\alpha}\circ\sigma^{-1}$ and $g\circ P_{\beta}\circ\tau
^{-1}\circ f$ do to a pure tensor $x_{1}\otimes x_{2}\otimes\cdots\otimes
x_{k}\in H^{\otimes k}$. For this purpose, we fix $x_{1},x_{2},\ldots,x_{k}\in
H$, and we define the $\left(  k-1\right)  $-tuple%
\[
\left(  y_{1},y_{2},\ldots,y_{k-1}\right)  :=\left(  x_{1},x_{2}%
,\ldots,\widehat{x_{r}},\ldots,x_{k}\right)  \in H^{k-1}.
\]
Then,
\[
y_{1}\otimes y_{2}\otimes\cdots\otimes y_{k-1}=x_{1}\otimes x_{2}\otimes
\cdots\otimes\widehat{x_{r}}\otimes\cdots\otimes x_{k}.
\]
We note that $p_{0}\left(  h\right)  =\epsilon\left(  h\right)  1_{H}$ for
each $h\in H$ (since $H$ is connected). Thus, in particular, $p_{0}\left(
x_{r}\right)  =\epsilon\left(  x_{r}\right)  1_{H}$. Now
\begin{align*}
&  \left(  P_{\alpha}\circ\sigma^{-1}\right)  \left(  x_{1}\otimes
x_{2}\otimes\cdots\otimes x_{k}\right) \\
&  =P_{\alpha}\left(  \sigma^{-1}\left(  x_{1}\otimes x_{2}\otimes
\cdots\otimes x_{k}\right)  \right) \\
&  =P_{\alpha}\left(  x_{\sigma\left(  1\right)  }\otimes x_{\sigma\left(
2\right)  }\otimes\cdots\otimes x_{\sigma\left(  k\right)  }\right)
\ \ \ \ \ \ \ \ \ \ \left(
\begin{array}
[c]{c}%
\text{by the definition of the}\\
\text{action of }\sigma^{-1}\text{ on }H^{\otimes k}%
\end{array}
\right) \\
&  =p_{\alpha_{1}}\left(  x_{\sigma\left(  1\right)  }\right)  \otimes
p_{\alpha_{2}}\left(  x_{\sigma\left(  2\right)  }\right)  \otimes
\cdots\otimes p_{\alpha_{k}}\left(  x_{\sigma\left(  k\right)  }\right)
\ \ \ \ \ \ \ \ \ \ \left(  \text{by the definition of }P_{\alpha}\right) \\
&  =\underbrace{p_{\alpha_{1}}\left(  x_{\sigma\left(  1\right)  }\right)
\otimes p_{\alpha_{2}}\left(  x_{\sigma\left(  2\right)  }\right)
\otimes\cdots\otimes p_{\alpha_{q-1}}\left(  x_{\sigma\left(  q-1\right)
}\right)  }_{\substack{=p_{\beta_{1}}\left(  y_{\tau\left(  1\right)
}\right)  \otimes p_{\beta_{2}}\left(  y_{\tau\left(  2\right)  }\right)
\otimes\cdots\otimes p_{\beta_{q-1}}\left(  y_{\tau\left(  q-1\right)
}\right)  \\\text{(by (\ref{pf.prop.mopis.reduce.c2.a1}) and
(\ref{pf.prop.mopis.reduce.c2.x1}))}}}\\
&  \ \ \ \ \ \ \ \ \ \ \otimes\underbrace{p_{\alpha_{q}}\left(  x_{\sigma
\left(  q\right)  }\right)  }_{\substack{=p_{0}\left(  x_{r}\right)
\\\text{(since }\alpha_{q}=0\\\text{and }\sigma\left(  q\right)  =r\text{)}%
}}\otimes\underbrace{p_{\alpha_{q+1}}\left(  x_{\sigma\left(  q+1\right)
}\right)  \otimes p_{\alpha_{q+2}}\left(  x_{\sigma\left(  q+2\right)
}\right)  \otimes\cdots\otimes p_{\alpha_{k}}\left(  x_{\sigma\left(
k\right)  }\right)  }_{\substack{=p_{\beta_{q}}\left(  y_{\tau\left(
q\right)  }\right)  \otimes p_{\beta_{q+1}}\left(  y_{\tau\left(  q+1\right)
}\right)  \otimes\cdots\otimes p_{\beta_{k-1}}\left(  y_{\tau\left(
k-1\right)  }\right)  \\\text{(by (\ref{pf.prop.mopis.reduce.c2.a2}) and
(\ref{pf.prop.mopis.reduce.c2.x2}))}}}\\
&  =p_{\beta_{1}}\left(  y_{\tau\left(  1\right)  }\right)  \otimes
p_{\beta_{2}}\left(  y_{\tau\left(  2\right)  }\right)  \otimes\cdots\otimes
p_{\beta_{q-1}}\left(  y_{\tau\left(  q-1\right)  }\right) \\
&  \ \ \ \ \ \ \ \ \ \ \otimes\underbrace{p_{0}\left(  x_{r}\right)
}_{=\epsilon\left(  x_{r}\right)  1_{H}}\otimes\,p_{\beta_{q}}\left(
y_{\tau\left(  q\right)  }\right)  \otimes p_{\beta_{q+1}}\left(
y_{\tau\left(  q+1\right)  }\right)  \otimes\cdots\otimes p_{\beta_{k-1}%
}\left(  y_{\tau\left(  k-1\right)  }\right) \\
&  =\epsilon\left(  x_{r}\right)  \cdot p_{\beta_{1}}\left(  y_{\tau\left(
1\right)  }\right)  \otimes p_{\beta_{2}}\left(  y_{\tau\left(  2\right)
}\right)  \otimes\cdots\otimes p_{\beta_{q-1}}\left(  y_{\tau\left(
q-1\right)  }\right) \\
&  \ \ \ \ \ \ \ \ \ \ \otimes1_{H}\otimes p_{\beta_{q}}\left(  y_{\tau\left(
q\right)  }\right)  \otimes p_{\beta_{q+1}}\left(  y_{\tau\left(  q+1\right)
}\right)  \otimes\cdots\otimes p_{\beta_{k-1}}\left(  y_{\tau\left(
k-1\right)  }\right)  .
\end{align*}
Comparing this with
\begin{align*}
&  \left(  g\circ P_{\beta}\circ\tau^{-1}\circ f\right)  \left(  x_{1}\otimes
x_{2}\otimes\cdots\otimes x_{k}\right) \\
&  =\left(  g\circ P_{\beta}\circ\tau^{-1}\right)  \left(  f\left(
x_{1}\otimes x_{2}\otimes\cdots\otimes x_{k}\right)  \right) \\
&  =\left(  g\circ P_{\beta}\circ\tau^{-1}\right)  \left(  \epsilon\left(
x_{r}\right)  \cdot\underbrace{x_{1}\otimes x_{2}\otimes\cdots\otimes
\widehat{x_{r}}\otimes\cdots\otimes x_{k}}_{=y_{1}\otimes y_{2}\otimes
\cdots\otimes y_{k-1}}\right) \\
&  \ \ \ \ \ \ \ \ \ \ \ \ \ \ \ \ \ \ \ \ \left(  \text{by the definition of
}f\right) \\
&  =\left(  g\circ P_{\beta}\circ\tau^{-1}\right)  \left(  \epsilon\left(
x_{r}\right)  \cdot y_{1}\otimes y_{2}\otimes\cdots\otimes y_{k-1}\right) \\
&  =\epsilon\left(  x_{r}\right)  \cdot\left(  g\circ P_{\beta}\circ\tau
^{-1}\right)  \left(  y_{1}\otimes y_{2}\otimes\cdots\otimes y_{k-1}\right)
\ \ \ \ \ \ \ \ \ \ \left(  \text{by }\mathbf{k}\text{-linearity of }g\circ
P_{\beta}\circ\tau^{-1}\right) \\
&  =\epsilon\left(  x_{r}\right)  \cdot g\left(  P_{\beta}\left(  \tau
^{-1}\left(  y_{1}\otimes y_{2}\otimes\cdots\otimes y_{k-1}\right)  \right)
\right) \\
&  =\epsilon\left(  x_{r}\right)  \cdot g\left(  P_{\beta}\left(
y_{\tau\left(  1\right)  }\otimes y_{\tau\left(  2\right)  }\otimes
\cdots\otimes y_{\tau\left(  k-1\right)  }\right)  \right) \\
&  \ \ \ \ \ \ \ \ \ \ \ \ \ \ \ \ \ \ \ \ \left(  \text{by the definition of
the action of }\tau^{-1}\text{ on }H^{\otimes\left(  k-1\right)  }\right) \\
&  =\epsilon\left(  x_{r}\right)  \cdot g\left(  p_{\beta_{1}}\left(
y_{\tau\left(  1\right)  }\right)  \otimes p_{\beta_{2}}\left(  y_{\tau\left(
2\right)  }\right)  \otimes\cdots\otimes p_{\beta_{k-1}}\left(  y_{\tau\left(
k-1\right)  }\right)  \right) \\
&  \ \ \ \ \ \ \ \ \ \ \ \ \ \ \ \ \ \ \ \ \left(  \text{by the definition of
}P_{\beta}\right) \\
&  =\epsilon\left(  x_{r}\right)  \cdot p_{\beta_{1}}\left(  y_{\tau\left(
1\right)  }\right)  \otimes p_{\beta_{2}}\left(  y_{\tau\left(  2\right)
}\right)  \otimes\cdots\otimes p_{\beta_{q-1}}\left(  y_{\tau\left(
q-1\right)  }\right) \\
&  \ \ \ \ \ \ \ \ \ \ \otimes1_{H}\otimes p_{\beta_{q}}\left(  y_{\tau\left(
q\right)  }\right)  \otimes p_{\beta_{q+1}}\left(  y_{\tau\left(  q+1\right)
}\right)  \otimes\cdots\otimes p_{\beta_{k-1}}\left(  y_{\tau\left(
k-1\right)  }\right) \\
&  \ \ \ \ \ \ \ \ \ \ \ \ \ \ \ \ \ \ \ \ \left(  \text{by the definition of
}g\right)  ,
\end{align*}
we obtain
\[
\left(  P_{\alpha}\circ\sigma^{-1}\right)  \left(  x_{1}\otimes x_{2}%
\otimes\cdots\otimes x_{k}\right)  =\left(  g\circ P_{\beta}\circ\tau
^{-1}\circ f\right)  \left(  x_{1}\otimes x_{2}\otimes\cdots\otimes
x_{k}\right)  .
\]
Since we have proved this for all $x_{1},x_{2},\ldots,x_{k}\in H$, we thus
conclude that the two $\mathbf{k}$-linear maps $P_{\alpha}\circ\sigma^{-1}$
and $g\circ P_{\beta}\circ\tau^{-1}\circ f$ agree on all pure tensors. Thus,
these two maps are identical. Hence, (\ref{pf.prop.mopis.reduce.c2-mid}) is
proved, and as we saw above, Claim 3 follows.
\end{proof}

Now, let us derive Proposition \ref{prop.mopis.reduce} from Claim 3.

Indeed, we define the \emph{length} $\ell\left(  \alpha\right)  $ of a weak
composition $\alpha$ to be the unique integer $k\in\mathbb{N}$ such that
$\alpha\in\mathbb{N}^{k}$.

We shall now prove Proposition \ref{prop.mopis.reduce} by induction on the
length $\ell\left(  \alpha\right)  $.

The \textit{base case} ($\ell\left(  \alpha\right)  =0$) is obvious (since in
this case, $\alpha$ is the empty tuple $\left(  {}\right)  $, and thus we have
$\left(  \beta,\tau\right)  =\left(  \alpha,\sigma\right)  $).

For the \textit{induction step}, we fix a positive integer $k$, and assume (as
the induction hypothesis) that Proposition \ref{prop.mopis.reduce} holds
whenever $\ell\left(  \alpha\right)  =k-1$. We must now prove that Proposition
\ref{prop.mopis.reduce} also holds whenever $\ell\left(  \alpha\right)  =k$.
For this purpose, we fix a weak mopiscotion $\left(  \alpha,\sigma\right)  $
with $\ell\left(  \alpha\right)  =k$, and we set $\left(  \beta,\tau\right)
=\operatorname*{red}\left(  \alpha,\sigma\right)  $. We must prove that
$p_{\alpha,\sigma}=p_{\beta,\tau}$.

If all entries of $\alpha$ are nonzero, then this is obvious (since in this
case we have $\left(  \beta,\tau\right)  =\operatorname*{red}\left(
\alpha,\sigma\right)  =\left(  \alpha,\sigma\right)  $). Thus, we WLOG assume
that not all entries of $\alpha$ are nonzero. In other words, some entry of
$\alpha$ is $0$.

Write $\alpha$ as $\alpha=\left(  \alpha_{1},\alpha_{2},\ldots,\alpha
_{k}\right)  $ (since $\alpha\in\mathbb{N}^{k}$). Then, there exists some
$q\in\left[  k\right]  $ such that $\alpha_{q}=0$ (since some entry of
$\alpha$ is $0$). Consider this $q$. Let $\left(  \alpha^{\prime}%
,\sigma^{\prime}\right)  :=\operatorname*{del}\left(  \alpha,\sigma,q\right)
$. Claim 3 (applied to $\left(  \alpha^{\prime},\sigma^{\prime}\right)  $
instead of $\left(  \beta,\tau\right)  $) then yields $p_{\alpha,\sigma
}=p_{\alpha^{\prime},\sigma^{\prime}}$. But Claim 1 yields
$\operatorname*{red}\left(  \alpha,\sigma\right)  =\operatorname*{red}\left(
\alpha^{\prime},\sigma^{\prime}\right)  $, so that $\left(  \beta,\tau\right)
=\operatorname*{red}\left(  \alpha,\sigma\right)  =\operatorname*{red}\left(
\alpha^{\prime},\sigma^{\prime}\right)  $. However, $\ell\left(
\alpha^{\prime}\right)  =k-1$ (by the definition of deletion), and thus our
induction hypothesis shows that Proposition \ref{prop.mopis.reduce} holds for
the weak mopiscotion $\left(  \alpha^{\prime},\sigma^{\prime}\right)  $
instead of $\left(  \alpha,\sigma\right)  $. Thus, we have $p_{\alpha^{\prime
},\sigma^{\prime}}=p_{\beta,\tau}$ (since $\left(  \beta,\tau\right)
=\operatorname*{red}\left(  \alpha^{\prime},\sigma^{\prime}\right)  $). Thus,
$p_{\alpha,\sigma}=p_{\alpha^{\prime},\sigma^{\prime}}=p_{\beta,\tau}$. Hence,
Proposition \ref{prop.mopis.reduce} is proved for our $\left(  \alpha
,\sigma\right)  $. This completes the induction step, and thus Proposition
\ref{prop.mopis.reduce} is proved.
\end{proof}
\end{fineprint}

\end{document}